\documentclass[11pt,english]{smfart} 
\usepackage{bm}
\usepackage[english,francais]{babel}
\usepackage{amscd}
\usepackage{amssymb,url,xspace,smfthm}
\usepackage{paralist}
\usepackage{datetime} 
\usepackage{colortbl}
\usepackage{color}
\usepackage[dvipsnames]{xcolor}
 \usepackage[colorlinks,linkcolor=BrickRed,
 citecolor=Green,
 urlcolor=Cerulean,hypertexnames=true,linktocpage]{hyperref} 
\usepackage{mathtools}
\usepackage{euscript}
\usepackage[all]{xy}
\usepackage{graphicx}

\usepackage[T1]{fontenc}
\usepackage{newtxtext}
\usepackage{newtxmath}
\usepackage{textcomp}
\usepackage[text={150mm,251mm},centering, marginparwidth=75pt]{geometry} 
\usepackage{comment}  
\usepackage{cite}  
\usepackage{thmtools}
\usepackage{enumitem}
\usepackage{letltxmacro} 
\usepackage{nameref}
\usepackage{cleveref}
\usepackage{calc}
\usepackage{interval}
\usepackage{manfnt}
\usepackage{tikz-cd}  
\usepackage{marginnote}
\usepackage{euscript}

\usepackage{smfhyperref}
\usepackage[hyperpageref]{backref}
\renewcommand{\backref}[1]{$\uparrow$~#1}

\usepackage{faktor}

\theoremstyle{plain}
\newtheorem{thmx}{Theorem}
\renewcommand{\thethmx}{\Alph{thmx}} 
\newtheorem{thm}{Theorem}[section]  
\newtheorem{lem}[thm]{Lemma}
\newtheorem{claim}[thm]{Claim}

\newtheorem{proposition}[thm]{Proposition}
\newtheorem{cor}[thm]{Corollary}

\newtheorem{corx}[thmx]{Corollary} 
\newtheorem{conjecture}[thm]{Conjecture}
\theoremstyle{definition}
\newtheorem{dfn}[thm]{Definition}

\theoremstyle{remark}
\newtheorem{rem}[thm]{Remark} 
\newtheorem{example}[thm]{Example}

\numberwithin{equation}{subsection}  

\theoremstyle{plain}
\newlist{thmlist}{enumerate}{1}
\setlist[thmlist]{wide = 0pt, labelwidth = 2em, labelsep*=0em, itemindent = 0pt, leftmargin = \dimexpr\labelwidth + \labelsep\relax, noitemsep,topsep = 1ex, font=\normalfont, label=(\roman*), ref=\thethm.(\roman{thmlisti})}

\addtotheorempostheadhook[thm]{\crefalias{thmlisti}{thm}}

\addtotheorempostheadhook[assumpsion]{\crefalias{thmlisti}{assumption}}

\addtotheorempostheadhook[cor]{\crefalias{thmlisti}{cor}}

\addtotheorempostheadhook[proposition]{\crefalias{thmlisti}{proposition}}

\addtotheorempostheadhook[dfn]{\crefalias{thmlisti}{dfn}}

\addtotheorempostheadhook[lem]{\crefalias{thmlisti}{lem}}
\addtotheorempostheadhook[main]{\crefalias{thmlisti}{main}}

\addtotheorempostheadhook[rem]{\crefalias{thmlisti}{rem}}

\newlist{thmenum}{enumerate}{1} 
\setlist[thmenum]{wide = 0pt, labelwidth = 2em, labelsep*=0em, itemindent = 0pt, leftmargin = \dimexpr\labelwidth + \labelsep\relax, noitemsep,topsep = 1ex, font=\normalfont, label=(\roman*), ref=\thethmx.(\roman{thmenumi})}
\crefalias{thmenumi}{thmx} 

\newlist{corlist}{enumerate}{1} 
\setlist[corlist]{wide = 0pt, labelwidth = 2em, labelsep*=0em, itemindent = 0pt, leftmargin = \dimexpr\labelwidth + \labelsep\relax, noitemsep,topsep = 1ex, font=\normalfont, label=(\roman*), ref=\thecorx.(\roman{corlisti})}
\crefalias{corlisti}{corx} 

\crefname{lem}{Lemma}{Lemmas} 
\crefname{conjecture}{Conjecture}{Conjectures}
\crefname{thm}{Theorem}{Theorems}
\crefname{proposition}{Proposition}{Propositions}
\crefname{dfn}{Definition}{Definitions}
\crefname{rem}{Remark}{Remarks}
\crefname{cor}{Corollary}{Corollaries}
\crefname{corx}{Corollary}{Corollaries}
\crefname{problem}{Problem}{Problems}
\crefname{thmx}{Theorem}{Theorems}
\crefname{claim}{Claim}{Claims}
\crefname{assumption}{Assumption}{Assumptions}
\crefname{main}{Main Theorem}{Main Theorems}

\def\ep{\varepsilon}
\def\O{\mathscr{O}}
\def\D{\mathscr{D}} 

\def\D{D}

\def\Res{{\rm Res}}
\def\Sp{{\rm Sp}}

\newcommand{\cS}{\mathcal{S}}
 
\newcommand{\cY}{\mathcal{Y}}  
\newcommand{\da}{\rightharpoonup}

\makeatletter
\newcommand*{\rom}[1]{\expandafter\@slowromancap\romannumeral #1@}
\makeatother

\makeatletter     
\newcommand{\crefnames}[3]{%
	\@for\next:=#1\do{%
		\expandafter\crefname\expandafter{\next}{#2}{#3}%
	}%
}
\makeatother

\crefnames{part,chapter,section}{\S}{\S\S}

\newcommand{\cA}{\mathcal A}
\newcommand{\cC}{\mathcal C}
\newcommand{\cD}{\mathcal D}

\newcommand{\cI}{\mathcal I}

\newcommand{\cO}{\mathcal O}

\newcommand{\bA}{\mathbb{A}}

\newcommand{\bC}{\mathbb{C}}
\newcommand{\bD}{\mathbb{D}}

\newcommand{\bF}{\mathbb{F}}

\newcommand{\bP}{\mathbb{P}}
\newcommand{\bQ}{\mathbb{Q}}
\newcommand{\bR}{\mathbb{R}}

\newcommand{\bZ}{\mathbb{Z}}

\newcommand{\xsp}{X^{\! \rm sp}}

\newcommand{\kg}{\mathfrak{g}}

\newcommand{\pN}[1]{\bar{N}_{#1}^{\partial}}

  \def\spec{\textrm{Spec}\,}

\def\End{{\rm \small  End}}

\def\Sym{{\text{Sym}}}

\def\sp{{\rm sp}}

\newcommand{\Spab}{\mathrm{Sp}_{\mathrm{sab}}}
\newcommand{\Sph}{\mathrm{Sp}_{\mathrm{h}}}
\newcommand{\Spp}{\mathrm{Sp}_{\mathrm{p}}}
\newcommand{\Spalg}{\mathrm{Sp}_{\mathrm{alg}}}

\def\Im{{\rm Im}\,}

\newcommand{\Hom}{{\rm Hom}}

\newcommand{\diae}{{}^\diamond\! E}

\newcommand{\ord}{{\rm ord}\,}

\newcommand{\ram}{{\rm ram}\,}

\title[Hyperbolicity and  representations of $\pi_1$]{Hyperbolicity and    fundamental groups of complex quasi-projective varieties} 

\date{\today} 

\author[B. Cadorel]{Benoit Cadorel} 
\email{benoit.cadorel@univ-lorraine.fr}
\address{Institut \'Elie Cartan de Lorraine, Universit\'e de Lorraine, F-54000 Nancy,
	France.}
\urladdr{http://www.normalesup.org/~bcadorel/} 
\author[Y. Deng]{Ya Deng}

\email{ya.deng@math.cnrs.fr}
\address{CNRS, Institut \'Elie Cartan de Lorraine, Universit\'e de Lorraine, F-54000 Nancy,
	France.}
\urladdr{https://ydeng.perso.math.cnrs.fr}

\author[K. Yamanoi]{Katsutoshi Yamanoi}

\email{yamanoi@math.sci.osaka-u.ac.jp}
\address{Department of Mathematics, Graduate School of Science, Osaka University, Toyonaka,  Osaka 560-0043, Japan} 
\urladdr{https://sites.google.com/site/yamanoimath/}


\begin{document} 
	\begin{abstract} 
		This paper   investigates the relationship between the hyperbolicity of complex quasi-projective varieties $X$ and the (topological) fundamental group $\pi_1(X)$ in the presence of a linear representation $\varrho: \pi_1(X) \to {\rm GL}_N(\mathbb{C})$. We present our main  results in three parts.
		
		Firstly, we show that if $\varrho$ is a big representation and the Zariski closure of $\varrho(\pi_1(X))$ in ${\rm GL}_N(\mathbb{C})$ is a semisimple algebraic group, then for any Galois conjugate variety $X^\sigma:=X\times_\sigma\mathbb{C}$ where $\sigma\in {\rm Aut}(\mathbb{C}/\mathbb{Q})$ is an automorphism of $\mathbb{C}$,  there exists a proper Zariski closed subset $Z \subsetneqq X^\sigma(\bC)$ such that any closed irreducible subvariety $V$ of $X^\sigma(\bC)$ not contained in $Z$ is of log general type, and any holomorphic map from the punctured disk $\bD^*$ to $X^\sigma(\bC)$ with image not contained in $Z$ does not have an essential singularity at the origin. In particular, all entire curves in $X^\sigma(\bC)$ lie on $Z$. We provide examples to illustrate the optimality of this condition.
		
		Secondly, assuming that $\varrho$ is big and reductive, we prove the generalized Green-Griffiths-Lang conjecture for all Galois conjugate varieties $X^\sigma$. Furthermore, if $\varrho$ is   large  in addition to being big,  we show that the special subsets of $X^\sigma$ that capture the non-hyperbolicity locus of $X^\sigma$ from different perspectives are equal, and this subset is a proper Zariski closed subset if and only if $X$ is of log general type. We also obtain a structure theorem for these algebraic varieties.
		
		Lastly, we prove that if $X$ is a special quasi-projective manifold in the sense of Campana or $h$-special, then $\varrho(\pi_1(X))$ is virtually nilpotent. We provides examples to demonstrate that this result is sharp and thus revise Campana's abelianity conjecture for smooth quasi-projective varieties. 
		
		To prove these theorems, we develop new features in non-abelian Hodge theory, geometric group theory, and Nevanlinna theory. Along the way, we also prove the generalized Green-Griffiths-Lang conjecture for quasi-projective varieties $X$ admitting a morphism $a:X \to A$ with $\dim X=\dim a(X)$ where $A$ is a semi-abelian variety, and a reduction theorem for reductive representations $\pi_1(X) \to {\rm GL}_N(K)$ where $K$ is a  non-Archimedean local field $K$.
	\end{abstract}  
	
	\maketitle
	\tableofcontents

	 \section{Main results} 
	 \subsection{Hyperbolicity and fundamental groups}
The study of hyperbolicity properties of complex algebraic varieties with big (topological) fundamental groups is a fundamental topic in algebraic geometry. In particular, for projective varieties whose fundamental groups admit a linear semisimple quotient, this topic has been extensively studied, with \cite{Zuo96,Yam10,CCE15} confirming the expectation that such varieties possess hyperbolicity properties.  A linear algebraic group $G$ over a field $K$  is called semisimple if it has no non-trivial connected normal solvable algebraic subgroups defined over the algebraic closure of $K$, and has positive dimension.   Specifically,  Campana-Claudon-Eyssidieux \cite[Theorem 1]{CCE15} proved that a smooth complex projective variety $X$ with a Zariski dense representation $\varrho:\pi_1(X)\to G(\mathbb{C})$, where $G$ is a semisimple linear algebraic group over $\mathbb{C}$, is of general type when $\varrho$ is big. In addition, the third author \cite[Proposition 2.1]{Yam10} proved that $X$ does not admit Zariski dense entire curves $f:\bC\to X$.

A representation $\varrho:\pi_1(X)\to G(\bC)$ is said to be \emph{big}, or \emph{generically large} in \cite{Kol95}, if for any closed irreducible subvariety $Z\subset X$ containing a \emph{very general} point of $X$, $\varrho({\rm Im}[\pi_1(Z^{\rm norm})\to \pi_1(X)])$ is infinite, where $Z^{\rm norm}$ denotes the normalization of $Z$ (see \cref{def:big representation}). 
It is worth noting that a stronger notion of largeness exists, where $\varrho$ is called \emph{large}   if $\varrho({\rm Im}[\pi_1(Z^{\rm norm})\to \pi_1(X)])$ is infinite for any closed subvariety $Z$ of $X$. In this paper, we aim to generalize and strengthen the above theorems by Campana-Claudon-Eyssidieux and the third author to complex quasi-projective varieties.
	 \begin{thmx}[=\cref{thm:20220819}]\label{main2}
	 	Let $X$ be a complex   quasi-projective normal variety and let $G$ be a semisimple algebraic group over $\bC$. If  $\varrho:\pi_1(X)\to G(\bC)$ is a big  and Zariski dense representation, then for any automorphism $\sigma\in {\rm Aut}(\bC/\bQ)$,  there is a proper Zariski closed subset $Z\subsetneqq X^\sigma$ where $X^\sigma$ is the Galois conjugate variety of $X$ under $\sigma$ such that
	 	\begin{thmenum}
	 		\item \label{main:log general type}    any closed  subvariety   of $X^\sigma$ not contained in $Z$ is  of log general type. In particular, $X^\sigma$ is of log general type.
\item \label{main:pseudo Picard}   Any holomorphic map $f:\bD^*\to X^\sigma$ from the punctured disk $\bD^*$ to $X^\sigma$ with $f(\bD^*)\not\subseteq Z$ extends to  a holomorphic map from the disk $\bD$ to a projective compactification $\overline{X^\sigma}$ of $X^\sigma$ (i.e. $X^\sigma$ is \emph{pseudo Picard hyperbolic}).   In particular,  all entire curves in $X^\sigma$ lie on $Z$ (i.e. $X^\sigma$ is \emph{pseudo Brody hyperbolic}).
	 	\end{thmenum} 
	 \end{thmx}
It is worth noting that every pseudo Picard hyperbolic variety is pseudo Brody hyperbolic (cf. \cref{lem:inclusion}).
We also mention that the two conditions for the representation $\varrho$ in \cref{main2} are essential to conclude the two statements in \cref{main2}, as discussed in \cref{rem:sharp}. Both statements in \cref{main2} are new even in the case where $X$ is projective, and the proof of \cref{main2} involves several new techniques (see \cref{rem:previous} for more details).

It is noteworthy that the condition of bigness for the representations $\varrho$ in \cref{main2} is not particularly restrictive, unlike the requirement for a large representation.  In fact, in \cref{lem:kollar} we demonstrate that any linear representation of $\pi_1(X)$ can be factored through a big representation after taking a finite \'etale cover. This result, combined with \cref{main2}, yields a factorization theorem for linear representations of $\pi_1(X)$.
 	 \begin{corx}[$\subsetneqq$\cref{cor:202304071}]\label{main}
 	 	Let $X$ be a complex    quasi-projective  normal variety and let $G$ be a semisimple algebraic group over $\bC$. 
If  $\varrho:\pi_1(X)\to G(\bC)$ is a Zariski dense representation, then there  exist
a  finite \'etale cover \(\nu:\widehat{X}\to X\), a birational and proper morphism \(\mu:\widehat{X}'\to \widehat{X}\), a dominant morphism $f:\widehat{X}'\to Y$  with connected general fibers, and a   big   and Zariski dense representation \(\tau : \pi_{1}(Y) \to G(\bC)\)  such that
 	\begin{itemize}
 	\item   \(f^{\ast} \tau = (\nu\circ\mu)^{\ast}\varrho\). 
 			\item There is a proper Zariski closed subset $Z\subsetneqq Y$ such that any closed   subvariety of $Y$ not contained in $Z$ is  of log general type.  
 			\item  $Y$ is pseudo Picard hyperbolic, and in particular pseudo Brody hyperbolic.  
 	\end{itemize}  
 In particular,   $X$ is not weakly special and does not contain Zariski-dense entire curves. 
 \end{corx} 
Note that by Campana \cite{Cam11}, a quasi-projective variety $X$ is \emph{weakly special} if for any finite \'etale cover $\widehat{X}\to X$ and any proper birational modification $\widehat{X}'\to \widehat{X}$, there exists no   dominant morphism $\widehat{X}'\rightarrow Y$  with $Y$ a positive-dimensional quasi-projective normal variety of log general type.

\cref{main} generalizes the previous work by Mok \cite{Mok92}, Corlette-Simpson \cite{CS08}, and Campana-Claudon-Eyssidieux \cite{CCE15}, in which they proved similar factorisation results.

\subsection{On the generalized Green-Griffiths-Lang conjecture}\label{subsec:20230427}
Building upon \cref{main2}, we further investigate the generalized Green-Griffiths-Lang conjecture (cf. \Cref{conj:GGL})and its relation to the non-hyperbolicity locus of a smooth quasi-projective variety $X$, under the weaker assumption that $\pi_1(X)$ admits a big and reductive representation. 
Specifically, we introduce four special subsets of $X$ that measure the non-hyperbolicity locus from different perspectives, as defined in \cref{def:special2}. Our main result, given in \cref{main:GGL}, establishes the equivalence of several properties of the conjugate variety $X^\sigma$ under the assumption that $\varrho:\pi_1(X)\to {\rm GL}_N(\bC)$ is a big and reductive representation, and for any automorphism $\sigma\in {\rm Aut}(\bC/\bQ)$. Additionally, we provide a further result regarding the special subsets, as stated in \cref{main:special}.
\begin{dfn}[Special subsets] \label{def:special2}
	Let $X$ be a smooth quasi-projective variety.
	\begin{thmlist}
		\item $\Spab(X) := \overline{\bigcup_{f}f(A_0)}^{\mathrm{Zar}}$, where $f$ ranges over all non-constant rational maps $f:A\dashrightarrow X$ from all semi-abelian varieties $A$ to $X$ such that $f$ is regular on a Zariski open subset $A_0\subset A$ whose complement $A\backslash A_0$ has codimension at least two;
		\item $\Sph(X) := \overline{\bigcup_{f}f(\mathbb{C})}^{\mathrm{Zar}}$, where $f$ ranges over all non-constant holomorphic maps from $\mathbb{C}$ to $X$;
		\item $\Spalg(X) := \overline{\bigcup_{V} V}^{\mathrm{Zar}}$, where $V$ ranges over all positive-dimensional closed subvarieties of $X$ which are not of log general type;
		\item $\Spp(X) := \overline{\bigcup_{f}f(\bD^*)}^{\mathrm{Zar}}$, where $f$ ranges over all holomorphic maps from the punctured disk $\bD^*$ to $X$ with essential singularity at the origin, i.e., $f$ has no holomorphic extension $\bar{f}:\mathbb D\to\overline{X}$ to a projective compactification $\overline{X}$.
	\end{thmlist}
	\end{dfn}

The first two sets $\Spab(X)$ and $\Sph(X)$ are introduced by Lang for the compact case.
He made the following two conjectures (cf. \cite[I, 3.5]{Lan97} and \cite[VIII, Conjecture 1.3]{Lan97}):
\begin{itemize}
\item
$\Spab(X)\subsetneqq X$ if and only if $X$ is of general type.
\item
$\Spab(X)=\Sph(X)$.
\end{itemize}
The first assertion implicitly include the following third conjecture:
\begin{itemize}
\item
$\Spab(X)=\Spalg(X)$.
\end{itemize}

The original two conjectures imply the famous strong Green-Griffiths conjecture that varieties of (log) general type are pseudo Brody hyperbolic.
Here we note that, by definition, $X$ is pseudo Brody hyperbolic if and only if $\Sph(X)\subsetneqq X$.
Similarly, $X$ is pseudo Picard hyperbolic if and only if $\Spp(X)\subsetneqq X$.

\begin{thmx}[=\cref{thm:GGL}]\label{main:GGL}
	Let $X$ be a complex  smooth  quasi-projective  variety admitting a big and reductive representation  $\varrho:\pi_1(X)\to {\rm GL}_N(\bC)$. 
Then for any  automorphism $\sigma\in {\rm Aut}(\bC/\bQ)$, the following properties are equivalent:
	\begin{enumerate}[font=\normalfont, label=(\alph*)] 
		\item $X^\sigma$ is of log general type. 
		\item   $\Spp(X^\sigma)\subsetneqq X^\sigma$.
		\item  $\Sph(X^\sigma)\subsetneqq X^\sigma$.
		\item  $\Spalg(X^\sigma)\subsetneqq X^\sigma$.
\item  $\Spab(X^\sigma)\subsetneqq X^\sigma$.
	\end{enumerate} 
\end{thmx} 

We note that the implication $(a)\implies (c)$ in \cref{main:GGL} establishes the strong Green-Griffiths conjecture for $X^\sigma$ for all automorphism $\sigma\in {\rm Aut}(\bC/\bQ)$, provided a big and reductive representation  $\varrho:\pi_1(X)\to {\rm GL}_N(\bC)$ exists.

As for the second and third conjectures of Lang, we obtain the following theorem under the stronger assumption when $\pi_1(X)$ admits a \emph{large} and reductive representation.

\begin{thmx}[=\cref{thm:special}]\label{main:special}
	Let $X$ be a smooth quasi-projective   variety admitting a large and reductive representation $\varrho:\pi_1(X)\to {\rm GL}_N(\bC)$. Then   for any  automorphism $\sigma\in {\rm Aut}(\bC/\bQ)$,  
	\begin{thmenum}
		\item   the four special subsets defined in \cref{def:special2} are the same, i.e., $$\Spalg(X^\sigma)=\Spab(X^\sigma)=\Sph(X^\sigma)=\Spp(X^\sigma).$$  
		\item These special subsets are conjugate under  automorphism   $\sigma$,  i.e., $$\Sp_{\bullet}(X^\sigma)=\Sp_{\bullet}(X)^\sigma,$$
		where $\Sp_{\bullet}$  denotes any of $\Spalg$, $\Spab$, $\Sph$ or $ \Spp$.
	\end{thmenum}
\end{thmx}

\begin{rem}
In the compact case, we should mention that several results have been recently announced on the topics of this paper.
In \cite{Sun22}, Sun claimed the pseudo Borel hyperbolicity (which is weaker than our pseudo Picard hyperbolicity) of projective manifolds $X$ with  $\pi_1(X)$ admitting Zariski dense and big representations into semisimple algebraic groups. 
Our \cref{main:special} is motivated by the recent work of Brunebarbe \cite{Bru22}.
He claimed that if $X$ is a projective variety supporting a large complex local system, then $\Sp_{\rm alg}(X)=\Spab(X)=\Sph(X)$ and these are proper subsets of $X$ iff $X$ is of general type.
We will discuss some (seemingly unfixed) issues in their proofs later.
See \cref{rem:20230417} and also \cref{rem:small ramification} for the case of \cite{Sun22}.
\end{rem}

\subsection{Fundamental groups of  special and $h$-special varieties}
We refer the readers to \cref{sec:spechspecial} for the definition of \emph{special}  and \emph{$h$-special} variaties. Campana's abelianity conjecture \cite[Conjecture 13.10.(1)]{Cam11} predicts that a smooth quasi-projective variety $X$ that is  {special} has a virtually abelian fundamental group.  
When a special variety $X$ is projective, it is known that all linear quotients of $\pi_1(X)$ are virtually abelian (cf. \cite[Theorem 7.8]{Cam04}).
The same conclusion is valid for any smooth projective variety $X$ which contains Zariski dense entire curves (cf. \cite[Theorem 1.1]{Yam10}).
It is natural to expect similar results for smooth quasi-projective varieties. However, we shall see an example of a quasi-projective surface that is special, that contains Zariski dense entire curves, whose fundamental group is linear and nilpotent but not virtually abelian (cf. \Cref{example}). This provides a counterexample to Campana's conjecture in the general case.

 Our third main result in this paper  is  as follows.  
  \begin{thmx}[=\cref{thm:VN}]\label{main5}
	Let $X$ be a   special  or $h$-special smooth quasi-projective variety.   Let  $\varrho:\pi_1(X)\to {\rm GL}_N(\bC)$ be a  linear representation.   Then  $\varrho(\pi_1(X)) $ is \emph{virtually nilpotent}.  
\end{thmx}
For the definition of $h$-special above, we refer the readers to \cref{defn:20230407}; it includes smooth quasi-projective varieties $X$ such that
\begin{itemize}
	\item $X$ admits Zariski dense entire curves, or 
	\item generic two points of $X$ are connected by the chain of entire curves.
\end{itemize}  

Note that the above theorem is sharp, as shown by \Cref{example}. Its proof is based on \cref{main2} and the following theorem:
\begin{thmx}[=\cref{thm:202210123}]\label{main:geomety group}
Let $X$ be an $h$-special or special  quasi-projective manifold.
Let $G$ be a connected, solvable algebraic group defined over $\mathbb C$.
Assume that there exists a Zariski dense representation $\varphi:\pi_1(X)\to G(\bC)$.
Then $G$ is nilpotent.
In particular, $\varphi(\pi_1(X))$ is nilpotent.
\end{thmx} 
When $X$ is a compact K\"ahler manifold, \cref{main:geomety group} was proved by Campana \cite{Cam01} and Delzant \cite{Del10} using different methods. The proof of \cref{main:geomety group} is based on results in Nevanlinna theory, Deligne's mixed Hodge theory and \cref{main}.
\subsection{A structure theorem}
If we replace the semi-simple algebraic group $G$ in \cref{main2} by a reductive one,   we  obtain the following structure theorem. 
\begin{thmx}[=\cref{thm:structure,thm:char}]
\label{thm:20230510}
	Let $X$ be a   smooth quasi-projective variety and let $\varrho:\pi_1(X)\to {\rm GL}_N(\bC)$ be a reductive and big representation. Then
	\begin{thmenum}
	\item the logarithmic Kodaira dimension $\bar{\kappa}(X)\geq 0$. Moreover, if $\bar{\kappa}(X)=0$, then   $\pi_1(X)$ is virtually abelian.
	\item There is a proper Zariski closed subset $Z$ of $X$ such that each non-constant morphism $\bA^1\to X$ has image in $Z$.
	\item After replacing $X$ by a finite \'etale cover and a birational modification, there are a semiabelian variety $A$, a quasi-projective manifold $V$, and a birational morphism $a:X\to V$ such that we have the following commutative diagram
		\begin{equation*}
			\begin{tikzcd}
				X \arrow[rr,    "a"] \arrow[dr,  "j"] & & V \arrow[ld, "h"]\\
				& J(X)&
			\end{tikzcd}
		\end{equation*}
	where $j$ is the logarithmic Iitaka fibration and $h:V\to J(X)$ is a  locally trivial fibration with fibers isomorphic to $A$.   Moreover, for a   general fiber $F$ of $j$, $a|_{F}:F\to  A$ is proper in codimension one. 
		\item \label{char abelian} If $X$ is   special or $h$-special, then $\pi_1(X)$ is virtually abelian, and after replacing $X$ by a finite \'etale cover, its quasi-Albanese morphism $\alpha:X\to \cA$ is birational.  
 		\end{thmenum} 
\end{thmx} 
When $X$ is projective, the third author proved \cref{char abelian} in \cite{Yam10} without assuming that the representation $\varrho$ is reductive. However, it is worth noting that when $X$ is only quasi-projective,   \cref{char abelian} might fail if $\varrho$ is not reductive (cf. \cref{rem:sharp abelian}).   
 \subsection{Results in non-abelian Hodge  and Nevanlinna theories}\label{sec:techniques}
 We believe that some of the new techniques developed in the proof of \cref{main2} are of significant interest in their own right. One such technique is a reduction theorem for Zariski dense representations $\varrho: \pi_1(X)\to G(K)$, where $G$ is a reductive algebraic group defined over a non-Archimedean local field $K$.
 \begin{thmx}[=\cref{thm:KZreduction}] \label{main3}
 	Let $X$ be a complex   quasi-projective normal variety, and let $\varrho:\pi_1(X)\to {\rm GL}_N(K)$ be a reductive representation where $K$ is non-archimedean local field.  Then there exists a quasi-projective normal variety $S_\varrho$ and a dominant morphism $s_\varrho:X\to S_\varrho$ with connected general fibers, such that  for any connected Zariski closed  subset $T$ of $X$, the following properties are equivalent:
 	\begin{enumerate}[label={\rm (\alph*)}]
 		\item \label{item bounded} the image $\rho({\rm Im}[\pi_1(T)\to \pi_1(X)])$ is a bounded subgroup of $G(K)$.
 		\item \label{item normalization} For every irreducible component $T_o$ of $T$, the image $\rho({\rm Im}[\pi_1(T_o^{\rm norm})\to \pi_1(X)])$ is a bounded subgroup of $G(K)$.
 		\item \label{item contraction}The image $s_\varrho(T)$ is a point.
 	\end{enumerate} 
 \end{thmx}
It is worth noting that if $X$ is projective, the equivalence between \Cref{item bounded} and \Cref{item contraction} has been established by Katzarkov \cite{Kat97}, Eyssidieux \cite[Proposition 1.4.7]{Eys04}, and Zuo \cite{Zuo96}. However, it seems that the implication of \Cref{item normalization} to \Cref{item contraction} is new even in the projective setting.
One of the building blocks of the proof of \cref{thm:KZreduction} is based on previous results by Brotbek, Daskalopoulos, Mese, and the second author \cite{BDDM} on the existence of harmonic mappings to Bruhat-Tits buildings (an extension of Gromov-Schoen's theorem to quasi-projective cases) and the construction of logarithmic symmetric differential forms via these harmonic mappings.

The following theorem is a crucial component in the proof of \cref{main2}.
 \begin{thmx}[=\cref{thm:log general type,thm:main33}] \label{main6}
	Let $X$ be a quasi-projective manifold.  Let $G$ be an almost simple algebraic group defined over a non-archimedean local field $K$.  Suppose that $\varrho:\pi_1(X)\to G(K)$ is a big and unbounded Zariski dense representation. Then:
	\begin{thmenum}
		\item  \label{main:lgt}$\Spalg(X)\subsetneqq X$. 
	\item  \label{main:PPH} $X$ is pseudo Picard hyperbolic.
\end{thmenum}
\end{thmx}
A significant building block in the proof of \cref{main6} is \cref{thm2nd} on the \emph{generalized Green-Griffiths-Lang conjecture} (cf. \cref{conj:GGL}). We present here a simplified version for clarity. 
\begin{thmx}[=\cref{cor:20221102}] \label{main:second}
	Let $X$ be a  quasi-projective variety. 
	Assume that there is a  morphism $a:X\to A$ such that $\dim X=\dim a(X)$ where $A$ is a semi-abelian variety (e.g., when $X$ has maximal quasi-Albanese dimension).  
Then the following properties are equivalent:  
	\begin{enumerate}[wide = 0pt,  noitemsep,  font=\normalfont, label=(\alph*)] 
			\item \label{being general type} $X$ is of log general type. 
			\item  \label{pseudo Picard} $\Spp(X)\subsetneqq X$.
			\item  $\Sph(X)\subsetneqq X$.
 	\item \label{strong LGT} $\Spalg(X)\subsetneqq X$.
\item 
$\Spab(X)\subsetneqq X$. 
\end{enumerate}
\end{thmx}
It is worth mentioning that showing \ref{being general type} implies \ref{strong LGT} requires using the implication of \ref{being general type} to \ref{pseudo Picard}, which follows from \cref{thm2nd}. The proof of \cref{thm2nd} is heavily based on Nevanlinna theory, and the entire \cref{sec:GGL} is dedicated to proving it.

 
 \medspace
 
 \subsection{Structure of the paper and further developments}
 \begin{figure}
 	\centering
 	\begin{tikzpicture}[>=stealth, node distance=1.7cm, every node/.style={rectangle, draw, align=center}]	
 		\node (A) {\cref{main6}};
 			\node (Z) [above left of = A] {\cref{main:second}};
 				\node (X) [above right of = A]  {\cref{main3}}; 
 				 		\node (W) [below of = A, yshift=0.6cm] {\cref{main2}};
 		\node (B) [below left of=W] {\cref{main}};
 		\node (C) [below right of=W] {\cref{main:GGL}};
 		\node (D) [below of= C] {\cref{main:special}}; 
 		\node (Y) [left of = B, xshift=-0.6cm]{\cref{main:geomety group}};
 		\node (E) [below of= C, , xshift=-2.6cm] {\cref{main5}};
 		\node (F) [right of=C, xshift=0.6cm] {\cref{thm:20230510}};

 			\path[->] (Z) edge (A);
 			 			\path[->] (X) edge (A);
 			 						\path[->] (A) edge (W);
 		\path[->] (W) edge (B);
 		\path[->] (W) edge (C);
 		\path[->] (C) edge (D);
 		\path[->] (Y) edge (E);
 		 		\path[->] (B) edge (E);
 		\path[->] (W) edge (F); 
 	\end{tikzpicture}
 	\caption{Relationships between Main Theorems}
 	\label{fig:relationships}
 \end{figure}
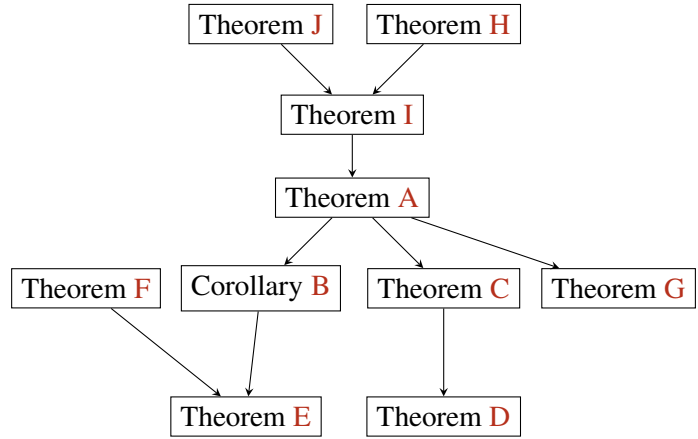 
 
This paper is long and comprehensive as there are several delicate issues in previous related works in the projective setting that require detailed explanation (see \cref{rem:20230417,rem:previous}). Given the complexity of the subject matter, we have devoted almost 100 pages to carefully and thoroughly exploring the issues at hand. However, we have taken care to ensure that the paper is as self-contained as possible. Since we believe that various techniques developed in this paper will have more applications in the future, our aim is to make our arguments accessible to readers who may not be familiar with the topic, all  without assuming prior knowledge of previous works on the compact cases.   Since the paper presents several results  from different perspectives, for the readers' convenience, we list in Figure \ref{fig:relationships} the  relationships between main theorems.  

\medspace
 
Let us conclude this section by highlighting some recent developments following the paper.  
  It is worth mentioning that the techniques developed in this paper have  substantial applications in these more recent works:
\begin{itemize}
	\item in \cite[Theorem A]{DY23}, the second and third authors constructed the Shafarevich morphism for complex reductive representations of fundamental groups of complex quasi-projective normal varieties.
	\item In \cite[Theorem A]{DY23b}, the second and third authors  constructed the Shafarevich morphism for linear representations in positive characteristic of fundamental groups of complex quasi-projective varieties. Furthermore, in \cite{DY23b} they established analogous results to those in \cref{main2,main:GGL,main:special,main5} for quasi-projective varieties whose fundamental groups admit linear (not necessarily reductive) representations in positive characteristic.
	\item   It is worth emphasizing the   significance of pseudo Picard hyperbolicity explored in this paper compared to   pseudo Brody one.   In \cite[Theorem F]{DY23b}, we prove a conjecture by Claudon-H\"oring-Koll\'ar in the linear case: let $X$ be a smooth complex projective variety whose universal covering  is quasi-projective. If there is a faithful representation $\varrho:\pi_1(X)\to {\rm GL}_N(K)$ with $K$ any field, then up to some finite \'etale cover, the Albanese map of $X$ is locally isotrivial with fibers simply connected.  The pseudo Picard hyperbolicity proved in \cref{main2} played a significant role in this result. 
\end{itemize}

\subsection*{Convention and notation.}In this paper, we use the following conventions and notations:
\begin{itemize}[noitemsep]
	\item Quasi-projective varieties and their closed subvarieties are usually assumed to be irreducible unless otherwise stated, while Zariski closed subsets might be reducible.
	\item Fundamental groups are always referred to as topological fundamental groups.
	\item If $X$ is a complex space, its normalization is denoted by $X^{\mathrm{norm}}$.
	\item An \emph{algebraic fiber space} is a dominant morphism between quasi-projective normal varieties whose general fibers are connected, but not necessarily proper or surjective.
	\item A birational map $f:X\to Y$ between quasi-projective normal varieties is \emph{proper in codimension one} if there is an open set $Y^\circ\subset Y$ with $Y\backslash Y^\circ$ of codimension at least two such that $f$ is proper over $Y^\circ$.
\end{itemize}

\subsection*{Acknowledgment.}
We would like to thank Michel Brion, Patrick Brosnan, Fr\'ed\'eric Campana,  Philippe Eyssidieux, Hisashi Kasuya, Chikako Mese,  Gopal Prasad, Guy Rousseau  and Carlos Simpson for helpful discussions.  We also thank 
Ariyan Javanpeykar
 and Yohan Brunebarbe for their valuable remarks on the paper.   B.C and Y.D acknowledge  the support by the ANR grant Karmapolis (ANR-21-CE40-0010). 
K.Y acknowledges the support by JSPS Grant-in-Aid for Scientific Research (C) 22K03286.

\section{Preliminaries and notation}\label{sec:pre}

\subsection{Logarithmic forms.} Let \((\overline{X}, D)\) be a smooth log-pair, {\em i.e.\ } the data of a smooth projective manifold \(\overline{X}\), and of a simple normal crossing divisor \(D\) on it. We will sometimes denote by \(T_{1}(\overline{X}, D) := H^{0}(\overline{X}, \Omega_{\overline{X}}(\log D))\) the space of logarithmic forms. Similarly, if \(X\) is a smooth quasi-projective manifold, we will write \(T_{1}(X) := T_{1}(\overline{X}, D)\), where \((\overline{X}, D)\) is any smooth log-pair compactifying \(X\). Note that \(T_{1}(X)\) depends only on \(X\), but not on the choice of \(\overline{X}\).

\subsection{Semi-abelian varieties and quasi-Albanese morphisms.} In this subsection we collect some results on quasi-Albanese morphisms of quasi-projective manifolds and semi-abelian varieties. We refer the readers to \cite{NW13,Fuj15} for further details. 

 A commutative complex algebraic  group $\cA$ is called a \emph{semi-abelian variety} if there is a short exact sequence of complex algebraic groups  
$$
1 \rightarrow H \rightarrow \cA \rightarrow \cA_0 \rightarrow 1
$$
where $\cA_0$ is an abelian variety and $H \cong (\bC^*)^\ell$.

Let \(X\) be a complex quasi-projective manifold. The quasi-Albanese variety of \(X\) is the semi-abelian variety
\[
	\cA_{X}=T_{1}(X)^*/\varrho(H_1(X,\bZ)) 
\]
where $\varrho(\gamma):=(\eta\mapsto \int_{\gamma}\eta)$. If we fix a base point \(\ast \in X\), and we denote by \(e \in \mathcal{A}_{X}\) the unit element, then there is a natural morphism of pointed varieties \((X, \ast) \to (\cA_{X}, e)\) given by integration along paths, exactly as in the projective setting. Also, any pointed morphism \((X, \ast) \to (B, e)\) to a semi-abelian variety factors uniquely through \(\mathcal{A}_{X}\) via a morphism of semi-abelian varieties \(\mathcal{A}_{X} \to B\).

Note that the image of the quasi-Albanese morphism is only a constructible set, which is neither closed nor open in general.
\medskip

The following lemma will permit to prove factorisation results for quasi-Albanese morphisms. 
\begin{lem} \label{lem:critpointalb}
	Let \(X\) be a quasi-projective manifold, and let \(q : X \to \cA_{X}\) be its quasi-Albanese morphism. Let \(S \subset T_{1}(X)\) be a set of logarithmic forms on \(X\), and let \(B \subset \mathcal{A}_{X}\) be the largest semi-abelian subvariety on which all \(\eta \in S\) vanish, using the natural identification \(T_{1}(\cA_{X}) \cong T_{1}(X)\). Let
	\[
		r : X \to \mathcal{A}_{X}/B
	\]
	be the quotient map. Then, for any morphism \(f : Y \to X\) from a   quasi-projective variety \(Y\), \(r(f(Y))\) is a point if and only for any \(\eta \in S\), one has \(f^{\ast}(\eta) = 0\).
\end{lem}
\begin{proof}
	After taking a resolution of singularities of $Y$, we assume that $Y$ is smooth. Choose base points on \(X\) and \(Y\) so that one has a diagram of pointed spaces
	\[
		\begin{tikzcd}
			X \arrow[r, "q"] & \cA_{X} \arrow[r, "p"] & \cA_{X}/B \\
			Y \arrow[u, "f"] \arrow[r, "u"] & \cA_{Y} \arrow[u, "g"]
		\end{tikzcd},
	\]
	where the semi-abelian varieties are pointed by their unit elements, and \(g, p\) are morphisms of algebraic groups.

	The previous diagram shows that \(p \circ q\) sends \(f(Y)\) to a point if and only if \(g(\cA_{Y}) \subset B\). By definition of \(B\), this is true if and only if \(g^{\ast}(\eta) = 0\) for any \(\eta \in S \subset T_{1}(\cA_{X}) = T_{1}(X)\). Since \(u^{\ast} g^{\ast} = f^{\ast} q^{\ast}\), and this \(q^{\ast}\) realizes the identification \(T_{1}(\cA_{X}) \cong T_{1}(X)\), this is equivalent to the second condition of the lemma.
\end{proof}
The above morphism $r:X\to \cA_X/B$ is called the \emph{partial quasi-Albanese morphism} induced by $S$.  

\begin{rem}
	 \cref{lem:critpointalb} enables us to define quasi-Albanese morphisms  for quasi-projective normal varieties $X$.     Let $Y$ be a desingularization of $X$. Let $S\subset T_1(Y)$ be the set of log one forms which vanish  on each fiber of $Y\to X$.  Then for the partial quasi-Albanese morphism $Y\to \cA$ induced by $S$, by \cref{lem:critpointalb} each fiber of $Y\to X$ is  contracted to one point.  Hence $Y\to \cA$ factors through $X\to \cA$. Note that this construction does not depend on the desingularization of $X$. Such $X\to \cA$ is called the quasi-Albanese morphism of $X$.
\end{rem}

We will also need two standard results permitting to evaluate the Kodaira dimension of subvarities of semi-abelian varieties (see \cite[Propositions 5.6.21 \& 5.6.22]{NW13}).  
\begin{proposition} \label{prop:Koddimabb}
	Let \(\cA\) be a semi-abelian variety. 
	\begin{enumerate}[label=(\alph*)]
		\item Let \(X \subset \cA\) be a closed subvariety. Then \(\overline{\kappa}(X) \geq 0\) with equality if and only if \(X\) is a translate of a semi-abelian subvariety.
		\item Let \(Z \subset \cA\) be a Zariski closed subset. Then \(\overline{\kappa}(\cA - Z) \geq 0\) with equality if and only if \(Z\) has no component of codimension \(1\).  \qed
	\end{enumerate}
\end{proposition}

\subsection{Covers and Galois morphisms.} \label{sec:covGal}

If \(X\) is a complex variety, its fundamental group will be denoted by \(\pi_{1}(X)\), and its universal cover will usually be denoted by \(\pi : \widetilde{X} \to X\).

\begin{dfn}[Galois morphism]
	 	A finite map $\gamma: X \rightarrow Y$ of varieties is called \emph{Galois  with group $G$} if there exists a finite group $G \subset \operatorname{Aut}(X)$ such that $\gamma$ is isomorphic to the quotient map. 
	 \end{dfn}

	 If \(Y\) is a complex manifold and \(p : X \to Y\) is an étale cover, then there exists a Galois cover \(p' : X' \to Y\) factoring through \(p\) such that any other such Galois cover factors through \(p'\). In the case where \(X\) is connected, the Galois cover \(p' : X' \to Y\) is described as follows.
Let $F$ be the fiber of $p$ over $y\in Y$.
Then $\pi_1(Y,y)$ acts on $F$ transitively by lifting a closed path through $y$ starting with a point $x\in F$.
This defines a group morphism \(\varrho : \pi_{1}(Y,y) \to \mathfrak{S}(F)\), where $\mathfrak{S}$ stands for the permutation.
Now the covering $p':X'\to Y$ corresponds to the finite $\pi_1(Y,y)$-set $\mathrm{Im}(\varrho)$.
Since this set is a coset space of the normal subgroup $\mathrm{Ker}(\varrho)$, $p':X'\to Y$ is a Galois cover.
Note that we have a surjective map $\mathrm{Im}(\varrho)\to F$ of finite $\pi_1(Y,y)$-sets defined by $\sigma\mapsto \sigma(x_0)$, where $x_0\in F$ is a fixed point.
Hence \(p' \) factoring through \(p\).

			 Now, if \(p : X \to Y\) is a finite morphism of normal quasi-projective varieties, one can form the normalization of \(X\) in the Galois closure of \(\mathbb{C}(X)/\mathbb{C}(Y)\). This is a finite Galois morphism \(p' : X' \to Y\), factoring through \(X\); over the locus \(Y_{\circ}\) where \(p\) is \'etale on $p^{-1}(Y_{\circ})$, it identifies with the Galois closure defined previously. In particular, \(p'\) is \'etale over \(Y_{\circ}\).

\subsection{Tame and pure imaginary harmonic bundles}\label{sec:tame}

\begin{dfn}[Higgs bundle]\label{Higgs}
	A \emph{Higgs bundle} on a complex manifold $X$  is a pair $(E,\theta)$ where $E$ is a holomorphic vector bundle  and $\theta:E\to E\otimes \Omega^1_X$ is a holomorphic one form with value in $\End(E)$, called the \emph{Higgs field},  satisfying $\theta\wedge\theta=0$.
\end{dfn}
Let  $(E,\theta)$ be a Higgs bundle  over a complex manifold $X$.  
Suppose that $h$ is a smooth hermitian metric of $E$.  Denote by $\nabla_h$  the Chern connection with respect to $h$, and by $\theta^\dagger_h$  the adjoint of $\theta$ with respect to $h$.  We write $\theta^\dagger$ for $\theta^\dagger_h$ for short.   The metric $h$ is  \emph{harmonic} if the connection  $\bD_1:=\nabla_h+\theta+\theta^\dagger$
is flat, i.e., if  $\bD_1^2=0$. 
\begin{dfn}[Harmonic bundle] A \emph{harmonic bundle} on  $X$ is
	a Higgs bundle $(E,\theta)$ endowed with a  harmonic metric $h$.
\end{dfn}

Let $\overline{X}$ be a compact complex manifold, $D=\sum_{i=1}^{\ell}D_i$ be a   simple normal crossing divisor of $\overline{X}$ and $X=\overline{X}\backslash D$ be the complement of $D$.
Let $(E,\theta,h)$ be a  harmonic bundle on $X$.	 Let $p$ be any point of $D$, and $(U;z_1,\ldots,z_n)$  be a coordinate system centered at  $p$ such that $D\cap U=(z_1\cdots z_\ell=0)$ . On $U$, we have the description:
\begin{align}\label{eq:local}
	\theta=\sum_{j=1}^{\ell}f_jd\log z_j+\sum_{k=\ell+1}^{n}f_kdz_k.
\end{align}

\begin{dfn}[Tameness]\label{def:tameness}
	Let $t$ be a formal variable. For any $j=1,\dots, \ell$, the characteristic polynomial $\det (f_j-t)\in \mathcal{O}(U\backslash D)[t]$, is a polynomial in $t$   whose coefficients are holomorphic functions. 
	If those functions can be extended
	to the holomorphic functions over $U$ 
for all $j$, then the harmonic bundle is called \emph{tame} at $p$.  A harmonic bundle is \emph{tame} if it is tame at each point.
\end{dfn} 

For a tame harmonic bundle  $(E,\theta,h)$ over $X$,  we prolong $E$ over $\overline{X}$ by a sheaf of $\cO_{\overline{X}}$-module $\diae_h$  as follows:  for any open set $U$ of $X$, 
\begin{align*} 
	\diae_h(U):=\{\sigma\in\Gamma(U\backslash D,E|_{U\backslash D})\mid |\sigma|_h\lesssim {\prod_{i=1}^{\ell}|z_i|^{-\ep}}\  \ \mbox{for all}\ \ep>0\}. 
\end{align*}  
In \cite{Moc07} Mochizuki proved that $\diae_h$ is locally free and that  $\theta$ extends to a morphism
$$
\diae_h\to \diae_h\otimes \Omega_{\overline{X}}(\log D),
$$
which we still denote by $\theta$.

\begin{dfn}[Pure imaginary, nilpotent residue]\label{def:nilpotency}
	Let $(E, h,\theta)$ be a tame harmonic bundle on $\overline{X}\backslash D$. The residue $\Res_{D_i}\theta$ induces an endomorphism of $\diae_h|_{D_i}$. Its characteristic polynomial has constant coefficients, and thus the eigenvalues are constant. We say that $(E,\theta,h)$ is \emph{pure imaginary} (resp. has  \emph{nilpotent residue}) if for each component $D_i$ of $D$, the   eigenvalues of $\Res_{D_i}\theta$ are all pure imaginary (resp. all zero).
\end{dfn}   
One can verify that \cref{def:nilpotency} does not depend on the compactification $\overline{X}$ of $\overline{X}\backslash D$. 

The following  theorem by Mochizuki will be used in \cref{sec:rigid}.
\begin{thm}[{Mochizuki \cite[Theorem 25.21]{Moc07b}}]\label{moc}
	Let $X$ be a smooth quasi-projective variety and let $(E,\theta,h)$ be a tame pure imaginary harmonic bundle on $X$. Then the flat bundle $(E, \nabla_h+\theta+\theta^\dagger)$ is semi-simple. Conversely, if $(V,\nabla)$ is a semisimple flat bundle on $X$, then there is a tame pure imaginary harmonic bundle $(E,\theta,h)$ on $X$ so that $(E, \nabla_h+\theta+\theta^\dagger)\simeq (V,\nabla)$. Moreover, when $\nabla$ is \emph{simple}, then any such harmonic metric $h$ is unique up to positive multiplication. 
\end{thm} 

 \subsection{Special varieties and $h$-special varieties} \label{sec:spechspecial}
Special varieties are introduced by Campana \cite{Cam04,Cam11} in his  remarkable program of classification of geometric orbifolds.  In this subsection we briefly recall definitions and properties of special varieties, and we refer the readers to \cite{Cam11} for more details. 
 
Let $f : X \to Y$ be a dominant morphism between quasi-projective smooth varieties with connected general fibers, that admits a compactification \(\overline{f} : \overline{X} \to \overline{Y}\), where \(\overline{X} = X \sqcup D\) (resp. \(\overline{Y} = Y \sqcup G\)) is a compactification with simple normal crossing boundary divisor.  We will consider $(\overline{X}| D)$ as a geometric orbifold  defined in \cite[D\'efinition 2.1]{Cam11} and $\bar{f}:(\overline{X}| D)\to \overline{Y}$ as  an orbifold morphism  defined in \cite[D\'efinition 2.4]{Cam11} .  In \cite[\S 2.1]{Cam11}, Campana defined the \emph{multiplicity divisor} $\Delta(\bar{f}, D)\subset \overline{Y}$ of $\overline{f}$, for which the \emph{orbifold base} of $\overline{f}$ is the pair $(\overline{Y}|\Delta(\bar{f}, D))$. Note that with this definition, one has \(\Delta(\bar{f}, D) \geq G\). 

 The \emph{Kodaira dimension} of $\overline{f}:(\overline{X}| D)\to \overline{Y}$, denoted by $\kappa(\overline{f}, D)$, is defined to be 
\begin{align} \label{eq:Kodaira}
\kappa(\overline{f}, D):=\inf \{\kappa(\overline{Y}'|\Delta(\bar{f}', D')) \},
\end{align} 
where $\overline{f}':\overline{X}'\to \overline{Y}'$ ranges over all \emph{birational models} of $\overline{f}$; i.e., $\overline{f}':\overline{X}'\to \overline{Y}'$ is an algebraic fiber space between smooth projective varieties   such that we have the following commutative diagram
\begin{equation*}
	\begin{tikzcd}
		\overline{X}' \arrow[r, "u"]\arrow[d, "\overline{f}'"] & \overline{X}\arrow[d, "\overline{f}"] \\
		\overline{Y}' \arrow[r, "v"] & \overline{Y} 
 	\end{tikzcd}
\end{equation*}
where $u$ and $v$ are  birational morphisms, and \(D' = u^{-1}(D)\) is a simple normal crossing divisor on \(\overline{X}'\). We note that $\kappa(\bar{f},D)$ is thus a birational invariant and does not depend on the choice of  the compactification  of $X$. Hence we  define the Kodaira dimension of $f:X\to Y$ to be
\begin{align}\label{eq:Kodaira2}
	 \kappa(Y,f):=\kappa(\bar{f},D).
\end{align} 

 \begin{dfn}[Campana's specialness]\label{def:special}
Let $X$ be a   quasi-projective variety. 
	 We say that $X$ is \emph{weakly special} if for any finite \'etale cover $\widehat{X}\to X$ and any proper birational modification $\widehat{X}'\to \widehat{X}$, there exists no dominant morphism  $\widehat{X}'\to Y$  with  connected general fibers such that $Y$ is a positive-dimensional quasi-projective normal variety of log general type. The variety $X$ is \emph{special} if for  any proper birational modification $\widehat{X}\to X$ and any  dominant morphism $f:\widehat{X}\to Y$ over a quasi-projective normal variety $Y$ with connected general fibers,  we have $\kappa(Y,f)<\dim  {Y}$.
 \end{dfn}

Campana defined $X$ to be \emph{$H$-special} if $X$ has vanishing Kobayashi pseudo-distance.
Motivated by \cite[Definition 9.1]{Cam04}, we introduce the following definition.

\begin{dfn}[$h$-special]\label{defn:20230407}
Let $X$ be a smooth quasi-projective variety.
We define the equivalence relation $x\sim y$ of two points $x,y\in X$ iff there exists a sequence of holomorphic maps $f_1,\ldots,f_l:\mathbb C\to X$ such that letting $Z_i\subset X$ to be the Zariski closure of $f_i(\mathbb C)$, we have 
$$x\in Z_1, Z_1\cap Z_2\not=\emptyset, \ldots, Z_{l-1}\cap Z_l\not=\emptyset, y\in Z_l.$$
We set $R=\{ (x,y)\in X\times X; x\sim y\}$.
We define $X$ to be \emph{hyperbolically special} ($h$-special for short) iff $R\subset X\times X$ is Zariski dense.
\end{dfn}

By definition, rationally connected projective varieties are $h$-special without refering a theorem of Campana and Winkelmann \cite{CW16}, who proved that all rationally connected projective varieties contain Zariski dense entire curves.  

\begin{lem}\label{lem:20230406}
If a smooth quasi-projective variety $X$ admits a Zariski dense entire curve $f:\mathbb C\to X$, then $X$ is $h$-special.
\end{lem}

\begin{proof}
Since $f:\mathbb C\to X$ is Zariski dense, we have $x\sim y$ for all $x,y\in X$.
Hence $R=X\times X$.
Thus $X$ is $h$-special.
\end{proof}

Note that the converse of \cref{lem:20230406} does not hold in general (cf. \Cref{ex:20230418}).
For the smooth case, motivated by Campana's suggestion \cite[11.3 (5)]{Cam11b}, we may ask whether a quasi-projective manifold $X$ is $h$-special  if and only if it is special. 
Campana proposed the following tantalizing \emph{abelianity conjecture} (cf. \cite[11.2]{Cam11b}).
 \begin{conjecture}[Campana]\label{conj:Campana}
 	A special smooth projective geometric orbifold   (e.g. quasi-projective manifold)  has \emph{virtually abelian} fundamental group.  
 \end{conjecture}
 In \Cref{example}  below we give an example of a special (and $h$-special) quasi-projective manifold with \emph{nilpotent} but not virtually abelian fundamental group, which thus disproves \cref{conj:Campana} for non-compact quasi-projective manifolds. Therefore, we revise Campana's conjecture as follows. 
 \begin{conjecture}\label{conj:revised}
 	An $h$-special or   special smooth quasi-projective variety has \emph{virtually nilpotent} fundamental group.  
 \end{conjecture}

\subsection{Fundamental groups of algebraic varieties}
Throughout this paper, we will make use of the following well-known result on fundamental groups from \cite[Theorem 2.1]{Ara16}:
\begin{lem}\label{lem:fun}
	Let $\mu:\widetilde{X}\to X$  be a bimeromorphic proper morphism between irreducible complex normal analytic variety.  Then $\mu_*:\pi_1(\widetilde{X})\to\pi_1(X)$ is surjective, and it is an isomorphism if both $\widetilde{X}$ and $X$ are smooth. Moreover, for any proper closed analytic subset $A\subset X$, $\pi_1(X\backslash A)\to \pi_1(X)$ is surjective.  \qed
\end{lem}

We also recall the following result. 
\begin{lem}[{see \cite[Proposition 1.3]{Cam91}, \cite[Proposition 2.10]{Kol95}}]\label{lem:finiteindex}
	Let \(f : X \to Y\) be a dominant morphism between   quasi-projective varieties, with \(Y\) normal. Then the image of \(f_{\ast} : \pi_{1}(X) \to \pi_{1}(Y)\) has finite index in \(\pi_{1}(Y)\). \qed
\end{lem}

We now give the definition of a big representation of fundamental groups, also known as a generically large representation as introduced by Koll\'ar in \cite{Kol95}:
\begin{dfn}[Big representation]\label{def:big representation}
	Let $X$ be a quasi-projective normal variety.  Let $\varrho:\pi_1(X)\to G(K)$ be a representation, where $G$ is an algebraic group defined over some field $K$. We say that $\varrho$ is a \emph{big representation} if there are at most countably many Zariski  closed subvarieties $Z_i\subsetneqq X$ so that for every positive dimensional    closed subvariety $Y\subset X$ so that   $Y\not\subset \cup Z_i$, the image $\varrho\big({\rm Im}[\pi_1(Y^{\rm norm})\to \pi_1(X))] \big)$ is infinite.  The points in $X-\cup_{i}Z_i$ are called \emph{very general points} in $X$.
\end{dfn} 
\begin{rem}
	In a more recent work by the second and third authors \cite{DY23}, it has been established that, for a quasi-projective normal variety $X$ and a reductive representation $\varrho:\pi_1(X)\to {\rm GL}_N(\bC)$, if we consider any closed subvariety $Z$, the image $\varrho\big({\rm Im}[\pi_1(Z^{\rm norm})\to \pi_1(X))] \big)$ is infinite if and only if $\varrho\big({\rm Im}[\pi_1(Z)\to \pi_1(X))] \big)$. Consequently, when the representation is reductive, in \cref{def:big representation} we can define the big representation without taking the normalization of $Y$. 
\end{rem}

 \subsection{Notions of pseudo Picard hyperbolicity}
Let us first recall the definition of pseudo Picard hyperbolicity introduced in \cite{Denarxiv}.  
Let $f:\bD^*\to X$ be a holomorphic map from the punctured disk $\bD^*$ to a quasi-projective variety $X$.
If $f$ extends to  a holomorphic map $\mathbb D\to\overline{X}$ from the disk $\bD$ to some projective compactification $\overline{X}$ of $X$, then the same holds for any projective compactification $\overline{X}'$ of $X$, as $\overline{X}$ and $\overline{X}'$ are birational.

\begin{dfn}[pseudo Picard hyperbolicity]\label{def:Picard}
	Let $X$ be a  smooth quasi-projective variety, and let $\overline{X}$ be a smooth projective compactification. 
 $X$ is called \emph{pseudo-Picard hyperbolic} if there is a   Zariski closed proper subset $Z\subsetneq X$ so that any holomorphic map  $f:\bD^*\to X$ with $f(\bD^*)\not\subset X$ extends to a holomorphic map $\bar{f}:\bD\to \overline{X}$. 
	If $Z=\varnothing$, $X$ is simply called \emph{Picard hyperbolic}.
\end{dfn}
As shown in \cite[Proposition 1.7]{CD21}, 
pseudo Picard hyperbolic varieties exhibit the following algebraic properties. 

\begin{proposition} \label{extension theorem} 
	Let $X$ be a smooth quasi-projective variety that is pseudo Picard hyperbolic. Then any meromorphic map $f:Y\dashrightarrow X$ from another smooth quasi-projective variety $Y$ to $X$ with $f(Y)\not\subset \mathrm{Sp_p}(X)$ is \emph{rational}. 
\end{proposition} 
Although we have stated this proposition only for dominant meromorphic maps $f:Y\dashrightarrow X$ in \cite[Proposition 1.7]{CD21}, the same proof works for the proof of \cref{extension theorem}.  We provide it here for completeness. 
\begin{proof}[Proof of \cref{extension theorem}]
Let $\overline{Y}$ be a smooth projective compactification of $Y$ such that $D:=\overline{Y}\backslash Y$ is a simple normal crossing divisor.   Let $\overline{X}$ be a smooth projective compactification of $X$.  Note that  $f$ is rational if and only if $f$ extends to a meromorphic map $\bar{f}:\overline{Y}\dashrightarrow \overline{X}$. It suffices to check that this property holds in a neighborhood of any point of \(D\).  By \cite[Theorem 1]{Siu75}, any meromorphic map from a Zariski open set $W^\circ$ of a complex manifold $W$ to a compact K\"ahler manifold $\overline{X}$ extends to a meromorphic map from $W$ to $\overline{X}$ provided that the codimension of $W-W^\circ$ is at least 2.  It then suffices to consider the extensibility of $f$ around smooth points on $D$. Pick any such point $p\in D$ and choose a coordinate system $(\Omega;z_1,\ldots,z_n)$ centered at $p$ such that $\Omega\cap D=(z_1=0)$. The theorem follows if we can prove that $f:\bD^*\times \bD^{n-1}\dashrightarrow X$ extends to a meromorphic map $\bD^{n}\dashrightarrow \overline{X}$. 

  Denote by $S$  the indeterminacy locus 
 of $f|_{\bD^*\times \bD^{n-1}}:\bD^*\times \bD^{n-1}\dashrightarrow X$, which is a closed subvariety of $\bD^*\times \bD^{n-1}$  of codimension at least two.  Since we assume that $f(Y)\not\subset\Sp_p(X)$, there is thus a dense open set $W\subset \bD^{n-1}$ such that for any $z\in W$, each slice $\bD^*\times \{z\}\not\subset S$ and $f(\bD^*\times \{z\}-S)\not\subset \Sp_p(X)$. Then the restriction   $f|_{\bD^*\times \{z\}}:\bD^*\times \{z\}\dashrightarrow X$ is well-defined and  holomorphic. Then  $f:\bD^*\times \{z\}\to X$   extends  to  a holomorphic map $\bD\times   \{z\}\to \overline{X}$ for each $z\in W$.   We then apply the theorem of Siu in \cite[p.442,  ($\ast$)]{Siu75} to conclude that $f|_{\bD^*\times \bD^{n-1}}$ extends to a meromorphic map  $\bD^{n}\dashrightarrow \overline{X}$. This implies that  $f$ extends to a meromorphic map $\bar{f}:\overline{Y}\dashrightarrow \overline{X}$. By the Chow theorem, $f$ is rational.   
\end{proof}

A direct consequence of \cref{extension theorem}  is the following uniquness of algebraic structure of pseudo Picard hyperbolic varieties. 

\begin{cor}
Let $X$ and $Y$ be smooth quasi-projective varieties such that there exists an analytic isomorphism $\varphi:Y^{\rm an}\to X^{\rm an}$ of associated complex spaces.
Assume that $X$ is pseudo Picard hyperbolic. 
Then $\varphi$ is an algebraic isomorphism. \qed
\end{cor}

A classical result due to Borel \cite{Bor72} and Kobayashi-Ochiai \cite{KO71} is that quotients of bounded symmetric domains by torsion-free lattices are Picard hyperbolic. The second author has proved a similar result for algebraic varieties that admit a complex variation of Hodge structures.
\begin{thm}[\protecting{\cite[Theorem A]{Denarxiv}}]\label{thm:PicardVHS}
	Let $X$ be a quasi-projective manifold. Assume that there is a complex variation of  Hodge structures on $X$ whose period mapping is injective at one  point. Then $X$ is pseudo Picard hyperbolic. \qed
\end{thm} 

We conclude this subsection by recalling from \cref{subsec:20230427} the strong Green-Griffiths and Lang conjectures, which are fundamental problems in the study of hyperbolicity of algebraic varieties (cf. \cite[I, 3.5]{Lan97} and \cite[VIII, Conj. 1.3]{Lan97}).
Taking into account \Cref{main:GGL,main:second}, we  include the pseudo Picard hyperbolicity into the statement as well, and formulate the \emph{generalized  Green-Griffiths-Lang conjecture} as follows.  
\begin{conjecture}\label{conj:GGL}
Let $X$ be a smooth quasi-projective variety.  Then the following properties are equivalent:
	\begin{thmlist} 
		\item  $X$ is of log general type;
	\item  $\Spp(X)\subsetneqq X$; 
	\item  $\Sph(X)\subsetneqq X$; 
	\item  
	$\Spab(X)\subsetneqq X$; 
	\item  
	$\Spalg(X)\subsetneqq X$.
	\end{thmlist}
\end{conjecture}
Note that if $X$ is of log general type, then the conjugate variety $X^\sigma:=X\times_\sigma\bC$ under  $\sigma\in {\rm Aut}(\bC/\bQ)$ is also of log general type.   Therefore, by \cref{conj:GGL}, if $X$ is pseudo Brody (Picard) hyperbolic, it is conjectured that $X^\sigma$ is also pseudo Brody (Picard) hyperbolic (cf. \cite[p. 179]{Lan97}). 
This problem remains quite open and is currently an active area of research.

 \section{Some factorisation results}\label{sec:fac}
 Throughout this paper an \emph{algebraic fiber space} $f:X\to Y$ is a dominant (not necessarily proper) morphism $f$ between quasi-projective normal varieties $X$ and $Y$ such that   general fibers of $f$ are connected.  
\begin{lem}[Quasi-Stein factorisation]\label{lem:Stein}
	Let $f:X\to Y$ be a morphism between quasi-projective manifolds. Then $f$ factors through   morphisms $\alpha:X\to S$ and $\beta:S\to Y$ such that
	\begin{enumerate}[label={\rm (\alph*)}]
		\item $S$ is a quasi-projective normal variety;
		\item   $\alpha$ is an algebraic fiber space;
		\item $\beta$ is a finite morphism.
	\end{enumerate}
Such a factorisation is unique.
\end{lem}
\begin{proof}
	Let $\overline{X}$ be a partial smooth compactification of $X$ such that $f$ extends to a projective morphism $\bar{f}:\overline{X}\to Y$. Take the Stein factorization of $\bar{f}$ and we obtain a proper surjective morphism  $\bar{\alpha}:\overline{X}\to S$ with connected fibers   and a finite morphism $\beta:S\to Y$. Then $\alpha:X\to S$ is defined to be the restriction of $\bar{\alpha}$ to $X$, which is dominant with connected general fibers. It is easy to see that this construction does not depend on the choice of $\overline{X}$. 
\end{proof}
The previous factorisation will be called \emph{quasi-Stein factorisation} in this paper. 
  \begin{lem}\label{lem:normal}
  		Let $f:X\to Y$ be a dominant morphism between connected quasi-projective manifolds such that general fibers are connected.  Then for a general fiber $F$, one has $ {\rm Im}[\pi_1(F)\to \pi_1(X)]\triangleleft \pi_1(X)$. 
  \end{lem}
\begin{proof}
 There is a Zariski open set $Y^\circ\subset Y$ such that $X^\circ:=f^{-1}(Y^\circ)$ is a topologically locally trivial fibration over $Y^\circ$.  Hence we have a short exact sequence
 $$
 \pi_1(F)\to \pi_1(X^\circ)\to \pi_1(Y^\circ)\to 0
 $$
 It follows that  $ {\rm Im}[\pi_1(F)\to \pi_1(X^\circ)]\triangleleft \pi_1(X^\circ)$. Note that $\pi_1(X^\circ)\to \pi_1(X)$ is surjective by \cref{lem:fun}.  Hence
$ {\rm Im}[\pi_1(F)\to \pi_1(X)]\triangleleft \pi_1(X)$. 
\end{proof}

\begin{lem}\label{lem:factor0}
	Let $f:X\to Y$ be a dominant morphism between  quasi-projective manifolds with  general fibers connected.  Let $\varrho:\pi_1(X)\to G(K)$ be a representation whose image is torsion free, where $G$ is a linear algebraic group defined on some field $K$.  If for the general fiber $F$, $\varrho({\rm Im}[\pi_1(F)\to \pi_1(X)])$ is trivial, then there is a commutative diagram
	\[
		\begin{tikzcd}
			X' \arrow[r, "\mu"] \arrow[d, "f'"] & X \arrow[d, "f"] \\
			Y' \arrow[r, "\nu"] & Y
		\end{tikzcd}
	\]
	where
	\begin{enumerate}[label=(\alph*)]
		\item \(\mu\) is a proper birational morphism,
		\item  \(\nu\) is a birational, not necessarily proper morphism ;
		\item  \(f'\) is dominant; 
	\end{enumerate} 
	 and a representation \(\tau : \pi_{1}(Y') \to G(K)\) such that $f'^*\tau=\mu^*\varrho$.  
	 \end{lem}
\begin{proof} 
	\noindent{\em Step 1. Compactifications and first reduction step.} 
	  We take a partial smooth compactification $\overline{X}$ of $X$  so that $f$ extends to a projective surjective morphism $\bar{f}:\overline{X}\to Y$ with connected fibers.   	 

	  \begin{claim} We may assume that \(\bar{f}:\overline{X} \to Y\) is equidimensional.
	  \end{claim}
	  Indeed, by Hironaka-Gruson-Raynaud's flattening theorem, there is a birational proper morphism $Y_1\to Y$ from a quasi-projective manifold $Y_1$ so that for the irreducible component  $T$ of $\overline{X}\times_YY_1$ which dominates $Y_1$, the induced morphism $f_{T} :=T\to Y_1$ is surjective, proper and flat.         In particular, the fibers of $f_{T}$ are equidimensional.  
	 Consider the normalization map $\nu:\overline{X}_1\to T$. Then the induced morphism $f_1:\overline{X}_1\to Y_1$ still has equidimensional fibers.  Write $\mu:\overline{X}_1\to \overline{X}$ for the induced proper birational morphism, and let $X_1:=\mu^{-1}(X)$.  Note that $\pi_1(X_1)\to \pi_1(X)$ is an isomorphism by \cref{lem:fun} below.

	 Then one has a diagram
	 \[
		 \begin{tikzcd}
		 X_{1} \arrow[r] \arrow[d] & X \arrow[d] \\
		 Y_{1} \arrow[r] & Y
		 \end{tikzcd}
         \]
	 where the horizontal maps are proper birational, and the two spaces on the left satisfy the hypotheses of the proposition if we take the representation induced on \(\pi_{1}(X_{1})\). Clearly, it suffices to show the result where \(X\) (resp. \(Y\)) is replaced by \(X_{1}\) (resp. \(Y_{1}\)). In the following, we may also replace \(\overline{X}\)  (resp. $Y$) by \(\overline{X}_{1}\) and $Y$ (resp. $Y_1$). 
	 \medskip
	
	\noindent
	{\em Step 2. Induced representation on an open subset of \(Y\).}  Consider a Zariski open set $Y^\circ\subset Y$ such that    $X^\circ:=f^{-1}(Y^\circ)$   is  a topologically locally trivial fibration over $Y^\circ$ with connected fibers $F$. Then  we have a short exact sequence
$$
\pi_1(F)\to \pi_1(X^\circ)\to \pi_1(Y^\circ)\to 0
$$
	By our assumption,  $\varrho({\rm Im}[\pi_1(F)\to \pi_1(X)])$ is trivial. Hence we can pass to the quotient, which yields a representation $\tau:\pi_1(Y^\circ)\to G(K)$ so that $\varrho|_{\pi_1(X^\circ)}= f^{\ast}\tau$. 
	\medskip

	\noindent
	{\em Step 3. Reducing \(Y\), we may assume that all divisorial components of \(Y - Y^{\circ}\) intersect \(f(X)\).} Denote by  $E$ the sum of prime divisors of $Y$ contained in the complement $Y\backslash Y^\circ$. We decompose $E=E_1+E_2$ so that $E_1$ is the sum of prime divisors of $E$ that do not intersect \(f(X)\).  We replace $Y$ by $Y\backslash E_1$.  Then for any prime divisor $P$ contained in $Y\backslash Y^\circ$, $f^{-1}(P)\cap X\neq \varnothing$.  
	\medskip

	\noindent
	{\em Step 4. Extension of the representation to the whole \(\pi_{1}(Y)\).}

	Let \(D\) be a divisorial component of \(Y - Y^{\circ}\). By what has been said above, \(f^{-1}(D) \neq \varnothing\). Since \(\bar{f} : \overline{X }\to Y\) is equidimensional, then for any prime component \(P\) of \(f^{-1}(D)\), the morphism \(f|_{P} : P \to D\) is dominant. Also, since \(X\) is normal, \(X\) is smooth at the general points  of \(P\).

This allows to find a point \(x \in P_{reg}\) (resp. \(y \in D_{reg}\)) with local coordinates \((z_{1}, \dotsc z_{m})\) (resp. \((w_{1}, \dotsc, w_{n})\)) around \(x\) (resp. \(y\)), adapted to the divisors, such that \(f^{\ast}(w_{1}) = z_{1}^{k}\) for some \(k \geq 1\).
	Hence the meridian loop $\gamma$ around the general point of \(P\)  is mapped to $\eta^k$ where $\eta$ is the meridian loop around \(D\). On the other hand, since \(\gamma\) is trivial in $\pi_1(X)$, it follows that
$$
0=\varrho(\gamma)=\tau(\eta^k).
$$    
Hence $\tau(\eta)$ is a torsion element. Since we have assumed that the image of $\varrho$ does not contain torsion element, then $\tau(\eta)$ has to be trivial. Hence $\tau$ extends to the smooth locus of $D$.

	Since this is true for any divisorial component \(D \subset Y - Y^{\circ}\), this shows that \(\tau\) extends to \(\pi_{1}(Y^{\circ\circ})\), where \(Y^{\circ} \subset Y^{\circ\circ}\) and \(Y - Y^{\circ\circ}\) has codimension \(\geq 2\) in \(Y\). However, since \(Y\) is smooth, we have \(\pi_{1}(Y) \cong \pi_{1}(Y^{\circ\circ})\), so \(\tau\) actually extends to \(\pi_{1}(Y)\).
\end{proof}
Note that the above proof is   more difficult than the compact cases, since $f(X)$ is only a constructible subset of $Y$ and $f$ is not proper. 


 Based on \cref{lem:factor0} we prove the following factorisation result which is important in proving \cref{main}. 
 \begin{proposition}\label{lem:kollar}
 	Let $X$ be a quasi-projective normal variety.  Let $\varrho:\pi_1(X)\to G(K)$ is a representation, where $G$ is a linear algebraic group defined over a field $K$ of zero characteristic.  Then there is a diagram
	 \[
		 \begin{tikzcd}
			 \widetilde{X} \arrow[r, "\mu"] \arrow[d, "f"] & \widehat{X} \arrow[r, "\nu"] & X\\
			 Y                                             &  &
		 \end{tikzcd}
	 \]
	 where \(Y\) and \(\widetilde{X}\) are quasi-projective manifolds, and
	 \begin{enumerate}[label=(\alph*)]
 		\item $\nu:\widehat{X}\to X$ is a finite étale cover;
		\item \(\mu : \widetilde{X} \to \widehat{X}\) is a birational proper morphism;
		\item $f : \widetilde{X} \to Y$ is a dominant morphism with connected general fibers;
 	\end{enumerate}
		such that there exists a big representation \(\tau : \pi_{1}(Y) \to G(K)\) with  $f^*\tau=(\nu\circ \mu)^*\varrho$.   \end{proposition}
Note that when $X$ is projective, this result is proved in \cite[Theorem 4.5]{Kol93}.   

\begin{proof}[Proo of \cref{lem:kollar}]
	{\em Step 1. We may assume that \(\varrho\) has a torsion free image.}
	Since the image $\varrho(\pi_1(X))$ is a finitely generated linear group, by a theorem of Selberg, there is a finite index normal subgroup $\Gamma\subset \varrho(\pi_1(X))$ which is torsion free.  Take an \'etale cover $\nu:\widehat{X}\to X$ with fundamental group $\pi_1(\widehat{X})=\varrho^{-1}(\Gamma)$.  Then the image of $\nu^*\varrho$ is torsion free.

	In the following, we may replace \(X\) by \(\widehat{X}\) and \(\varrho\) by \(\nu^{\ast}\varrho\) to assume that \(\mathrm{Im}(\varrho)\) is torsion free.
	\medskip

\noindent	{\em Step 2. We find a  suitable model of the Shafarevich map.}
	Denote by $H:={\ker} (\varrho)$, which is a normal subgroup of $\pi_1(X)$.  We apply \cite[Corollary 3.5 \& Remark 4.1.1]{Kol93} to conclude that there is a normal quasi-projective variety ${\rm Sh}^H(X)$ and a dominant rational map ${\rm sh}_{X}^H:X\dashrightarrow {\rm Sh}^H(X)$ so that
\begin{enumerate}[label=(\roman*)]
	\item there is a Zariski open set $X^\circ\subset X$ which does not meet the indeterminacy locus of ${\rm sh}_{X}^H|_{X^{\circ}}$ such that the fibers of ${\rm sh}_{X}^H|_{X^{\circ}}$
  	are closed in $X$; 
	\item  ${\rm sh}_{X}^H : X^{\circ} \to \mathrm{Sh}^{H}(X)$ has connected general fibers;
	\item \label{VG} there are at most countably many  closed subvarieties $Z_i\subsetneqq X$ so that for  every    closed subvariety $W\subset X$ such that   $W\not\subset \cup Z_i$, the image $\varrho\big({\rm Im}[\pi_1(W^{\rm norm})\to \pi_1(X))] \big)$ is finite (and thus trivial since the image of $ \varrho$ is torsion free by Step 1) if and only if ${\rm sh}_{X}^H(W)$ is a point. 
\end{enumerate}

Let us first take a resolution of singularities $Y_1\to {\rm Sh}^H(X)$, and then a birational proper morphism $X_1\to X$ from a quasi-projective manifold $X_1$ which resolves the indeterminacy of $X\dashrightarrow Y_1$. Then the induced dominant morphism
 $
f_1:X_1\to  Y_1
$ 
	fulfills the conditions of \cref{lem:factor0}. Thus, one has a commutative diagram as follows:
	\[
		\begin{tikzcd}
			X_{1}' \arrow[r, "\mu"] \arrow[d, "f_{1}'"] & X_{1} \arrow[r, "q"] \arrow[d, "f_{1}"] & X \arrow[d, dashed] \\ 
			Y_{1}' \arrow[r, "\nu"]            &  Y_{1} \arrow[r] & \mathrm{Sh}^{H}(X) 
		\end{tikzcd}
	\]
	where 
	\begin{enumerate}[label=(\alph*)]
	\item \(\mu\) is a proper birational morphism; 
	\item \(\nu\) is a birational (not necessarily proper) morphism,  
	\item  \(f_{1}'\) is a dominant morphism.
	\end{enumerate}
	Moreover, one has a representation $\tau:\pi_1(Y_{1}')\to G(K)$ so that \((f_{1}')^*\tau\) identifies with the pullback \(\mu^{\ast} q^{\ast} \varrho\).  
	
Therefore, there   are at most countably many  closed subvarieties $Z'_i\subsetneqq X_1'$ such that for  each   closed subvariety $Z\subset X_1'$ with 
\begin{itemize}
	\item $Z\not\subset \cup Z'_i$;
	\item   $f_1'(Z)$ is not a point,
\end{itemize}
 the map $Z\to (q\circ\mu)(Z)$ is proper birational, and ${\rm sh}_X^H((q\circ\mu)(Z))$ is not a point.  Write $\widetilde{Z}:=(q\circ\mu)(Z)$.  By \Cref{VG} $\varrho\big({\rm Im}[\pi_1(\widetilde{Z}^{\rm norm})\to \pi_1(X))] \big)$ is infinite.  By \cref{lem:fun} below, $\pi_1(Z^{\rm norm})\to \pi_1(\widetilde{Z}^{\rm norm})$ is surjective. Hence $(q\circ\mu)^*\varrho\big({\rm Im}[\pi_1( {Z}^{\rm norm})\to \pi_1(X_1'))] \big)$ is infinite. 
	\medskip

	\noindent
	{\em Step 3. We check that \(\tau\) is big and obtain the required diagram.} Let \(\widetilde{X} := X_{1}'\), \(Y := Y_{1}'\) and \(f := f_{1}'\). To simplify the notation, let us denote by the same letter \(\varrho\) the pullback \(\varrho : \pi_{1}(\widetilde{X}) \to G(K)\). Hence $\varrho=f^*\tau$. Recall that by Step 2 there   are at most countably many  closed subvarieties $Z'_i\subsetneqq \widetilde{X}$ such that for  every   closed subvariety $Z\subset \widetilde{X}$ with $Z\not\subset \cup Z'_i$ and $f(Z)$ not a point, the image $\varrho\big({\rm Im}[\pi_1(Z^{\rm norm})\to \pi_1(\widetilde{X}))] \big)$ is infinite. 
	
	Since $f:\widetilde{X}\to Y$ is dominant with connected general fibers,   for a subvariety $W$ of $Y$ containing a very general point, there is an irreducible component $Z$ of $f^{-1}(W)$ which dominates $W$ and $Z\not\subset \cup Z'_i$. Hence $\tau({\rm Im}[\pi_1(Z^{\rm norm})\to \pi_1(Y)])=\varrho({\rm Im}[\pi_1(Z^{\rm norm})\to \pi_1(\widetilde{X})])$ is infinite.  Since the morphism $\pi_1(Z^{\rm norm})\to \pi_1(Y)$ factors through $\pi_1(W^{\rm norm})\to \pi_1(Y)$, it follows that  $\tau({\rm Im}[\pi_1(W^{\rm norm})\to \pi_1(Y)])$ is infinite. Therefore $\tau$ is big.
	\end{proof}

 \section{Some results on varieties with logarithmic Kodaira dimension zero and maximal quasi-Albanese dimension} \label{sec:quasi}
In this section we prove a result on \cref{conj:Campana} by Campana, which will play important role in this paper.
We first prove a lemma. 
\begin{lem}  \label{lem:same Kd}
	Let $f:X\to Y$ be a dominant birational morphism of  quasi-projective manifolds. 
	Let $E\subset X$ be a Zariski closed subset such that $\overline{f(E)}$ has codimension at least two.  
	Assume $\bar{\kappa}(Y)\geq 0$.
	Then $\bar{\kappa}(X-E)=\bar{\kappa}(X)$. 
\end{lem}
\begin{proof}   
	Let $\bar{f}:\bar{X}\to \bar{Y}$ be a proper birational morphism which extends $f:X\to Y$, where $\bar{X}$ and $\bar{Y}$ are smooth projective compactifications of $X$ and $Y$ such that   $B=\bar{Y}-Y$ is a simple normal crossing divisor. Since the logarithmic Kodaira dimension is  birationally invariant, we may further blow-up  \(\bar{X}\) to assume that 
	\begin{itemize}
		\item $\bar{E}+(\bar{X}-X)$ is a simple normal crossing  divisor,
	\end{itemize} 
	where $\bar{E}\subset \bar{X}$ is the closure of $E$.
	Then we note that $\bar{X}-X$ is also simple normal crossing.
	Write $\bar{X}-X=\bar{f}^{-1}(B)+D$, where $D$ is a reduced divisor on $\bar{X}$ and $\bar{f}^{-1}(B)=\mathrm{supp}\bar{f}^*B$.
	Then we have $f^*K_{\bar{Y}}(B)+\bar{E}\leq K_{\bar{X}}(\bar{f}^{-1}(B))$.
	So we write $K_{\bar{X}}(\bar{f}^{-1}(B))=f^*K_{\bar{Y}}(B)+\bar{E}+F$, where $F$ is effective.
	Then we have
	$$
	K_{\bar{X}}(\bar{f}^{-1}(B))+D+\bar{E}=\bar{f}^*K_{\bar{Y}}(B)+D+2\bar{E}+F
	$$
	By the assumption $\bar{\kappa}(Y)\geq 0$, one has $nK_{\bar{Y}}(B)\geq 0$ for some positive integer $n>0$.
	Therefore,
	$$
	n(\bar{f}^*K_{\bar{Y}}(B)+D+2\bar{E}+F)\leq 2n(\bar{f}^*K_{\bar{Y}}(B)+D+\bar{E}+F)=2n(K_{\bar{X}}(\bar{f}^{-1}(B))+D).
	$$
	Thus 
	$n(K_{\bar{X}}(\bar{f}^{-1}(B))+D+\bar{E})\leq 2n(K_{\bar{X}}(\bar{f}^{-1}(B))+D) $. 
	Recall that $\bar{f}^{-1}(B)+D+\bar{E}$ and $\bar{f}^{-1}(B)+D$ are both simple normal crossing divisors.  	It follows 
	$\bar{\kappa}(X-E)\leq \bar{\kappa}(X)$. 
	Hence $\bar{\kappa}(X-E)=\bar{\kappa}(X)$.  
\end{proof}

\begin{lem}\label{lem:abelian pi0}
	Let $\alpha:X\to \cA$ be  a  (possibly non-proper)  birational morphism  from a quasi-projective manifold $X$ to a semi-abelian variety $\cA$ with   $\overline{\kappa}(X)=0$.  Then  there exists a Zariski closed subset $Z\subset \cA$ of codimension at least two such  that $\alpha$ is   isomorphic
over $\cA\backslash Z$.   
\end{lem}

\begin{proof}
Since $\alpha:X\to \cA$ is birational,  we remove the exceptional locus of $X\to \mathcal{A}$ from $X$ to get $X^\circ\subset X$ such that $\alpha: X^\circ\to \alpha(X^\circ)$ is an isomorphism. 
By $\bar{\kappa}(\cA)=0$, we may apply \cref{lem:same Kd} to get $\bar{\kappa}(X^\circ)=\bar{\kappa}(X)=0$.  
	By \cref{prop:Koddimabb}, $\mathcal{A}-\alpha(X^\circ)$ has codimension  at least two. \end{proof}

\begin{lem}\label{lem:abelian pi}
	Let $\alpha:X\to \cA$ be  a  (possibly non-proper)  morphism  from a quasi-projective manifold $X$ to a semi-abelian variety $\cA$ with   $\overline{\kappa}(X)=0$. 
Assume that $\dim X=\dim \alpha(X)$.	
	 Then  $ \pi_1(X)$ is   abelian.  
	 \end{lem}
\begin{proof}
{\em Step 1. We may assume that \(\alpha\) is dominant and birational.}
Consider the quasi-Stein factorisation $X\stackrel{h}{\to} Y\stackrel{k}{\to} \cA$  of $\alpha$, where $h$  is birational (might not proper), and $k$ is a finite morphism.  Since   $\bar{\kappa}(X)= 0$, one has $$0=\bar{\kappa}(X)\geq\bar{\kappa}(Y)\geq \bar{\kappa}(k(Y))\geq 0$$ where the last inequality follows from \cref{prop:Koddimabb}.     
Hence $\bar{\kappa}(Y)=\bar{\kappa}(k(Y))=0$.  By  \cref{prop:Koddimabb}    and Kawamata \cite[Theorem~26]{Kaw81},  $k(Y)$ is a semi-abelian variety and  $k$ is finite \'etale. In conclusion, it now suffices to prove the proposition with \(\cA\) replaced by \(Y\). 
Hence in the following, we may assume that $\alpha$ is dominant and birational.

\noindent
{\em Step 2.}
We apply \cref{lem:abelian pi0} to get the isomorphism $\alpha|_{X^{\circ}}:X^{\circ}\to \cA^{\circ}$, where $\cA-\cA^\circ$ of codimension at least two.
 Let $\bar{\alpha}:\overline{X}\to \cA$  be a  proper birational morphism which extends $\alpha:X\to \cA$ such that $\overline{X}$ is smooth.   One gets the following commutative diagram
	\begin{equation*}
		\begin{tikzcd}
			\pi_1( X^{\circ})\arrow[r, "\simeq"] \arrow[d] & \pi_1(\cA^{\circ})\arrow[dd, "\simeq"] \\
			\pi_1(X)  \arrow[d]& \\
			\pi_1(\bar{\alpha}^{-1}(\cA)) \arrow[r, "\simeq"] & \pi_1(\cA)
		\end{tikzcd}
	\end{equation*}
	where the two rows are isomorphisms because they are induced by   proper birational morphisms  between smooth quasi-projective varieties. The map \(\pi_{1}(X^\circ) \to \pi_{1}(X)\) is surjective since it is induced by the inclusion of a dense Zariski open subset. It follows that all the groups in the previous diagram are isomorphic, so \(\pi_{1}(X) \cong \pi_{1}(\cA)\) is abelian.
\end{proof}

\begin{lem}\label{lem:ZD}
Let $\alpha:X\to \cA$ be  a  (possibly non-proper)  morphism  from a quasi-projective manifold $X$ to a semi-abelian variety $\cA$ with   $\overline{\kappa}(X)=0$. 
Assume that $\dim X=\dim \alpha(X)$ and $\dim X>0$.
Then $X$ admits a Zariski dense entire curve $\mathbb C\to X$.
In particular, $X$ is $h$-special. 
\end{lem}

\begin{proof} 
As in the step 1 of the proof of \cref{lem:abelian pi}, we may assume that $\alpha:X\to \cA$ is birational.
We apply \cref{lem:abelian pi0} to get the isomorphism $\alpha|_{X^{\circ}}:X^{\circ}\to \cA^{\circ}$, where $\cA-\cA^\circ$ of codimension at least two.
Set $Z=\cA-\cA^{\circ}$.
We shall show that $\cA^{\circ}=\cA-Z$ contains a Zariski dense entire curve.
	Let $\varphi:\mathbb C\to \cA$ be a one parameter group such that $\varphi(\mathbb C)\subset \cA$ is Zariski dense.
	We define $F:\mathbb C\times Z\to \cA$ by $F(c,z)=z+\varphi(c)$.
	Then $F$ is a holomorphic map.
	Let $Z_1\subset Z_2\subset \cdots$ be a sequence of compact subsets of $Z$ such that $\cup_nZ_n=Z$.
	Let $K_n=\{|z|\leq n\}\subset \mathbb C$.
	Then $F(K_n\times Z_n)\subset \cA$ is a compact subset.
	Set $O_n=\cA\backslash F(K_n\times Z_n)$.
	Then $O_n\subset \cA$ is an open subset.
	Note that $\dim (\mathbb C\times Z)<\dim \cA$.
	Hence $O_n$ is a dense subset.
	Hence by the Baire category theorem, the intersection $\cap_nO_n$ is dense in $\cA$.
	In particular, we may take $x\in \cap_nO_n$.
	We define $f:\mathbb C\to \cA$ by $f(c)=x+\varphi(c)$.
	Then $f(\mathbb C)\cap Z=\emptyset$.
	Indeed suppose contary $f(c)\in Z$.
	Then $F(-c,f(c))=f(c)+\varphi(-c)=x$.
	Since $(-c,f(c))\in K_n\times Z_n$ for sufficiently large $n$, we have $x\in F(K_n\times Z_n)$, so $x\not\in O_n$.
	This contradicts to the choice of $x$. 
By \cref{lem:20230406}, $X$ is $h$-special.
\end{proof}

\begin{lem}\label{lem:20230509} 
Let $\alpha:X\to \cA$ be  a  (possibly non-proper)  morphism  from a quasi-projective manifold $X$ to a semi-abelian variety $\cA$ with   $\overline{\kappa}(X)=0$. 
Assume that $\dim X=\dim \alpha(X)$ and $\dim X>0$.
Then $\Spab(X)=X$.
\end{lem}

\begin{proof}
As in the step 1 of the proof of \cref{lem:abelian pi}, we may assume that $\alpha:X\to \cA$ is birational.
We apply \cref{lem:abelian pi0} to get the isomorphism $\alpha|_{X^{\circ}}:X^{\circ}\to \cA^{\circ}$, where $\cA-\cA^\circ$ of codimension at least two.
Then we have the inverse $\cA^{\circ}\to X$ of the isomorphism $X^{\circ}\to \cA^{\circ}$.
Thus $\Spab(X)=X$.
\end{proof}

\begin{example}\label{example:20221105}
	This example is a quasi-projective surface $X$ with maximal quasi-Albanese dimension such that $X$ is $h$-special, special, and $\bar{\kappa}(X)=1$.

	Let $C_1$ and $C_2$ be elliptic curves which are not isogenus.
	Set $A=C_1\times C_2$.
	Let $p_1:A\to C_1$ and $p_2:A\to C_2$ be the first and second projections, respectively.
	Let $\widetilde{A}=\mathrm{Bl}_{(0,0)}A$ be the blow-up with respect to the point $(0,0)\in A$.
	Let $E\subset \widetilde{A}$ be the exceptional divisor and let $D\subset \widetilde{A}$ be the proper transform of the divisor $p_1^{-1}(0)\subset A$.
	Set $X=\widetilde{A}-D$.
	Then by \cref{lem:same Kd}, we have $\bar{\kappa}(X)=\bar{\kappa}(X-E)$.
	On the other hand, we have $X-E=(C_1-\{0\})\times C_2$.
	Hence $\bar{\kappa}(X)=1$.
	
	Next we construct a Zariski dense entire curve $f:\mathbb C\to X$.
	Let $\pi_1:\mathbb C\to C_1$ and $\pi_2:\mathbb C\to C_2$ be the universal covering maps.
	We assume that $\pi_1(0)=0$ and $\pi_2(0)=0$.
	Set $\Gamma=\pi_1^{-1}(0)$, which is a lattice in $\mathbb C$.
	We define an entire function $h(z)$ by the Weierstrass canonical product as follows
	$$
	h(z)=\prod_{\omega\in \Gamma}\left( 1-\frac{z}{\omega}\right)e^{P_2(z/\omega)}.
	$$
	We consider $\pi_2\circ h:\mathbb C\to C_2$.
	Then for $\omega\in \Gamma$, we have $\pi_2\circ h(\omega)=0$.
	Hence $\Gamma\subset (\pi_2\circ h)^{-1}(0)$.
	Thus $\pi_1^{-1}(0)\subset (\pi_2\circ h)^{-1}(0)$.
	We define $g:\mathbb C\to A$ by $g(z)=(\pi_1(z),\pi_2\circ h(z))$.
	Then $g^{-1}(p_1^{-1}(0))\subset g^{-1}((0,0))$.
	Let $f:\mathbb C\to \widetilde{A}$ be the map induced from $g$.
	Then $f^{-1}(D+E)\subset f^{-1}(E)$.
	By $\pi_1'(z)\not=0$ for all $z\in \mathbb C$, we have $f^{-1}(D\cap E)=\emptyset$.
	Hence $f^{-1}(D)=\emptyset$.
	This yields $f:\mathbb C\to X$.
	By the Bloch-Ochiai theorem, the Zariski closure of $g:\mathbb C\to A$ is a translate of an abelian subvariety of $A$.
	Since $C_1$ and $C_2$ are not isogenus, non-trivial abelian subvarieties are $A$, $C_1\times\{0\}$ and $\{0\}\times C_2$.
	Hence $f$ is Zariski dense. 
	Hence $X$ is $h$-special (cf. \cref{lem:20230406}).
	
	Next we show that $X$ is special.
	If $X$ is not special, then  by definition after replacing $X$ by a proper birational modification, there is an algebraic fiber space $g:X\to C$ from $X$ to a  quasi-projective curve such that the orbifold base $(C,\Delta)$ of $g$ defined in \cref{sec:spechspecial} is of log general type, hence hyperbolic.  However, the composition $g\circ f$ is a orbifold entire curve of  $(C,\Delta)$. This contradicts with the hyperbolicity of $(C,\Delta)$.  Therefore, $X$ is special. \end{example}

 \section{Generalized GGL conjecture for algebraic varieties with maximal quasi-Albanese dimension}  \label{sec:GGL}
  In this section, $A$ is a semi-abelian variety and $Y$ is a Riemann surface with a proper surjective holomorphic map $\pi:Y\to\mathbb C_{>\delta}$, where $\mathbb C_{>\delta}:=\{z\in\bC\mid \delta<|z|\}$ with some fixed positive constant $\delta>0$.
 The purpose of this section is to prove the following theorem.
 The notations in the statement will be given in \cref{subsec:notion Nevanlinna}.
  
 \begin{thm}\label{thm2nd}
 	Let $X$ be a smooth quasi-projective variety which is of log general type.
Assume that there is a morphism $a:X\to A$ such that $\dim X=\dim a(X)$.
 	Then there exists a proper Zariski closed set $\Xi\subsetneqq X$ with the following property:
 	Let $f:Y\to X$ be a holomorphic map such that $N_{\ram\pi}(r)=O(\log r)+o(T_f(r))||$ and that $f(Y)\not\subset \Xi$.
	Then $f$ does not have essential singularity over $\infty$, i.e.,	there exists an extension $\overline{f}:\overline{Y}\to\overline{X}$ of $f$, where $\overline{Y}$ is a Riemann surface such that  $\pi:Y\to \mathbb C_{>\delta}$ extends to a proper map $\overline{\pi}:\overline{Y}\to \mathbb C_{>\delta}\cup\{\infty\}$  
	and $\overline{X}$ is a compactification of $X$.	
 \end{thm}

 If we apply this theorem for $Y=\mathbb C_{>\delta}$, then we conclude that $X$ is pseudo Picard hyperbolic, where $X$ is the same as in the theorem.

A consequence of \cref{thm2nd} is the following result on \Cref{conj:GGL}.
\begin{cor}\label{cor:20221102}
Let $X$ be a smooth quasi-projective variety. 
Assume that there is a morphism $a:X\to A$ such that $\dim X=\dim a(X)$.
 Then the following properties are equivalent:  
 \begin{enumerate}[wide = 0pt,  noitemsep,  font=\normalfont, label=(\alph*)] 
 	\item \label{being general type1} $X$ is of log general type; 
 	\item  \label{pseudo Picard1} $\Spp(X)\subsetneqq X$; 
 	\item \label{pseudo Brody}$\Sph(X)\subsetneqq X$; 
\item \label{spab}
$\Spab(X)\subsetneqq X$; 
 	\item \label{strong LGT1} 
$\Spalg(X)\subsetneqq X$.
 \end{enumerate}
\end{cor} 

To prove \cref{cor:20221102}, we start from the following general facts of the special subsets.

\begin{lem}\label{lem:inclusion} 
	Let $X$ be a smooth quasi-projective   variety. Then  one has  $\Spab\subseteq\Sph\subseteq\Spp$. 
\end{lem}
\begin{proof}
$\Spab\subseteq\Sph$ follows from  \cref{lem:ZD}. For any non-constant holomorphic map $f:\bC\to X$,   the holomorphic map 
\begin{align*}
	g:\bC^*\to X\\
	z\mapsto f(\exp(\frac{1}{z}))
\end{align*} has essential singularity at 0. Note that the Zariski closure of $g(\bD^*)$ coincides with the Zariski closure of $f$. It follows that   $\Sph\subseteq\Spp$. 
\end{proof}

 \begin{proof}[Proof of \cref{cor:20221102}]  \ref{being general type1}$\implies$\ref{pseudo Picard1}.  It follows from \cref{thm2nd} directly. 
 	
 	 \ref{pseudo Picard1}$\implies$\ref{pseudo Brody}. Note that pseudo Picard hyperbolicity always implies pseudo Brody one by \cref{lem:inclusion}.

\ref{pseudo Brody}$\implies$\ref{spab}.
This follows from \cref{lem:inclusion}.

 \ref{spab}$\implies$\ref{strong LGT1}.  
Let $E\subsetneqq X$ be a proper Zariski closed subset such that $a|_{X\backslash E}:X\backslash E\to A$ is quasi-finite.
Set $\Xi=\Spab(X)\cup E$.
 Under the hypothesis  \ref{spab}, we have $\Xi\subsetneqq X$. 
 
Now we prove that all closed subvarieties $V\subset X$ with $V\not\subset \Xi$ are of log-general type.
Note that $\dim V=\dim a(V)$. 
Hence $\bar{\kappa}(V)\geq 0$ by \cref{prop:Koddimabb}.    
We prove that the logarithmic Iitaka fibration $V\dashrightarrow W$ is birational.
So assume contrary that a very general fiber $F$ is positive dimensional.
We have $F\not\subset \Xi$, hence $F\not\subset \Spab(X)$, and the logarithmic Kodaira dimension satisfies $\bar{\kappa}(F)=0$.
Then $a|_F:F\to A$ satisfies $\dim a(F)=\dim F$. 
Hence by \cref{lem:20230509}, we have $\Spab{F}=F$.
This contradicts to $F\not\subset \Spab(X)$.
Hence $V$ is of log general type. 

\ref{strong LGT1}$\implies$\ref{being general type1}.  This is obvious. 
 \end{proof} 
 Therefore, \cref{cor:20221102}   proves  \cref{conj:GGL} for quasi-projective varieties with maximal quasi-Albanese dimension. 
 
 \medspace
 
 In this section, after introducing the notations in Nevanlinna theory in the following subsection (cf. \cref{subsec:notion Nevanlinna}), we shall prove \cref{thm2nd} in \cref{subsec:4.2} to \cref{subsec:4.7}.
 The proof of \cref{thm2nd} is based on the arguments of \cite{NWY13} and \cite{Yam15}.
 The proof of \cref{thm2nd} is independent from the other part of this paper.
 So the reader may skip \cref{subsec:4.2} to \cref{subsec:4.7}.
   
  \subsection{Some notions in Nevanlinna theory}\label{subsec:notion Nevanlinna}
We shall recall some elements of Nevanlinna theory (cf. \cite{NW13}).  
For $r>2\delta$, define
 $
  Y(r)=\pi^{-1}\big( \mathbb C_{>2\delta}(r) \big)
$ 
 where $\mathbb C_{>2\delta}(r)=\{z\in \bC\mid  2\delta<|z|<r\}$. 
 In the following, we assume that $r>2\delta$.
 The \emph{ramification counting function} of the covering $\pi:Y\to  \bC_{>\delta}$ is defined by
$$
N_{{\rm ram}\,  \pi}(r):=\frac{1}{{\rm deg} \pi}\int_{2\delta}^{r}\left[\sum_{y\in {Y}(t)} \ord_y \ram \pi \right]\frac{dt}{t},
$$ 
where $\ram \pi\subset Y$ is the ramification divisor of $\pi:Y\to\mathbb C_{>\delta}$.

Let $X$ be a projective variety and let $Z$ be a closed subscheme of $X$. Let $f: Y \rightarrow X$ be a holomorphic map such that $f(Y)\not\subset Z$. 
 Since $Y$ is one dimensional, the pull-back $f^{*} Z$ is a divisor on $Y$. The counting function and truncated function are defined to be 
$$
\begin{gathered}
	N_f(r, Z):=\frac{1}{\operatorname{deg}  \pi} \int_{2\delta}^{r}\left[\sum_{y \in Y(t)} \operatorname{ord}_{y} f^{*} Z\right] \frac{d t}{t}, \\
 {N}_f^{(k)}(r, Z):=\frac{1}{\operatorname{deg} \pi} \int_{2\delta}^{r}\left[\sum_{y \in Y(t)} \min \left\{k, \operatorname{ord}_{y} f^{*} Z\right\}\right] \frac{d t}{t} ,
\end{gathered}
$$
where $k\in\mathbb Z_{\geq 1}$ is a positive integer.
We define the proximity function by
$$
m_f(r, Z):=\frac{1}{\operatorname{deg} \pi} \int_{y\in \pi^{-1}(\{|z|=r\})}\lambda_Z(f(y))\ \frac{d\mathrm{arg}\ \pi(y)}{2\pi},
$$
where $\lambda_Z:X-\mathrm{supp}\ Z\to \mathbb R_{\geq 0}$ is a Weil function for $Z$ (cf. \cite[Prop 2.2.3]{Yam04}).

Let $L$ be a line bundle on $X$.
Let $f:Y\to X$ be a holomorphic map.
We define the order function $T_f(r,L)$ as follows.
First suppose that $X$ is smooth.
We equip with a smooth hermitian metric $h_L$, and let $c_1(L,h_L)$ be the curvature form of $(L, h_L)$.  Set
$$
T_f(r,L):=\frac{1}{\operatorname{deg} \pi} \int_{2\delta}^{r}\left[\int_{Y(t)} f^{*}c_1(L,h_L)\right] \frac{d t}{t}.
$$
This definition is independent of the choice of  the hermitian metric up to a function  $O(\log r)$.  
For the general case, let $V\subset X$ be the Zariski closure of $f(Y)$ and let $V'\to V$ be a smooth model of $V$.
This induces a morphism $p:V'\to X$. 
We have a natural lifting $f':Y\to V'$ of $f:Y\to V$.
Then we set
$$
T_f(r,L):=T_{f'}(r,p^*L)+O(\log r).
$$
This definition does not depend on the choice of $V'\to V$ up to a function  $O(\log r)$.

\begin{thm}[First Main Theorem]\label{thm:first}
Let $X$ be a projective variety and let $D$ be an effective Cartier divisor on $X$.
Let $f:Y\to X$ be a holomorphic curve such that $f(Y)\not\subset D$.
Then
$$N_f(r,D)+m_f(r,D)=T_f(r,\mathcal{O}_X(D))+O(\log r).$$
\end{thm}

\begin{proof}
Let $V\subset X$ be the Zariski closure of $f(Y)$ and let $p:V'\to V$ be a desingularization.
Let $f':Y\to V'$ be the canonical lifting of $f$.
Replacing $X$, $D$ and $f$ by $V'$, $p^*D$ and $f'$, respectively, we may assume that $X$ is smooth.
\par

Let $||\cdot ||$ be a smooth Hermitian metric on $\mathcal{O}_X(D)$ and let $s_D$ be the associated section for $D$.
By the Poinca{\'e}-Lelong formula, we have 
$$2dd^c\log (1/||s_D\circ f(y)||)=-\sum_{y\in Y}(\mathrm{ord}_yf^{*}D)\delta _y+f^{*}c_1(\mathcal{O}_X(D),||\cdot ||),$$
where $\delta_y$ is Dirac current suported on $y$.
Integrating over $Y(t)$, we get 
$$
2\int_{Y(t)}dd^c\log (1/||s_D\circ f(y)||)=-\sum_{y\in Y(t)}\mathrm{ord}_yf^{*}D+\int_{Y(t)}f^{*}c_1(\mathcal{O}_X(D), ||\cdot ||)
$$
Hence, we get
\begin{equation*}
\begin{split}
-N_f(r,D)+T_f(r,\mathcal{O}_X(D))
&=\lim_{\delta'\to \delta+}\frac{2}{\deg\pi}\int_{2\delta'}^r\frac{dt}{t}\int_{\pi^{-1}\big( \mathbb C_{>2\delta'}(t) \big)}dd^c\log \left( \frac{1}{||s_D\circ f(z)||}\right)\\
&=\lim_{\delta'\to \delta+}\frac{2}{\deg\pi}\int_{2\delta'}^r\frac{dt}{t}\int_{\partial \pi^{-1}\big( \mathbb C_{>2\delta'}(t) \big)}d^c\log \left( \frac{1}{||s_D\circ f(z)||}\right)\\
&=m_f(r,D)-m_f(2\delta,D)-C_f\log (r/2\delta),
\end{split}
\end{equation*}
where we set
$$
C_f=\lim_{\delta'\to \delta+}\frac{2}{\deg\pi}\int_{\pi^{-1}\{ |z|=2\delta'\}}d^c\log \left( \frac{1}{||s_D\circ f(z)||}\right),$$
which is a constant independent of $r$. 
\end{proof}

For any effective divisor $D_L \in  |L|$ with $f(Y)\not\subset D_L$,  the First Main theorem (cf. \cref{thm:first}) implies the following Nevanlinna inequality:
\begin{align}\label{eqn:20221120}
 	N_f(r, D_L) \leqslant T_f(r, L)+O(\log r) .
\end{align}
When $L$ is an ample line bundle on $X$, then we have $T_f(r,L)+O(\log r)>0$.
If $L'$ is another ample line bundle on $X$, then we have $T_f(r,L')=O(T_f(r,L))+O(\log r)$.
We write $T_f(r)$ for short instead of $T_f(r,L)$ when the order of magnitude as $r\to\infty$ is concerned.

\begin{lem}\label{lem:20230411}
Let $X$ be a smooth projective variety and let $f:Y\to X$ be a holomorphic map.
Assume that $T_f(r)=O(\log r)||$ and $N_{\ram\pi}(r)=O(\log r)||$. 	Here the symbol $||$ means that the stated estimate holds for $r>2\delta$ outside some exceptional interval with finite Lebesgue measure. 
 Then $f$ does not have essential singularity over $\infty$.
\end{lem}

\begin{proof}
For $s>2\delta$, we set $\nu(s)=\sum_{y\in {Y}(s)} \ord_y \ram \pi$.
Then for $2\delta <s<r$, we have
\begin{equation}\label{eqn:20230410}
\begin{split}
N_{{\rm ram}\,  \pi}(r)-N_{{\rm ram}\,  \pi}(s)
&=\frac{1}{{\rm deg} \pi}\int_{s}^{r}\left[\sum_{y\in {Y}(t)} \ord_y \ram \pi \right]\frac{dt}{t}
\\
&\geq \frac{1}{{\rm deg} \pi}\int_{s}^{r}\nu(s)\frac{dt}{t}=\frac{\nu(s)}{{\rm deg} \pi}(\log r-\log s).
\end{split}
\end{equation}
Set $K=\varliminf_{r\to\infty}N_{\ram\pi}(r)/\log r$.
By $N_{\ram\pi}(r)=O(\log r)||$, we have $K<\infty$.
By \eqref{eqn:20230410}, we have $\nu(s)\leq K{\rm deg} \pi$.
Since $s>2\delta$ is arbitraly, the ramification divisor $\ram\pi\subset Y$ consists of finite points.
Thus we may take $s_0>\delta$ such that $\pi^{-1}(\mathbb C_{>s_0})\to \mathbb C_{>s_0}$ is unramified covering.
Thus $\pi^{-1}(\mathbb C_{>s_0})$ is a disjoint union of punctured discs.
Hence we may take an extension $\overline{\pi}:\overline{Y}\to \mathbb C_{>\delta}\cup\{\infty\}$ of the covering $\pi:Y\to \mathbb C_{>\delta}$.

In the following, we replace $\delta$ by $s_0$ and take a connected component of $\pi^{-1}(\mathbb C_{>s_0})$.
Then we may assume that $\pi:Y\to \mathbb C_{>\delta}$ is unramified and $Y$ is a punctured disc.
Let $\omega$ be a smooth positive $(1,1)$-form on $X$.
For $s>2\delta$, we set $\alpha(s)=\int_{Y(s)}f^*\omega$.
Then by the similar computation as in \eqref{eqn:20230410}, the assumption $T_f(r)=O(\log r)||$ yields that $\alpha(s)$ is bounded on $s>2\delta$.
Thus $\int_Yf^*\omega<\infty$.
By \cref{lem:picard} below, we conclude the proof.
\end{proof}
The following lemma is well-known to the experts. We refer the readers to \cite[Lemma 3.3]{CD21} for a simpler proof based on Bishop's theorem. 
\begin{lem}\label{lem:picard}
Let $X$ be a compact K\"ahler manifold and let $f:\mathbb D^*\to X$ be a holomorphic map from the punctured disc $\mathbb D^*$.
Let $\omega$ be a smooth K\"ahler form on $X$.
Suppose $\int_{\mathbb D^*}f^*\omega<\infty$.
Then $f$ has holomorphic extension $\bar{f}:\mathbb D\to X$. \qed
\end{lem}

\begin{lem}\label{lem:20230415}
Let $X$ be a smooth projective variety and let $f:Y\to X$ be a holomorphic map.
If $\varliminf_{r\to\infty}T_f(r)/\log r<+\infty$, then $T_f(r)=O(\log r)$.
\end{lem}

\begin{proof}
For $s>2\delta$, we set $\alpha(s)=\int_{Y(s)}f^*\omega$, where $\omega$ is a smooth positive $(1,1)$-form on $X$.
Then by the similar computation as in \eqref{eqn:20230410}, the assumption $\varliminf_{r\to\infty}T_f(r)/\log r<+\infty$ yields that $\alpha(s)$ is bounded, i.e., there exists a positive constant $C>0$ such that $\alpha(s)<C$ for all $s>2\delta$.
Hence for all $r>2\delta$, we have
$$
\frac{1}{\operatorname{deg} \pi} \int_{2\delta}^{r}\left[\int_{Y(t)} f^{*}\omega\right] \frac{d t}{t}<\frac{C}{\operatorname{deg} \pi}\log \frac{r}{2\delta}.
$$
Thus we get $T_f(r)=O(\log r)$.
\end{proof}

 \subsection{Lemma on logarithmic derivatives} 
Let $\pi _Y:Y\to \mathbb C_{>\delta}$ be a Riemann surface with a proper surjective holomorphic map.
Set $D_Y=\pi_Y^*(\partial /\partial z)$.
Then $D_Y$ is a meromorphic vector field on $Y$.
For a meromorphic function $f:Y\to\mathbb P^1$ on $Y$, we set $f'=D_Y(f)$.
For $r>2\delta$, we set as follows:
$$
m(r,f)=\frac{1}{\deg \pi_Y}\int_{y\in \pi_Y^{-1}(\{|z|=r\})}\log \sqrt{1+|f(y)|^2}\frac{d\arg \pi_Y(y)}{2\pi},
$$
$$
T(r,f)=\frac{1}{\deg \pi_Y}\int_{2\delta}^r\frac{dt}{t}\int_{Y(t)}f^*\omega_{\mathrm{F.S.}},
$$
where $\omega_{\mathrm{F.S.}}$ is the Fubini-Study form:
$$
\omega_{\mathrm{F.S.}}=\frac{1}{(1+|w|^2)^2}\frac{\sqrt{-1}}{2\pi}dw\wedge d\overline{w}.
$$
Note that $m(r,f)$ is a proximity function function for $f:Y\to\mathbb P^1$ with respect to the divisor $(\infty)$ on  $\mathbb P^1$, and $
T(r,f)$ is a order function with respect to the ample line bundle $\mathcal{O}_{\mathbb P^1}(1)$.

\begin{lem}\label{lem:202311191}
For $n\in\mathbb Z_{\geq 1}$, we have
$T(r,f^n)=nT(r,f)+O(\log r)$.
\end{lem}

\begin{proof}
Let $\varphi:\mathbb P^1\to\mathbb P^1$ be defined by $\varphi(w)=w^n$, then we have $f^n=\varphi\circ f$.
By $\varphi^*\mathcal{O}_{\mathbb P^1}(1)=\mathcal{O}_{\mathbb P^1}(n)$, we obtain our lemma.
\end{proof}

\begin{lem}\label{lem:202311192}
$m(r,f'/f)\leq 3T(r,f)+O(\log r)\ ||$
\end{lem}

\begin{proof}
Set 
$$
f^{\#}=\frac{|f'|}{1+|f|^2}
$$
and
$$
m^{\#}(r,f)=\frac{1}{\deg \pi _Y}\int_{y\in \pi_Y^{-1}(\{|z|=r\})}\log \sqrt{1+(f^{\#}(y))^2} \frac{d \arg \pi _Y(y)}{2\pi}.
$$
We first show
\begin{equation}\label{eqn:1aa}
m^{\#}(r,f)\leq T(r,f)+O(\log r)\ ||.
\end{equation}
Indeed using convexity of $\log$, we have
\begin{equation*}
m^{\#}(r,f)\leq \frac{1}{2}\log \left( 1+\frac{1}{\deg \pi _Y}\int_{y\in \pi_Y^{-1}(\{|z|=r\})}f^{\#}(y)^2\frac{d \arg \pi _Y(y)}{2\pi}\right) .
\end{equation*}
Using polar coordinate, we get
$$
\int_{Y(r)}f^*\omega_{\mathrm{F.S.}}=\frac{1}{\pi}\int_{2\delta}^rtdt\int_{y\in \pi_Y^{-1}(\{|z|=t\})}f^{\#}(y)^2d\arg\pi_Y(y).
$$
This shows
$$
\frac{1}{2r}\frac{d}{dr}\left( r\frac{d}{dr}T(r,f)\right) =\frac{1}{\deg \pi _Y}\int_{y\in \pi_Y^{-1}(\{|z|=r\})}f^{\#}(y) ^2\frac{d \arg \pi _Y(y)}{2\pi}.
$$
Hence using Borel's growth lemma \cite[p. 13]{NW13}, we get
\begin{equation*}
\begin{split}
m^{\#}(r,f)&\leq  \frac{1}{2}\log \left( 1+\frac{1}{2r}\frac{d}{dr}\left( r\frac{d}{dr}T(r,f)\right)\right) \\
&\leq  \frac{1}{2}\log \left( 1+\frac{1}{2r}\left( r\frac{d}{dr}T(r,f)\right)^{1+\delta}
\right) \ ||_{\delta}\\
&\leq  \frac{1}{2}\log \left( 1+\frac{1}{2}r^{\delta}T(r,f)^{(1+\delta )^2}\right) \ ||_{\delta}\\
&\leq  T(r,f) +O(\log r)\ ||.
\end{split}
\end{equation*}
This shows our estimate \eqref{eqn:1aa}.
Here we take $\delta =1$ in the final equation.

\par

Now by
\begin{multline*}
\log \sqrt{1+\left\vert \frac{1}{f}\right\vert^2}+\log \sqrt{1+|f|^2}+\log \sqrt{ 1+\left( \frac{|f'|}{1+|f|^2}\right) ^2}
\\
=\log \sqrt{ \left(|f|+\frac{1}{|f|}\right)^2+\left\vert \frac{f'}{f}\right\vert ^2} 
\geq
\log\sqrt{1+\left\vert \frac{f'}{f}\right\vert ^2},
\end{multline*}
we get
$$
m(r,f'/f)\leq m(r,1/f)+m(r,f )+m^{\#}(r,f).
$$
Using the first main theorem (cf. \cref{thm:first}) and \eqref{eqn:1aa}, we get \cref{lem:202311192}.
\end{proof}

We prove Nevanlinna's lemma on logarithmic derivatives (cf. \cite[Lemma 1.2.2]{NW13}) in the following form.
\begin{thm}\label{thm:20231119}
$m(r,f'/f)=o(T(r,f))+O(\log r)\ ||.
$
\end{thm}

\begin{proof}
We first show that for every $\varepsilon>0$, we have
\begin{equation}\label{eqn:202311193}
m(r,f'/f)\leq \varepsilon T(r,f))+O_{\varepsilon}(\log r)\ ||_{\varepsilon}.
\end{equation}
Here $O_{\varepsilon}$ indicates that the implicit constant in the Landau symbol $O$ may depend on $\varepsilon$.

Indeed we take $n\in\mathbb Z_{\geq 1}$ so that $3/n<\varepsilon$.
We set $f_n=\sqrt[n]{f}$.
Then $f_n$ is a multi-valued meromorphic function on $Y$.
Let $\pi_{Y_n}:Y_n\to \mathbb C_{>\delta}$ be the Riemann surface for $f_n$.
Both $f$ and $f_n$ are considered as meromorphic funtions on $Y_n$.
We have $f'/f=nf_n'/f_n$.
By \cref{lem:202311191}, we get
$$
m(r, f'/f)= m(r,f_n'/f_n )+O(1)\leq 3T(r,f_n)+O(\log r)\ ||. 
$$
Hence by \cref{lem:202311192}, we have 
$$
m(r, f'/f)\leq \frac{3}{n}T(r,f)+O(\log r)\ ||. 
$$
This shows \eqref{eqn:202311193}.

Suppose that $\varliminf_{r\to\infty}T(r,f)/\log r<+\infty$.
Then by \cref{lem:20230415}, we have $T(r,f)=O(\log r)$.
Then by \eqref{eqn:202311193}	, we have $m(r, f'/f)			=O(\log r)$, in particular $m(r,f'/f)=o(T(r,f))+O(\log r)\ ||$.
	
Next we assume $\varliminf_{r\to\infty}T(r,f)/\log r=+\infty$.
Then $\log r=o(T(r,f))$.		
		Hence by \eqref{eqn:202311193}	, we have 
		\begin{equation}\label{eqn:20231120}
		m(r, f'/f)			\leq \ep T(r, f)+	o(T(r,f))\ ||_{\varepsilon}
		\end{equation}
		for all $\varepsilon>0$.
		This implies $m(r, f'/f)=o(T(r,f))\ ||$.
		Indeed we take a sequence $2\delta=r_0<r_1<r_2<\cdots$ with $r_n\to\infty$ as follows.
	By \eqref{eqn:20231120}, we have	$m(r, f'/f)\leq \frac{1}{n}T(r,f)$ for all $r>2\delta$ outside some exceptional set $E_n\subset (2\delta,\infty)$ with $|E_n|<\infty$.
	We take $r_n$ such that $|(r_n,\infty)\cap E_n|<1/2^n$.
	We set $\varepsilon(r)=1/n$ if $r_n\leq r<r_{n+1}$, and $\varepsilon(r)=1$ if $r_0<r<r_1$.
	Then $\varepsilon(r)\to0$ if $r\to\infty$.
	Set $E=(r_0,r_1)\cup\bigcup\left((r_n,r_{n+1})\cap E_n\right)$.
	Then we have $m(r, f'/f)\leq \varepsilon(r)T(r,f)$ for all $r>2\delta$ outside $E$, and $|E|<r_1+1$.
		Hence we get $m(r, f'/f)=o(T(r,f))\ ||$, in particular we get $m(r,f'/f)=o(T(r,f))+O(\log r)\ ||$.
		This conclude the proof of the theorem.	
		\end{proof}
 
Let $V$ be a smooth projective variety and let $D\subset V$ be a simple normal crossing divisor.
 Let $T(V;\log D)$ be the logarithmic tangent bundle. 
Set $\overline{T}(V;\log D)=P(T(V;\log D)\oplus \mathcal{O}_V)$, which is a smooth compactification of $T(V;\log D)$.
Let $\partial T(V;\log D)\subset \overline{T}(V;\log D)$ be the boundary divisor.
Let $f:Y\to V$ be a holomorphic map such that $f(Y)\not\subset D$.
Then we get a derivative map $j_1(f):Y\to \overline{T}(V;\log D)$.
By \eqref{thm:20231119}, we obtain the following estimate
\begin{equation}\label{eqn:202311195}
m_{j_1 (f)}(r,\partial T(V;\log D))=o(T(r,f))+O(\log r)\ ||.
\end{equation}
For the proof of this estimate, we refer the readers to \cite[Thm 5.1.7 (2)]{Yam04}, where we use \cref{thm:20231119} instead of  \cite[Thm 2.5.1 (2)]{Yam04}.
(Here we simply denote $j_1 (f)$ instead of "$j_1^{\log} (f)$" in \cite{Yam04}.)
More generally, by the same manner using \cref{thm:20231119}, the estimates in \cite[Thm 5.1.7]{Yam04} are valid with the error term $S(r,f)=o(T(r,f))+O(\log r)\ ||$.
 
  \subsection*{Convention and notation.} 
  In the rest of this section, we use the following convention and notation.
Let $\Sigma$ be a quasi-projective variety and let $\overline{\Sigma}$ be a projective compactification.
We denote by $f:Y\da \Sigma$ a holomorphic map $\bar{f}:Y\to \overline{\Sigma}$ such that $\bar{f}^{-1}(\Sigma)\not=\emptyset$.
Moreover for a Zariski closed set $W\subset \Sigma$, we denote by $f(Y)\sqsubset W$ if $\bar{f}(Y)\subset \overline{W}$, where $\overline{W}\subset \overline{\Sigma}$ is the Zariski closure.
 
 Let $A$ be a semi-abelian variety and let $S$ be a projective variety.
 For $f:Y\da A\times S$, we set
 $$
 \pN{f}(r):=\frac{1}{\operatorname{deg}  \pi} \int_{2\delta}^{r}\mathrm{card}\left(Y(t)\cap \bar{f}^{-1}(\partial A\times S)\right)\frac{d t}{t}.
 $$
 This definition does not depend on the choice of $\overline{A}$.
 
 For $f:Y\da A\times S$, we define $f_A:Y\da A$ and $f_S:Y\to S$ by the compositions of $f$ and the projections $A\times S\to A$ and $A\times S\to S$, respectively.

 \subsection{Preliminaries for the proof of \cref{thm2nd}}
 \label{subsec:4.2}
 
 Let $V$ be a smooth algebraic variety.
 Let $TV$ be the tangent bundle of $V$.
 Then we have $TV=\mathbf{Spec}(\mathrm{Sym}\ \Omega^1_V)$. 
Set $\overline{T}V=P(TV\oplus \mathcal{O}_V)$, which is a smooth compactification of $TV$.
 Then we have $\overline{T}V=\mathbf{Proj}((\mathrm{Sym}\ \Omega^1_V)\otimes_{\mathcal{O}_V}\mathcal O_V[\eta ])$.
Let $Z\subset V$ be a closed subscheme.
 We define a closed subscheme $Z^{(1)}\subset \overline{T}(V)$ as follows (cf. \cite[p. 38]{Yam04}): 
 Let $U\subset V$ be an affine open subset.
 Let $(g_1,\cdots ,g_l)\subset \Gamma (U,\mathcal O_V)$ be the defining ideal of $Z\cap U$.
 Then $Z^{(1)}\cap \overline{T}(U)$ is defined by the homogeneous ideal
 $$(g_1,\cdots ,g_l,dg_1\otimes 1,\cdots ,dg_l\otimes 1)\subset \Gamma (U,(\mathrm{Sym}\ \Omega^1 _V)\otimes_{\mathcal O_V}\mathcal O_V[\eta ]).
 $$
 We glue $Z^{(1)}\cap \overline{T}(U)$ to define $Z^{(1)}$.
 If $Z\subset V$ is a closed immersion of a smooth algebraic variety, then $Z^{(1)}\cap T(V)=T(Z)$.

 Let $S$ be a projective variety, and let $W\subset A\times S$ be a Zariski closed set.
 Let $q:A\times S\to S$ be the second projection and let $q_W:W\to S$ be the restriction of $q$ on $W$.
 Given $f:Y\da A\times S$, we denote by $\Sigma\subset S$ the Zariski closure of $f_S(Y)\subset S$.
 We define $e(f,W)$ to be the dimension of the generic fiber of $q_W^{-1}(\Sigma)\to \Sigma$.
 We set $e(f,W)=-1$ if the generic fiber of $q_W^{-1}(\Sigma)\to \Sigma$ is an empty set.

 Now assume that $W\subset A\times S$ is irreducible and $\overline{q(W)}=S$.
 We also assume that $\dim \mathrm{St}(W)=0$, where the action $A\curvearrowright A\times S$ is defined by $(a,s)\mapsto (a+\alpha,s)$ for $\alpha\in A$.
 Since $W$ and $S$ are integral, and $q_W$ is dominant, there exists a non-empty Zariski open subset $W^o\subset W$ such that $q_W$ is a smooth morphism over $W^o$.
 Let $S^o\subset S$ be a non-empty Zariski open subset such that for each $s\in S^o$, (1) every irreducible component of $q_W^{-1}(s)$ has non-trivial intersection with $W^o$, and (2) the stabilizer of every irreducible component of $q_W^{-1}(s)$ is 0-dimensional.
 
 Assume $S$ is smooth.
 We note 
 \begin{equation}\label{eqn:20221203}
 \overline{T}(A\times S)=A\times S',
 \end{equation}
  where $S'=\overline{\mathrm{Lie}(A)\times TS}$.
 We may define $e(j_1f,W^{(1)})$ from $j_1f:Y\da A\times S'$ and $W^{(1)}\subset A\times S'$.

 \begin{lem}\label{lem:202210041}
 	Assume that $S$ is smooth and projective, and that $W\subset A\times S$ is irreducible with $\overline{q(W)}=S$.
 	Let $f:Y\da A\times S$ satisfies $e(j_1f,W^{(1)})=e(f,W)$ and $f_S(Y)\not\subset S- S^o$.
 	Then we have
 	$$T_{f_A}(r)=O(T_{f_S}(r))+O(\log r).$$
 \end{lem}

 \begin{rem}
 	In the statement of the lemma, by $T_{f_A}(r)$ we mean $T_{f_A}(r,L)$ for some ample line bundle on a compactification $\overline{A}$.
 	So this notion has ambiguity, but we have $T_{f_A}(r,L')=O(T_{f_A}(r,L))+O(\log r)$ for other choices.
 	Hence the term $O(T_{f_A}(r))+O(\log r)$ has fixed meaning.
 	Similar for $O(T_{f_S}(r))+O(\log r)$.
 \end{rem}
 
 \begin{proof}[Proof of \cref{lem:202210041}]
 	We define $\mathcal{Z}(W)\subset A\times S\times S$ by
 	$$
 	\mathcal{Z}(W)=\{ (a,s_1,s_2)\in A\times S\times S;\dim (q_W^{-1}(s_1)\cap (a+q_W^{-1}(s_2)))\geq \dim W-\dim S\} ,
 	$$
 	where we consider $q_W^{-1}(s_1)$ and $q_W^{-1}(s_2)$ as subvarieties of $A$.
 	For $s\in S$, let $\mathcal{Z}_{s}(W)\subset A\times S$ be the fiber of the composite map $\mathcal{Z}(W)\hookrightarrow A\times S\times S\to S$ over $s\in S$, where the second map $A\times S\times S\to S$ is the third projection.

 	\begin{claim}\label{claim:20221018}
 		Let $s\in S^o$.
 		Let $g:\mathbb D \to A\times S$ be a holomorphic map such that $e(j_1f,W^{(1)})=e(f,W)$ with $g(0)=(a,s)\in A\times S^o$.
 		Then $g(\mathbb D )\subset \mathcal{Z}_{s}(W)+a$.
 	\end{claim}
 	
 	\begin{proof}
 		This is proved in the sublemma of \cite[Lemma 3]{Yam15} when $A$ is compact.
 		The same proof works for our situation.
 		We note that the condition "$j_1(g)(\mathbb D )\subset \Theta (W)$" appears in the sublemma above follows from $e(j_1f,W^{(1)})=e(f,W)$.
 	\end{proof}

 	Now we take $y\in Y$ such that $f_S(y)\in S^o$ and $f(y)=(a,s)\in A\times S^o$.
 	Since the stabilizer of every irreducible component of $q_W^{-1}(s)$ is 0-dimensional, 
 	the restriction of the natural map 
 	$$q|_{\mathcal{Z}_s(W)} :\mathcal{Z}_s(W)\to S$$ 
 	over $S^o$ is quasi-finite.
 	Let $\overline{\mathcal{Z}_s(W)}\subset \overline{A}\times S$ be the Zariski closure, where $\overline{A}$ is an equivariant compactification of $A$.
 	Let $\mu:\overline{\mathcal{Z}_s(W)}\to S$ be the projection.
 	Then $\mu$ is generically-finite onto its image $\mu(\overline{\mathcal{Z}_s(W)})$.
 	By \cref{claim:20221018}, we have $f(Y)\subset \overline{\mathcal{Z}_s(W)}+a$.
 	Then by \cite[Lemma 4]{Yam15}, we have $T_{f_A}(r)=O(T_{f_S}(r))+O(\log r)$.
 \end{proof}

 We recall $S'=\overline{\mathrm{Lie}A\times TS}$, so that $\overline{T(A\times S)}=A\times S'$.
 
 \begin{lem}\label{lem:202210042}
 	For $f:Y\da A\times S$, we have
 	$$T_{(j_1f)_{S'}}(r)=O( T_{f_S}(r)+N_{\ram \pi _Y}(r)+\pN{f}(r))+O(\log r)+o(T(r,f_A)) \ ||,$$
 	where $j_1(f)_{S'}:Y\to S'$ is the composition of $j_1(f):Y\da A\times S'$ and the second projection $A\times S'\to S'$.
 \end{lem}
 
 \begin{proof}
(cf. \cite[Lemma 2]{Yam15})
Let $\overline{A}$ be an equivariant compactification, which is smooth and projective.
Set $D=(\partial A)\times S$.
Then $D$ is a simple normal crossing divisor on $\overline{A}\times S$.
Let $T(\overline{A}\times S;\log D)$ be the logarithmic tangent bundle.
We set $\overline{T}(\overline{A}\times S;\log D)=P(T(\overline{A}\times S;\log D)\oplus \mathcal{O}_{\overline{A}\times S})$, which is a smooth compactification of $T(\overline{A}\times S;\log D)$.
We set $E=\overline{T}(\overline{A}\times S;\log D)-T(\overline{A}\times S;\log D)$, which is a divisor on $\overline{T}(\overline{A}\times S;\log D)$.
By $T(A\times S)\subset \overline{T}(\overline{A}\times S;\log D)$, we have $j_1f:Y\to \overline{T}(\overline{A}\times S;\log D)$.
By \cite[(2.4.8)]{Yam04} and \eqref{eqn:202311195}, we have 
$$
T_{j_1f}(r,E)\leq N_{\ram \pi}(r)+\pN{f}(r)+O(\log r)+o(T_f(r))\ ||.
$$ 
Note that $T(\overline{A};\log\partial A)=\overline{A}\times \mathrm{Lie}A$ (cf. \cite[Prop 5.4.3]{NW13}).
Hence we have $\overline{A}\times S'=\overline{T}(\overline{A}\times S;\log D)$ and $S'=\mathbf{Proj}((\mathrm{Sym}\ \Omega^1_S)\otimes_{\mathcal{O}_S}\mathcal O_S[\eta,dz_1,\ldots,dz_{\dim A} ])$, where $\{dz_1,\ldots,dz_{\dim A}\}\subset H^0(\overline{A},\Omega^1_{\overline{A}}(\log \partial A))$ is a basis.
Let $F\subset S'$ be the divisor defined by $\eta=0$.
Then we have $p^*F=E$, where $p:\overline{A}\times S'\to S'$ is the second projection.
Hence we have
$$T_{(j_1f)_{S'}}(r,F)\leq N_{\ram \pi _Y}(r)+\pN{f}(r)+O(\log r)+o(T(r,f_A)) \ ||.$$
By $\mathrm{Pic}(S')=\mathrm{Pic}(S)\oplus \mathbb Z[F]$, we obtain our lemma.
 \end{proof}

 \subsection{Refinement of log Bloch-Ochiai Theorem}
 
 Let $\mathcal{S}_0(A)$ be the set of all semi-abelian subvarieties of $A$.
 Let $W\subset A\times S$ be a Zariski closed subset.
 For $B\in \mathcal{S}_0(A)$, we set
 \begin{equation}\label{eqn:20221101}
 W^B=\{ x\in W;\ x+B\subset W\}.
\end{equation}
 Then $W^B=\cap_{b\in B}(b+W)\subset W$ is a Zariski closed subset.
 When $B=\{0\}$, we have $W^{\{ 0\}}=W$.
 
 \begin{proposition}\label{pro:20220807}
 	Let $S$ be a projective variety.
 	Let $W\subset A\times S$ be an irreducible Zariski closed subset.
 	Then there exists a finite subset $P\subset \mathcal{S}_0(A)$ such that 
 	for every $f:Y\da W$, there exists $B\in P$ such that $f(Y)\sqsubset W^B$ and
 	\begin{equation}\label{eqn:20220921}
 		T_{q_B\circ f_A}(r)=O(N_{\ram \pi}(r)+T_{f_S}(r)+\pN{f}(r))+O(\log r)+o(T_f(r))||,
 	\end{equation}
	where $q_B:A\to A/B$ is the quotient.
 \end{proposition}
 
 \begin{proof}
 	Given $f:Y\da W$, we have $e(f,W)\not=-1$, hence $e(f,W)\in\{0,1,\ldots,\dim A\}$.
 	The proof of \cref{pro:20220807} easily reduces to the following claim by letting $P=P_{\dim A}$.
 	
 	\medskip
 	
 	\begin{claim}
 		In \cref{pro:20220807}, let $\dim S=l$.
 		Let $k\in \{0,1,\ldots,\dim A\}$.
 		Then there exists a finite subset $P_k\subset \mathcal{S}_0(A)$ such that for every $f:Y\da W$ with $e(f,W)\leq k$, there exists $B\in P_k$ such that $f(Y)\sqsubset W^B$ and \eqref{eqn:20220921}. 
 	\end{claim}
 	\medskip
 	
 	We prove this claim by the induction on the pair $(k,l)\in\mathbb Z_{\geq 0}^2$, where we consider the dictionary order on $\mathbb Z^2$.
 	So we assume that the claim is true if $(e(f,W),\dim S)<(k,l)$ and prove the case $(e(f,W),\dim S)=(k,l)$.
	
	Let $q:A\times S\to S$ be the second projection.
	Replacing $S$ by the Zariski closure of $q(W)$, we may assume that $q(W)\subset S$ is dominant.

 	To see the conclusion of the induction hypothesis, we assume that $f_S(Y)\subset V$ for some proper Zariski closed subset $V\subsetneqq S$.
 	Let $V_1,\ldots ,V_l$ be the irreducible components of $V$.
 	For $j=1,\ldots ,l$, we set $W_j=W\cap (A\times V_j)$.
 	Let $W_j^1,\ldots,W_j^{t_j}$ be the irreducible components of $W_j$.
 	Then by the induction hypothesis, there exists a finite subset $P_{V_j,W_j^i}\subset \mathcal{S}_0(A)$ such that if $f(Y)\sqsubset W_j^i$ and $e(f,W_j^i)\leq k$, then there exists $B\in P_{V_j,W_j^i}$ such that $f(Y)\sqsubset (W_j^i)^B$ and \eqref{eqn:20220921}.
 	Set $P_V=\cup_j\cup_iP_{V_j,W_j^i}$.
 	Then if $f_S(Y)\subset V$ and $e(f,W)\leq k$, then there exists $B\in P_{V}$ such that $f(Y)\sqsubset W^B$ and \eqref{eqn:20220921}.

 	Now we first consider the case that $S$ is smooth and $\mathrm{St}^0(W)=\{0\}$.
 	Let $f:Y\da W$ satisfy $e(f,W)=k$.
 	Suppose $f_S(Y)\subset S\backslash S^o$, then by the above consideration, there exists $B\in P_{S\backslash S^o}$ such that $f(Y)\sqsubset W^B$ and \eqref{eqn:20220921}.
 	So we consider the case $f_S(Y)\not\subset S\backslash S^o$.
 	We consider the first jet $j_1f:Y\da \overline{T}(A\times S)=A\times \overline{(\mathrm{Lie}A\times TS)}$.
 	Set $S'=\overline{\mathrm{Lie}A\times TS}$.
 	If $e(j_1f,W^{(1)})=k=e(f,W)$, then \cref{lem:202210041} yields 
 	$$T_{f_A}(r)=O(T_{f_{S}}(r))+O(\log r).$$
	This shows \eqref{eqn:20220921} for $B=\{0\}$.
	 	If $e(j_1f,W^{(1)})<k=e(f,W)$, then the induction hypothesis yields a finite set $P'_{k-1}\subset \mathcal{S}_0(A)$ such that there exists $B\in P'_{k-1}$ such that $j_1f(Y)\sqsubset (W^{(1)})^B$ and 
 	$$T_{q_B\circ f_A}(r)=O(N_{\ram\pi}(r)+\pN{f}(r)+T_{(j_1f)_{S'}}(r))+O(\log r)+o(T_f(r))||.$$
 	By \cref{lem:202210042}, we have 
 	$$T_{(j_1f)_{S'}}(r)=O(N_{\ram\pi}(r)+\pN{f}(r)+T_{f_S}(r))+O(\log r)+o(T_{f_A}(r))||.$$
 	Hence we have $f(Y)\sqsubset W^B$ and \eqref{eqn:20220921}.
 	We set $P=P'_{k-1}\cup P_{S\backslash S^o}\cup\{0\}$.
 	This concludes the proof of the claim if $S$ is smooth and $\mathrm{St}^0(W)=\{0\}$.
 	
 	Next we remove the assumption that $S$ is smooth.
 	Let $\widetilde{S}\to S$ be a smooth modification which is an isomorphism outside a proper Zariski closed set $E\subsetneqq S$.
 	If $f_S(Y)\subset E$, then there exists $B\in P_{E}$ such that $f(Y)\sqsubset W^B$ and \eqref{eqn:20220921}.
 	If $f_S(Y)\not\subset E$, then there exists a unique lift $f:Y\da A\times \widetilde{S}$.
 	Let $\widetilde{W}\subset A\times \widetilde{S}$ be the proper transform of $W$.
 	Then by the consideration above, there exists $P'\subset \mathcal{S}_0(A)$ such that $f(Y)\sqsubset \widetilde{W}^B$ and \eqref{eqn:20220921} for some $B\in P'$.
	 	We set $P=P'\cup P_E$ to conclude the proof of the claim when $\mathrm{St}^0(W)=\{0\}$.
 	
 	Finally we remove the assumption $\mathrm{St}^0(W)=\{0\}$.
 	Suppose that $\mathrm{St}^0(W)\not=\{0\}$.
 	Set $C=\mathrm{St}^0(W)$.
 	Let $f':Y\da W/C\subset (A/C)\times S$ be the induced map.
 	Then we have $e(f',W/C)<e(f,W)$.
 	Hence by the induction hypothesis, there exists $P'\subset \mathcal{S}_0(A/C)$ such that $f'(Y)\sqsubset (W/C)^{B'}$ and \eqref{eqn:20220921} for some $B'\in P'$.	 	We set $P$ by $B\in P$ iff $C\subset B$ and $B/C\in P'$.
 	This concludes the proof of the claim.
 	Thus the proof of \cref{pro:20220807} is completed.
 \end{proof}

 \begin{cor}\label{cor:20220806}
 	In Proposition \ref{pro:20220807}, we may take $P$ so that $\mathrm{St}^0(W)\subset B$ for all $B\in P$.
 	Moreover there exists a proper Zariski closed set $\Xi\subsetneqq W$ such that for every $f:Y\da W$, either one of the followings holds:
 	\begin{enumerate}
 		\item
 		$f(Y)\sqsubset \Xi$.
 		\item
 		$T_{q_{\mathrm{St}^0(W)}\circ f_A}(r)=O(N_{\ram\pi}(r)+\pN{f}(r)+T_{f_S}(r))+O(\log r)+o(T_f(r))||$,
 		where $q_{\mathrm{St}^0(W)}:A\to A/\mathrm{St}^0(W)$ is the quotient map
 	\end{enumerate}
 \end{cor}
 
 \begin{proof}
 	We apply Proposition \ref{pro:20220807} for $W/\mathrm{St}^0(W)\subset (A/\mathrm{St}^0(W))\times S$ to get $P_0\subset \mathcal{S}_0(A/\mathrm{St}^0(W))$.
 	Then we define $P\subset \mathcal{S}_0(A)$ by $B\in P$ iff $\mathrm{St}^0(W)\subset B$ and $B/\mathrm{St}^0(W)\in P_0$.
	We set $P'=P\backslash \{\mathrm{St}^0(W)\}$ and $\Xi=\cup_{B\in P'}W^B$.
 	Then $\Xi$ is a proper Zariski closed set.
 \end{proof}

 \subsection{Second main theorem with weak truncation}
 \label{sec:20221125}
 
 Let $\mathcal{S}(A)$ be the set of all positive dimensional semi-abelian subvarieties of $A$.
 Hence $\mathcal{S}(A)=\mathcal{S}_0(A)\backslash \{\{0\}\}$.

 \begin{proposition}\label{pro:202208062}
 Let $\overline{A}$ be an equivariant compactification of $A$ such that $\overline{A}$ is smooth and projective.
 	Let $W\subsetneqq \overline{A}\times S$ be a closed subscheme, where $S$ is a projective variety.
 	Then there exist a finite subset $P\subset \mathcal{S}(A)\backslash\{A\}$ and a positive integer $\rho\in\mathbb Z_{>0}$ with the following property:
 	Let $f:Y\da A\times S$ satisfies $f(Y)\not\sqsubset \mathrm{supp}W$ and $f_S(Y)\not\subset p(\mathrm{Sp}_AW)$, where $\mathrm{Sp}_AW=\cap_{a\in A}(a+W)\subset \overline{A}\times S$ and $p:\overline{A}\times S\to S$ is the second projection.
 	Then either one of the followings are true:
 	\begin{enumerate}
 		\item
 		There exists $B\in P$ such that
 		$$T_{q\circ f_A}(r)=O(T_{f_S}(r)+N_{\ram\pi}(r)+\pN{f}(r))+O(\log r)+o(T_{f}(r))||,$$
 		where $q:A\to A/B$ is the quotient map.
 		\item
 		The following estimate holds:
 		\begin{multline}
 		m_f(r,W)+N_f(r,W)-N^{(\rho)}_f(r,W)\\
 		=O(N_{\ram\pi}(r)
 		+\pN{f}(r)+T_{f_S}(r))+O(\log r)+o(T_f(r))||.
 		\end{multline}
 	\end{enumerate}
 \end{proposition}
 
 \begin{proof}
 For $W\subset \overline{A}\times S$, we set $W_o=W\cap (A\times S)$.
 	Given $f:Y\da A\times S$, we have $e(f,W_o)\in\{-1,0,\ldots,\dim A\}$.
 	The proof of \cref{pro:202208062} easily reduces to the following claim by letting $P=P_{\dim A}$, $\rho=\rho_{\dim A}$. 
 	
 	\begin{claim}
 		In \cref{pro:202208062}, let $\dim A=m$ and $\dim S=l$.
 		Let $k\in \{-1,0,\ldots,\dim A\}$.
 		Then there exist a finite subset $P_k\subset \mathcal{S}(A)\backslash \{A\}$ and a positive integer $\rho_k$ with the following property.
 		Let $f:Y\da A\times S$ with $e(f,W_o)\leq k$ satisfies $f(Y)\not\sqsubset W$ and $f_S(Y)\not\subset p(\mathrm{Sp}_AW)$.
 		Then either one of the assertions of \cref{pro:202208062} is true, where $P$ is replaced by $P_k$ in the first assertion and $\rho$ is replaced by $\rho_k$ in the second assertion.
 	\end{claim}
 	
 	We prove this claim by the induction on the triple $(m,k,l)\in\mathbb Z_{\geq 0}\times \mathbb Z_{\geq -1}\times\mathbb Z_{\geq 0}$ with the dictionary order.  
 	So we assume that the claim is true if $(\dim A,e(f,W_o),\dim S)<(m,k,l)$ and prove the case $(\dim A,e(f,W_o),\dim S)=(m,k,l)$.

	To see the conclusion of the induction hypothesis, we assume that $f:Y\da A\times S$ satisfies $f_S(Y)\subset V$ for some proper Zariski closed subset $V\subsetneqq S$.
 	Let $V_1,\ldots ,V_l$ be the irreducible components of $V$.
 	For $j=1,\ldots ,l$, we set $W_j=W\cap (A\times V_j)$.
 	Then by the induction hypothesis, there exist a finite subset $P_{j,k}\subset \mathcal{S}(A)\backslash \{A\}$ and $\rho_{j,k}\in\mathbb Z_{\geq 1}$ such that if $f:Y\da A\times S$ satisfies $e(f,W)= k$ and $f_S(Y)\subset V_j$, then either one of the assertions of \cref{pro:202208062} is true, where $P$ is replaced by $P_{j,k}$ in the first assertion and $\rho$ is replaced by $\rho_{j,k}$ in the second assertion.
 	Set $P_{V,k}=\cup_jP_{j,k}$ and $\rho_{V,k}=\max_{j}\rho_{j,k}$.
 	Then if $f_S(Y)\subset V$ and $e(f,W)= k$, either one of the assertions of \cref{pro:202208062} is true, where $P$ is replaced by $P_{V,k}$ in the first assertion and $\rho$ is replaced by $\rho_{V,k}$ in the second assertion.

 	Now we first assume the followings
	\begin{enumerate}
	\item
	$q(W)=S$, where $q:\overline{A}\times S\to S$ is the second projection,
	\item
	$W$ is reduced, and irreducible,
		\item
	$S$ is smooth.
		\end{enumerate}

 	First we consider the case $W_o\not=\emptyset$.
	Set $C=\mathrm{St}^0(W_o)$.
	We define $S^o\subset S$ from $W_o/C\subset (A/C)\times S$ (cf. \cref{lem:202210041}).
		We consider $f:Y\da A\times S$ with $e(f,W_o)=k$ such that $f(Y)\not\sqsubset W$ and $f_S(Y)\not\subset p(\mathrm{Sp}_AW)$.
		If $f_S(Y)\subset S\backslash S_o$, then by the above consideration, there exists $B\in P_{S\backslash S_o,k}$ such that the first assertion of \cref{pro:202208062}	is valid, or the second estimate of \cref{pro:202208062} holds for $\rho=\rho_{S\backslash S_o,k}$.
		So we assume that $f_S(Y)\not\subset S\backslash S_o$.
		We consider the first jet $j_1f:Y\da A\times S'$, where $S'=\overline{\mathrm{Lie}(A)\times TS}$.
		
		We consider the case $e(j_1f,W_o^{(1)})=e(f,W_o)$.		
		Note that $C=\mathrm{St}^0(W_o^{(1)})$, where $W_o^{(1)}\subset A\times S'$.
		Let $f':Y\to (A/C)\times S$ be the composition of $f$ and the projection $A\times S\to (A/C)\times S$.
	Then we have $e(f,W_o)=e(f',W_o/C)+\dim C$ and $e(j_1f,W_o^{(1)})=e(j_1f',(W_o/C)^{(1)})+\dim C$. 		
Hence $e(f',W_o/C)=e(j_1f',(W_o/C)^{(1)})$.
 Thus \cref{lem:202210041} yields that 
 \begin{equation}\label{eqn:202212015}
 T_{q_C\circ f_A}(r)=O(T_{f_S}(r))+O(\log r),
 \end{equation}
 where $q_C:A\to A/C$ is the quotient.
 Since we are assuming that $W\to S$ is dominant, we have $C\not=A$, for otherwise $f(Y)\sqsubset W$.
Thus if $\dim C>0$,  the first assertion of \cref{pro:202208062} is true, provided $C\in P_k$.
If $\dim C=0$, then \eqref{eqn:202212015} yields that $T_{f_A}(r)= O(T_{f_S}(r))+O(\log r)$.  
 	Hence $m_f(r,W)+N_f(r,W)=O(T_{f_S}(r))+O(\log r)$.
 	This is stronger than the second assertion of \cref{pro:202208062}.

	So we assume $e(j_1f,W_o^{(1)})<e(f,W_o)$.
 	Then the induction hypothesis yields $P_{k-1}'\subset \mathcal{S}(A)\backslash\{A\}$ and $\rho_{k-1}'$ such that either the first assertion of  \cref{pro:202208062} for $P=P_{k-1}'$ or the estimate
 	\begin{multline*}
 		m_{j_1f}(r,W_{\log}^{(1)})+N_{j_1f}(r,W_{\log}^{(1)})-N^{(\rho_{k-1}')}_{j_1f}(r,W_{\log}^{(1)})
 		\\
 		=O(N_{\ram\pi}(r)+\pN{f}(r)
 		+T_{j_1f_{S'}}(r))+O(\log r)+o(T_f(r))||
 	\end{multline*}
 	holds.
	Here $W_{\log}^{(1)}\subset \overline{A}\times S'=\overline{T}(\overline{A}\times S;\log \partial (A\times S))$ is defined in \cite[Section 5]{Yam04} so that $W_{\log}^{(1)}\cap (A\times S')=W_o^{(1)}$.
			By the same argument for the proof of \cite[Lemma 5]{Yam15}, using \cite[Thm 5.1.7]{Yam04}, we have
 	\begin{multline*}
 		m_f(r,W)+N_f(r,W)-N^{(\rho_{k-1}'+1)}_f(r,W)
 		\\
 		\leq m_{j_1f}(r,W_{\log}^{(1)})+N_{j_1f}(r,W_{\log}^{(1)})-N^{(\rho_{k-1}')}_{j_1f}(r,W_{\log}^{(1)})+N_{\ram\pi}(r)+o(T_{f}(r))||.
 	\end{multline*}
	 	By \cref{lem:202210042}, we have 
 	$$T_{(j_1f)_{S'}}(r)=O(T_{f_S}(r)+N_{\ram\pi}(r)+\pN{f}(r))+O(\log r)+o(T_f(r))||.$$
	Thus we get
\begin{multline*}
 		m_f(r,W)+N_f(r,W)-N^{(\rho_{k-1}'+1)}_f(r,W)=O(T_{f_S}(r)+N_{\ram\pi}(r)+\pN{f}(r))+O(\log r)+o(T_f(r))||. 	
		\end{multline*}		
This concludes the proof of the induction step for our case $W_o\not=\emptyset$.
Here we set $P_k=(P_{S\backslash S_o,k}\cup P_{k-1}' \cup \{C\})\backslash \{0\}$ and $\rho_k=\rho_{S\backslash S_o,k}+\rho_{k-1}'+1$.

 	Next we consider the case $W_o=\emptyset$.
 	In this case, we have $W\subset \partial A\times S$ and $e(f,W_o)=-1$.
 	Let $I\subset A$ be the isotropy group for $W$ and let $D\subset \partial A$ be the irreducible component of $\partial A$ such that $W\subset D\times S$.
 	Then $D$ is an equivariant compactification of $A/I$.
 	By \cite[Lem A.11]{Y22}, there exist an $A$-invariant Zariski open set $U\subset \overline{A}$ and an equivariant map $\psi:U\to D$ such that $D\subset U$ and $\psi$ is an isomorphism over $D\subset U$.
 	We define $g:Y\da (A/I)\times S\subset D\times S$ by $g=(\psi\circ f_A,f_S)$.
 	By $\mathrm{Sp}_A(W)=\mathrm{Sp}_{A/I}(W)$, we have $g_S(Y)\not\subset p(\mathrm{Sp}_{A/I}W)$.
Suppose $g(Y)\not\sqsubset W$.
Note that $\dim (A/I)<\dim A$.
	Hence by the induction hypothesis, there exist $P'\subset \mathcal{S}(A/I)\backslash \{A/I\}$ and $\rho'\in\mathbb Z_{>0}$, which are independent of the choice of $f$, such that either the first assertion of \cref{pro:202208062} or the following estimate 
 	$$m_g(r,W)+N_g(r,W)-N^{(\rho')}_g(r,W)=O(N_{\ram\pi}(r)+\pN{g}(r)
 	+T_{f_S}(r))+O(\log r)+o(T_f(r))||
 	$$
 	holds.
 	By $\pN{g}(r)\leq \pN{f}(r)$, $m_f(r,W)\leq m_g(r,W)$ and $\mathrm{ord}_yf^{*}W\leq \mathrm{ord}_yg^{*}W$ for all $y\in Y$, this estimate implies the second assertion of \cref{pro:202208062} for $\rho=\rho'$.
 	If $g(Y)\sqsubset W$, then we apply \cref{pro:20220807} to get $P''\subset \mathcal{S}_0(A/I)$, which is independent of the choice of $f$.
 	Then there exists $B\in P''$ such that 
 	$$T_{q\circ g_{A/I}}(r)=O(T_{f_S}(r)+N_{\ram\pi}(r)+\pN{f}(r))+O(\log r)+o(T_{f}(r))||,
 	$$
	where $q:A/I\to (A/I)/B$ is the quotient.
 	We note $B\not=A/I$, for otherwise we have $g(Y)\sqsubset W^{A/I}$, which contradicts to $g_S(Y)\not\subset p(\mathrm{Sp}_{A/I}W)$.
 	We define $P_k$ by $B\in P_k$ iff $I\subset B$ and $B/I\in P'\cup (P''\backslash \{A/I\})$.
 	This concludes the proof of the induction step for our case $W_o=\emptyset$.
Thus we have completed the induction step for the case that $W\to S$ is dominant, $W$ is irreducible and reduced, and $S$ is smooth.

	We remove these three assumptions on $W$ and $S$.
 First we remove the first assumption.
	Suppose $q(W)\subsetneqq S$, where $q:\overline{A}\times S\to S$ is the second projection.
	If $f_{S}(Y)\not\subset q(W)$, then we have 
	$$m_f(r,W)+N_f(r,W)=O(T_{f_S}(r))+O(\log r).$$
	This is stronger than the second assertion of \cref{pro:202208062}.
	Thus we may consider the case $f_{S}(Y)\subset q(W)$.
	We replace $S$ by $q(W)$.
	Then, by the induction hypothesis, we get our claim.	
	Hence the first assumption is removed.	
	
	Next we remove the second assumption.
	Let $W_1,\ldots W_n$ be the irreducible components of $\mathrm{supp} W$.
 	Then for each $j\in \{1,\ldots,n\}$, by the argument above, there exist a finite subset $P_{j,k}\subset \mathcal{S}(A)\backslash \{A\}$ and a positive integer $\rho_{j,k}$ such that for every $f:Y\da A\times S$ with $e(f,W_o)\leq k$, either the first assertion of \cref{pro:202208062} for $P=P_{j,k}$ or the following estimate holds:
 	\begin{multline}\label{eqn:9.3.1}
 		m_f(r,W_j)+N_f(r,W_j)-N^{(\rho _{j,k})}_f(r,W_j)
 		\\
 		=O(T_{f_S}(r)+N_{\ram \pi _Y}(r)+\pN{f}(r))+O(\log r)+o( T_f(r)\ ||.
 	\end{multline}
 	There is a positive integer $l$ such that
 	$$\big( \mathcal I_{W_1}\cdots \mathcal I_{W_n}\big) ^l\subset \mathcal I_W,
 	$$
 	where $\mathcal I_W\subset \mathcal O_{\overline{A}\times S}$ (resp. $\mathcal I_{W_i}$) is the defining ideal sheaf of $W$ (resp. $W_i$).
 	Then we have
 	\begin{equation}\label{eqn:9.3.2}
 		m_f(r,W)\leq l\sum_{i=1}^nm_f(r,W_i).
 	\end{equation}
 	For $y\in Y$, we have
 	$$\ord _yf^*W\leq l\sum_{i=1}^n\ord _yf^*W_i$$
 	so setting $\widetilde{\rho}_k=\rho _{1,k}+\cdots +\rho _{n,k}$, we get
 	\begin{equation*}
 		\begin{split}
 			\max \{ 0,\ord _yf^*W-l\widetilde{\rho}_k \}&\leq \max \left\{ 0,l\sum_{i=1}^n\left( \ord _yf^*W_i-\rho _{i,k}\right) \right\}\\
 			&\leq l\sum_{i=1}^n\max \left\{ 0,\left( \ord _yf^*W_i-\rho _{i,k}\right) \right\} .
 		\end{split}
 	\end{equation*}
 	Hence we get
 	\begin{equation}\label{eqn:9.3.3}
 		N_f(r,W)-N^{(l\widetilde{\rho}_k)}_f(r,W)\leq  l\sum_{i=1}^n\left( N_f(r,W_i)-N^{(\rho _{i,k})}_f(r,W_i)\right) .
 	\end{equation}
 	By \eqref{eqn:9.3.1}, \eqref{eqn:9.3.2}, \eqref{eqn:9.3.3}, we get
 	\begin{equation*}
 		m_f(r,W)+N_f(r,W)-N^{(l\widetilde{\rho}_k )}_f(r,W) \\
 		\leq O(T_{f_S}(r)+N_{\ram\pi _Y}(r)+\pN{f}(r))+O(\log r)+o(T_f(r))\ ||.
 	\end{equation*}
 	Thus we have removed the assumption that $W$ is irreducible and reduced.
 	Here we set $\rho_k=l\widetilde{\rho}_k$ and $P_k=P_{1,k}\cup\cdots\cup P_{n,k}$.

The assumption that $S$ is smooth is removed similarly as in the proof of \cref{pro:20220807}.
 	This completes the induction step of the proof of the claim.
 	Thus the claim is proved.
 	The proof of \cref{pro:202208062} is completed.
 \end{proof}

 \subsection{Intersection with higher codimensional subvarieties}

 \begin{lem}\label{lem:202209151}
 Let $L$ be an ample line bundle on $\overline{A}$, where $\overline{A}$ is a smooth equivariant compactification.
 	Let $Z\subset A\times S$ be a closed subscheme whose codimension is greater than one, where $S$ is a projective variety.
 	Let $\varepsilon>0$.
 	Then there exist a finite subset $P\subset \mathcal{S}(A)\backslash\{A\}$ and a proper Zariski closed subset $E\subsetneqq S$ with the following property:
 	Let $f:Y\da A\times S$ satisfies $f(Y)\not\sqsubset \mathrm{supp} Z$.
 	Then either one of the followings is true:
 	\begin{enumerate}
 		\item
 		$
 		N^{(1)}_f(r,Z)\leq \varepsilon T_{f_A}(r,L)+O_{\varepsilon}(N_{\ram\pi}(r)+\pN{f}(r)
 		+T_{f_S}(r))+O(\log r)+o(T_f(r))||_{\varepsilon}$.
 		\item
 		$f_S(Y)\subset E$.
 		\item
 		There exists $B\in P$ such that 
 		$$
 		T_{q_B\circ f_A}(r)=O(N_{\ram\pi}(r)+\pN{f}(r)+T_{f_S}(r))+O(\log r)+o(T_f(r))||, 
 		$$
 		where $q_B:A\to A/B$ is the quotient map.
 	\end{enumerate}
 \end{lem}

 \begin{proof}
 We first consider the case that 
\begin{itemize}
\item
$q(Z)\subset S$ is dense, where $q:A\times S\to S$ is the second projection,
\item
$Z$ is irreducible,

\item
$S$ is smooth.
\end{itemize}
 	
	We use higher jet spaces.
 	Let $J_l(A\times S)$ be the $l$-th jet space.
	We have the natural splitting $J_l(A\times S)=A\times (J_l(A\times S)/A)$ induced from the splitting $TA=A\times\mathrm{Lie}A$.
	There exists a partial compactification $J_l(A\times S)\subset \bar{J}_l(A\times S)$ such that the natural map $\bar{J}_l(A\times S)\to A\times S$ is projective and $\bar{J}_l(A\times S)$ is $\mathbb Q$-factorial (\cite[2.4]{Yam04}).	
We have the natural splitting $\bar{J}_l(A\times S)=A\times (\bar{J}_l(A\times S)/A)$.
When $l=1$, this reduces to \eqref{eqn:20221203}.

 	\begin{claim}\label{claim:20221021}
	There exist a sequence of positive integers $n(1),n(2),n(3),\cdots$ and an ample line bundle $L_o$ on $\overline{A}$ with the following conditions:
\renewcommand{\labelenumi}{(\theenumi )}
\begin{enumerate}
\item $\frac{n(l)}{l}\to 0 $ when $l\to \infty$.
\item For $l\geq 1$, let $S_l\subset J_l(A\times S)/A$ be an irreducible Zariski closed subset whose
image under the projection $J_l(A\times S)/A\to S$ is Zariski dense.
Let $\overline{S_l}\subset \bar{J}_l(A\times S)/A$ be the compactification of $S_l$ in $\bar{J}_l(A\times S)/A$.
Then there exists an effective Cartier divisor $F_l\subset \overline{A}\times \overline{S_l}$ with the following properties:
\begin{enumerate}
\item $F_l$ corresponds to a global section of
$$H^0(\overline{A}\times \overline{S_l},p_l^*L_o^{\otimes n(l)}\otimes q _l^*M_l),$$
where $p_l:\overline{A}\times \overline{S_l}\to \overline{A}$ is the first projection, $q _l:\overline{A}\times \overline{S_l}\to \overline{S_l}$ is the second projection, and $M_l$ is an ample line bundle on $\overline{S_l}$.

\item Let $f:\mathbb D \to A\times S$ be a holomorphic map such that $j_l(f)(\mathbb D )\subset A\times S_l$ and $j_l(f)(\mathbb D )\not\subset F_l$.
Assume $f(0)\in Z$, then $\ord _{0}j_{l}(f)^*F_l\geq l+1$.
\end{enumerate}
\end{enumerate}
 	\end{claim}
 	
 	\begin{proof}
	When $A$ is compact, this is \cite[Prop 6]{Yam15}.	
	The same proof of \cite[Prop 6]{Yam15} works for our situation.
	See also \cite[Lemma 6.5.42]{NW13}.
		\end{proof}

 	Now we are given $\varepsilon>0$.
	Let $\mu>0$ be a positive integer such that $L^{\otimes\mu}\otimes L_o^{-1}$ is ample.
 	We take $l\in\mathbb N$ such that 
 	\begin{equation}\label{eqn:201308171}
 		n(l)/(l+1)<\varepsilon/\mu .
 	\end{equation}
 	Let $\mathcal{V}$ be the set of all irreducible Zariski closed subsets of $J_l(A\times S)/A$.
 	Let $\mathcal{W}\subset \mathcal{V}$ be the subset of $\mathcal{V}$ which consists of $W\in \mathcal{V}$ whose image under the projection $J_l(A\times S)/A\to S$ is Zariski dense in $S$.
 	We define the sequence
 	$\mathcal{V}_1, \mathcal{V}_2, \ldots$
 	of subsets of $\mathcal{V}$ and the sequence
 	$\mathcal{W}_1, \mathcal{W}_2, \ldots$
 	of subsets of $\mathcal{W}$
 	by the following inductive rule. 
 	Put $\mathcal{V}_1 =\mathcal{W}_1= \{ J_l(A\times S)/A\}$. 
 	For each $W\in \mathcal{W}_i$, let $F_{W}\subset A\times W$ be the divisor obtained in \cref{claim:20221021}.
 	Set $(F_W)^A=A\times W'$, where $W'\subset W$ is a proper Zariski closed subset.
 	Put
 	$$
 	\mathcal{V}_{i+1}=\bigcup_{W\in \mathcal{W}_i} \{ V\in \mathcal{V};\text{$V$ is an irreducible component of $W'$}\} ,
 	$$
 	$$
 	\mathcal{W}_{i+1}=\mathcal{V}_{i+1}\cap \mathcal{W}.
 	$$
 	Since the number of the irreducible components of $W'$ is finite,
 	each $\mathcal{V}_i$ is a finite set. 
 	Since $\dim W'  < \dim W$, we have $\mathcal{V}_i=\mathcal{W}_i = \emptyset$ for
 	$i \geq \dim (J_l(A\times S)/A)+2$.
 	We apply \cref{pro:20220807} for $F_W\subset A\times W$ to get $P_{W}\subset \mathcal{S}_0(A)$.
 	(More precisely, we apply \cref{pro:20220807} for each irreducible component of $F_W$.) 
 	Set
 	$$
 	P=\bigcup_{i}\bigcup_{W_i\in \mathcal{W}_i}P_{W_i}\backslash \{\{0\},A\}.
 	$$
 	Then $P\subset \mathcal{S}(A)-\{ A\}$ is a finite subset.
 	For $V\in \mathcal{V}_i\backslash \mathcal{W}_i$, let $S_{V}\subset S$ be the Zariski closure of the image of $V$ under the projection $J_l(A\times S)/A \to S$.
 	Then by the construction of $\mathcal{W}_i$, we have $S_{V}\not= S$.
 	Set
 	$$
 	E=\bigcup_{i}\bigcup_{V\in \mathcal{V}_i\backslash \mathcal{W}_i}S_{V}.
 	$$
 	Then $E\subset S$ is a proper Zariski closed subset.
 	
 	\par

 	Now let $f:Y\da A\times S$ with $f(Y)\not\sqsubset Z\cup q^{-1}(E)$. 
	Let $q':J_l(A\times S)\to J_l(A\times S)/A$ be the projection under the splitting $J_l(A\times S)=A\times (J_l(A\times S)/A)$.
	 	We may take $W_i\in \mathcal{W}_i$ such that $q'\circ j_l(f)(Y)\sqsubset W_i$ but $q'\circ j_l(f)(Y)\not\sqsubset W_{i+1}$ for all $W_{i+1}\in \mathcal{W}_{i+1}$.
 	Then we have
 	\begin{equation}\label{eqn:130504}
 		j_l(f)(Y)\not\sqsubset (F_{W_i})^A.
 	\end{equation}
 	We consider the three possible cases.
 	
 	\par
 	
 	{\it Case 1.} $j_l(f)(Y)\not\sqsubset F_{W_i}$.
 	We remark that $j_l(f)$ hits the boundary of $\bar{J}_l(A\times S)$ only at the ramification points of $\pi _Y:Y\to \mathbb C$ and the points in $\bar{f}^{-1}(\partial A\times S)$, where $\bar{f}:Y\to \overline{A}\times S$ is the extension of $f$.
 	Hence, we have
 	$$(l+1)N^{(1)}_f(r,Z)\leq N_{j_l(f)}(r,F_{W_i})+(l+1)N_{\ram \pi _Y}(r)+(l+1)\pN{f}(r).$$
 	We recall that $F_{W_i}$ corresponds to $H^0(\overline{A}\times W_i,p_l^*L_o^{\otimes n(l)}\otimes q _l^*M_l)$, where $M_l$ is an ample line bundle on $W_i$ and $q_l:A\times W_i\to W_i$ is the second projection.
 	Hence by the Nevanlinna inequality (cf. \eqref{eqn:20221120}), we have
 	$$
 	N_{j_l(f)}(r,F_{W_i})\leq n(l)T_f(r, p_l^*L_o)+T_{j_l(f)}(r,q_l^*M_l)+O(\log r).
 	$$
 	By the similar argument for the proof of \cref{lem:202210042}, we get 	
	\begin{equation}\label{eqn:12092710}
 		T_{ j_l(f)}(r,q_l^*M_l)=O (T_{f_S}(r)+N_{\ram \pi _Y}(r)+\pN{f}(r))+O(\log r)+o(T_f(r))\ ||.
 	\end{equation}
 	Hence we obtain
 	\begin{equation*}\label{eqn:1209271}
 		N^{(1)}_f(r,Z)\leq \varepsilon T_{f_A}(r,L)+O (T_{f_S}(r)+N_{\ram \pi _Y}(r)+\pN{f}(r))+O(\log r)+o(T_f(r))\ ||.
 	\end{equation*}

 	\par

 	{\it Case 2. }$j_l(f)(Y)\sqsubset F_{W_i}$.
 	We apply \cref{pro:20220807} to get $B\in P_{W_i}$.
 	By $j_lf(Y)\not\sqsubset (F_{W_i})^A$, we have $B\not=A$.
 	
 	\par 
 	
 	{\it Case 2-1.}
 	$B=\{0\}$.
 	Then by \cref{pro:20220807}, we have
 	$$
 	T_{f_A}(r)=O(N_{\ram\pi}(r)+\pN{f}(r)+T_{j_lf_{W_i}}(r))+O(\log r)+o(T_f(r))||.
 	$$
 	Thus by \eqref{eqn:12092710}, we have
 	$$
 	T_f(r)=O (T_{f_S}(r)+N_{\ram \pi _Y}(r)+\pN{f}(r))+O(\log r)+o(T_f(r))\ ||.
 	$$
 	Hence
 	$$
 	N_f(r,Z)=O (T_{f_S}(r)+N_{\ram \pi _Y}(r)+\pN{f}(r))+O(\log r)+o(T_f(r))\ ||.
 	$$
 	
 	\par

 	{\it Case 2-2.} $B\not=\{0\}$.
	Then $B\in P$.
 	By \cref{pro:20220807}, we have
 	$$
 	T_{q_B\circ f_A}(r)
 	=O(N_{\ram\pi}(r)+\pN{f}(r)+T_{j_lf_{W_i}}(r))+O(\log r)+o(T_f(r))||.
 	$$
 	Thus by \eqref{eqn:12092710}, we have
 	$$
 	T_{q_B\circ f_A}(r)=O(N_{\ram\pi}(r)+\pN{f}(r)+T_{f_S}(r))+O(\log r)+o(T_f(r))||.
 	$$
 	
 	\par
 	
 	We combine the three cases above to conclude the proof of the lemma (\cref{lem:202209151}) under the three assumptions above.
	
	We remove the three assumptions.
	First suppose $q(Z)\subsetneqq S$.
	We set $E=\overline{q(Z)}$.
	Then $E\subsetneqq S$ is a Zariski closed set.
	If $f_S(Y)\not\subset E$, then we have $m_f(r,Z)+N_f(r,Z)=O(T_{f_S}(r))+O(\log r)$.
	This is stronger than the first assertion in \cref{lem:202209151}.
	Thus we have removed the assumption that $q(Z)\subset S$ is dense.	
	
	Next we remove the assumption that $Z$ is irreducible.
	Let $Z=Z_1\cup \cdots\cup Z_k$ be the irreducible decomposition.
	We apply \cref{lem:202209151} for $Z_i$ and $\varepsilon/k$ to get $P_i\subset \mathcal{S}(A)\backslash \{A\}$ and $E_i\subsetneqq S$.
	We set $P=\cup_{i}P_i$ and $E=\cup_iE_i$.
	Then $P\subset \mathcal{S}(A)\backslash\{A\}$ is a finite subset and $E\subsetneqq S$ is a proper Zariski closed subset.
	If $f:Y\da A\times S$ does not satisfy the second and the third assertions in \cref{lem:202209151}, then we have 
	$$
 		N^{(1)}_f(r,Z_i)\leq (\varepsilon/k) T_{f_A}(r,L)+O_{\varepsilon}(N_{\ram\pi}(r)+\pN{f}(r)
 		+T_{f_S}(r))+O(\log r)+o(T_f(r))||_{\varepsilon}.
		$$
		By $N^{(1)}_f(r,Z)\leq \sum_{i=1}^kN^{(1)}_f(r,Z_i)$, we have removed the assumption that $Z$ is irreducible.	

Finally we remove the assumption that $S$ is smooth.
Let $E_o\subset S$ be the singular locus of $S$ and let $\tau:S'\to S$ be a smooth modification, which is isomorphic outside $E_o$.		Let $Z'\subset A\times S'$ be the Zariski closure of $Z\cap (A\times (S\backslash E_o))$.
Then the codimension of $Z'\subset A\times S'$ is greater than one.
We apply  \cref{lem:202209151} for $Z'$ to get $P\subset \mathcal{S}(A)\backslash \{A\}$ and $E'\subsetneqq S'$.
We set $E=E_o\cup\tau(E')$.
If $f:Y\da A\times S$ does not satisfy the second and the third assertions in \cref{lem:202209151}, then we have a lift $f':Y\to A\times S'$ and 
$$
 		N^{(1)}_{f'}(r,Z')\leq \varepsilon T_{f_A}(r,L)+O_{\varepsilon}(N_{\ram\pi}(r)+\pN{f}(r)
 		+T_{f_S}(r))+O(\log r)+o(T_f(r))||_{\varepsilon}.
		$$	
		By 
		$$N^{(1)}_f(r,Z)\leq N^{(1)}_{f'}(r,Z')+N^{(1)}_{f_S}(r,E_o)\leq N^{(1)}_{f'}(r,Z')+O(T_{f_S}(r))+O(\log r),$$
		we have removed the assumption that $S$ is smooth.
		The proof of \cref{lem:202209151} is completed.		\end{proof}
 
 \medskip

 Let $Z\subset A\times S$ be a closed subscheme.
 For $f:Y\da A\times S$, we recall $e(f,Z)$ from \cref{subsec:4.2}.

 \begin{proposition}\label{pro:202208061}
 	Let $Z\subset A\times S$ be a closed subscheme, where $S$ is a projective variety.
	Let $L$ be an ample line bundle on $\overline{A}$, where $\overline{A}$ is a smooth equivariant compactification. 	Let $\varepsilon>0$.
 	Then there exists a finite subset $P\subset \mathcal{S}(A)\backslash\{A\}$ with the following property:
 	Let $f:Y\da A\times S$ satisfies $f(Y)\not\sqsubset Z$.
 	Then either one of the followings is true:
 	\begin{enumerate}
 		\item
 		$
 		N^{(1)}_f(r,Z)\leq \varepsilon T_{f_A}(r,L)+O_{\varepsilon}(N_{\ram\pi}(r)+\pN{f}(r)
 		+T_{f_S}(r))+O(\log r)+o(T_f(r))||_{\varepsilon}$.
 		\item
 		$e(f,Z)\geq \dim A-1$.
 		\item
 		There exists $B\in P$ such that 
 		$$
 		T_{q_B\circ f_A}(r)=O(N_{\ram\pi}(r)+\pN{f}(r)+T_{f_S}(r))+O(\log r)+o(T_f(r))||, 
 		$$
 		where $q_B:A\to A/B$ is the quotient map.
 	\end{enumerate}
 \end{proposition}
 
 \begin{proof}
 	We prove by Noether induction on $S$.
 	So we assume that the the assertion is true for all irreducible $V\subsetneqq S$ and $Z|_{A\times V}\subset A\times V$.
 We are given $\varepsilon>0$.
 
 We set 
 $$
 \Psi(A,Z)=\{ x\in A\times S;\ \dim ((x+A)\cap Z)\geq \dim A-1\}.
 $$
 Note that $\Psi(A,Z)/A\subset S$ is a constructible set.

 First we cosnider the case that $\Psi(A,Z)/A\subset S$ is Zariski dense.
 Then there exists a non-empty Zariski open set $U\subset S$ such that $U\subset \Psi(A,Z)/A$.
 Set $V=S\backslash U$.
 Then $V\subsetneqq S$ is a proper Zariski closed subset.
 If $f_S(Y)\not\subset V$, then we have $e(f,Z)\geq \dim A-1$.
 Hence the second assertion of  \cref{pro:202208061} is valid.
 If $f_S(Y)\subset V$, then the induction hypothesis yields a finite subset $P'\subset \mathcal{S}(A)\backslash\{A\}$ so that one of the three assertions in \cref{pro:202208061} is valid.
Hence our proposition is proved by setting $P=P'$.

 Next we consider the case $V_o=\overline{\Psi(A,Z)/A}\subsetneqq S$.
 Let $Z'\subset A\times S$ be the Zariski closure of $Z\cap (A\times (S\backslash V_o))$.
 Then the codimension of $Z'\subset A\times S$ is greater than one.
 We apply \cref{lem:202209151} for $Z'\subset A\times S$ to get $E\subsetneqq S$ and $P'\subset \mathcal{S}(A)\backslash \{A\}$.
 	Then $E\subsetneqq S$.
 	Let $V_1,\ldots,V_l$ be the irreducible components of $E\cup V_o$.
 	By the induction hypothesis applied for each $V_i\subsetneqq S$, we get $P_i\subset \mathcal{S}(A)\backslash\{ A\}$.
 	We set $P=P'\cup \cup_iP_i$.
 	Now let $f:Y\da A\times S$ satisfies $f(Y)\not\sqsubset Z$.
 	Assume that the second and the third assertions of  \cref{pro:202208061} are not valid.
 If $f_S(Y)\subset V_i$ for some $V_i$, then by the induction hypothesis, we get the first assertion of \cref{pro:202208061}.
Hence we assume that $ f_S(Y)\not\subset V_i$ for all $V_i$.
 	Then we have $f_S(Y)\not\subset E$.
 	By \cref{lem:202209151}, we get
 	$$
 		N^{(1)}_{f}(r,Z')\leq \varepsilon T_{f_A}(r,L)+O_{\varepsilon}(N_{\ram\pi}(r)+\pN{f}(r)
 		+T_{f_S}(r))+O(\log r)+o(T_f(r))||_{\varepsilon}.
		$$	
		By 
		$$N^{(1)}_f(r,Z)\leq N^{(1)}_{f}(r,Z')+N^{(1)}_{f_S}(r,V_o)\leq N^{(1)}_{f}(r,Z')+O(T_{f_S}(r))+O(\log r),$$
		we get the first assertion of \cref{pro:202208061} for the case that $ f_S(Y)\not\subset V_i$ for all $V_i$.
		Thus the proof is completed.
 \end{proof}

 \subsection{Second main theorem with truncation level one}\label{subsec:4.7a}
 In this subsection, we assume that $S$ is smooth and projective.
 In particular, all divisors on $A\times S$ are Cartier divisors.
 
 \begin{lem}\label{lem:20220915}
  	Let $D\subset A\times S$ be a reduced divisor.
	Let $L$ be an ample line bundle on $\overline{A}$, where $\overline{A}$ is a smooth equivariant compactification. 	Let $\varepsilon>0$.
 	Then there exist a finite subset $P\subset \mathcal{S}(A)\backslash\{A\}$ and a proper Zariski closed subset $E\subsetneqq S$ with the following property:
 	Let $f:Y\da A\times S$ satisfies $f(Y)\not\sqsubset D$.
 	Then either one of the followings is true:
 	\begin{enumerate}
 		\item
 		$
 		N^{(2)}_f(r,D)-N^{(1)}_f(r,D)\leq \varepsilon T_{f_A}(r,L)+O_{\varepsilon}(N_{\ram\pi}(r)+\pN{f}(r)
 		+T_{f_S}(r))+O(\log r)+o(T_f(r))||_{\varepsilon}$.
 		\item
 		$f(Y)\subset A\times E$.
 		\item
 		There exists $B\in P$ such that 
 		$$
 		T_{q_B\circ f_A}(r)=O(N_{\ram\pi}(r)+\pN{f}(r)+T_{f_S}(r))+O(\log r)+o(T_f(r))||, 
 		$$
 		where $q_B:A\to A/B$ is the quotient map.
 	\end{enumerate}
 \end{lem}

 \begin{proof}
 	We prove in the following three cases.
 	
 	{\it Case 1.}\ 
 	$D$ is irreducible and $\mathrm{St}^0(D)=\{0\}$.
 	We recall a Zariski open set $S^o\subset S$ (cf. \cref{lem:202210041}).
 	By replacing $S^o$ by a smaller non-empty Zariski open set, we may assume that $D_s\subset A$ is a reduced divisor for all $s\in S^o$.
 	Set $E=S\backslash S^o$.
	Set $S'=\overline{(\mathrm{Lie}A\times TS)}$. 	
	We apply \cref{pro:202208061} for $D^{(1)}\subset A\times S'$ and $\varepsilon>0$ to get a finite subset $P\subset \mathcal{S}(A)\backslash\{A\}$.
 	
 	Let $f:Y\da A\times S$ satisfies $f(Y)\not\sqsubset D$.
 	We assume that the assertions 2 and 3 of the statement do not occur.
	Then $e(f,D)=\dim A-1$.
 	If $e(j_1f,D^{(1)})= \dim A-1$, then by \cref{lem:202210041}, we have $T_f(r)=O(T_{f_{S}}(r))+O(\log r)$.
 	Hence
 	$$
 	m_f(r,D)+N_f(r,D)=O(T_{f_S}(r))+O(\log r).
 	$$
 	This is stronger than the assertion 1.
 	Thus in the following, we assume $e(j_1f,D^{(1)})<\dim A-1$.
 	Then by \cref{pro:202208061}, we get
 	$$
 	N^{(1)}_{j_1f}(r,D^{(1)})\leq \varepsilon T_{f_A}(r,L)+O_{\varepsilon}(N_{\ram\pi}(r)+\pN{f}(r)
 	+T_{(j_1f)_{S'}}(r))+O(\log r)+o(T_f(r))||_{\varepsilon}.
 	$$
 	By  \cref{lem:202210042} and
	 	$$
 	N^{(2)}_f(r,D)-N^{(1)}_f(r,D)\leq N^{(1)}_{j_1f}(r,D^{(1)})+N_{\ram\pi}(r)+\pN{f}(r),
 	$$
 	we conclude the proof in the case 1.
 	
 	{\it Case 2.}\
 	$D$ is irreducible, but $\mathrm{St}^0(D)$ is general.
 	We set $B=\mathrm{St}^0(D)$.
 	We apply the case 1 above for $A'=A/B$, $D'=D/B$ to get $P'\subset \mathcal{S}(A')\backslash\{ A'\}$ and $E\subsetneqq S$.
 	We define $P\subset \mathcal{S}(A)\backslash \{A\}$ by $C\in P$ iff $B\subset C$ and $C/B\in P'$.
 	
 	{\it Case 3.}\
 	We treat the general case.
 	Let $D=D_1+\cdots+D_k$ be the irreducible decomposition.
 	Set $Z_{ij}=D_i\cap D_j$.
 	Then $Z_{ij}\subset A\times S$ has codimension greater than one.
 	We apply \cref{pro:202208061} for $Z_{ij}$ to get $P_{ij}\subset \mathcal{S}(A)\backslash \{A\}$ and $E_{ij}\subsetneqq S$ such that $\overline{\Psi(A,Z_{ij})}=A\times E_{ij}$.
 	We apply the case 2 above for each $D_i$ to get $P_i\subset \mathcal{S}(A)\backslash \{A\}$ and $E_i\subsetneqq S$.
 	Then we set $P=\cup_iP_i\cup \cup_{i,j}P_{ij}$ and $E=\cup_{i}E_i\cup \cup_{i,j}E_{ij}$.
 \end{proof}
 
 \begin{lem}\label{lem:20220909}
 	Let $D\subset A\times S$ be a reduced divisor.
	Let $\overline{A}$ be a smooth equivariant compactification such that $p(\mathrm{Sp}_A\overline{D})\subsetneqq S$ is a proper Zariski closed set, where $p:\overline{A}\times S\to S$ is the second projection and $\overline{D}\subset \overline{A}\times S$ is the Zariski closure.
	Let $L$ be an ample line bundle on $\overline{A}$.
 Let $\varepsilon>0$.
 	Then there exist
 		a finite subset $P\subset \mathcal{S}(A)\backslash \{ A\}$, and a proper Zariski closed subset $E\subsetneqq S$
 	such that for every $f:Y\da A\times S$ with $f(Y)\not\sqsubset D$, either one of the followings holds:
 	\begin{enumerate}
 		\item
 		There exists $B\in P$ such that 
 		$$
 		T_{q_B\circ f_A}(r)=O(N_{\ram\pi}(r)+\pN{f}(r)+T_{f_S}(r))+O(\log r)+o(T_f(r))||, 
 		$$
 		where $q_B:A\to A/B$ is the quotient map.
 		\item
 		$f(Y)\sqsubset A\times E$.
 		\item
 		\begin{multline*}
 		m_f(r,\overline{D})+N_f(r,\overline{D})\leq N^{(1)}_f(r,\overline{D})
 			+\varepsilon T_{f_A}(r,L)\\
			+O_{\varepsilon}(N_{\ram\pi}(r)+\pN{f}(r)
 			+T_{f_S}(r))+O(\log r)||.
			\end{multline*}
 	\end{enumerate}
 \end{lem}

 \begin{proof}
 	We apply \cref{pro:202208062} for $\overline{D}\subset \overline{A}\times S$ to get a finite subset $P_1\subset \mathcal{S}(A)\backslash\{A\}$ and $\rho\in\mathbb Z_{>0}$.
 	We apply \cref{lem:20220915} for $D\subset A\times S$ and $\varepsilon'=\varepsilon/(\rho-1)$ to get a finite subset $P_2\subset \mathcal{S}(A)\backslash\{A\}$ and $E'\subsetneqq S$.
 	We set $P=P_1\cup P_2$ and $E=p(\mathrm{Sp}_A\overline{D})\cup E'$.

 	Now let $f:Y\da A\times S$ satisfies $f(Y)\not\sqsubset D$.
 	We assume that the assertions 1 and 2 of our lemma do not hold.
 	By \cref{pro:202208062}, we get
 	\begin{equation*}
 		m_f(r,D)+N_f(r,D)-N^{(\rho)}_f(r,D)=O(N_{\ram\pi}(r)
 		+\pN{f}(r)+T_{f_S}(r))+O(\log r)+o(T_f(r))||.
 	\end{equation*}
 	We apply \cref{lem:20220915} to get
 	$$N^{(2)}_f(r,D)-N^{(1)}_f(r,D)\leq \varepsilon' T_{f_A}(r,L)+O(N_{\ram\pi}(r)+\pN{f}(r)
 	+T_{f_S}(r))+O(\log r)+o(T_f(r))||.$$
 	By $N^{(\rho)}_f(r,D)-N^{(1)}_f(r,D)\leq (\rho-1) (N^{(2)}_f(r,D)-N^{(1)}_f(r,D))$, we get
 	\begin{multline}\label{eqn:20220915}
 		m_f(r,D)+N_f(r,D)-N^{(1)}_f(r,D)\leq (\rho-1)\varepsilon'T_{f_A}(r,L)
 		\\
 		+
 		O(N_{\ram\pi}(r)+\pN{f}(r)
 		+T_{f_S}(r))+O(\log r)+o(T_f(r))||.
 	\end{multline}
 	This completes the proof.
 \end{proof}

  \subsection{Varieties of maximal quasi-albanese dimension} 
  \label{subsec:4.7}

 Let $B\subset A$ be a semi-abelian subvariety and let
 $\overline{A}$ be an equivariant compactification.
 
 \begin{dfn}
 We say that $\overline{A}$ is compatible with $B$ iff there exists an equivariant compactification $\overline{A/B}$ such that the quotient map $q_{B}:A\to A/B$ extends to a morphism $\overline{q_B}:\overline{A}\to \overline{A/B}$.
 \end{dfn}

 Let $W\subset A\times S$ be a $B$-invariant closed subvariety.
 Then we have $W/B\subset (A/B)\times S$.
 If $\overline{A}$ is compatible with $B$, then $\overline{q_B}$ yields $p:\overline{A}\times S\to(\overline{A/B})\times S$.

 	\begin{lem}\label{lem:20221124}
 	Let $W\subset A\times S$ be a $B$-invariant closed subvariety.
 	Let $D\subset W$ be a reduced divisor.
 	Then there exists a smooth, projective, equivariant compactification $\overline{A}$ which is compatible with $B$ and $p(\mathrm{Sp}_B\overline{D})\subsetneqq \overline{W/B}$, where $\overline{D}\subset \overline{A}\times S$ is the Zariski closure.
	\end{lem}
 	
\begin{proof}
We consider $W\to W/B$.
We take the schematic generic point $\eta\in W/B$.
Let $D_{\eta}\subset W_{\eta}$ be the schematic generic fiber.
Let $\bar{\eta}:\mathrm{Spec}\overline{\mathbb C(\eta)}\to W/B$ be the geometric point.
Then $W_{\bar{\eta}}=B\times_{\mathbb C}\overline{\mathbb C(\eta)}$ is a semi-abelian variety over $\overline{\mathbb C(\eta)}$.
By \cite[Lem 12.7]{Y22}, 
there exists a smooth equivariant compactification $\overline{W_{\bar{\eta}}}$ such that $\overline{D_{\bar{\eta}}}\subset \overline{W_{\bar{\eta}}}$ satisfies 
\begin{equation}\label{eqn:20221028}
\mathrm{Sp}_{W_{\bar{\eta}}}\overline{D_{\bar{\eta}}}=\emptyset.
\end{equation}

We remark that there exists a smooth equivariant compactification $\overline{B}$ such that $\overline{W_{\bar{\eta}}}$ is obtained by the base change of $\overline{B}$.
Namely
$$\overline{W_{\bar{\eta}}}=\overline{B}\times_{\mathbb C}\overline{\mathbb C(\eta)}.$$
To see this, we consider the canonical extension $0\to T_B\to B\to B_0\to 0$, where $T_B$ is an algebraic torus and $B_0$ is an abelian variety.
By \cite[Lem A.6]{Y22}, there exists a smooth equivariant compactification $\overline{(T_B)_{\bar{\eta}}}$ such that $\overline{W_{\bar{\eta}}}=(\overline{(T_B)_{\bar{\eta}}}\times W_{\bar{\eta}})/(T_B)_{\bar{\eta}}$, where $(T_B)_{\bar{\eta}}=T_B\times_{\mathbb C}\overline{\mathbb C(\eta)}$.
By Sumihiro's theorem, $\overline{(T_B)_{\bar{\eta}}}$ is a torus embedding associated to some complete fan (cf. \cite[Sec. A.2]{Y22}).
Hence there exists a smooth equivariant compactification $\overline{T_B}$ such that $\overline{(T_B)_{\bar{\eta}}}=\overline{T_B}\times_{\mathbb C}\overline{\mathbb C(\eta)}$.
Let $\overline{B}=(\overline{T_B}\times B)/T_B$.
Then $\overline{B}\times_{\mathbb C}\overline{\mathbb C(\eta)}=\overline{W_{\bar{\eta}}}$ as desired.
By \cite[Lem A.8]{Y22}, we may assume moreover that $\overline{B}$ is projective by replacing $\overline{B}$ by a modification $\hat{B}\to \overline{B}$, for the property \eqref{eqn:20221028} remains true under this modification.
Then $\overline{T_B}$ is projective.

Next we shall construct an equivariant compactification $\overline{A}$ which is compatible with $B$ and the general fibers of $\overline{A}\to\overline{A/B}$ are $\overline{B}$ constructed above.
We have the canonical extensions $0\to T_B\to B\to B_0\to 0$, $0\to T_A\to A\to A_0\to 0$ and $0\to T_{A/B}\to A/B\to A_0/B_0\to 0$.
 Then we have $0\to T_B\to T_A\to T_{A/B}\to 0$.
 We may take a section $T_{A/B}\to T_A$ so that $T_A=T_B\times T_{A/B}$.
 Set $A'=A/T_{A/B}$.
 Then $B\subset A'$ and $A'/B=A_0/B_0$.
 The $B$-torsor $A\to A/B$ is the pull-back of $A'\to A'/B$ by $A/B\to A'/B$. 
 
 Let $\overline{B}$ be the compactification above so that $\overline{B}=(\overline{T_B}\times B)/T_B$.
 Set $\overline{A'}=(\overline{T_B}\times A')/T_B$.
 Then $\overline{A'}=(\overline{B}\times A')/B$. 
 Then the general fibers of $\overline{A'}\to A'/B$ are $\overline{B}$.
 Let $\overline{T_{A/B}}$ be a smooth projective equivariant compactification.
Set $\overline{A}=((\overline{T_B}\times \overline{T_{A/B}})\times A)/(T_B\times T_{A/B})$ and $\overline{A/B}=(\overline{T_{A/B}}\times (A/B))/T_{A/B}$.
 Then $\overline{A}$ is compatible with $B$ and the general fibers of $\overline{A}\to\overline{A/B}$ are $\overline{B}$.
Moreover $\overline{A}$ and $\overline{A/B}$ are smooth and projective by the proof of \cite[Lem A.8]{Y22}.

Now let $\overline{D}\subset \overline{W}$ be the compactification of $D$.
Then we claim
\begin{equation}\label{eqn:202210281}
(\overline{D})_{\bar{\eta}}=\overline{D_{\bar{\eta}}}
\end{equation}
in $(\overline{W})_{\bar{\eta}}
$.
To see this, we note that the induced map $\overline{W}_{\eta}\to \overline{W}$ is homeomorphism onto its image (cf. \cite[\href{https://stacks.math.columbia.edu/tag/01K1}{Tag 01K1}]{stacks-project}). 
Hence $(\overline{D})_{\eta}=\overline{D_{\eta}}$.
Note that $(\overline{W})_{\bar{\eta}}\to (\overline{W})_{\eta}$ is a closed morphism (cf.  \cite[\href{https://stacks.math.columbia.edu/tag/01WM}{Tag 01WM}]{stacks-project}).
Hence $(\overline{D_{\eta}})_{\bar{\eta}}=\overline{D_{\bar{\eta}}}$.
Combining these two observations, we have
$(\overline{D})_{\bar{\eta}}=(\overline{D_{\eta}})_{\bar{\eta}}=\overline{D_{\bar{\eta}}}$.
This shows \eqref{eqn:202210281}.
Since $\mathrm{Sp}_B\overline{D}\subset \overline{W}$ is $B$-invariant, $(\mathrm{Sp}_B\overline{D})_{\bar{\eta}}\subset (\overline{W})_{\bar{\eta}}$ is $B_{\bar{\eta}}$-invariant.
Moreover $(\mathrm{Sp}_B\overline{D})_{\bar{\eta}}\subset (\overline{D})_{\bar{\eta}}$.
Hence $(\mathrm{Sp}_B\overline{D})_{\bar{\eta}}\subset\mathrm{Sp}_{B_{\bar{\eta}}}((\overline{D})_{\bar{\eta}})$.
Combining this with \eqref{eqn:20221028} and \eqref{eqn:202210281},
we have
$(\mathrm{Sp}_B\overline{D})_{\bar{\eta}}\subset\mathrm{Sp}_{B_{\bar{\eta}}}\overline{D_{\bar{\eta}}}=\emptyset$.
Hence $p(\mathrm{Sp}_B\overline{D})\subset \overline{W/B}$ is a proper Zariski closed set.
\end{proof}

  Let $X$ be a quasi-projective variety.
  Let $D\subset X$ be a reduced (Weil) divisor.
  Let $E\subset D$ be a subset.
  Let $f:Y\to X$ be a holomorphic map such that $f(Y)\not\subset D$.
 We set
  $$
  \overline{N}_f(r, E)=\frac{1}{\mathrm{deg}\ \pi} \int_{2\delta}^{r}\mathrm{card}\ (Y(t)\cap f^{-1}E)\ \frac{d t}{t} .
  $$
  When $D$ is a Cartier divisor, then we have $\overline{N}_f(r,D)=N^{(1)}_f(r,D)$.

 Let $S$ be a projective variety and let $B\subset A$ be a semi-abelian subvariety.
 We denote by $I_B$ the set of all holomorphic maps $f:Y\to A\times S$ such that the following three estimates hold:
 \begin{itemize}
 \item 
$T_{q_B\circ f_A}(r)=O(\log r)+o(T_f(r))||$, where $q_B:A\to A/B$ is the quotient,
 \item
$N_{\ram\pi}(r)=O(\log r)+o(T_f(r))||$,
\item
$T_{f_S}(r)=O(\log r)+o(T_f(r))||$.
 \end{itemize}
 
 \begin{rem}\label{rem:20221203}
 If $f\in I_{\{0\}}$, then $f:Y\to A\times S$ does not have essential singularity over $\infty$.
 Indeed, $f\in I_{\{0\}}$ yields $T_{f_A}(r)=O(\log r)+o(T_f(r))||$ and  $T_{f_S}(r)=O(\log r)+o(T_f(r))||$.
 Hence $T_f(r)=O(\log r)||$, thus $N_{\ram\pi}(r)=O(\log r)||$.
 These two estimates implies that $f$ does not have essential singularity over $\infty$ (cf. \cref{lem:20230411} ). 
  \end{rem}
  
  For a proper birational morphism $\varphi:V\to W$, we denote by $\mathrm{Ex}(\varphi)\subset V$ the exceptional locus of $\varphi$.
  If $f:Y\to W$ is a holomorphic map such that $f(Y)\not\subset \varphi(\mathrm{Ex}(\varphi))$, then there exists a unique lift $g:Y\to V$ of $f$.
  
  \begin{lem}\label{lem:20220909gg}
  Let $B\subset A$ be a semi-abelian subvariety.
  Let $W\subset A\times S$ be a $B$-invariant, closed subvariety, where $S$ is a projective variety.
  Let $D\subset W$ be a reduced Weil divisor.
 Then there exist a smooth projective equivariant compactification $A\subset \overline{A}$ and a proper birational morphism $\varphi:V\to \overline{W}$, where $V$ is smooth and $\varphi^{-1}(\overline{D}\cup \partial W)\subset V$ is a simple normal crossing divisor, such that the following property holds:
 Let $Z\subset W$ be a Zariski closed subset such that $\mathrm{codim}(Z,W)\geq 2$ and $Z\subset D$.
 	 Let $L$ be an ample line bundle on $V$.
 	 Let $\varepsilon>0$.
Then there exist
\begin{itemize}
\item
 		a finite subset $P\subset \mathcal{S}(B)\backslash \{ B\}$,
 		\item
 		a proper Zariski closed subset $\Phi\subsetneqq W$ with $\varphi(\mathrm{Ex}(\varphi))\subset \overline{\Phi}$
 		\end{itemize}
 	such that for every $f:Y\to W$ with $f(Y)\not\subset D\cup \Phi$ and $f\in I_B$, either one of the followings holds:
 	 	\begin{enumerate}
 		\item
 		There exists $C\in P$ such that $f\in I_C$.
 		\item
		Let $f':Y\to V$ be the lift of $f$.
		Then
 		$$T_{f'}(r,K_V(\varphi^{-1}(\overline{D}\cup \partial W)))\leq \overline{N}_f(r,D\backslash Z)
 			+\varepsilon T_{f'}(r,L)+O(\log r)+o(T_{f'}(r))||.$$
 	\end{enumerate}
 \end{lem}

  \begin{proof}
  By \cref{lem:20221124}, we may take an equivariant compactification $\overline{A}$ such that
 \begin{itemize}
 \item
  $\overline{A}$ is smooth and projective, 
  \item
  $\overline{A}$ is compatible with $B$, and
  \item $p(\mathrm{Sp}_B\overline{D})\subsetneqq \overline{W/B}$, where $p:\overline{W}\to \overline{W/B}$ is induced from $\overline{A}\to\overline{A/B}$.
  \end{itemize}
  Let $U\subset W/B$ be the smooth locus of $W/B$.
  Then $p^{-1}(U)\subset \overline{W}$ is smooth.
  We take a smooth modification $\varphi_o:V_o\to \overline{W}$ which is isomorphism over $p^{-1}(U)$.
 We apply the embedded resolution of Szab{\'o}  \cite{Sza94} 
 for $\varphi_o^{-1}(\overline{D}+\partial W)\subset V_o$ to get a smooth modification $\tau:V\to V_o$ such that $\widetilde{D}=(\varphi_o\circ \tau)^{-1}(\overline{D}+\partial W)$ is simple normal crossing and $\tau$ is an isomorphism outside the locus where $\varphi_o^{-1}(\overline{D}+\partial W)$ is not simple normal crossing.
 Let $\varphi:V\to \overline{W}$ be the composite of $\tau:V\to V_o$ and $\varphi_o:V_o\to \overline{W}$.
 We have 
 \begin{equation}\label{eqn:20221126}
 K_{V}(\varphi^{-1}(\overline{D}\cup \partial W)))
\leq
 \tau^*K_{V_o}(\varphi_o^{-1}(\overline{D}+\partial W))+E
 \end{equation}
 for some effective divisor $E\subset V$ such that $\tau(E)\subset V_o$ is contained in the locus where $\varphi_o^{-1}(\overline{D}+\partial W)$ is not simple normal crossing.
 Since $p^{-1}(U)\cap \partial W$ is simple normal crossing, we have $p^{-1}(U)\cap \tau(E)\subset \overline{D}$.

 Let $Z\subset W$ be a Zariski closed subset such that $\mathrm{codim}(Z,W)\geq 2$ and $Z\subset D$. 
 We take a closed subscheme $\widehat{Z}\subset V_o$ such that
 \begin{itemize}
 \item
 $\mathrm{supp}\widehat{Z}=\overline{p^{-1}(U)\cap (\tau(E)\cup\varphi_o^{-1}(Z))}$,
 \item
 $(p\circ \varphi)^{-1}(U)\cap E\subset (p\circ\varphi)^{-1}(U)\cap \tau^*\widehat{Z}$, as closed subschemes of $(p\circ\varphi)^{-1}(U)$.
 \end{itemize}
 Then we have $\mathrm{supp}(\widehat{Z})\subset \varphi_o^{-1}(\overline{D})$ and $\mathrm{codim}(\widehat{Z},V_o)\geq 2$.

 Let $\Psi:B\times \overline{W}\to \overline{W}$ be the $B$-action.
 We may take a smooth subvariety $S'\subset W$ such that the induced map $S'\to U$ is {\'e}tale.
 	Then $\Psi$ yields an {\'e}tale map 
 	$$\psi:B\times S'\to p^{-1}(U)\cap W.$$
 Let $\overline{B}\subset\overline{A}$ be the compactification.
 Then $\overline{B}$ is an equivariant compactification.
 The map $\psi$ extends to a map
 $$
 \bar{\psi}:\overline{B}\times S'\to p^{-1}(U).
 $$ 
 Since $\overline{A}$ is compatible with $B$, $\bar{\psi}$ is {\'e}tale.
 We take a smooth compactification $S'\subset \Sigma$ such that the maps $S'\hookrightarrow W$ and $S'\to W/B$ extend to $\Sigma\to \overline{W}$ and $\sigma:\Sigma\to \overline{W/B}$.
 	Then $\bar{\psi}:\overline{B}\times S'\to p^{-1}(U)\subset V_o$ induces a rational map $\widehat{\psi}:\overline{B}\times \Sigma\dashrightarrow V_o$.

 	\[
		\begin{tikzcd}
		\overline{B}\times S' \arrow[r, hookrightarrow] \arrow[d] \arrow[bend left, "\bar{\psi}"]{rrr}&	\overline{B}\times \Sigma \arrow[r, dashrightarrow]
		\arrow[dashed, bend left, "\widehat{\psi}" ']{rr}
		 \arrow[d, "q"] & \overline{W} \arrow[d, "p"] & V_o\arrow[l,"\varphi_o" ] &V\arrow[l,"\tau" ] \\
		S' \arrow[r, hookrightarrow] &	\Sigma \arrow[r, "\sigma"] & \overline{W/B}& &
		\end{tikzcd}
	\]
We take a closed subscheme $I\subset \overline{B}\times\Sigma$ such that $q(\mathrm{supp} I)\subset \Sigma\backslash S'$ and the ratinal map $\widehat{\psi}:\overline{B}\times\Sigma\dashrightarrow V_o$ extends to a morphism 
$$\widetilde{\psi}:\mathrm{Bl}_I(\overline{B}\times\Sigma)\to V_o.$$

Let $L$ be an ample line bundle on $V$.
Let $L_{V_o}$ be an ample line bundle on $V_o$.
Then, there exists a positive constant $c_1>0$ such that for every $g:Y\to V$, we have the estimate:
\begin{equation}\label{eqn:202211303}
T_{\tau\circ g}(r,L_{V_o})\leq c_1T_g(r,L)+O(\log r).
\end{equation}
Let $L_{\overline{B}}$ be an ample line bundle on $\overline{B}$.
Let $\kappa:\mathrm{Bl}_I(\overline{B}\times\Sigma)\to \overline{B}$ be the composite of $\mathrm{Bl}_I(\overline{B}\times\Sigma)\to \overline{B}\times\Sigma$ and the first projection $\overline{B}\times\Sigma\to\overline{B}$.
Then, there exists a positive constant $c_2>0$ such that for every $g:Y\to V_o$ with $g(Y)\not\subset \widetilde{\psi}(\mathrm{Bl}_I(\overline{B}\times\Sigma)\backslash (\overline{B}\times S'))$, we have the estimate (cf. \cite[Lemma 1]{Yam15}):
\begin{equation}\label{eqn:202211304}
T_{g'}(r,\kappa^*L_{\overline{B}})\leq c_2T_g(r,L_{V_o})+O(\log r),
\end{equation}
where $g':Y\to \mathrm{Bl}_I(\overline{B}\times\Sigma)$ is a lift of $g$.

Let $D'\subset \overline{B}\times \Sigma$ be the reduced divisor defined by the Zariski closure of $\bar{\psi}^{-1}(\overline{D}\cap p^{-1}(U))\subset \overline{B}\times S'$.
Then we have
$$
q(\mathrm{Sp}_BD')\subsetneqq \Sigma.
$$
Let $Z'\subset \overline{B}\times \Sigma$ be the schematic closure of $\bar{\psi}^{*}(\widehat{Z}\cap p^{-1}(U))\subset \overline{B}\times S'$.
Then $\mathrm{codim}(Z',\overline{B}\times \Sigma)\geq 2$ and $\mathrm{supp}Z'\subset D'$.
Hence we have $q(\mathrm{Sp}_BZ')\subsetneqq \Sigma$.
We apply \cref{pro:202208062} for $Z'\subset \overline{B}\times \Sigma$ to get $P_0\subset \mathcal{S}(B)\backslash\{B\}$ and $\rho\in \mathbb Z_{>0}$.

  Let $\varepsilon>0$.
  We set $\varepsilon'=\frac{\varepsilon}{c_1c_2(\rho+2)}$.
  We may apply \cref{lem:20220909} for $D'\subset \overline{B}\times \Sigma$, $L_{\overline{B}}$ and $\varepsilon'$ to get $P_1\subset \mathcal{S}(B)\backslash\{B\}$ and $E_1\subsetneqq \Sigma$.
  We apply \cref{lem:202209151} for $Z'\subset \overline{B}\times \Sigma$, $L_{\overline{B}}$ and $\varepsilon'$ to get $P_2\subset \mathcal{S}(B)\backslash\{B\}$ and $E_2\subsetneqq \Sigma$.
  We take a non-empty Zariski open set $U'\subset U$ such that $S'\to U$ is finite over $U'$.
 We set 
  $$\Phi=\big( \varphi(\mathrm{Ex}(\varphi))\cup p^{-1}(\overline{W/B}\backslash U')\cup p^{-1}(\sigma (E_1\cup E_2\cup q(\mathrm{Sp}_BZ'))\big) \cap W.$$
  We set $P=P_0\cup P_1\cup P_2$.
  
  Now let $f:Y\to W$ be a holomorphic map such that $f(Y)\not\subset D\cup \Phi$ and $f\in I_B$.
  We assume that the first assertion of \cref{lem:20220909gg} does not valid.
  By $p\circ f(Y)\not\subset \overline{W/B}\backslash U'$, we may take a lift 
  $$g:Y'\to \overline{B}\times \Sigma$$ 
  such that 
  $$g^{-1}(\partial B\times \Sigma)\subset g_{\Sigma}^{-1}(\Sigma\backslash S').$$
Hence we have 
\begin{equation*}
 	\pN{g}(r)\leq N_{g_{\Sigma}}(r,\Sigma\backslash S')=O(\log r)+o(T_f(r))||,
 \end{equation*}
 	\begin{equation*}
 	T_{g_{\Sigma}}(r)=O(\log r)+o(T_f(r))||.
 	\end{equation*}
 	Note that the covering $\mu:Y'\to Y$ is unramified outside $g_{\Sigma}^{-1}(\Sigma\backslash S')$.
 	Hence we have
 	\begin{equation*}
 	N_{\ram\pi'}(r)=O(\log r)+o(T_f(r))||,
 	\end{equation*}
   where $\pi':Y'\to \mathbb C_{>\delta}$ is the composite $\pi'=\pi\circ \mu$.

   Hence by \cref{lem:20220909}, we get 
 \begin{equation*}
 		m_{g}(r,D')+N_{g}(r,D')-\overline{N}_{g}(r,D')\leq\varepsilon' T_{g_B}(r,L_{\overline{B}})
 		+O(\log r)+o(T_f(r))||.
 	\end{equation*}
 By \cref{pro:202208062} 
 	\begin{equation*}
 		m_{g}(r,Z')+N_{g}(r,Z')\leq N^{(\rho)}_{g}(r,Z')+
 		O(\log r)+o(T_f(r))||.
 	\end{equation*}
 	By \cref{lem:202209151}, we get
 	$$
 	N^{(1)}_{g}(r,Z')\leq \varepsilon' T_{g_B}(r,L_{\overline{B}})
 	+O(\log r)+o(T_f(r))||.
 	$$
 	Hence
 	$$
 	m_{g}(r,Z')+N_{g}(r,Z')\leq \rho\varepsilon'T_{g_B}(r,L_{\overline{B}})
 +O(\log r)+o(T_f(r))||.
 	$$  
 Note that $K_{\overline{B}\times \Sigma}(\partial (B\times \Sigma))=q^*K_{\Sigma}$.
 Hence
  \begin{multline}\label{eqn:202211301}
 		m_{g}(r,Z')+N_{g}(r,Z')+T_{g}(r,K_{\overline{B}\times \Sigma}(D'+\partial (B\times \Sigma)))\leq \overline{N}_{g}(r,D'\backslash Z')
 		\\
 		+(\rho+2)\varepsilon' T_{g_B}(r,L_{\overline{B}})
 		+O(\log r)+o(T_f(r))||.
 	\end{multline}
 	Since $\bar{\psi}:\overline{B}\times S'\to p^{-1}(U)$ is {\'e}tale, we have
 $$
 \bar{\psi}^*K_{V_o}(\varphi_o^{-1}(\overline{D}+\partial W))=K_{\overline{B}\times \Sigma}(D'+\partial (B\times \Sigma))|_{\overline{B}\times S'}.
 $$
 Let $f':Y\to V$ be the lift of $f:Y\to W$.
 Then $\tau\circ f':Y\to V_o$ is the lift of $f$.
By \cite[Lemma 3.1]{Yam06} we have  
 \begin{equation}
 T_{g}(r,K_{\overline{B}\times \Sigma}(D'+\partial (B\times \Sigma)))=T_{\tau\circ f'}(r,K_{V_o}(\varphi_o^{-1}(\overline{D}+\partial W)))+O(\log r)+o(T_f(r))||.
 \end{equation}
 Similarly we have (cf. \cite[Lemma 1]{Yam15})
 	\begin{equation}
 	m_{g}(r,Z')+N_{g}(r,Z')=m_{\tau\circ f'}(r,\widehat{Z})+N_{\tau\circ f'}(r,\widehat{Z})+O(\log r)+o(T_f(r))||,
 	\end{equation}
 	 \begin{equation}\label{eqn:202211302}
 	 \overline{N}_{g}(r,D'\backslash Z')\leq \overline{N}_{f}(r,D\backslash Z)+O(\log r)+o(T_f(r))||.
 	 \end{equation}
Hence combinning \eqref{eqn:202211301}--\eqref{eqn:202211302}, we get
 \begin{multline}
 		m_{\tau\circ f'}(r,\widehat{Z})+N_{\tau\circ f'}(r,\widehat{Z})+T_{\tau\circ f'}(r,K_{V_o}(\varphi_o^{-1}(\overline{D}+\partial W)))
 		\leq \overline{N}_{f}(r,D\backslash Z)
 		\\
 		+(\rho+2)\varepsilon' T_{g_B}(r,L_{\overline{B}})
 		+O(\log r)+o(T_f(r))||.
 	\end{multline}
 	By the choice of $\widehat{Z}\subset V_o$, we have $(p\circ \varphi)^{-1}(U)\cap E\subset (p\circ\varphi)^{-1}(U)\cap \tau^*\widehat{Z}$.
 	Hence
 	$$
 	m_{f'}(r,E)+N_{f'}(r,E)\leq m_{\tau\circ f'}(r,\widehat{Z})+N_{\tau\circ f'}(r,\widehat{Z})+O(\log r)+o(T_f(r))||.
 	$$
 	Hence by \eqref{eqn:20221126}, we get
 	 \begin{multline}
 		T_{f'}(r,K_{V}(\varphi^{-1}(\overline{D}\cup \partial W)))
 		\leq \overline{N}_{f}(r,D\backslash Z)
 		\\
 		+(\rho+2)\varepsilon' T_{g_B}(r,L_{\overline{B}})
 		+O(\log r)+o(T_f(r))||.
 	\end{multline}
 	By \eqref{eqn:202211303} and \eqref{eqn:202211304}, we have $T_{g_B}(r,L_{\overline{B}})\leq c_1c_2T_{f'}(r,L)$.
 	Thus we get the second assertion of \cref{lem:20220909gg}.
 	The proof is completed.
 \end{proof}

 Let $V$ be a $\mathbb Q$-factorial, projective variety.
 Let $F\subset V$ be an effective Weil divisor.
 Then there exists a positive integer $k$ such that $kF$ is a Cartier divisor.
 Let $g:Y\to V$ be a holomorphic map.
 Then we set
 $$
 T_g(r,F)=\frac{1}{k}T_g(r,\mathcal{O}_V(kF))+O(\log r).
 $$
 This definition does not depend on the choice of $k$.

  \begin{lem}\label{cor:20220805}
  Let $B\subset A$ be a semi-abelian subvariety.
 	Let $X$ be a normal variety with a finite map $a:X\to A\times S$, where $S$ is a projective variety.
 	Let $D\subset X$ be a reduced Weil divisor.
 	Then there exists a compactification $\overline{X}$ and a proper birational morphsim $\varphi:\overline{X}'\to \overline{X}$, where $\overline{X}'$ is $\mathbb Q$-factorial, such that the following property holds:
	Let $Z\subset X$ be a Zariski closed set such that $\mathrm{codim}(Z,X)\geq 2$ and $Z\subset D$.
	 	Let $L$ be a big line bundle on $\overline{X}'$ and let $\varepsilon>0$.
 	Then there exist  
 	\begin{itemize}
 		\item 
 		a finite subset $P\subset \mathcal{S}(B)\backslash\{B\}$,
 		\item
 		a proper Zariski closed set $\Xi\subsetneqq X$ with $\varphi(\mathrm{Ex}(\varphi))\subset \overline{\Xi}$
 	\end{itemize}
 	such that for every holomorphic map $f:Y\to X$ with $f(Y)\not\subset D\cup\Xi$ and $a\circ f\in I_B$, either one of the followings holds:
 	\begin{enumerate}
 		\item
 		There exists $C\in P$ such that $a\circ f\in I_C$.
 		\item
 		Let $H\subset \overline{X}'$ be a reduced Weil divisor defined by $\varphi^{-1}(\overline{D\cup X_{\mathrm{sing}}\cup \partial X})=H\cup \Gamma$, where $\Gamma\subset \overline{X}'$ is a Zariski closed subset of codimension greater than one.
		Let $g:Y\to \overline{X}'$ be the lift of $f$. 		
		Then
 		\begin{equation*}
 			T_g(r,K_{\overline{X}'}(H))
 			\leq \overline{N}_f(r,D\backslash Z)+\varepsilon T_{g}(r,L)
 			\\
 			+O(\log r)+o(T_g(r))||.
 		\end{equation*}
 	\end{enumerate}
 \end{lem}

 \begin{proof}
 We set $W=a(X)$.
 	Since $a$ is finite, $W\subset A\times S$ is a Zariski closed set.
 	We first consider the case that $B\not\subset \mathrm{St}^0(W)$.
 	We apply \cref{cor:20220806} to get $\Xi\subset W$.
 	We set $B'=B\cap \mathrm{St}^0(W)$.
 	Then $B'\subsetneqq B$.
 	We set $P=\{ B'\}$.
 	Now let $f:Y\to X$ satisfies $f(Y)\not\subset D\cup\Xi$ and $a\circ f\in I_B$.
 	By \cref{cor:20220806} we have $a\circ f\in I_{\mathrm{St}^0(W)}$.
 	Hence $a\circ f\in I_{B'}$.
 	This shows \cref{cor:20220805}, provided $B\not\subset \mathrm{St}^0(W)$.
 	In the following, we consider the case $B\subset \mathrm{St}^0W$.

 We first find $\overline{X}$ and $\overline{X}'$.
 Let $R\subset W$ be a reduced Weil divisor such that $a:X\to W$ is {\'e}tale outside $R$.
 We take a reduced Weil divisor $D'\subset W$ such that $a(D\cup X_{\mathrm{sing}})\cup R\subset D'$.
 By \cref{lem:20220909gg},  there exist a compactification $W\subset \overline{W}$ and a proper birational morphism $\psi:V\to \overline{W}$, where $V$ is smooth and $\psi^{-1}(\overline{D'}\cup \partial W)\subset V$ is a simple normal crossing divisor.
Let $\overline{X}\to \overline{W}$ be obtained from $X\to W$ by the normalization.
 Let $p:\overline{X}'\to V$ be obtained from the base change and normalization.

 		\begin{equation*}
 	\begin{tikzcd}
 		X \arrow[r,hookrightarrow] \arrow[d, "a" ']	& \overline{X} \arrow[d]  &\overline{X}' \arrow[l,  "\varphi" ' ] \arrow[d, "p"] \\
 		W\arrow[r,hookrightarrow]  & \overline{W}    &V \arrow[l, "\psi"] 
 		\end{tikzcd}
 	\end{equation*}
	Note that $p:\overline{X}'\to V$ is unramified outside $\widetilde{D}=\psi^{-1}(\overline{D'}\cup \partial W)$.
Since $\widetilde{D}$ is simple normal crossing, \cite[Lemma 2]{NWY13} yields that $\overline{X}'$ is $\mathbb
 	Q$-factorial.

	Let $Z\subset X$ be a Zariski closed set such that $\mathrm{codim}(Z,X)\geq 2$ and $Z\subset D$. 	We set 
 	$$Z'=a(Z\cup X_{\mathrm{sing}}
)\cup (W\cap (\psi\circ p)(\overline{X}'_{\mathrm{sing}})).$$
Then $\mathrm{codim}(Z',W)\geq 2$ and $Z'\subset D'$.
Let $L$ be a big line bundle on $\overline{X}'$.
Let $L_{V}$ be an ample line bundle on $V$.
Then there exists a positive integer $l\in\mathbb Z_{\geq 1}$ such that $p^*L_{V}(E)=L^{\otimes l}$ for some effective divisor $E\subset \overline{X}'$.
Let $\varepsilon>0$.
We set $\varepsilon'=\varepsilon/l$.

Now by \cref{lem:20220909gg} for $Z'\subset D'$, $L_{V}$ and $\varepsilon'$, we get a finite subset $P\subset \mathcal{S}(B)\backslash \{ B\}$ and a proper Zariski closed subset $\Phi\subsetneqq W$.
We set 
$$\Xi=a^{-1}(D'\cup \Phi)\cup (X\cap \varphi(E))\cup (X\cap \varphi(\mathrm{Ex}(\varphi))).$$
Let $f:Y\to X$ such that $a\circ f\in I_B$ and $f(Y)\not\subset \Xi$.
Assume that the first assertion of \cref{cor:20220805} does not valid.
Let $g:Y\to \overline{X}'$ be the lift of $f$.
Then $p\circ g:Y\to V$ is the lift of $a\circ f:Y\to W$.
Hence by \cref{lem:20220909gg}, we get 
\begin{equation}\label{eqn:202211261}
T_{p\circ g}(r,K_V(\widetilde{D}))\leq \overline{N}_{a\circ f}(r,D'\backslash Z')
 			+\varepsilon' T_{p\circ g}(r,L_V)+O(\log r)+o(T_{p\circ g}(r))||
\end{equation}
We define a reduced divisor $F$ on $\overline{X'}$ by $p^{-1}(\widetilde{D})=F+H$.
By the ramification formula, we have
 	\begin{equation*}
	K_{\overline{X'}}(F+H)=p^*K_{V}(\widetilde{D}).
	\end{equation*}
 	Thus we have
 	\begin{equation*}
	T_{g}(r,K_{\overline{X'}}(F+H)) 	=T_{p\circ g}(r,K_{V}(\widetilde{D}))+O(\log r).
 	\end{equation*}
	Combining this estimate with \eqref{eqn:202211261} and $p^*L_{V}(E)=L^{\otimes l}$, we get
	\begin{equation}\label{eqn:20221201}
T_{g}(r,K_{\overline{X'}}(F+H))\leq \overline{N}_{a\circ f}(r,D'\backslash Z')
 			+\varepsilon T_{g}(r,L)+O(\log r)+o(T_{ g}(r))||
\end{equation}

 	We estimate the right hand side of \eqref{eqn:20221201}.
	We have
	 	\begin{equation}\label{eqn:202212011}
 	\overline{N}_{a \circ
 		f}(r,\D'\backslash Z')
 		\leq \overline{N}_{g}(r,F\backslash  \overline{X}'_{\mathrm{sing}})+\overline{N}_{g}(r,H\backslash  \varphi^{-1}(Z\cup X_{\mathrm{sing}}))
 	\end{equation}
 	We set $H'$ by $H=H'+\varphi^{-1}(\partial X+\overline{D})$.
 	Then $\varphi(H')\subset X_{\mathrm{sing}}$.
 	Hence by $\overline{N}_f(r,\partial X)=0$, 
 	we have
 	\begin{equation}\label{eqn:202208051} 	\overline{N}_g(r,H\backslash \varphi^{-1}(Z\cup X_{\mathrm{sing}}))\leq \overline{N}_f(r,D\backslash Z).
 	\end{equation}
	Since $\overline{X}'$ is $\mathbb Q$-factorial, we may take a positive integer $k$ such that $kF$ is a Cartier divisor.
 	Since $F\cap (\overline{X}'\backslash \overline{X}'_{\mathrm{sing}})$
 	is a Cartier divisor
 	on $\overline{X}'\backslash \overline{X}'_{\mathrm{sing}}$, we have
 	$$
 	k\ \mathrm{ord}_z g^{*}F=\mathrm{ord}_z g^{*}(kF)
 	$$
 	for $z\in g^{-1} (\overline{X}'\backslash \overline{X}'_{\mathrm{sing}})$,
 	and hence
 	$$
 	k\min \{ 1,\mathrm{ord}_z g^{*}F\}
 	\leq \mathrm{ord}_z g^{*}(kF).
 	$$
 	Thus we get
 	\begin{equation*}
 		k\overline{N}_g(r,F\backslash \overline{X}'_{\mathrm{sing}})\leq
 		N_g(r,kF).
 	\end{equation*}
 	By the Nevanlinna inequality (cf. \eqref{eqn:20221120}), we have
 	$$
 	N_g(r,kF)\leq T_{g}(r,kF)+O(\log r).
 	$$
 	Hence
 	\begin{equation*}
 		\overline{N}_g(r,F\backslash \overline{X}'_{\mathrm{sing}})\leq T_{g}(r,F)+O(\log r).
 	\end{equation*}
 	Combining this with \eqref{eqn:202212011} and \eqref{eqn:202208051}, we get
	$$
	\overline{N}_{a \circ
 		f}(r,\D'\backslash Z')\leq \overline{N}_f(r,D\backslash Z)+T_{g}(r,F)+O(\log r).	$$
		Combining this estimate with \eqref{eqn:20221201}, we get the second assertion of \cref{cor:20220805}.
This concludes the proof.
\end{proof}

 \begin{proposition}\label{prop:20220902}
 	Let $\Sigma$ be a smooth quasi-projective variety which is of log general type.
 	Assume that there is a morphism $a:\Sigma\to A\times S$ such that $\dim \Sigma=\dim a(\Sigma)$, where $S$ is a projective variety.
 	Let $B\subset A$ be a semi-abelian subvariety.
 	Then there exist a finite subset $P\subset \mathcal{S}(B)\backslash\{B\}$ and a proper Zariski closed set $\Phi\subsetneqq \Sigma$ with the following property:
 	Let $f:Y\to \Sigma$ be a holomorphic map such that $a\circ f\in I_B$ and $f(Y)\not\subset \Phi$. 
		Then either one of the followings holds:
 	\begin{enumerate}
 		  		\item
 		There exists $C\in P$ such that $a\circ f\in I_C$.
		\item
 		$f$ does not have essential singularity over $\infty$.
 	\end{enumerate}
 \end{proposition}

 \begin{proof}
 	Let $W\subset A\times S$ be the Zariski closure of $a(\Sigma)$.
 	Let $\pi :X\to W$ be the normalization with respect to $a$
 	and let $\varphi :\Sigma\to X$ be the induced map.
 	Then $\varphi$ is birational.
 	Let $\overline{\Sigma}$ be a smooth partial compactification such that $\varphi$
 	extends to a projective morphism $\overline{\varphi}:\overline{\Sigma}\to X$ and
 	$\overline{\Sigma}\backslash \Sigma$ is a divisor on $\overline{\Sigma}$. 
 	Since $\overline{\varphi}$ is birational and $X$ is normal, there exists a
 	Zariski closed subset $Z\subset X$ whose codimension is greater than one
 	such that $\overline{\varphi}:\overline{\Sigma}\to X$ is an isomorphism over $X\backslash Z$.
 	In particular $X\backslash Z$ is smooth.
 	Let $D$ be the Zariski closure of
 	$(X\backslash Z)\cap \bar{\varphi}(\bar{\Sigma}\backslash \Sigma)$ in $X$.
 	Then $D$ is a reduced divisor on $X$ and $X\backslash (Z\cup D)$
 	is of log-general type.
 	Thus by \cite[Lemma 4]{NWY13}, $X\backslash (X_{\mathrm{sing}}\cup D)$
 	is of log-general type.
 	Note that $(D\cap Z)\cup X_{\mathrm{sing}}\subset Z$.

 	We apply \cref{cor:20220805} to get $\overline{X}$ and $\psi:\overline{X}'\to \overline{X}$.
	We define a reduced divisor $H\subset \overline{X}'$ to be 	
	$\psi^{-1}(\overline{D\cup X_{\mathrm{sing}}\cup \partial X})=H\cup \Gamma$, where $\Gamma\subset \overline{X}'$ is a Zariski closed subset of codimension greater than one.	
	Since $X\backslash (D\cup X_{\mathrm{sing}})$ is of log-general type, we deduce that $\psi ^{-1}(X\backslash (D\cup
 	X_{\mathrm{sing}}))=\overline{X}'\backslash (H\cup \Gamma)$ is also of log-general type.
 	Thus by \cite[Lemma 3]{NWY13}, $K_{\overline{X}'}(H)$ is big.
	
	By Kodaira's lemma, there exist an effective divisor $E\subsetneqq \overline{X}'$ and a positive integer $l\in\mathbb Z_{\geq 1}$ such that $lK_{\overline{X}'}(H)-E$ is ample.	
	Hence if $g:Y\to \overline{X}'$ satisfies $g(Y)\not\subset E$, then 
		$$T_{g}(r)= O( T_{g}(r,K_{\overline{X'}}(H)))+O(\log r).$$

By \cref{cor:20220805} applied to $Z\cap D$, $L=K_{\overline{X}'}(H)$ and $\varepsilon=1/2$, we get $P\subset \mathcal{S}(B)\backslash\{B\}$ and $\Xi\subsetneqq X$.
We set $\Phi=\varphi^{-1}(\Xi\cup\psi(E))$.
Let $f:Y\to \Sigma$ be a holomorphic map such that $a\circ f\in I_B$ and $f(Y)\not\subset \Phi$.
Then $\varphi\circ f(Y)\not\subset \Xi$. 
Suppose that the first assertion of \cref{prop:20220902} is not valid.
Then by \cref{cor:20220805}, we get
\begin{equation*}
 			T_{g}(r,K_{\overline{X}'}(H))
 			\leq \overline{N}_{\varphi\circ f}(r,D\backslash Z)+\frac{1}{2}T_{g}(r,K_{\overline{X}'}(H))
 			\\
 			+O(\log r)+o(T_{g}(r))||,
 		\end{equation*}
		where $g:Y\to \overline{X}'$ is the lift of $\varphi\circ f:Y\to X$.
We have
 	$$
 	\overline{N}_{\varphi\circ f}(r,D)=\overline{N}_{\varphi\circ f}(r,D\cap Z).$$
	Hence we get
	$$
	T_{g}(r,K_{\overline{X}'}(H))=O(\log r)+o(T_{g}(r))||.
		$$ 	 	
	Thus $T_{\varphi\circ f}(r)=O(\log r) ||$.
		This shows $T_{a\circ f}(r)=O(\log r) ||$. 
		Hence $a\circ f\in I_{\{0\}}$.
 	Hence by \cref{rem:20221203}, $f$ does not have essential singularity over $\infty$.
 \end{proof}

The following theorem implies \cref{thm2nd}, when $S$ is a single point.
 
 \begin{thm}\label{thm:20221201}
 	Let $X$ be a smooth quasi-projective variety which is of log general type.
Assume that there is a morphism $a:X\to A\times S$ such that $\dim X=\dim a(X)$, where $S$ is a projective variety.
 	Then there exists a proper Zariski closed set $\Xi\subsetneqq X$ with the following property:
 	Let $f:Y\to X$ be a holomorphic map with the following three properties:
	\begin{enumerate}[label=(\alph*)]
	\item \label{item1}
	$T_{(a\circ f)_S}(r)=O(\log r)+o(T_f(r))||$, 
	\item \label{item2}
	$N_{\ram\pi}(r)=O(\log r)+o(T_f(r))||$,
	\item \label{item3}
	$f(Y)\not\subset \Xi$.
	\end{enumerate}
	Then $f$ does not have essential singularity over $\infty$.
 \end{thm}

 \begin{proof}
 We take a Zariski closed set $E\subsetneqq X$ such that if $g:Y\to X$ satisfies $g(Y)\not\subset E$, then $T_{g}(r)=O(T_{a\circ g}(r))+O(\log r)$.
 	We take a sequence $\mathcal{P}_i\subset \mathcal{S}(A)$ of finite sets as follows.
 	Set $\mathcal{P}_1=\{A\}$.
 	Given $\mathcal{P}_i$, we define $\mathcal{P}_{i+1}$ as follows.
 	For each $B\in\mathcal{P}_i$, we apply \cref{prop:20220902} to get $\Phi_B\subsetneqq X$ and $P_B\subset \mathcal{S}(B)\backslash\{B\}$.
 	We set $\mathcal{P}_{i+1}=\cup_{B\in \mathcal{P}_i}P_B$ and $\Xi_{i+1}=\cup_{B\in\mathcal{P}_i}\Phi_B$.
 	Set $\Xi=E\cup \cup_i\Xi_i$.
 	Then $\Xi\subsetneqq X$ is a proper Zariski closed set.
 	
 	Let $f:Y\to X$ satisfy the three properties in \cref{thm:20221201}.
	Then $f\in I_A$.
	We take $i$ which is maximal among the property that there exists $B\in \mathcal{P}_i$ such that $f\in I_B$. 	
	Then by $f(Y)\not\subset \Phi_B\subset \Xi$,
 	\cref{prop:20220902} implies that $f$ does not have essential singularity over $\infty$.
 \end{proof}

\begin{cor}\label{cor:GGL}
	Let $X$ be a smooth quasi-projective variety and let $a:X\to A\times S^\circ$ be a morphism such that $\dim X=\dim a(X)$, where $S^\circ$ is a smooth quasi-projective variety ($S$ can be a point). Write $b:X\to S^\circ$ as the composition of $a$ with the projection map $A\times S^\circ\to S^\circ$. Assume that $b$ is dominant.
	\begin{thmlist}
		\item\label{coritem1} 
Suppose $S^\circ$ is pseudo Picard hyperbolic.
If $X$ is of log general type, then $X$  is pseudo Picard hyperbolic.
		\item \label{coritem2} Suppose $\Spalg(S^\circ)\subsetneqq S^\circ$.
		If $\Spab(X)\subsetneqq X$, then $\Spalg(X)\subsetneqq X$.   
	\end{thmlist} 
\end{cor}

\begin{proof}
	Let $S$ be a smooth projective variety that compactifies $S^\circ$. \\
\noindent{\em Proof of \cref{coritem1}:}	Since $S^\circ$ is pseudo Picard hyperbolic, there exists a proper Zariski closed subset $Z\subset S^\circ$ such that each holomorphic map $g:\bD^*\to S^\circ$ with $g(\bD^*)\not\subset Z$ has no essential singularity at $\infty$. By applying \cref{thm:20221201} to $X$, we obtain a proper Zariski closed set $\Xi\subsetneqq X$ that satisfies the property given in \cref{thm:20221201}. Set $\Sigma:=b^{-1}(Z)\cup \Xi$.  Since $b$ is dominant, $\Sigma$ is a proper Zariski closed subset of $X$.  Then for any holomorphic map $f:\bD^*\to X$ with $f(\bD^*)\not\subset \Sigma$, we have $b\circ f(\bD^*)\not\subset Z$ and $f(\bD^*)\not\subset \Xi$. It follows that
	 $
	T_{b\circ f}(r)=O(\log r).
	$ 
	Hence $f$ verifies Properties \Cref{item1} and \Cref{item3} in   \cref{thm:20221201}. 
Note that Property \Cref{item2}  in \cref{thm:20221201} is automatically satisfied as  $N_{\ram\pi}(r)=0$. We can now apply \cref{thm:20221201} to conclude that $f$ does not have an essential singularity over $\infty$. Therefore, $X$ is pseudo Picard hyperbolic.

\medspace

\noindent {\em Proof of \cref{coritem2}:}   
Since $\dim X = \dim a(X)$, there is a proper Zariski closed subset $\Upsilon \subsetneqq X$ such that the restriction of $a|_{X \backslash \Upsilon}:X \backslash \Upsilon\to A\times S^\circ$ is quasi-finite. 
Let $\Xi := b^{-1}(\Spalg(S^\circ)) \cup \Spab(X) \cup \Upsilon$.  
Then by the assumptions $\Spalg(S^\circ)\subsetneqq S^\circ$ and $\Spab(X)\subsetneqq X$, we have $\Xi\subsetneqq X$.

Let $V$ be a closed subvariety of $X$ that is not contained in $\Xi$. Then we have
\begin{equation}\label{eqn:20230510}
\Spab(V)\subsetneqq V.
\end{equation} 
In the following, we shall prove that $V$ is of log general type to conclude $\Spalg(X)\subset \Xi$.

We first show $\bar{\kappa}(V)\geq 0$.
Note that for a general fiber $F$ of $b|_{V}:V\to S^\circ$, we have $\dim F=\dim c(F)$, where $c:X\to A$ is the composition of $a$ with the projection map $A\times S^\circ\to A$. 
By \cref{prop:Koddimabb}, it follows that $\bar{\kappa}(\overline{c(F)})\geq 0$, and hence $\bar{\kappa}(F)\geq 0$. 
Also note that $b(V)$ is not contained in $\Spalg(S^\circ)$, and thus $\overline{b(V)}$ is of log general type.  
We use Fujino's addition formula for logarithmic Kodaira dimensions \cite[Theorem 1.9]{Fuj17} to conclude that $$\bar{\kappa}(V)\geq \bar{\kappa}(F)+\bar{\kappa}(\overline{b(V)})\geq 0.$$  
Hence we have proved $\bar{\kappa}(V)\geq 0$.
We may consider the logarithmic Iitaka fibration of $V$.

Next we show $\bar{\kappa}(V)=\dim V$.
After replacing $V$ with a birational modification, we can assume that $V$ is smooth and that the logarithmic Iitaka fibration $j:V\to J(V)$ is regular. Assume contrary that $V$ is not of log general type. Note that for a very general fiber $F$ of $j$, the followings hold:
   \begin{enumerate}[label*=(\alph*)]
      \item \label{item:positive} $\dim F>0$;
	\item  $F$ is smooth;
	\item   \label{item:zero}   $\bar{\kappa}(F)=0$;
	\item \label{item:avoid} $b(F)\not\subset \Spalg(S^\circ)$;
	\item \label{item:same dim} $F\not\subset \Upsilon$.   
\end{enumerate} 
By \Cref{item:avoid}, we have $\bar{\kappa}(\overline{b(F)})=\dim b(F)\geq 0$. 
\Cref{item:same dim} implies that for a general fiber $Y$ of $b|_{F}:F\to S^\circ$, we have $\dim Y=\dim c(Y)$. 
Using \cref{prop:Koddimabb}, we have $\bar{\kappa}(\overline{c(Y)})\geq 0$, and thus $\bar{\kappa}(Y)\geq 0$. 
Using \cite[Theorem 1.9]{Fuj17} again, we can conclude that $$\bar{\kappa}(F)\geq \bar{\kappa}(Y)+\bar{\kappa}(\overline{b(F)})\geq 0.$$
\Cref{item:zero} implies that $\bar{\kappa}(Y)=0$ and $\bar{\kappa}(\overline{b(F)})=0$. Hence $b(F)$ is a point. 
This implies that $\dim F=\dim c(F)$. 
Combining with \Cref{item:positive,item:zero}, \cref{lem:20230509} yields $\Spab(F)=F$. 
Since $F$ is a very general fiber of $j:V\to J(V)$, we get $\Spab(V)=V$.
This contradicts to \eqref{eqn:20230510}.
Thus we have proved that $V$ is of log general type, hence $\Spalg(X)\subset \Xi$.
\end{proof}

\begin{proposition}\label{prop:20230405}
Let $D\subset A$ be a reduced divisor such that $\mathrm{St}(D)=\{ a\in A; a+D=D\}$ is finite.
Let $Z\subset D$ be a Zariski closed subset such that $\mathrm{codim}(Z,A)\geq 2$.
Then there exists a proper Zariski closed subset $\Xi\subsetneqq A$ with the following property:
Let $f:Y\to A$ be a holomorphic map with the following three properties:
	\begin{enumerate}[label=(\alph*)]
	\item  
	$N_{\ram\pi}(r)=O(\log r)+o(T_f(r))||$, 
	\item  
	$f(Y)\not\subset \Xi\cup D$,
	\item
	$\mathrm{ord}_yf^*D\geq 2$ for all $y\in f^{-1}(D\backslash Z)$.	\end{enumerate}
	Then $f$ does not have essential singularity over $\infty$.
\end{proposition}

\begin{proof}
For a semi-abelian variety $B\subset A$, we denote by $J_B$ the set of all holomorphic maps $f:Y\to A$ such that 
 \begin{itemize}
\item
$f\in I_B$,  
\item
$f(Y)\not\subset D$, and
\item
$\mathrm{ord}_yf^*D\geq 2$ for all $y\in f^{-1}(D\backslash Z)$. \end{itemize}
Given a semi-abelian variety $B\subset A$, we first prove the following claim.

\begin{claim}\label{claim:20230405}
There exist a finite subset $P_B\subset \mathcal{S}(B)\backslash\{B\}$ and a proper Zariski closed set $\Phi_B\subsetneqq A$ with the following property:
 	Let $f:Y\to A$ be a holomorphic map such that $f\in J_B$ and $f(Y)\not\subset \Xi_B$. 
		Then either one of the followings holds:
 	\begin{enumerate}[label=(\alph*)]
 		  		\item
 		There exists $C\in P_B$ such that $f\in J_C$.
		\item
 		$f$ does not have essential singularity over $\infty$.
 	\end{enumerate}
\end{claim}

\begin{proof}[Proof of \cref{claim:20230405}]
Since $\mathrm{St}(D)$ is finite, $A\backslash D$ is of log-general type.
We apply \cref{lem:20220909gg} to get a smooth projective equivariant compactification $\overline{A}$ and a proper birational morphism $\varphi:V\to \overline{A}$, where $V$ is smooth and $\varphi^{-1}(\overline{D}\cup \partial A)$ is a simple normal crossing divisor.
	Since $A\backslash D$ is of log-general type, we deduce that $\varphi ^{-1}(A\backslash D)=V\backslash \varphi^{-1}(\overline{D}\cup \partial A)$ is also of log-general type.
 	Thus $K_{V}(\varphi^{-1}(\overline{D}\cup \partial A))$ is big.
	
	By Kodaira's lemma, there exist an effective divisor $E\subsetneqq \overline{A}'$ and a positive integer $l\in\mathbb Z_{\geq 1}$ such that $L=lK_{V}(\varphi^{-1}(\overline{D}\cup \partial A))-E$ is ample.	
	Hence if $g:Y\to V$ satisfies $g(Y)\not\subset E$, then 
		$$T_{g}(r)= O( T_{g}(r,K_{V}(\varphi^{-1}(\overline{D}\cup \partial A))))+O(\log r).$$

By \cref{lem:20220909gg} applied to $(Z\cup \varphi (\mathrm{Ex}(\varphi)))\cap D$, $L$ and $\varepsilon=1/3l$, we get a finite subset $P_B\subset \mathcal{S}(B)\backslash\{B\}$ and a proper Zariski closed set $\Phi_B\subsetneqq A$ with $\varphi (\mathrm{Ex}(\varphi))\subset \overline{\Phi_B}$.
We set $\Xi_B=\Phi_B\cup\varphi(E)$.
Let $f:Y\to A$ be a holomorphic map such that $f\in J_B$ and $f(Y)\not\subset \Xi_B$.
Then $ f(Y)\not\subset D\cup \Xi_B$ and $f\in I_B$.
Suppose that the first assertion of \cref{claim:20230405} is not valid.
Then $f\not\in I_C$ for all $C\in P_B$.
By \cref{lem:20220909gg}, we get
\begin{equation*}
 			T_{g}(r,K_{V}(\varphi^{-1}(\overline{D}\cup \partial A)))
 			\leq \overline{N}_{f}(r,D\backslash (Z\cup \varphi (\mathrm{Ex}(\varphi))))+\frac{1}{3l}T_{g}(r,L)
 			\\
 			+O(\log r)+o(T_{g}(r))||,
 		\end{equation*}
		where $g:Y\to V$ is the lift of $ f:Y\to A$.
By $g(Y)\not\subset E$, we have 
$$
T_{g}(r,L)\leq lT_{g}(r,K_{V}(\varphi^{-1}(\overline{D}\cup \partial A)))+O(\log r).
$$
Hence we get
\begin{equation*}
 			\frac{2}{3}T_{g}(r,K_{V}(\varphi^{-1}(\overline{D}\cup \partial A)))
 			\leq \overline{N}_{f}(r,D\backslash (Z\cup \varphi (\mathrm{Ex}(\varphi))))
 			\\
 			+O(\log r)+o(T_{g}(r))||.
 		\end{equation*}

We estimate the first term of the right hand side.
Since $\overline{A}$ is smooth and equivariant, $\partial A$ is a simple normal crossing divisor.
Hence we may decompose as $\varphi^{-1}(\overline{D}\cup \partial A)=H+F$ so that $F=\varphi^{-1}(\partial A)$.
Then $H=\overline{\varphi^{-1}(D)}$.
The induced map $V\backslash (F\cup \mathrm{Ex}(\varphi))\to A\backslash \varphi(\mathrm{Ex}(\varphi))$ is isomorphic.
Hence by $\mathrm{ord}_yf^*D\geq 2$ for all $y\in f^{-1}(D\backslash Z)$, we have
 	$$
 	2\overline{N}_{f}(r,D\backslash (Z\cup \varphi (\mathrm{Ex}(\varphi))))\leq T_g(r,H).
$$
By $\bar{\kappa}(V\backslash F)=0$ and $g(Y)\not\subset \mathrm{Ex}(\varphi)$, we have $T_g(r,K_V(F))+O(\log r)>0$.
Hence
$$
2\overline{N}_{f}(r,D\backslash (Z\cup \varphi (\mathrm{Ex}(\varphi))))\leq T_g(r,K_V(H+F))+O(\log r).
$$
	Hence we get
	$$
	T_{g}(r,K_{V}(H+F))=O(\log r)+o(T_{g}(r))||.
		$$ 	 	
	Thus $T_{g}(r)=O(\log r) ||$.
		This shows $T_{f}(r)=O(\log r) ||$. 
		Hence $f\in I_{\{0\}}$.
 	Hence by \cref{rem:20221203}, $f$ does not have essential singularity over $\infty$.
\end{proof}

 	We take a sequence $\mathcal{P}_i\subset \mathcal{S}(A)$ of finite sets as follows.
 	Set $\mathcal{P}_1=\{A\}$.
 	Given $\mathcal{P}_i$, we define $\mathcal{P}_{i+1}$ as follows.
 	For each $B\in\mathcal{P}_i$, we apply \cref{claim:20230405} to get $\Xi_B\subsetneqq A$ and $P_B\subset \mathcal{S}(B)\backslash\{B\}$.
 	We set $\mathcal{P}_{i+1}=\cup_{B\in \mathcal{P}_i}P_B$ and $\Xi_{i+1}=\cup_{B\in\mathcal{P}_i}\Xi_B$.
 	Set $\Xi=\cup_i\Xi_i$.
 	Then $\Xi\subsetneqq A$ is a proper Zariski closed set.
 	
 	Let $f:Y\to A$ satisfy the three properties in \cref{prop:20230405}.
	Then $f\in J_A$.
	We take $i$ which is maximal among the property that there exists $B\in \mathcal{P}_i$ such that $f\in J_B$. 	
	Then by $f(Y)\not\subset \Xi_B\subset \Xi$,
 	\cref{claim:20230405} implies that $f$ does not have essential singularity over $\infty$.
\end{proof}
 
 \begin{rem} 
The proof of \cref{thm2nd} is based on the arguments of \cite{Yam15} and \cite{NWY13}
.
Compared with the compact case treated in \cite{Yam15}, the lack of Poincar{\'e} reducibility theorem is a major difficulity to treat the non-compact case.
We use a more general ``cover'' than {\'e}tale cover to overcome this problem.
In \cite{NWY08}, we use ``flat cover''
 (cf. the commutative diagram just after \cite[(5.9)]{NWY08}. 
An important issue to use this flat cover is to construct a lift of holomorphic maps onto the flat cover so that the order function of the lift is bounded by that of original one (cf. \cite[Lemma 5.8]{NWY08}.
In \cite{NWY08}, we only treat the holomorphic maps from the complex plane $\mathbb C$, so thanks to the simply connectedness of $\mathbb C$, we may construct such lift easily.
In this paper, we are considering holomorphic maps from the covering $\pi:Y\to  \bC_{>\delta}$, which is not simply connected.
So, in this paper, we chose another cover $\sigma:\Sigma\to \overline{W/B}$ in the proof of \cref{lem:20220909gg}.
After the base change by this cover, we get the lift $g:Y'\to \overline{B}\times \Sigma$ of $f:Y\to W$ from the covering space $Y'\to Y$. 
Then we apply the results of \cref{subsec:4.2}--\cref{subsec:4.7a} to this lift $g$.
We remark that the map $g_{\overline{B}}:Y'\to \overline{B}$ may hit the boundary $\partial B$.
This is the reason to treat the situation $f:Y\da A\times S$ in \cref{subsec:4.2}--\cref{subsec:4.7a}.
 \end{rem}

	 \section{A reduction theorem for  non-Archimedean representation of $\pi_1$}\label{sec:reduction}
In this section we shall prove \cref{main3}. It is based on two results in    \cite{BDDM}: 
\begin{itemize}
	\item  the existence of $\varrho$-equivariant harmonic mapping $u:\widetilde{X}\to \Delta(G)$ from the universal cover $\widetilde{X}$ of $X$ to the Bruhat-Tits building of  $G$ such that the energy of $u$ at infinity has  logarithmic   growth;
	\item  the construction of logarithmic symmetric differential forms of $X$ via this harmonic mapping $u$.  
\end{itemize} 
  When $X$ is compact, \cref{main3} is proved by Katzarkov \cite{Kat97}, Zuo \cite{Zuo96} and Eyssidieux \cite{Eys04}.   When $X$ is  non-compact, there are several delicate and technical issues which occur, and we provide  complete details as possible during  the proof of \cref{main3}.  
\subsection{Some recollections on buildings}

Let \(K\) be a non-archimedian local field, and let \(G\) be a semi-simple group defined over \(K\). One can attach to \(G\) data its {\em Bruhat-Tits building} \(\Delta(G)\); this is a simplicial complex obtained by glueing affine spaces isometric to \(\mathbb{R}^{N}\) (called the {\em appartments}), where \(N = \mathrm{rk}_{K}(G)\). It is a ${\rm CAT}(0)$ space.  We refer the readers to \cite{AB08} for more details. 

There is a natural continuous action of \(G(K)\) on \(\Delta(G)\) which acts transitively on the appartments, and which is such that the stabilizer of any point in \(\Delta(G)\) is a bounded subgroup of \(G(K)\)  by \cref{lem:BT} below.

Fix an appartment \(A \subset \Delta(G)\). The affine Weyl group $W_{\rm aff}\subset {\rm Isom}(A)$ of $\Delta(G)$ and its finite reflection subgroup $W := W_{\rm aff}\cap GL(A)$ both act on \(A\). Also, if \(g \in G(K)\) is an isometry leaving \(A\) invariant, the restriction \(g|_{A}\) is induced by an element of \(W_{\rm aff}\).

The root system $\Phi=\{\alpha_1,\ldots,\alpha_m\}\subset A^*-\{0\}$ of $\Delta(G)$ is fixed under the action of \(W\):
\begin{align*}
	\{w^*\alpha_1,\ldots,w^*\alpha_m\}=\{\alpha_1,\ldots,\alpha_m\}\quad \mbox{for any}\ w\in W.
\end{align*}
In other words, $W$ acts on $\Phi$ by permutations. Note that $W_{\rm aff}=W\ltimes \Lambda$, where $\Lambda$ is a lattice which acts on $A$ by translations. Hence, it follows that 
\begin{align}\label{eq:coin}
	\{w^*d\alpha_1,\ldots,w^*d\alpha_m\}=\{d\alpha_1,\ldots,d\alpha_m\}\quad \mbox{for any}\ w\in W_{\rm aff}.
\end{align} 
Here we consider $d\alpha_i$ as global real one-forms on $A$. 
\medskip

\subsection{Some finiteness criterion for subgroups of almost simple algebraic groups}
We begin with the following definition. 
\begin{dfn}[Bounded subgroup]
	Let $G$ be a  semisimple algebraic group over the non-archimedean local field $K$. Fix an embedding $G\to {\rm GL}_N$. A subgroup $H$ of $G(K)$ is \emph{bounded} if  there is an upper bound on the absolute values of the matrix entries in ${\rm GL}_N(K)$ of the elements of $H$, otherwise it is called \emph{unbounded}. 
\end{dfn}
\begin{lem}[\protecting{\cite[11.40]{AB08}}]\label{lem:AB}
	Let $G$ be a  semisimple algebraic group over a non-archimedean local field $K$.   Let $H$ be a  subgroup  of $G(K)$. Then the following properties are equivalent. 
	\begin{itemize}
		\item $H$ is bounded.
		\item $H$ is contained in a compact subgroup of $G(K)$.
		\item $H$ fixes a point in $\Delta(G)$.
	\end{itemize}
\end{lem}
 	A representation $\varrho:\pi_1(X)\to G(K)$ is \emph{(un)bounded} if its image $\varrho(\pi_1(X))$ is  a (un)bounded subgroup of $G(K)$.

The following finiteness criterion will be used to prove \cref{genericallyfinite}, which is the cornerstone of \cref{main6}.
\begin{lem}\label{lem:BT}
	Let $G$ be an almost simple algebraic group over the non-archimedean local field $K$.  Let $\Gamma\subset G(K)$ be a finitely generated subgroup so that
	\begin{itemize}
		\item  it is a Zariski dense subgroup in $G$,
		\item it is not contained in any bounded subgroup of $G(K)$. 
	\end{itemize} 
	Let $\Upsilon$ be a normal subgroup of $\Gamma$ which is \emph{bounded}. Then $\Upsilon$ must be finite. 
\end{lem}
\begin{proof} 
To prove the lemma, we may assume that $G$ is connected.  	As $G(K)$ acts on $\Delta(G)$ transitively, we denote by  $R\subset \Delta(G)$ the set of fixed points of $\Upsilon$. Since $\Upsilon$ is bounded, $R$ is not empty. $R$ is moreover closed and convex. 	 
	We will prove that $R$ is moreover invariant under $\Gamma$, i.e. $\gamma R\subset R$ for $\gamma \in \Gamma$. 
	
	For every $u\in \Upsilon$ and $\gamma\in \Gamma$, one has  $\gamma^{-1}u\gamma\in \Upsilon$ since $\Upsilon\triangleleft \Gamma$.  Then for every $x\in R$, one has $u(\gamma x)=\gamma ((\gamma^{-1}u\gamma)x)=\gamma x$.  Hence $R$ is  invariant under $\Gamma$.
	
	If $R$ is bounded, by Bruhat-Tits' fixed point theorem (see \cite[Theorem 11.23]{AB08}), $\Gamma$ fixes a point in $R$. Then $\Gamma$ is a bounded subgroup of $G(K)$ by \cref{lem:AB}, which contradicts with our assumption. Hence $R$ is unbounded. 
	
Consider the   compactification $\overline{\Delta(G)}$ of the building $\Delta(G)$ by adding points at infinity.  A point at infinity in $\overline{\Delta(G)}$ is the equivalent class of geodesic rays in $\Delta(G)$ (see \cite[\S 11.8.1]{AB08}). Write $\partial \Delta(G):=\overline{\Delta(G)}-\Delta(G)$.  Note that the action of $G(K)$ on $\Delta(G)$ induces a natural action on $\overline{\Delta(G)}$.   Let $\widetilde{R}= \overline{\Delta(G)}^\Upsilon$ be the set of fixed points  of $\Upsilon$ in the compactification $\overline{\Delta(G)}$.  Then $\widetilde R$ is closed convex. If $\widetilde{R} \cap \partial\Delta(G)$ is empty, then $\partial\Delta(G)$ is covered by some cones (in finite number by compactness of $\partial\Delta(G)$) not intersecting $\widetilde R$. 
	These cones define the topology on $\overline{\Delta(G)}$ (the cone topology). Let O be an origin in $\Delta(G)$, such a cone is the set $C(y,\epsilon)$ of all $\xi$ in $\overline{\Delta(G)}$ such that the line segment (or geodesic ray) $[O,\xi]$ contains a point $x$ at distance less than $\epsilon$ from a fixed point $y \in \Delta(G)$. To get a basis of the topology, one has also to consider open balls in $\Delta$ as bounded cones.
	So if $\widetilde{R} \cap \partial\Delta(G)$ is empty, $\partial\Delta(G)$ is in the finite union of some cones $C(y,1/n)$ with $d(O,y)=n$   for some (great) $n$, and these cones do not intersect $\widetilde R$. Therefore $R=\widetilde{R}$ is in the ball $B(O,n)$. This contradicts with the unboundness of $R$. Hence $\widetilde{R}\cap \partial\Delta(G)\neq\varnothing$.  In other words, $\Upsilon$ fixes a point at infinity.   
	$\Upsilon$ is thus contained in $P(K)$ where $P\subset G$ is a proper parabolic subgroup of $G$. 
Write $H\subset G$ to be the Zariski closure of $\Upsilon$ in $G$. Then $H\subsetneq G$.   Since   $\Upsilon\triangleleft \Gamma$ and $\Gamma$ is Zariski dense in $G$, it follows that $H$ is a normal subgroup of $G$. 
 	Since 
	$G$ is assumed to be almost simple, we conclude that $H$, hence $\Upsilon$ is finite. This proves the lemma.   
\end{proof} 
\begin{rem}
It is worth acknowledging that  \cref{lem:BT} presented here was inspired by a discussion between the second author and Brunebarbe at a non-abelian Hodge Theory conference held in Saint-Jacut, France in June 2022.  After the second author sent the proof of \cref{lem:BT} to Brunebarbe, he informed the second author that his proof is similar to ours, which was later published in \cite{Bru22}. We are grateful for this fruitful discussion and the contributions made by Brunebarbe to this area.
\end{rem}

\subsection{Harmonic mappings} \label{sec:harmmapp}
The construction of $s_\varrho:X\to S_\varrho$ in \cref{main3} will be based on an existence theorem for harmonic mapping into Bruhat--Tits buildings. We refer to \cite{BDDM} for a precise definition of harmonic mapping with values in a Euclidean building.

\begin{thm}[{\cite[Theorem~2.1]{BDDM}}] \label{thm:BDDMex}
	Let $X$ be a quasi-projective manifold.  Let $G$ be a semi-simple group defined over a non-archimedean local field $K$.  Let $\varrho:\pi_1(X)\to G(K)$ be a Zariski dense representation.
	Then there exists a $\varrho$-equivariant harmonic mapping \(u : \widetilde{X} \to \Delta(G)\), which is locally Lipschitz and pluriharmonic. Moreover, its energy has at most logarithmic growth at infinity. 
\end{thm}

With this notation, recall that   \(x \in \widetilde{X}\) is called a {\em regular point}  of $u$  if it admits an open neighborhood \(U\) such that \(u(U) \subset A\) for some appartment \(A\). We say that a point \(x \in X\) is regular if some (equivalently, any) of its preimages in \(\widetilde{X}\) is regular. The regular points of $u$ in \(X\) form a non-empty open subset \(X^{\circ}\); denote by \(\mathcal{S}(u)\) its complementary. By the deep theorem of Gromov-Schoen, $\cS(u)$ has small Hausdorff dimension. 
\begin{thm}[\protecting{\cite[Theorem 6.4]{GS92}}]\label{thm:GS}
	$\cS(u)$ is a closed subset of $X^\circ$  with  \emph{Hausdorff  dimension}    at most $2\dim_{\bC}X-2$. \qed
\end{thm} 

\subsection{Proof of \cref{main3}} \label{sec:KZthm}
\begin{thm}[=\Cref{main3}] \label{thm:KZreduction}
	Let $X$ be a complex quasi-projective normal variety,  and let $G$ be a reductive algebraic group defined over a non-archimedean local field $K$. Let $\varrho:\pi_1(X)\to G(K)$ be a Zariski-dense representation. Then there exist  a quasi-projective normal variety $S_\varrho$ and a dominant morphism $s_\varrho:X\to S_\varrho$ with connected general fibers, such that for any connected Zariski closed  subset $T$ of $X$, the following properties are equivalent:
	\begin{enumerate}[label*={\rm (\alph*)}]
		\item  the image $\rho({\rm Im}[\pi_1(T)\to \pi_1(X)])$ is a bounded subgroup of $G(K)$.
		\item  For every irreducible component $T_o$ of $T$, the image $\rho({\rm Im}[\pi_1(T_o^{\rm norm})\to \pi_1(X)])$ is a bounded subgroup of $G(K)$.
		\item The image $s_\varrho(T)$ is a point.
	\end{enumerate} 
\end{thm}
We divide its proof into 8 steps. Throughout Steps 1-7, we assume that $X$ is smooth, and in Step 8, we address the case when $X$ is singular.  Steps 1-5 are dedicated to the situation where $G$ is semisimple, while in Steps 6-7, we establish the theorem for reductive $G$.  Steps 5-8 are independent of the other part of the paper and the readers may skip them.    It is worth highlighting that \cref{thm:KZreduction} will play a crucial role as a cornerstone in the   paper  \cite{DY23} by the second and third authors on the Shafarevich conjecture for reductive representations over quasi-projective varieties.
\begin{proof} 
	\noindent
	{\em Step 1. Construction of logarithmic symmetric forms on \(X\)}.\  	A large part of  arguments in this step is detailed in \cite[Section~3]{BDDM}, so we will give the necessary details, and refer to {\em loc. cit.} for more details. 
In the following, by a multiset, we mean a pair $(S, m)$ where $S$ is the underlying set and $m: S\to \mathbb {Z}_{\geq 0}$ is a function, giving the multiplicity, that is, the number of occurrences, of the element $s\in S$ in the multiset as the number $m(s)$.
We simply write $\{s_1,s_2,\ldots,s_k\}$ for a multiset so that each $s_i$ apears $m(s_i)$-times in this description.

	We start with the following definition. 
	\begin{dfn}[Multivalued section]
		Let $X$ be a complex manifold, and let $E$ be a holomorphic vector bundle on $X$. A  \emph{multivalued} holomorphic section of $E$ on $X$ consists of the following data:
		\begin{itemize}
			\item  an open covering $\{U_i\}_{i\in I}$ of $X$, 
			\item     
multisets $\{\omega_{i1},\ldots,\omega_{im} \}$ of holomorphic sections in $H^0(U_i, E|_{U_i})$
		\end{itemize} 
		satisfying that if $U_i\cap U_j\neq \varnothing$, then $$\{\omega_{i1}|_{U_i\cap U_j},\ldots,\omega_{im}|_{U_{i}\cap U_j} \}=\{\omega_{j1}|_{U_i\cap U_j},\ldots,\omega_{jm}|_{U_{i}\cap U_j} \} $$
counting multiplicities.
We shall abusively write $\{\omega_1,\ldots,\omega_m\}$ to denote a multivalued  holomorphic section on $X$. 
	\end{dfn}
	\begin{claim}\label{claim:multivalued}
		Assume that $\{\omega_1,\ldots,\omega_m\}$ and $\{\eta_1,\ldots,\eta_m\}$ are two multivalued holomorphic sections of  a holomorphic vector bundle $E$  on a connected complex manifold $X$. If there is an open set $U\subset X$ over which they coincide. Then $\{\omega_1,\ldots,\omega_m\}$  coincides with  $\{\eta_1,\ldots,\eta_m\}$ globally.  
	\end{claim}
	\begin{proof}[Proof of \cref{claim:multivalued}]
		Denote by $q:E\to X$ be projective map.	Let $\lambda\in H^0(E, q^*E)$ be the Liouville section defined by $\lambda(e)=e$ for any $e\in E$. Consider the section
		$$
		P(\lambda):=\prod_{i=1}^{m}(\lambda-q^*\omega_i)\in H^0(E, q^*\Sym^mE)
		$$
		which is globally defined. Then the zero scheme (counting multiplicities)  $Z$  of $P(\lambda)$ is the graph  (counting multiplicities)  of $\{\omega_1,\ldots,\omega_m\}$.  It is an analytic subscheme of the complex manifold $E$.  Similarly,  we consider the zero scheme  $Z'$ of  the section
		$$
		P'(\lambda):=\prod_{i=1}^{m}(\lambda-q^*\eta_i)\in H^0(E, q^*\Sym^mE)
		$$
		which is the graph of $\{\eta_1,\ldots,\eta_m\}$. By our assumption,  $P(\lambda)=P(\lambda')$ over the open set $q^{-1}(U)$. Since $E$ is connected,  $P(\lambda)=P(\lambda')$ over the whole $E$. Hence $Z$ and $Z'$ coincide, which implies our claim. 
	\end{proof}

	If we denote by \(r\) the \(K\)-rank of \(G\), the appartments of \(\Delta(G)\) are all isometric to \(\mathbb{R}^{r}\); let \(A\) be one of these apartments. Let $\pi:\widetilde{X}\to X$ be the universal covering map.

	Let \(u : \widetilde{X} \to \Delta(G)\) be the harmonic map provided by Theorem~\ref{thm:BDDMex}. For any regular point $x\in X$ of $u$, there exists an open simply-connected neighborhood  $U$ of $x$   so that
	\begin{itemize}
		\item  the inverse image $\pi^{-1}(U)=\coprod_{i\in I}\Omega_i$ is a union of disjoint open sets in $\widetilde{X}$,  each of which is mapped homeomorphically onto $U$ by $\pi$.
		\item For any $i\in I$,  there is an apartment $A_i$ of $\Delta(G)$ such that $u(\Omega_i)\subset A_i$. Let \(g_{i} \in G\) be such that \(g_{i} \cdot A = A_{i}\), and denote \(u_{i} := g_{i} \circ u \circ (\pi|_{\Omega_{i}})^{-1} : U \to A\).
	\end{itemize}
	
	By the canonical isomorphism between the affine Weyl group \(W_{\mathrm{aff}}\) and the stabilizer of \(A\), one sees that two 
applications \(u_{i}\) and \(u_{j}\) differ by post-composition with an element of \(W_{\mathrm{aff}}\). Letting \(\{\alpha_{1}, \ldots, \alpha_{n}\}\) be the system of roots on \(A\), this implies with \eqref{eq:coin} that the multiset
	\[
 	\{u_{i}^{\ast} (d \alpha_{1})^{(1,0)}_{\mathbb{C}},
	\dotsc,
	u_{i}^{\ast} (d \alpha_{n})^{(1, 0)}_{\mathbb{C}}\}
	\] 
	is independent of \(i\).

	Since \(u\) is pluriharmonic, this gives a multivalued holomorphic one forms $\{\omega_1,\ldots,\omega_n\}$ on the open set $X^\circ:= X-\cS(u)$. Let $\{U_i\}_{i\in I}$ be an open covering of $X^\circ$ such that $\{\omega_{i1},\ldots,\omega_{in}\} \subset \Gamma(U_i,\Omega_{U_i})$ is the realization of $\{\omega_1,\ldots,\omega_n\}$ over $U_i$.    Note that it might happens that $\omega_{ij}=\omega_{i\ell}$ for $j\neq\ell$. 
	
Recall that  $X\backslash X^\circ$ is of Hausdorff codimension at least two by \cref{thm:GS}. Hence $X^\circ$ is connected. 
We take the underlying set of the multiset $\{\omega_{i1},\ldots,\omega_{in}\}$ by ignoring multiplicities.
Namely, we reorder 
the holomorphic one forms $\{\omega_{i1},\ldots,\omega_{in}\}$  on each $U_i$  such that  there exists an integer $m$ with $0\leq m \leq n$  and    over any open set $U_i$ of the covering $\{U_i\}_{i\in I}$  one has
\begin{itemize}
	\item    for any $j,\ell\in \{1,\ldots,m\}$ with $j\neq \ell$, one has  $\omega_{ij}\neq \omega_{i\ell}$. 
	\item For every $m<\ell\leq n$, there exists $j\in \{1,\ldots,m\}$   such that $\omega_{ij}=\omega_{i\ell}$.   
\end{itemize} 
We note that the sets  of holomorphic sections $\{\omega_{i1},\ldots,\omega_{im} \}_{i\in I}$  generate a multivalued holomorphic one form on $X^\circ$, denoted by $\{\omega_1,\ldots,\omega_m\}$.

 	Let $T$ be a formal variable. Then one can form the product
	\begin{align}\label{eq:construction}
		\prod_{k=1}^{m}\big(T-\omega_k\big)=:T^m+\sigma_{1}T^{m-1}+\cdots+\sigma_{m} 
	\end{align}
	where the coefficients \(\sigma_{k}\) are elements of the spaces $H^0(X^{\circ}, {\rm Sym}^k\Omega_{X^{\circ}})$.  
It is proved in \cite[Proposition~3.2]{BDDM} that
  $\sigma_{k}$ extends to a logarithmic symmetric form $H^0(\overline{X}, \Sym^k\Omega_{\overline{X}}(\log D))$, which is still be denoted by $\sigma_k$. 
  
	\medskip 
	
	\noindent
	{\em Step 2.  Construct an intermediate Galois covering of $X$.}   
	Denote by $p: \Omega_{ \overline{X}}(\log D)\to\overline{X}$ the projection map. Let $ {\lambda}  \in H^0(\Omega_{\overline{X}}(\log D),  {p}^*\Omega_{\overline{X}}(\log D))$ be the  Liouville section  on $\Omega_{ \overline{X}}(\log D)$.  For the section
	$$
	P( {\lambda}):=	\lambda^m+p^*\sigma_1\lambda^{m-1}+\cdots+p^*\sigma_m\in H^0(\Omega_{\overline{X}}(\log D), \bar{p}^*\Omega_{\overline{X}}(\log D)),
	$$
	let $\overline{Z}$ be an irreducible component of  the  zero locus $(P(\lambda)=0)$ in $\Omega_{ \overline{X}}(\log D)$ which dominates $\overline{X}$ under $p$.     Then $p|_{\overline{Z}}:\overline{Z}\to \overline{X}$ is a finite (affine) surjective morphism.  Hence $\overline{Z}$ is a projective variety.   
	Since $X\backslash X^\circ$ is of Hausdorff codimension at least two by \cref{thm:GS}, it follows that $X^\circ$ is connected. Then 
	$\overline{Z}$ gives rise to a multivalued section $\{\varphi_{1},\ldots,\varphi_{k}\}$ of $\Omega_{X^\circ}$ on $X^\circ$  such that  
	\begin{itemize}
		\item $\{\varphi_{1},\ldots,\varphi_{k}\}\subset \{\omega_1,\ldots,\omega_m\}$. 
		\item $Z^\circ:=p^{-1}(X^\circ)$ is the graph of $\{\varphi_{1},\ldots,\varphi_{k}\}$. 
	\end{itemize} 
	Set  
	\begin{align}\label{eq:pi}
		Q(\lambda)=\prod_{j=1}^{k}(\lambda-p^*\varphi_i)=:\lambda^k+\delta_1\lambda^{k-1}+\cdots+\delta_k
	\end{align} 
	with $\delta_i\in H^0(X^\circ, \Sym^{i}\Omega_{X^{\circ}})$.  
	By \cite[Proposition 3.2]{BDDM} again,  $\delta_i$ extends to logarithmic forms in $(\overline{X}, D)$.  Therefore, the zero locus of $Q(\lambda)$ contains $\overline{Z}$ and coincides with $\overline{Z}$ over $Z^\circ:=p^{-1}(X^\circ)$. In \cref{claim:zero locus} we will prove that the zero locus $(Q(\lambda)=0)=\overline{Z}$.

	\medskip

	Let $q: \overline{W} \to \overline{Z}$ be the normalization of \(\overline{Z}\) in the Galois closure of the extension $\bC(\overline{Z})/\bC(\overline{X})$.    
	Denote by $H={\rm Aut}(\overline{W}/\overline{X})$.   Then $\overline{X}=\overline{W}/H$.  Write  $W^\circ:=q^{-1}(Z^\circ)$,   $\varpi:=\lambda|_{Z^{\circ}}$, and $\pi:=p|_{\overline{Z}}\circ q$.     
	\begin{claim}\label{claim:permutes2}
		There are $g_{1},\ldots, g_{k}\in H$ so that
		$$
		\{(q\circ g_{1})^*\varpi,\ldots, (q\circ g_{k})^*\varpi\}= \{ \pi^*\varphi_{1},\ldots,   \pi^*\varphi_{k}\}.
		$$ 
		Moreover, for every $g\in H$, one has
		$$
		(q\circ g)^*\varpi\in\{ \pi^*\varphi_{1},\ldots,   \pi^*\varphi_{k}\}.
		$$ 
	\end{claim}
	\begin{proof}[Proof of \cref{claim:permutes2}]
	Set $Z:=\overline{Z}\cap \Omega_X$.  	Let us denote by $X'\subset  {X}$ the  Zariski open set over which ${{Z}}\to {X}$ is \'etale.
  Set $X^\circ_\circ:=X'\cap X^\circ$.   It follows from \cref{sec:covGal} that $\overline{W}\to \overline{X}$ is \'etale over $X^\circ_\circ$.  Pick  any point $x\in X^\circ_\circ$. We can take a connected analytic open set $U\subset X^\circ_\circ$ containing $x$ such that 
		\begin{itemize}
			\item $p^{-1}(U)=\coprod_{j=1}^{k}U_j$, where $U_j$ is the graph of some $\varphi_{j}\in \Gamma(U, \Omega_U)$ with $p|_{U_j}:U_j\to U$ isomorphic;
			\item $q^{-1}(U_j)=\coprod_{\alpha=1}^{M}U_{j\alpha}$ with $q|_{U_{j\alpha}}:U_{j\alpha}\to U_j$  isomorphic;
			\item for any $U_{i\alpha}$ and $U_{j\beta}$, there is $g\in H$ so that  $g|_{U_{i\alpha}}:U_{i\alpha}\to U_{j\beta}$  is an isomorphism;
			\item for any $g\in H$ and any $U_{i\alpha}$,  $g(U_{i\alpha})=U_{j\beta}$ for some $U_{j\beta}$.
		\end{itemize} 
Since $X'\backslash X^\circ$ is of Hausdorff codimension at least two by \cref{thm:GS},    it follows that    $W^\circ$ is   connected.  Fix some $U_{i\alpha}$ and by \cref{claim:multivalued} it suffices to prove the claim over it.    Pick any $g\in H$. Then $g(U_{i\alpha})=U_{j\beta}$ for some $U_{j\beta}$.  Note that the isomorphism $g:U_{i\alpha}\to U_{j\beta}$ factors through
		$$
		U_{i\alpha}\stackrel{q|_{U_{i\alpha}}}{\longrightarrow} U_i\stackrel{p|_{U_{i}}}{\longrightarrow}U\stackrel{(p|_{U_{j}})^{-1}}{\longrightarrow} U_{j}\stackrel{(q|_{U_{j\beta}})^{-1}}{\longrightarrow}U_{j\beta}.
		$$ 
		By our construction,  $\varpi|_{U_j}=\lambda|_{U_j}=p^*\varphi_j|_{U_j}$.  
		Hence $	\big((q\circ g)^*\lambda\big)|_{U_{i\alpha}}=\pi^*\varphi_{j}|_{U_{i\alpha}}$.  This proves the second statement. 
		
		Pick any $j\in \{1,\ldots,k\}$.  By the above construction, there exist  some $g_j\in H$  and $U_{j\beta}$ such that $g_j|_{U_{i\alpha}}:U_{i\alpha}\to U_{j\beta}$  is an isomorphism. By the above argument,   $	\big((q\circ g_j)^*\varpi\big)|_{U_{i\alpha}}=\pi^*\varphi_{j}|_{U_{i\alpha}}$.  This proves the first statement.  
	\end{proof}
	
	By the above claim, the multivalued section  $\{ \pi^*\varphi_{1},\ldots,   \pi^*\varphi_{k}\}
	$ coincides with  the set of \emph{global} sections in $W^\circ$  defined by 
	$$
	\{\eta^\circ_{1},\ldots,\eta^\circ_{k}\}:=\{(q\circ g_{1})^*\varpi,\ldots, (q\circ g_{k})^*\varpi\}  \subset H^0(W^{\circ}, \pi^*\Omega_{W^\circ}).
	$$  Moreover,  any $g\in H$ acts on $	\{\eta^\circ_{1},\ldots,\eta^\circ_{k}\}$ as a permutation. 
	
	\begin{claim}\label{claim:extension2}  $
		\{\eta^\circ_{1},\ldots,\eta^\circ_{k}\}$
		extends to global sections $\{\eta_{1},\ldots,\eta_{k}\}\subset H^0(\overline{W}, \pi^*\Omega_{\overline{X}}(\log D))$. For any $g\in H$, $g$ acts on $\{\eta_{1},\ldots,\eta_{k}\}$ as a permutation.  
	\end{claim} 
	\begin{proof}[Proof of \cref{claim:extension2}] 
		For any open set $U\subset \overline{W}$  and any holomorphic section  $s\in \Gamma(U, \pi^*T_{\overline{X}}(-\log D)|_U)$.  Write $U^\circ:=U\cap W^\circ$.   It suffices to prove  that $ \eta_i^\circ(s)\in \Gamma(U^\circ, \cO_{U^\circ}) $ extends to a holomorphic function in $U$.  By \cref{claim:permutes2}, one has
		$$
		\prod_{i=1}^{k}(T- \eta_i^\circ)=(T^k+\pi^*\delta_1T^{k-1}+\cdots+\pi^*\delta_k)|_{U^\circ}
		$$
		which implies that
		$$
		\prod_{i=1}^{k}(T-\eta_i^\circ(s)) =\big(T^k+\pi^*\delta_1(s)T^{k-1}+\cdots+\pi^*\delta_k(s))|_{U^\circ},
		$$ 
		Recall that $\delta_i\in H^0( \overline{X}, \Sym^i\Omega_{ \overline{X}}(\log D) )$. It follows that the $\pi^*\sigma_i(s)$ are holomorphic functions on $\cO(U)$.   After shrinking $U$, we may assume that  $M:=2\max_{k=1}^{m}\sup_K |\pi^*\sigma_k(s)|^{\frac{1}{k}}$ is finite. Then by the classical inequalities between the norms of roots and coefficients of a polynomial, one has
		$$
		|\eta_i^\circ(s)|(z)\leq M
		$$ 
		for all $z\in   U^\circ$.   
		Recall that $\overline{W}\backslash W^\circ$    is a  closed subset of Hausdorff codimension at least two. By the removable singularity of Shiffman \cite[Lemma 3.2.(ii)]{Shi68},   $\eta_i^\circ(s)|_{U^\circ}$  extends to holomorphic functions  on  $U$.  The first statement is proved. 
		
		For any $g\in H$ and $\eta_i$, it follows that $g^*\eta_i\in H^0(\overline{W}, \pi^*\Omega_{\overline{X}}(\log D))$. By \cref{claim:permutes2}, $g^*\eta_i^\circ=\eta_j^\circ$ for some $j\in \{1,\ldots,m\}$. By the continuity, $g^*\eta_i=\eta_j$. 
		The second statement is proved. 
	\end{proof}  
	
	\begin{claim}\label{claim:ramified1}
		The Galois morphism $\pi:\overline{W}\to \overline{X}$ is \'etale outside 
$$\{z\in \overline{W} \mid \exists \eta_i\neq\eta_j \ \mbox{with}\ (\eta_i-\eta_j)(z)=0 \}.$$
	\end{claim}
	\begin{proof}[Proof of \cref{claim:ramified1}] 
 	For the Galois morphism $\pi:\overline{W}\to \overline{X}$,  we denote by $\nu:\pi^*\Omega_{\overline{X}}(\log D)\to \overline{W}$ the natural morphism.   
		
		Consider the graph variety $\Xi\subset \pi^*\Omega_{\overline{X}}(\log D)$ defined by sections $\eta_{1},\ldots,\eta_{k}$.    Let $\lambda'\in H^0(\pi^*\Omega_{\overline{X}}(\log D), \nu^*\pi^*\Omega_{\overline{X}}(\log D))$ be the Liouville section.  Consider the section
		$$
		Q'(\lambda'):=	\lambda'^{k}+\nu^*\pi^*\delta_1\lambda'^{k-1}+\cdots+\nu^*\pi^*\delta_k\in H^0(\pi^*\Omega_{\overline{X}}(\log D), \nu^*\pi^*{\rm Sym}\Omega_{\overline{X}}(\log D))
		$$
		defined on the total space $\pi^*\Omega_{\overline{X}}(\log D)\to \overline{W}$ where $\delta_i\in H^0(\overline{X},\Sym^i\Omega_{\overline{X}}(\log D))$ are defined in \eqref{eq:pi}. Since $\eta_i\neq\eta_j$ for $i\neq j$,  it follows that the zero scheme of    $Q'(\lambda')$ is reduced which is nothing but $\Xi$.    By this construction, $\nu|_{\Xi}:\Xi\to \overline{W}$ is a finite morphism and  \'etale outside 
		$$\mathcal{R}:=\{z\in \overline{W} \mid \exists \eta_i\neq\eta_j \ \mbox{with}\ (\eta_i-\eta_j)(z)=0 \}.$$
		Since any $g\in H$ acts on $\{\eta_{1},\ldots,\eta_{k}\}$ as a permutation, one obtains $\pi^{-1}\pi(\mathcal{R})=\mathcal{R}$.   
		
		Consider the zero scheme  $\Sigma$   of the section
		$$
		Q(\lambda):=	\lambda^k+p^*\sigma_1\lambda^{k-1}+\cdots+p^*\sigma_k\in H^0(\Omega_{\overline{X}}(\log D), p^*\Sym^m\Omega_{\overline{X}}(\log D))
		$$
		defined on the total space $p:\Omega_{\overline{X}}(\log D)\to X$, where $\lambda$ be the  Liouville 1-form on $\Omega_{\overline{X}}(\log D)$.   By this construction, one can see that   
		\begin{equation*}
			\begin{tikzcd}
				\Xi=\overline{W}\times_{\overline{X}}\Sigma \arrow[r]	\arrow[d, "\nu|_{\Xi}"] &  \Sigma\arrow[d, "p|_{\Sigma}"]\\
				\overline{W}\arrow[r, "\pi"] &	    \overline{X}
			\end{tikzcd}
		\end{equation*}
		Since each irreducible component of $\Xi$ is the graph of some $\eta_i$,  it is isomorphic to $\overline{W}$. It follows that any irreducible component of $\Sigma$ is dominant over $\Xi$ under $p|_{\Sigma}$. However, $\overline{Z}$ is an irreducible component of $\Sigma_{\rm red}$ which coincide over $X^\circ$. This implies that $\Sigma_{\rm red}=\overline{Z}$.    Recall that $\nu|_\Xi:\Xi\to \overline{W}$ is    \'etale outside $\mathcal{R}$, and $\mathcal{R}=\pi^{-1}(\pi(\mathcal{R}))$.   By \cite[\href{https://stacks.math.columbia.edu/tag/0475}{Tag 0475}]{stacks-project}, $p|_{\Sigma}$  is unramified over $(p|_{\Sigma})^{-1}(\overline{X}\backslash\mathcal{R})$,  which is moreover \'etale  by \cite[\href{https://stacks.math.columbia.edu/tag/0GS7}{Tag 0GS7}]{stacks-project}.   By \cref{sec:covGal}, we conclude that $\pi:\overline{W} \to \overline{X}$ is a finite \'etale morphism outside $\mathcal{R}$.  
	\end{proof}
	
	From the above proof we obtain the following result. 
	\begin{claim}\label{claim:zero locus}
		The zero locus of $Q(\lambda)$ is $\overline{Z}$.  \qed
	\end{claim}

	\noindent {\it Step 3. Construct the spectral covering associated to the harmonic mapping.} 
In this step we will construct a finite Galois cover $\pi: \xsp\to X$ from a connected quasi-projective normal variety $\xsp$ such that   multivalued holomorphic section  $\{\pi^*\omega_1,\ldots,\pi^*\omega_m\}$ over $\pi^{-1}(X^\circ)$ becomes   single valued.   
	
	Consider the zero locus of $P(\lambda)$. Let $\overline{Z}_1,\ldots,\overline{Z}_N$ be its irreducible components which dominate $\overline{X}$ under $p:\Omega_{ \overline{X}}(\log D)\to \overline{X}$.  Write    $Z^\circ_i:=(p|_{\overline{Z}_i})^{-1}(X^\circ)$. They generate $N$ multivalued sections $\{\omega_{i1},\ldots,\omega_{ik_i}\}_{i=1,\ldots,N}$ of $\Omega_{X^{\circ}}$ over $X^\circ$ such  that 
	\begin{itemize}
		\item  \begin{align} \label{eq:partition}	\{\omega_1,\ldots,\omega_m\}=\coprod_{i=1}^{N}\{\omega_{i1},\ldots,\omega_{ik_i}\}
		\end{align} 
		\item  $Z^\circ_i$ is the graph of $\{\omega_{i1},\ldots,\omega_{ik_i}\}$ \end{itemize} 
	Set 
	\begin{align*} 
		P_i(\lambda)=\prod_{j=1}^{k_i}(\lambda-p^*\omega_{ij})=:	\lambda^{k_i}+p^*\sigma_{i1}\lambda^{k_i-1}+\cdots+p^*\sigma_{ik_i} 
	\end{align*} 
	with $\sigma_{ij}\in H^0(X^\circ, \Sym^{j}\Omega_{X^{\circ}})$.    
	Then $P(\lambda)=\prod_{i=1}^{N}P_i(\lambda)$ and     $\sigma_{ij}$ extends to logarithmic forms  $H^0(\overline{X}, \Sym^{j}\Omega_{\overline{X}}(\log D))$. By \cref{claim:zero locus} the zero locus of $P_i(\lambda)$  is $\overline{Z}_i$.  Therefore, the zero locus of $P(\lambda)$ is $\cup_{i=1}^{N}\overline{Z}_i$. Let  $q_i:\overline{W}_i\to \overline{Z}_i$ be the Galois closure of the finite morphism $p_i:\overline{Z}_i\to \overline{X}$.  Then the natural morphism $\pi_i:\overline{W}_i\to \overline{X}$ is a Galois cover with the Galois group $H_i$. 
	
	Consider the  normalization $\widehat{W}$ of $   \overline{W_1}\times_{\overline{X}}\times\cdots\times_{\overline{X}}  \overline{W_N}$, which might be not connected. The induces morphism $ \widehat{W}\to \overline{X}$ is thus also a Galois morphism with Galois group $\widehat{G}:=H_1\times\cdots\times H_N$. Therefore, $\widehat{W}$ is a disjoint union of isomorphic projective normal varieties.  We pick one connected component,  denoted by $\overline{\xsp}$.  
	Define a subgroup 
	$$
	\widetilde{G}:=\{g\in \widehat{G}\mid g(\overline{\xsp})=\overline{\xsp} \}. 
	$$
	Then $\pi:\overline{\xsp}\to \overline{X}$ is also a Galois morphism  with the Galois group $\widetilde{G}$.   Denote by $r_i:\overline{\xsp}\to \overline{W_i}$ the natural map, which is a finite surjective morphism. \begin{dfn}[Spectral cover]\label{def:spectral}
		Write $\xsp:=\pi^{-1}(X)$.  The   Galois morphism $\xsp\to X$  is called the \emph{spectral cover} of $X$ with respect to the multivalued one forms $\{\omega_1,\ldots,\omega_m\}$ defined on $X^\circ$. 
	\end{dfn} 
	 By \cref{claim:permutes2,claim:extension2},  there are global sections  $$
	\{\eta _{i1},\ldots,\eta_{ik_i}\} \subset H^0(\overline{W_i}, \pi_i^*\Omega_{\overline{X}}(\log D)).
	$$  such that it coincides  with the multivalued section $\{\pi_i^*\omega_{i1},\ldots,\pi_i^*\omega_{ik_i}\}$ over $\pi_i^{-1}(X^\circ)$.   By \cref{claim:ramified1} we know that $\overline{W_i}\to  \overline{X}$ is \'etale outside 
	\begin{align}\label{eq:ramification locus}
		\mathcal{R}_i:=\{z\in \overline{W}_i \mid \exists \eta_{ik}\neq\eta_{i\ell} \ \mbox{with}\ (\eta_{ik}-\eta_{i\ell})(z)=0 \}.
	\end{align} 
 	Denote by  
	\begin{align*} 
		\{\eta_1,\ldots,\eta_m\}:=\coprod_{i=1}^{N}\{r_i^*\eta_{i1},\ldots,r_i^*\eta_{ik_i}\}\subset H^0(\overline{\xsp}, \pi^*\Omega_{\overline{X}}(\log D)) 
	\end{align*} 
	By \cref{claim:permutes2} and \eqref{eq:partition}, we obtain
	\begin{claim}\label{claim:extension}
		The multivalued section	$ \{ \pi^*\omega_{1},\ldots,   \pi^*\omega_{m}\}$ coincides with 	 $
		\{\eta_1,\ldots,\eta_m\}\subset H^0(\overline{\xsp}, \pi^*\Omega_{ \overline{X}}(\log D))  
		$ 
		over $\pi^{-1}(X^\circ)$. 
		Any $g\in  \widetilde{G}={\rm Aut}(\overline{\xsp}/\overline{X})$ acts on $\{\eta_1,\ldots,\eta_m\}$ by a permutation.  \qed
	\end{claim} 
	\begin{dfn}[Spectral one forms]\label{def:spectral one forms}
		These one forms $\{\eta_1,\ldots,\eta_m\}\subset   H^0(\overline{\xsp}, \pi^*\Omega_{\overline{X}}(\log D))$ will be called \emph{spectral one forms} induced by the above $\varrho$-equivariant harmonic mapping $u$.
	\end{dfn}
 	\cref{claim:ramified1} enables us to describe the ramification of the spectral covering $\xsp\to X$.
	\begin{claim}\label{claim:ramified}
		The Galois morphism $\pi:\overline{\xsp}\to \overline{X}$ is \'etale outside $$R:=\{z\in \overline{\xsp} \mid \exists \eta_i\neq\eta_j \ \mbox{with}\ (\eta_i-\eta_j)(z)=0 \},$$
		which satisfies $\pi^{-1}(\pi(R))=R$. 
	\end{claim}
	\begin{proof}[Proof of \cref{claim:ramified}]
		We know that $\overline{W}_i\to \overline{X}$ is \'etale outside $\mathcal{R}_i$ with $\mathcal{R}_i$ defined in \eqref{eq:ramification locus}.  Therefore, $\overline{\xsp}\to \overline{X}$ is \'etale outside  $\cup_{i=1}^{N}r_i^{-1}(\mathcal{R}_i)$.  Note that $\cup_{i=1}^{N}r_i^{-1}(\mathcal{R}_i)\subset R$. Hence $\overline{\xsp}\to \overline{X}$ is \'etale  outside $R$.  Since  any $g\in  \widetilde{G}={\rm Aut}(\overline{\xsp}/\overline{X})$ acts on $\{\eta_1,\ldots,\eta_m\}$ by a permutation  by \cref{claim:extension}, it follows that $\pi^{-1}(\pi(R))=R$. 
	\end{proof} 

	\medspace

\noindent 	{\em Step 4. Partial quasi-Albanese morphism via spectral one forms.}  
Let $\mu:\widetilde{\xsp}\to \overline{\xsp}$   be a	$\widetilde{G}$-equivariant resolution of singularities such that $E:=({\pi}\circ \mu)^{-1}(D)$  is a simple normal crossing divisor.
	Write $Y:=({\pi}\circ \mu)^{-1}(X)$, which is a quasi-projective manifold. One has \(T_{1}(Y) \cong H^0(\widetilde{\xsp}, \Omega_{\widetilde{\xsp}}(\log E))\). 
	Let us denote by  
	$
	\Omega^1(X^{\! \rm sp})\subset T_1(Y)
	$ 
	the $\bC$-subspace of $T_1(Y)$ which vanishes on each fiber of $\mu|_{Y}:Y\to X^{\! \rm sp}$. Since each $\eta\in T_1(Y)$ can be seen as a linear form on  $T_1(Y)^*$,  we may consider the largest semi-abelian subvariety $A$ of the quasi-Albanese variety $\cA_Y$ of $Y$ so that   all $\eta\in \Omega^1(X^{\! \rm sp})$ vanishes on $A$. 
	Define 
	$
	\cA_{X^{\! \rm sp}}:= \cA_Y/A. 
	$ 
	Since $\mu$ is $\widetilde{G}$-equivariant, it follows that $\Omega^1(X^{\! \rm sp})$ is stable under the induced $\widetilde{G}$-action on $T_1(Y)$. Therefore, for the natural $\widetilde{G}$-action on $\cA_Y$ induced by the $\widetilde{G}$-action on $Y$, it gives rise to a  $\widetilde{G}$-action on $\cA_{\xsp}$.

	\begin{claim}\label{claim:factor}
		The composition $Y\to \cA_Y\to \cA_{X^{\! \rm sp}}$ factors through $\alpha_1:X^{\! \rm sp}\to \cA_{X^{\! \rm sp}}$.   Moreover, $\alpha_1$ is $\widetilde{G}$-equivariant. 
	\end{claim} 
	\begin{proof}[Proof of \cref{claim:factor}]
		Since $\xsp$ is normal,  each  fiber $F$ of $Y\to \xsp$ is connected. By definition   $\eta|_{F}=0$ for any $\eta\in \Omega^1(\xsp)$.  It follows from \cref{lem:critpointalb} that $F$ is mapped to one point under the morphism $Y\to \cA_{\xsp}$.  Hence it factors through $\xsp\to \cA_{\xsp}$.   The second one is easy to show since all our previous constructions are \(\widetilde{G}\)-equivariant.
	\end{proof}

	By \cref{claim:extension}, the spectral one forms $\{\eta_1,\ldots,\eta_m\}\subset \Omega^1(X^{\! \rm sp})$.  
	Define $A'\subset \cA_Y$ to be   the largest semi-abelian variety contained in $\cA_Y$ on which each $\eta_i$ vanishes.  Hence $A'\supset A$, which implies that there is a  morphism $\cA_{X^{\! \rm sp}}\to \cA_Y/A'$.    We denote by $\cA:=\cA_Y/A'$ the quotient, which is also a semi-abelian variety. Consider the morphism $q:Y \to \cA$ which is the composition of the quasi-Albanese morphism $\alpha_Y:Y\to \cA_Y$ and the quotient map $\cA_Y\to \cA$.   By the above claim, it follows that $q$ factors through $\alpha: X^{\! \rm sp}\to \cA$. 
	\begin{dfn}[Partial quasi-Albanese morphism]\label{def:partial}
		The above morphism $\alpha: \xsp\to \cA$ is called the partial quasi-Albanese morphism  induced by the spectral one forms $\{\eta_1,\ldots,\eta_m\}$.  
	\end{dfn}
	\begin{claim}\label{claim:trivial}
		For every   closed subvariety $F$ of $\xsp$,  $\alpha(F)$ is a point if and only if $\eta_i|_F$ is trivial for each spectral one form $\eta_i$.
	\end{claim}
	\begin{proof}
		It suffices to apply Lemma~\ref{lem:critpointalb} with \(X\) replaced by \(Y\), \(S\) replaced by \(\{\eta_{1}, \dotsc, \eta_{m}\}\) and \(B\) by \(A'\).
	\end{proof}
	Since $\{\eta_1,\ldots,\eta_m\}$ is invariant under $\widetilde{G}$, it follows that the $\widetilde{G}$-action on $\cA_Y$ gives rise to a $\widetilde{G}$-action on $\cA$. Hence $\alpha$ is $\widetilde{G}$-equivariant. As $X=X^{\! \rm sp}/\widetilde{G}$, this gives rise to a morphism $\beta:X\to \cA/\widetilde{G}$ satisfying the following commutative diagram 
	\begin{eqnarray*}
		\begin{tikzcd}
	 &	&&& S_{\nu^*\varrho} \arrow[dr, bend left=10]& \\
&\widetilde{Y}\arrow[dl]\arrow[dr,"\widetilde{\nu}"']\arrow[r,"p_Y"]&	Y\arrow[r, "\mu"]\arrow[rru, bend left=10,  "\varphi"] \arrow[dr,"\nu"']&		X^{\! \rm sp} \arrow[rr, "\alpha"] \arrow[d, "\pi"] & &  \cA\arrow[d]\\
\Delta(G)&&\widetilde{X}\arrow[ll, "u"]\arrow[r,"p_X"'] &			X\arrow[rr, "\beta"] \arrow[dr, "s_\varrho"'] && \cA/\widetilde{G}\\
&&	&		 &S_\varrho\arrow[ur]&
		\end{tikzcd}
	\end{eqnarray*}   
Let $s_\varrho:X\rightarrow S_\varrho$ be the quasi-Stein factorisation of $\beta$, and let $\nu:Y\rightarrow X$ be the composition of $Y\rightarrow \xsp$ and $\xsp\rightarrow X$. Suppose $\widetilde{X}$ is the universal covering of $X$, and let $\widetilde{Y}$ be a connected component of $\widetilde{X}\times_XY$. Then, the induced map $\widetilde{\nu}:\widetilde{Y}\rightarrow \widetilde{X}$ is proper and  $p_Y:\widetilde{Y}\rightarrow Y$ an unramified covering.  Let $\varphi:Y\rightarrow S_{\nu^*\varrho}$ be the quasi-Stein factorization of $q:Y\rightarrow \cA$. 

\medspace

\noindent {\em Step 5. Property of the reduction $s_\rho:X\to S_\rho$.}  
\begin{claim}\label{claim:bounded subgroup}
	Let $T$ be any connected Zariski closed subset of $X$. Then the following properties are equivalent.
	\begin{enumerate}[label*=(\arabic*)]
	\item The image $s_\varrho(T)$ is a point in $S_\varrho$.			\item  The image $\rho({\rm Im}[\pi_1(T)\to \pi_1(X)])$ is a bounded subgroup of $G(K)$.
		\item  For every irreducible component $T_o$ of $T$, the image $\rho({\rm Im}[\pi_1(T_o^{\rm norm})\to \pi_1(X)])$ is a bounded subgroup of $G(K)$.
		\end{enumerate} 	
\end{claim} 
\begin{proof}
	For the representation $\nu^*\varrho:\pi_1(Y)\to G(K)$,  it is known from \cite{Eys04,DMunique} that that $u\circ\widetilde{\nu}: \widetilde{Y}\to \Delta(G)$ is a  $\nu^*\varrho$-equivariant harmonic mapping of \emph{logarithmic energy growth} (cf. \cite{BDDM} for the definition) which is pluriharmonic.  By the definition of the singular set of a harmonic mapping, we have $\cS(u\circ\widetilde{\nu})\subset \nu^{-1}(\cS(u))$. According to \cref{thm:GS}, the Hausdorff codimension of $\cS(u)$ is at least two. As $\nu:Y\to X$ is a surjective generically finite morphism, the Hausdorff codimension of $\widetilde{\nu}^{-1}(\cS(u))$ is also at least two.  
	Let $F$ be \emph{any} connected component  of an arbitrary  fiber of $\varphi: Y\to S_{\nu^*\varrho}$.  Let $\widetilde{F}$ be   any connected component  of $ p_Y^{-1}(F)$. We will prove that the image $u\circ\widetilde{\nu}(\widetilde{F})$ is a single point in $\Delta(G)$. 
	
	  Pick a general point  $x$ on the smooth locus of $\widetilde{F}$, and we take a relatively compact coordinate system $(U;z_1,\ldots,z_n;\phi)$ centered at $x$ such that 
	  \begin{itemize}
	  	\item $\phi:U\to \bD^n$ is a biholomorphic map;
	  	\item  $\phi(\widetilde{F}\cap U)= \{0\}\times \bD^k$.
	  \end{itemize}  For every $t\in \bD^{n-k}$, we denote by  $U_t:= \phi^{-1}(\{t\}\times \bD^k)$ and  $h_t:=u\circ\widetilde{\nu}|_{U_t}$. Then $h_t:U_t\to \Delta(G)$ are harmonic mappings for every $t\in \bD^{n-k}$.  Since   the Hausdorff dimension   $ \widetilde{\nu}^{-1}(\cS(u))\cap U$ is at most $  2n-2$, by \cite[Collary 4.(i)]{Shi68}   there is a dense set  $\Omega\subset  \bD^{n-k}$ such that  for every $t\in \Omega$ the Hausdorff dimension of $U_t\cap \widetilde{\nu}^{-1}(\cS(u))$ is at most $2k-2$. In particular, $U_t^\circ:=U_t\backslash \widetilde{\nu}^{-1}(\cS(u))$  is an open dense set  of  $U_t$.  For   $t\in \Omega$, $\{p_Y^*\eta_m|_{U_t^\circ},\ldots, p_Y^*\eta_1|_{U_t^\circ},\}$ correspond to complex differentials of $h_t$, and thus there are positive constants $\ep$ and $C$ such that
the energy of $h_t$ 
\begin{align} \label{eq:energybound}
C\cdot 	E^{h_t}\geq  \int_{U_t}\sum_{j=1}^{m}\sqrt{-1}p_Y^*\eta_j|_{U_t^\circ}\wedge p_Y^*\overline{\eta}_j|_{U_t^\circ}\wedge\omega^{k-1} \geq C^{-1}\cdot E^{h_t}
\end{align}
for every $t\in \bD_\ep^{n-k}\cap \Omega$.  
Here $\omega$ is some K\"ahler metric on $\widetilde{Y}$, and $\bD_\ep$ denotes the disk of radius $\ep$.   On the other hand, since $u\circ \widetilde{\nu}$ is Lipschitz on the boundary $\phi^{-1}(\bD^{n-k}\times \partial (\bD^k))$,   it can be shown using \cite[Lemma 2.19]{BDDM} that $t\mapsto E^{h_t}$ is a continuous function. Furthermore, since    $\int_{U_t}\sum_{j=1}^{m}\sqrt{-1}p_Y^*\eta_j|_{U_t^\circ}\wedge p_Y^*\overline{\eta}_j|_{U_t^\circ}\wedge\omega^{k-1}$ is also a continuous function with respect to $t$,   it follows that \eqref{eq:energy} holds for \emph{any}  $t\in \bD_\ep^{n-k}$.   Since $F$ is a fiber of $\varphi: Y\to S_{\nu^*\varrho}$, by \cref{claim:trivial}, it follows that  $\eta_i|_{F}\equiv 0$ for  $i=1,\ldots,m$. Consequently,  we have $p_Y^*\eta_i|_{U_0}$ for each $i$, which implies  $E^{h_0}=0$ by  \eqref{eq:energybound}.  As $x$ is a general point on $\widetilde{F}$, it follows that the energy density of $u\circ\widetilde{\nu}|_{\widetilde{F}}:\widetilde{F}\to \Delta(G)$ is zero almost everywhere on $\widetilde{F}$. Thus, the energy of   $u\circ\widetilde{\nu}|_{\widetilde{F}}$ is zero.   Therefore, $u\circ\widetilde{\nu}(\widetilde{F})$ is a point $P$ in the building  $\Delta(G)$. It is important to note that  $\widetilde{F}$ is a connected component of $p_Y^{-1}(F)$.  This observation  implies that $\nu^*\rho({\rm Im}[\pi_1(F)\to \pi_1(Y)])$ fixes the point  $P$. By \cref{lem:AB}, $\nu^*\rho({\rm Im}[\pi_1(F)\to \pi_1(Y)])$ is a bounded subgroup of $G(K)$.
	
	\medspace
	
	Let $T$ be any connected Zariski closed subset of $X$. We claim that there exists a connected Zariski closed subset $W$ of $\nu^{-1}(T)$ that surjects onto $T$. Moreover, if $T$ is irreducible, then $W$ can be chosen to be irreducible.
	
	Since $\pi:\xsp\to X$ is a Galois covering, any connected component $W_1$ of $\pi^{-1}(T)$ is mapped surjectively onto $T$ via $\pi$. Note that all fibers of $\mu:Y\to\xsp$ are connected because $\xsp$ is normal. Therefore, $W:=\mu^{-1}(W_1)$ is a connected subset of $\nu^{-1}(T)$ that maps surjectively onto $T$ via $\nu$. When $T$ is irreducible, we just replace $W$ by one of its  irreducible component  which surjects onto $T$. 

 Let $\widetilde{W}$ be a connected component    of $p_Y^{-1}(W)$.  By the above choice of $W$, we know that $\widetilde{W}$ is a connected component of $(p_X\circ\widetilde{\nu})^{-1}(T)$. 
 Then $\widetilde{\nu}(\widetilde{W})$ is contained in a  connected component  $\widetilde{T}$ of $p_X^{-1}(T)$.    Let $\{\widetilde{W}_i\}_{i\in I}$ be the set of  connected components  of $p_Y^{-1}(W)$ which are  contained in $\widetilde{\nu}^{-1}(\widetilde{T})$.   Since $\widetilde{\nu}$ is proper and surjective,   we have $\bigcup_{i\in I}\widetilde{\nu}(\widetilde{W}_i)=\widetilde{T}$.      

	\medspace
	
	\noindent {\em Proof of (1) $\Rightarrow$ (2)}. 
		If $s_\varrho(T)$ is a point, then    $W$ is contained in some connected component  $F$ of a fiber of $\varphi$.   
	By the above result, it follows that $u\circ\tilde{\nu}(\widetilde{W}_i)$ is a point for each $i\in I$.  Since $T$ is connected and $\bigcup_{i\in I}\widetilde{\nu}(\widetilde{W}_i)=\widetilde{T}$, it follows that   $u(\widetilde{T})$ is a point $P$ in $\Delta(G)$. As a result,  $\varrho\big({\rm Im}[\pi_1(T)\to \pi_1(X)] \big)$ fixes $P$ and by \cref{lem:AB}   $\varrho\big({\rm Im}[\pi_1(T)\to \pi_1(X)] \big)$ is   bounded.
	
	\medspace

\noindent {\em Proof of (2) $\Rightarrow$ (3)}.  This obviously follows from ${\rm Im}[\pi_1(T_o^{\rm norm})\to \pi_1(X)]\subset {\rm Im}[\pi_1(T)\to \pi_1(X)]$.

\medskip
	
		\noindent {\em Proof of (3) $\Rightarrow$ (1)}.  
We first assume that $T$ is irreducible. Then by the above argument  there exists an irreducible  closed subvariety $W\subset Y$ which maps surjectively onto $T$ via $\nu$.  	If  $\varrho\big({\rm Im}[\pi_1(T^{\rm norm})\to \pi_1(X)] \big)$ is bounded, then $\nu^*\varrho\big({\rm Im}[\pi_1(W^{\rm norm})\to \pi_1(Y)] \big)$ is also bounded.  Let $\widehat{W}\to W$ be a resolution of singularities and $\widetilde{{W}}$ be a connected component  of $\widehat{W}\times_Y\widetilde{Y}$. 
We  denote by $p:\widetilde{{W}}\to \widetilde{Y}$ and $g:\widehat{W}\to Y$  the induced   maps.  
 \begin{equation*}
	\begin{tikzcd}
		\widetilde{W} \arrow[r, "p"]\arrow[d] & \widetilde{Y}\arrow[d,"p_Y"]\\
		\widehat{W}\arrow[r, "g"] & Y
	\end{tikzcd}
\end{equation*} 
Note that $\widetilde{W}\to \widehat{W}$ is a Galois cover.  We can deduce from \cite{BDDM} that $u\circ\widetilde{\nu}\circ p$ is  a  $(\nu\circ g)^*\varrho$-equivariant harmonic mapping of  {logarithmic energy growth}, where   $(\nu\circ g)^*\varrho: \pi_1(\widehat{W})\to G(K)$, whose image is assumed to be bounded.  
	Hence  $(\nu\circ g)^*\varrho(\pi_1(\widetilde{W}))$ fixes a point $P\in \Delta(G)$ by \cref{lem:AB}. The constant map $\widetilde{W}\to P$ is   also a $(\nu\circ g)^*\varrho$-equivariant harmonic mapping. Though harmonic mappings are not unique,  their energy densities are  unique by \cite[Proposition 2.18]{BDDM} if they have  {logarithmic energy growth} at infinity. It follows that  the energy density of $u\circ\widetilde{\nu}\circ p$ is zero and thus $u\circ\widetilde{\nu}\circ p(\widetilde{W})$ is point in $\Delta(G)$. 
	
	To show that $s_{\varrho}(T)$ is a point, it is equivalent to prove that $\varphi(W)$ is a point. By \cref{claim:trivial} it suffices to prove that $\eta_i|_W\equiv 0$  for each $i$.   Pick a general point  $x$ on the smooth locus of $W$, and we take a relatively compact coordinate system $(U;z_1,\ldots,z_n)$ of $Y$ containing $x$ such that  $U\simeq \bD^n$ and $W\cap U\simeq \{0\}\times \bD^k$ under this trivialization. Let  $\Omega$ be a connected component of $p^{-1}(W\cap U)$, which is isomorphic to $\bD^k$.   Denote by $h_0:=u\circ\widetilde{\nu}\circ p|_{\Omega}$.  
By employing the same argument as at the beginning of the proof of the claim, there exists  a  positive constant  $C$ such that  the following inequality holds:
	\begin{align}\label{eq:energy}
		E^{h_0}\geq C\int_{\Omega}\sum_{j=1}^{m}\sqrt{-1}p^*\eta_j \wedge p^*\bar{\eta}_j \wedge\omega^{k-1}. 
	\end{align} 
Here $\omega$ is some K\"ahler metric on $\widetilde{Y}$.  Since $u\circ\tilde{\nu}\circ p:\widetilde{W}\to \Delta(G)$ is constant, 	we have $E^{h_0}=0$.  \eqref{eq:energy} implies that $p^*\eta_j|_{\Omega}\equiv 0$.  Since $\Omega$ is an  analytic open set of $\widetilde{W}$, we can deduce that $\eta_i|_{W}\equiv 0$ for every $i$. Consequently, we can conclude that  $\varphi(W)$ is a point by \cref{claim:trivial}.  This proves that  $s_\varrho(T)$ is a point.  We prove  that $(3)\implies (1)$ when $T$ is irreducible.
	
Assume now $T$ is reducible.  Let $T_1,\ldots,T_k$ be its irreducible components. By the assumption that $\varrho({\rm Im}[\pi_1(T_i^{\rm norm})\to \pi_1(X))])$  is bounded for each $i$, it follows that $s_\varrho(T_i)$ is a point. Since $T$ is connected, it implies that $s_\varrho(T)$ is also a point. We prove  that $(3)\implies (1)$  completely.  
\end{proof}

\medspace

\noindent {\it Step 6. Construction of $s_\varrho$ when $G$ is an algebraic torus.}    
Let $G$ be an algebraic torus over $K$.  Let $a:X\to \cA$ be the quasi-albanese morphism.
Then $\varrho:\pi_1(X)\to G(K)$ factors $\pi_1(X)\to \pi_1(\cA)$.
Let $\tau:\pi_1(\cA)\to G(K)$ be the induced representation. 
Let $\mathcal{B}$ be the set of all semi-abelian subvarieties $B\subset \cA$ such that $B\in \mathcal{B}$ iff $\tau(\mathrm{Im}[\pi_1(B)\to \pi_1(\cA)])\subset G(K)$ is bounded.
\begin{claim}
	There exists a unique maximal element $B_o\in\mathcal{B}$, i.e., for every $B\in \mathcal{B}$, we have $B\subset B_o$.
\end{claim}
\begin{proof}
If $B_1,B_2\in \mathcal{B}$, then $B_1\cdot B_2\in \mathcal{B}$.
Indeed, $\mathrm{Im}[\pi_1(B_1\cdot B_2)\to\pi_1(\cA)]$ is generated by $\mathrm{Im}[\pi_1(B_1)\to\pi_1(\cA)]$ and $\mathrm{Im}[\pi_1(B_1)\to\pi_1(\cA)]$ whose images under $\varrho$ are both contained in the unique maximal compact subgroup of $G(K)$.  
\end{proof}
Set $A=\cA/B_o$.
Let $\alpha:X\to A$ be the induced map.
Let $s_{\varrho}:X\to S_{\varrho}$ be the quasi-Stein factorisation of $\alpha$.
We shall show that this $s_{\varrho}$ satisfies the desired property. 

For this purpose, let $g:T\to X$ be a morphism from a smooth quasi-projective variety $T$.
Then $a\circ g:T\to \cA$ factors a quasi-albanese morphism $T\to A_T$.
Let $B_T\subset \cA$ be the image of the induced map $A_T\to \cA$.
Then $B_T$ is a  translate of a semi-abelian subvariety.
We have a surjection $\pi_1(T)\to \pi_1(B_T)$ so that $\pi_1(T)\to \pi_1(\cA)$ factors $\pi_1(T)\to \pi_1(B_T)\to \pi_1(\cA)$.
Then
$$
\varrho(\mathrm{Im}[\pi_1(T)\to\pi_1(X)])=\tau(\mathrm{Im}[\pi_1(B_T)\to\pi_1(\cA)]).$$
Thus $\varrho(\mathrm{Im}[\pi_1(T)\to\pi_1(X)])\subset G(K)$ is bounded if and only if $\tau(\mathrm{Im}[\pi_1(B_T)\to\pi_1(\cA)])\subset G(K)$ is bounded, i.e, $B_T\subset B_o\cdot x$ for some $x\in \cA$.
Hence $\varrho(\mathrm{Im}[\pi_1(T)\to\pi_1(X)])\subset G(K)$ is bounded if and only if $\alpha\circ g:T\to A$ is a single point.

We now assume that $T$ is a connected Zariski closed subset of $X$. Let $T_1,\ldots,T_k$ be all its irreducible components.  Let $\mu_i:\widetilde{T}_i\to T_i$ be a desingularization.  If $\varrho({\rm Im}[\pi_1({T}_i^{\rm norm})\to \pi_1(X)])$ is bounded for each  $i$, then   $\varrho({\rm Im}[\pi_1(\widetilde{T}_i)\to \pi_1(X)])$   is also bounded. By the above argument we know that $\alpha\circ\mu_i(\widetilde{T}_i)$ is a point for each $i$.    Since $T$ is connected,   it follows that $\alpha(T)$, and thus $s_\varrho(T)$ is a point.   

Assume now $s_\varrho(T)$ is a point. Then  
 $a(T)$ is contained in $B_o\cdot x$ for some $x\in \cA$. Then 
$$
\varrho(\mathrm{Im}[\pi_1(T)\to\pi_1(X)])\subset \tau(\mathrm{Im}[\pi_1(B_o\cdot x)\to\pi_1(\cA)])$$
which is bounded.  
Hence $s_{\varrho}$ satisfies the desired property of the theorem.

\medspace

\noindent {\it Step 7. Construction of $s_\varrho$ when $G$ is reductive.}   
Now we assume that $G$ is reductive.  Let $Z$ be the central torus of $G$ and let $G'$ be the derived group of $G$. Then  the natural morphism
$$
a:G\to G/G'\times G/Z
$$
is an isogeny, where $T:=G/G'$ is an algebraic torus and $H:=G/Z$ is semisimple.     Let $\varrho_1:\pi_1(X)\to T(K)$ and $\varrho_2:\pi_1(X)\to H(K)$ be the induced  representation obtained by the composition of $a\circ\varrho$ with the projections $T(K)\times H(K)\to T(K)$ and $T(K)\times H(K)\to H(K)$ respectively.  Therefore,   $\varrho:\pi_1(X)\to G(K)$ is bounded if and only if both $\varrho_1$ and $\varrho_2$ are bounded.  At Steps 5 and 6, we have morphisms $s_{\varrho_i}: X \rightarrow S_{\varrho_i}$ for $i = 1, 2$, such that for any connected Zariski closed subset $Z$ of $X$, the image $s_{\varrho_i}(Z)$ is a point if and only if $\varrho_i(\operatorname{Im}[\pi_1(Z) \rightarrow \pi_1(X)])$ is bounded, where $i=1,2$. This is also equivalent to $\varrho_i(\operatorname{Im}[\pi_1(Z_o^{\operatorname{norm}}) \rightarrow \pi_1(X)])$ being bounded for any irreducible component $Z_o$ of $Z$.  Let $s_\varrho:X\to S_{\varrho}$  be the quasi-Stein factorization of the morphism 
\begin{align*}
	X&\to S_{\varrho_1}\times S_{\varrho_2}\\
	x&\mapsto (s_{\varrho_1}(x), s_{\varrho_2}(x)).
\end{align*}
 Then we can verify that $s_\varrho$ is the desired reduction map in the theorem.

 \medspace
 
 \noindent {\it Step 8. We do not assume that $X$ is smooth.}
 Let $\mu:Y\to X$ be a resolution of singularities. Then $\mu^\varrho:\pi_1(Y)\to G(K)$ is also a Zariski dense representation. By Step 7, the desired reduction map $s_{\mu^*\varrho}:Y\to S_{\mu^*\varrho}$ exists.  
 
 Let $F$ be any fiber of $\mu$, which is connected and compact as $X$ is normal. Obviously, $\mu^*\varrho({\rm Im}[\pi_1(F)\to \pi_1(Y)])=\{e\}$. Therefore, $s_{\mu^*\varrho}(F)$ is a point. Hence there exists a morphism $s_\varrho:X\to S_{\mu^*\varrho}$ such that $s_\varrho\circ\mu=s_{\mu^*\varrho}$. We will prove that it satisfies the required property in the theorem. 
 
   Let $T':=\mu^{-1}(T)$. Since $X$ is normal, it follows that each fiber of $\mu$ is connected. Hence the natural morphism $T'\to T$ has connected fibers. By \cite[Lemma 3.44]{DY23}, we know that  $\pi_1(T')\to \pi_1(T)$ is surjective.   
   
 \medspace
 
 	\noindent {\em Proof of (c) $\Rightarrow$ (a)}.  Since $s_{\mu^*\varrho}(T')=s_\varrho(T)$ is a point, it follows that
 	$$
 	 \varrho({\rm Im}[\pi_1(T)\to \pi_1(X)]) = \mu^*\varrho({\rm Im}[\pi_1(T')\to \pi_1(Y)]) 	 
 	$$
 	is bounded. 
 	
 	\medspace
 	
 	 	\noindent {\em Proof of (a) $\Rightarrow$ (b)}. This is obvious.
 	 		
 	 		\medspace
 	 		
 	 		 	\noindent {\em Proof of (b) $\Rightarrow$ (c)}.  We take  an irreducible component, denoted as $T_o'$, of $T'$ that dominates $T_o$. Considering
 	 		 	$$\mu^*\varrho({\rm Im}[\pi_1(T_o')\to \pi_1(Y)])\subset  	 \varrho({\rm Im}[\pi_1(T_o)\to \pi_1(X)]),$$ 
 	 		 we can conclude that $\mu^*\varrho({\rm Im}[\pi_1(T_o')\to \pi_1(Y)])$ is   bounded.   By Step 7, $s_{\mu^*\varrho}(T_o')$ is a point. This leads  to  the conclusion that $s_\varrho(T_o)$ is a  point as well. Given the connectedness of $T$, we  conclude that $s_\varrho(T)$ as a point.

 We complete the proof of the theorem.   
\end{proof}

We finish this section with some remarks  on the above long proof. 
\begin{rem}
	For the purpose of the proof of \cref{main2}, we  only need to study the properties of the morphism $\alpha:\xsp\to \cA$ together with the precise information on the ramification locus $\pi: {\xsp}\to  {X}$ in \cref{claim:ramified}. The reduction $s_\varrho:X\to S_\varrho$ will be an important tool in studying the linear Shafarevich conjecture for quasi-projective varieties (see \cite{Eys04,EKPR12} and the very recent work by Green-Griffiths-Katzarkov \cite{GGK22}).
	
	Note that  the  spectral cover $\xsp$ of $X$ constructed above has the advantage that it is always irreducible, while in the literature like    \cite{Eys04}  spectral covers  might  be reducible.  In \cite{Kli03} Klingler has a completely different way to  construct the spectral cover, which is more visual. However,   the ramification locus cannot be described as  \cref{claim:ramified}.  We stress here that \cref{claim:ramified}  is crucial in the proof of \cref{main6}. 
\end{rem}

\section{Hyperbolicity of algebraic varieties admitting non-Archimedean representations of their fundamental groups}\label{sec:spectral}
This section is devoted to prove \cref{main6}. In  \cref{sub:spectral}  we prove that for the spectral cover of $\xsp\to X$ associated to $\varrho$ in \cref{main6}  defined in \cref{def:spectral}, $\xsp$ is of log general type and  the partial quasi-Albanese morphism induced by its spectral one forms defined in \cref{def:spectral one forms} is generically finite onto its image. In \cref{subsection:spread1} we  prove \cref{main:lgt} by spreading  the positivity of $\xsp$ to $X$. \cref{subsection:log derivative,subsection:ramification} we estimate the ramification counting function for   the covering $\cY\to {\bC}_{>\delta}$ defined in \ref{figure:curve} induced by the spectral covering $\xsp$ and  the holomorphic map $f:{\bC}_{>\delta}\to X$.  In \cref{subsection:spread2} we prove \cref{main:PPH} by applying \cref{thm:spectral cover,thm2nd,prop:Yam10}. 
 \subsection{Positivity  of log canonical bundle of  spectral cover} \label{sub:spectral}
 \begin{thm}\label{thm:spectral cover}
	 Let $X$ be a quasi-projective manifold.  Assume that there is a Zariski dense representation $\varrho:\pi_1(X)\to G(K)$ where $G$ is an almost simple algebraic group defined over a non-archimedean local field $K$.   When $\varrho$ is big and unbounded, then \begin{thmlist}
	 	\item \label{spectral general type}the spectral cover $\xsp$ of $X$ defined in  \cref{def:spectral} is of log general type.  
	 	\item  \label{genericallyfinite}  Let $\alpha:\xsp\to \cA$  be the partial quasi-Albanese map induced by   the spectral one forms $\{\eta_1,\ldots,\eta_m\}\subset  H^0( {\xsp}, \pi^*\Omega_{ {X}})  $  defined in \cref{def:partial}. Then $\dim \xsp=\dim \alpha(\xsp)$. 
	 	\end{thmlist}
  \end{thm}
\begin{proof}
	\noindent
	{\em Step 1. We replace \(\xsp\) by a smooth model.} 
	We will use the notations in \cref{sec:KZthm}. Let \(\mu:Y \to \xsp\) be a resolution of singularities, and let \(\nu : Y \to X\) be the composite map. Denote by \(\widetilde{\nu} : \widetilde{Y} \to \widetilde{X}\) the map induced at the level of universal covers; one checks right away that \(u \circ \widetilde{\nu} : \widetilde{Y} \to \widetilde{X}\) is a pluriharmonic, \(\nu^{\ast} \varrho\)-equivariant map.

	Since the image of \(\nu_{\ast} : \pi_{1}(Y) \to \pi_{1}(X)\) has finite index by \cref{lem:finiteindex}, the Zariski closure $H$ of   \(\nu^{\ast}\varrho(\pi_1(Y))\) contains the identity component $G^o$ of $G$, hence is also almost simple.  We have also:
	\begin{claim}
		The representation \(\nu^{\ast}\varrho\) is big.
	\end{claim}
\begin{proof}
	For any closed   subvariety $Z\subset Y$ containing a  very general point in $Y$, its image $\nu(Z)$ is   closed subvariety  passing to a very general point in $X$.  Since $ Z\to \nu(Z)$ is surjective, ${\rm Im}[\pi_1(Z^{\rm norm})\to \pi_1(\nu(Z)^{\rm norm})]$  has   finite index in  $\pi_1(\nu(Z)^{\rm norm})$ by Lemma~\ref{lem:finiteindex}.  
	 Hence  $\nu^*\varrho({\rm Im}[\pi_1(Z^{\rm norm})\to \pi_1(Y)])$ has finite index in $\varrho({\rm Im}[\pi_1(\nu(Z)^{\rm norm})\to \pi_1(X)])$.  
	 Since \(\nu(Z)\) contains a very general point of $X$, this latter group is infinite, which   implies that $\nu^*\varrho({\rm Im}[\pi_1(Z^{\rm norm})\to \pi_1(Y)])$ is infinite. This proves our claim.
\end{proof}
	 
	 \medskip

	 \noindent
	 {\em Step 2. We show that \(\overline{\kappa}(Y) \geq 0\).} Let \(X^{\circ}\subset X\) be the regular set of $\varrho$-equivariant harmonic mapping $u$ introduced in   \Cref{sec:harmmapp}, and let \(Y^{\circ} := \nu^{-1}(X^{\circ})\).  Let \(F\subset Y\) be a connected component of a general fiber of \(q : Y \to \cA\), where  \(q : Y \to \cA\) is the composite map of $\alpha:\xsp\to \cA$ and $\mu:Y\to \xsp$. \cref{claim:trivial} implies that,  for every fiber $F$ of $Y\to \cA$ one has \(\eta_{i}|_{F} = 0\).
	 
	 Take a connected component  \(F\) of a  general fiber of $Y\to \cA$.   Then  by the proof of  \cref{claim:bounded subgroup}  $\nu^*\varrho({\rm Im}[\pi_1(F)\to \pi_1(Y)])$  is a bounded subgroup of $H(K)$.  However, $\nu^*\varrho({\rm Im}[\pi_1(F)\to \pi_1(Y)])$ is a normal subgroup of $\nu^*\varrho(\pi_1(Y))$ by \cref{lem:normal}.  Since $\nu^*\varrho:\pi_1(Y)\to H(K)$ is Zariski dense and unbounded,   \cref{lem:BT} yields  the finiteness of $\nu^*\varrho({\rm Im}[\pi_1(F)\to \pi_1(Y)])$.  However,   \(\nu^{\ast}\varrho\) is assumed to be big. Then \(F\) must be a point. This implies that  $q:Y\to \cA$, hence  $\alpha:\xsp\to \cA$ is generically finite onto its image.  Let $Z$ to be the Zariski closure of $q(Y)$.  By \cref{prop:Koddimabb}, $\bar{\kappa}(Z)\geq 0$.   Hence $\overline{\kappa}(Y)\geq 0$ for $q:Y\to Z$ is dominant and generically finite. 
	 \medskip

	\noindent
	{\em Step 3. We show that the log--Iitaka fibration of \(Y\) is trivial.}
	We may replace \(Y\) with a birational modification so that the log-Iitaka fibration of \(Y\) is well-defined as a dominant morphism $f:Y\to B$ with connected general fibers. Note that a very general fiber $F$ of $f$ is a connected smooth quasi-projective variety with $\overline{\kappa}(F)=0$. Moreover, since \(q:Y\to \cA\) is generically finite onto its image, so is its restriction $q|_{F}:F\to \cA$. By \cref{lem:abelian pi}, it follows that $\pi_1(F)$ is abelian.  Hence $\nu^*\varrho({\rm Im}[\pi_1(F)\to \pi_1(Y)])$   is an abelian subgroup of $G(K)$.   Note that $\pi_1(F)\triangleleft \pi_1(Y)$ by \cref{lem:normal}. Since $\nu^*\varrho: \pi_1(Y)\to H(K)$ is Zariski dense and $H$ is   almost simple, it follows that the Zariski closure of $\nu^*\varrho({\rm Im}[\pi_1(F)\to \pi_1(Y)])$  is either  finite or a finite index subgroup of $H$. The second case cannot happen since  $\nu^*\varrho({\rm Im}[\pi_1(F)\to \pi_1(Y)])$   is abelian. Therefore, $\nu^*\varrho({\rm Im}[\pi_1(F)\to \pi_1(Y)])$  must be finite. Now, since \(\nu^*\varrho\) is big, this implies that $F$ is a point. We conclude that $Y$, hence $X^\sp$  is of log general type. 
\end{proof}

  \subsection{Spread positivity from spectral covering}\label{subsection:spread1}
  In this subsection, based on \cref{claim:ramified,spectral general type,cor:20221102} together with the celebrated work \cite{CP19} we prove that  for $X$ in \cref{thm:spectral cover} any  closed   subvariety of $X$ passing to a \emph{general} point   is   of log general type.    

\begin{thm}\label{thm:log general type}
 	 	Let $X$ be a quasi-projective manifold.  Let $G$ be an almost simple algebraic group defined over a non-archimedean local field $K$.  Let $\varrho:\pi_1(X)\to G(K)$ be a Zariski dense representation which is big and unbounded. 
Then there exists a proper Zariski closed subset $E\subsetneqq X$ such that any closed subvariety $V\subset X$ not contained in $E$  is of log general type. 
\end{thm}

 \begin{proof} 
Let $\pi:\overline{\xsp}\to \overline{X}$ be the spectral cover associated to $\varrho$ defined in \cref{def:spectral}, and $\{\eta_1,\ldots,\eta_m\}\in H^0(\overline{\xsp}, \pi^*\Omega_{ \overline{X}}(\log D))$ the resulting spectral one forms defined in \cref{def:spectral one forms}.  Let $\mu:  \overline{Y}\to \overline{\xsp}$ be a resolution of singularities with $Y:=\mu^{-1}(\xsp)$. Set $\omega_i:=\mu^*\eta_i$.  By \cref{genericallyfinite} the partial quasi-Albanese morphism $\alpha:\xsp\to \cA$ associated to  $\{\eta_1,\ldots,\eta_m\}$ satisfies that $\dim \xsp=\dim \alpha(\xsp)$.   In particular, the     quasi-Albanese morphism  $\alpha_Y:Y\to \cA_Y$ satisfies $\dim Y=\dim \alpha_Y(Y)$.  

We first determine the proper Zariski closed subset $E\subsetneqq X$ in the proposition.  Write $Z_{ij}:=(\omega_i-\omega_j=0)$ for $\omega_i\neq \omega_j$, which is a proper Zariski closed subset  of $Y$.  
	For the quasi-Albanese variety $\cA_Y$ of $Y$, we know that there exists log one forms $\widetilde{\omega}_i\in T_1(\cA_Y)$ so that
	$
	\omega_i=\alpha_Y^*\widetilde{\omega}_i
	$ for every $i$. Consider the maximal semi-abelian subvariety $B$ of $\cA_Y$  on which $\widetilde{\omega}_i-\widetilde{\omega}_j$ vanishes.  Denote by $q:\cA_Y\to \cA_Y/B$ the quotient map.  For the composed morphism $\beta:Y\to \cA_Y/B$,  by \cref{lem:critpointalb}, for any closed    subvariety $F\subset Y$, one has $\beta(F)$ is a point if and only if $(\omega_i-\omega_j)|_{F}=0$.  Then for any irreducible component  $Z$ of $Z_{ij}$,  
	$
	\beta(Z)
	$ is a point $x\in A':=\cA_Y/B$.   Since $Z_{ij}$ has   finitely many irreducible components, it follows that $\beta(Z_{ij})$ is a set of finite points  $F_{ij}$ in $A'$. Denote by $E_{ij}:=\beta^{-1}(F_{ij})$, which is a proper Zariski closed subset of $Y$.   Set  $E_1: =\pi\big( \bigcup_{\omega_i\neq \omega_j}E_{ij}\big)$ in $X$, which is a proper Zariski closed subset of $X$.  
 By \cref{cor:20221102}, 
 there exists a proper Zariski closed subset $\Xi\subsetneqq Y$ such that all entire curves $\mathbb C\to Y$ are contained in $\Xi$.
 By replacing $\Xi$ by a larger proper Zariski closed set of $Y$, we may assume that $Y\to \mathcal{A}_Y$ is quasi-finite outside $\Xi$. 
 
 \begin{claim}\label{claim:loggeneraltype}
 All closed subvarieties $F\subset Y$ with $F\not\subset \Xi$ are of log-general type.
 \end{claim}
\begin{proof}[Proof of \cref{claim:loggeneraltype}]
This is proved in \cref{cor:20221102}.
\end{proof}
 We set $E_2=\pi(\Xi)$, which is a proper Zariski closed subset of $X$.
 We define $E\subsetneqq X$ by $E=E_1\cup E_2$.

 Let $V\subset X$ be a closed subvariety such that $V\not\subset E$.
 Let $W\to V$ be a smooth modification and let $\overline{W}$ be a smooth projective compactification of $W$ so that $D_{\overline{W}}:=\overline{W}-W$ is a simple normal crossing divisor and $(\overline{W},D_{\overline{W}})\to (\overline{X}, D)$ is a log morphism.  
 Let $\overline{S}$ be a normalization of an irreducible component of $\overline{W}\times_{\overline{X}}\overline{\xsp}$. 
 \begin{equation*}
 	\begin{tikzcd}
 		\overline{S}  \arrow[r, "g"] \arrow[d, "p"]&  \overline{\xsp}\arrow[d, "\pi"]\\
 		\overline{W} \arrow[r] & \overline{X}
 	\end{tikzcd}
 \end{equation*}
By the construction,   $S:=p^{-1}(W)$ is not contained in $\Xi$ and thus by \cref{claim:loggeneraltype} $S$ is of log general type. 
 \begin{claim}\label{claim:Galois}
 The finite morphism	$p:\overline{S}\to \overline{W}$ is a Galois morphism with the Galois group $G'\subset \widetilde{G}$ where $\widetilde{G}$ is the Galois group of the  Galois morphism $\overline{\xsp}\to \overline{X}$. The morphism $g:\overline{S}\to \overline{\xsp}$ is $G'$-equivariant. 
 \end{claim}
\begin{proof}[Proof of \cref{claim:Galois}]
	It is obvious that   the base change $\overline{W}\times_{\overline{X}}\overline{\xsp}\to \overline{W}$ is also a Galois covering with the Galois group $\widetilde{G}$. Then the normalization of $\overline{W}\times_{\overline{X}}\overline{\xsp}$ is a disjoint union of isomorphic quasi-projective normal variety, and it admits  an induced  $\widetilde{G}$-action. Define
	 $$
	G':=\{h\in \widetilde{G}\mid h(\overline{S})=\overline{S}\},
	$$
	 and one can see that $G'$ acts  transitively on the fibers of $p$. $p$ is therefore a Galois covering with Galois group $G'$. 
\end{proof}
 Let $\psi_i$ be the pull back of $\eta_i\in H^0(\overline{\xsp}, \pi^*\Omega_{\overline{X}}(\log D))$ by $\overline{S}\to \overline{\xsp}$. 
 Then   we can consider  $\psi_i$ as a section in $H^0(\overline{S}, p^*\Omega_{\overline{W}}(\log D_{\overline{W}}))$. 
 Let $I$ be the set of all $(i,j)$ such that
 \begin{itemize}
 	\item $\eta_i-\eta_j\neq 0$;
 	\item   the image of $S\to \xsp$ intersects with $Z'_{ij}:=\{z\in\overline{\xsp}\mid (\eta_i-\eta_j)(z)=0\}$. 
 \end{itemize}
 
 \begin{claim}
 For $(i,j)\in I$, $\psi_i-\psi_j\not=0$ in $H^0(\overline{S}, p^*\Omega_{\overline{W}}(\log D_{\overline{W}}))$.
 \end{claim}
	\begin{proof}
		Assume by contradiction that $\psi_i-\psi_j=0$ in $H^0(\overline{S}, p^*\Omega_{\overline{W}}(\log D_{\overline{W}}))$.
		Then the image of $S\to \xsp\to \mathcal{A}_Y/B$ is a single point $x$ by \cref{lem:critpointalb}.
		By $V\not\subset E_1$, the image of $S\to \xsp$ is not contained in $E_{ij}$.
		Hence $x$ is not contained in $\beta(Z_{ij})$.
		This contradicts to the assumption $(i,j)\in I$.
	\end{proof}

 We set
 \begin{align*} 
 	 		R':=\{z\in \overline{S} \mid \exists   (i,j)\in I  \mbox{ with } (\psi_i-\psi_j)(z)=0 \}.
 	 	\end{align*}  
By the claim above, $R'$ is a proper Zariski closed subset of $\overline{S}$.  	Denote by  $R_0$ the ramification locus of $p:\overline{S}\to\overline{W}$.  By the purity of branch locus of finite morphisms,  we know that $R_0$ is a (Weil) divisor, and thus $p(R_0)$ is also a divisor.  
Moreover $R_0=p^{-1}\big(p(R_0)\big)$ since $p$ is Galois with Galois group $G'$.  Denote by $E$ the sum of prime components of $p(R_0)$ which intersect with $W$.  One observes that $S-p^{-1}(E)\to W-E$ is finite \'etale.

Recall that  by \cref{claim:ramified} 	 $\pi:\overline{\xsp}\to \overline{X}$  is \'etale over $\overline{\xsp}-R$, where $
	R:=\cup_{\eta_i\neq\eta_j}Z_{ij}'.$  Note that $\mu^{-1}(Z_{ij}')=Z_{ij}$.     Since the base change of an \'etale morphism   is also étale, it follows that 
$p$ is \'etale over $\overline{S}-g^{-1}(R)$.  Hence $R_0\subset g^{-1}(R)$.   Note that for $(i,j)\not\in I$, the image of $S\to \xsp$ does not intersect with $Z'_{ij}$, so such $Z'_{ij}$ does not contribute to $g^{-1}(R)\cap S$. It follows that 
 $
R_0\cap S\subset R'\cap S.
$

Since  $\overline{S}\to \overline{\xsp}$ is $G'$-equivariant, it follows that any $h\in G'$ acts on $\{\psi_i-\psi_j\}_{(i,j)\in I}\subset H^0(\overline{S}, p^*\Omega_{\overline{W}}(\log D_{\overline{W}}))$ as a permutation by \cref{claim:extension} and our choice of $I$. 
 	 	Define a section
 	 	$$
 	 	\sigma:=\prod_{h\in G'}\prod_{(i,j)\in I}h^*(\psi_i-\psi_j)\in H^0(\overline{S}, \Sym^{N}p^*\Omega_{\overline{W}}(\log D_{\overline{W}})), 
 	 	$$
 	 	which is non-zero and vanishes at  $R'$ by our choice of $I$. 
 	 	Then it is invariant under the $G'$-action and thus descends to a section
 	 	$$
 	 	\sigma^{G'}\in H^0(\overline{W}, \Sym^{N}\Omega_{\overline{W}}(\log D_{\overline{W}}))
 	 	$$
 	 	so that $p^*\sigma^{G'}=\sigma$.      Since $R_0\cap S\subset R'\cap S$ and $p^{-1}(p(R_0))=R_0$,     $\sigma^{G'}$ vanishes at   the divisor $E$. This implies that there is a non-trivial morphism
 	 	\begin{align}\label{eq:CP1}
 	 		\mathcal{O}_{\overline{W}}(E)\to \Sym^{N}\Omega_{\overline{W}}(\log D_{\overline{W}}).
 	 	\end{align} 
Recall that by \cref{claim:loggeneraltype} and our construction of $S$, $S$ is of log general type.  Since	$S-p^{-1}(E)\to W-E$ is finite \'etale, by \cref{lem:KodairaDim} below  $W\backslash E$ is also of log general type.  By \cite[Lemma 3]{NWY13}, $K_{\overline{W}}+E+D_{\overline{W}}$ is big.  
	Together with \eqref{eq:CP1} we can apply \cite[Corollary 8.7]{CP19} to conclude that $K_{\overline{W}}+ D_{\overline{W}}$ is big.
		Hence $W$, so $V$ is of log general type. 
 	 \end{proof}

\begin{lem}\label{lem:KodairaDim}
	Let $f':U\to V$ be a finite \'etale morphism between quasi-projective manifolds. If $\bar{\kappa}(V)\geq 0$,  then the logarithmic Kodaira dimension $\bar{\kappa}(U)=\bar{\kappa}(V)$. 
\end{lem}
\begin{proof}
	We first take a smooth projective compactification $Y$ of $V$  so that $D_Y:=Y-V$ is a simple normal crossing divisor.  By \cref{sec:covGal}, there is a normal projective variety $X$ compactifying $U$ so that $f'$ extends to a finite morphism $f:X\to Y$.  Let $\mu:Z\to X$ be a strict desingularization so that $\mu^{-1}(U)\simeq U$ and $D_Z:=Z-\mu^{-1}(U)$ is   a simple normal crossing divisor.  Write $g=f\circ\mu$.  
	\begin{claim}\label{claim:exceptional}
		$E:=K_{Z}+D_Z-g^*(K_Y+D_Y)$ is an effective \emph{exceptional} divisor.  
	\end{claim}
	\begin{proof}[Proof of \cref{claim:exceptional}]
		Since $g:(Z,D_Z)\to (Y,D_Y) $ is a log morphism, $E$ is effective. 	Let $D_Y^{\rm sing}$ be the singularity of $D_Y$  which is a Zariski closed subset of $Y$ of codimension at least two.  Write $Y^\circ:=Y-D_Y^{\rm sing}$, and $X^\circ:=f^{-1}(Y^\circ)$. Note that $X^\circ$ is smooth, and $D_X^\circ:=X^\circ-U$ is a smooth divisor in $X^\circ$.  Moreover, it follows from the proof of \cite[Lemma A.12]{Den22} that at any $x\in D_X^\circ$ we can take a holomorphic coordinate  $(\Omega;x_1,\ldots,x_n)$ around $x$ with $D^\circ_X\cap \Omega=(x_1=0)$ and a holomorphic coordinate  $(\Omega';y_1,\ldots,y_n)$ around $f(x)$ with $D_Y\cap \Omega'=(y_1=0)$  so that 
		\begin{align}\label{eq:ramified}
			f(x_1,\ldots,x_n)=(x_1^k,x_2,\ldots,x_n).
		\end{align}
		This implies that $K_{X^\circ}+D_X^\circ=f^*(K_{Y^\circ}+D_Y^\circ)$. Since $\mu$ is a strict desingularization, it follows that $\mu^{-1}(X^\circ)\simeq X^\circ$.   Therefore,   $\mu (K_{Z}+D_Z-g^*(K_Y+D_Y))$ is contained in $f^{-1}(D_Y^{\rm sing})$  which is of codimension at least two. 
	\end{proof}
Therefore, one has
	$$
	\kappa(K_Z+D_Z)=	\kappa(g^*(K_Y+D_Y)+E)=\kappa(g^*(K_Y+D_Y))=	\kappa(K_Y+D_Y). 
	$$ 
	where the first equality follows from the above claim, the second one is due to  \cite[Example 2.1.16]{Laz04} and the last one follows from \cite[Lemma 5.13]{Uen75}.   This concludes that $\bar{\kappa}(U)=\bar{\kappa}(V)$. 
\end{proof}

\subsection{Lemma on logarithmic derivative for logarithmic one forms} \label{subsection:log derivative}
 Let $D$ be a simple normal crossing divisor on a projective manifold $\overline{Y}$. Let $\cY$ be a Riemann surface with a proper surjective holomorphic map $p_\cY:\cY\to\mathbb C_{>\delta}$.
Let $g:\cY\to \overline{Y}$ be a holomorphic map such that $g(\cY)\not\subset D$. 
Let $\omega \in H^{0}\left(\overline{Y}, \Omega_{\overline{Y}}(\log D)\right)$ be a logarithmic \(1\)-form.  Set $\eta=g^{*} \omega / p_\cY^{*}(d z)$. Then $\eta$ is a meromorphic function on $\cY$.  Then $\eta$ induces a holomorphic map $g_\eta:\cY\to \bP^1$.  We define 
$$
m(r, \eta):=\frac{1}{\operatorname{deg} p_\cY} \int_{y\in p_\cY^{-1}(\{|z|=r\})}\lambda_{\infty}(g_\eta(y))\ \frac{d\mathrm{arg}\ p_\cY(y)}{2\pi},
$$
where $\lambda_{\infty}:\bP^1-\infty \to \mathbb R_{\geq 0}$ is a Weil function for the point $\infty$ of $\bP^1$ (cf. \cite[Prop 2.2.3]{Yam04}).

We claim 
\begin{align}\label{eq:Noguchi}
	 m(r,\eta)=O(\log r)+o(T_g(r))\lVert. 
\end{align}  

We prove this. 
The logarithmic one form $\omega$ defines a morphism $\varphi:T(\overline{Y};\log D)\to \mathbb A^1$, which extends to a rational map $\overline{\varphi}:\overline{T}(\overline{Y};\log D)\dashrightarrow \mathbb P^1$.
By $g(\cY)\not\subset D$, we get the lifting $j_1(g):\cY\to \overline{T}(\overline{Y};\log D)$.
Then we have $\eta=\overline{\varphi}\circ j_1(g)$.
Hence by \cite[Prop 2.3.2 (6)]{Yam04}, we have
$$
m(r, \eta)=O(m_{j_1(g)}(r,\partial T(\overline{Y};\log D)))+O(1).$$
Hence by \eqref{eqn:202311195}, we get \eqref{eq:Noguchi}.

It follows from the First Main theorem (cf. \cref{thm:first}) that
\begin{align}\label{eq:First} 
 T_{g_\eta}(r, \cO_{\bP^1}(1))=N_{g_\eta}(r,\infty)+ m(r,\eta)+O(\log r).
\end{align}

\subsection{Estimate for ramification counting functions}\label{subsection:ramification}
Let $X$ be a quasi-projective manifold and let  $\varrho:\pi_1(X)\to G(K)$  be a Zariski dense representation where $G$ is a simple algebraic group defined over a non-archimedean local field  $K$.  Assume  that $\varrho$ is unbounded and that $\varrho$ is a big representation. By  \cref{claim:ramified}, there is a finite Galois cover $\xsp\to X$  and  algebraic one forms 
$$
\eta_1,\ldots,\eta_m\in H^0(\xsp, \pi^*\Omega_X) 
$$
so that $\xsp\to X$ is unramified  outside  $R:=\{z\in X \mid \exists \eta_i\neq\eta_j \quad \mbox{with}\quad (\eta_i-\eta_j)(z)=0 \}$, which is a Zariski closed   subset of $\xsp$. Moreover, one has $\pi^{-1}(\pi(R))=R$.  Consider a resolution of singularities $\mu:  Y\to \xsp$ and a projective compactification $\overline{Y}$ of $Y$ with $D:=\overline{Y}\backslash Y$ a simple normal crossing divisor. Then $\omega_i:=\mu^*\eta_i\in H^0(Y, \pi^*\Omega_X)\cap H^0(\overline{Y}, \Omega_{\overline{Y}}(\log D))$ by \cref{claim:extension}, where $\pi:Y\to X$ denotes  the composition of $\mu:Y\to \xsp$ and $\xsp\to X$.  

Consider a holomorphic map $f:\bC_{>\delta}\to X$.  The generically finite proper morphism $\pi:Y\to X$ induces a surjective finite holomorphic map   $p_{\cY}:{\cY}\to \bC_{>\delta}$  from a Riemann surface $\cY$ to  $\bC_{>\delta}$ and  a holomorphic map $g:\cY\to Y$  satisfying
\begin{equation}\label{figure:curve}
 	\begin{tikzcd}
	\cY\arrow[r, "g"] \arrow[d, "p_{\cY}"] & Y\arrow[d, "\pi"]\\
	\bC_{>\delta}\arrow[r, "f"] & X
\end{tikzcd}
\end{equation}
Note that $p_\cY$ is unramified outside $(\pi\circ g)^{-1}(R)$.     It follows that $\ram p_\cY\subset \cup_{\omega_i\neq \omega_j}g^{-1}(\omega_i-\omega_j=0)$, where we consider $\omega_i$ as  sections in $H^0(Y, \pi^*\Omega_X)$. 


 \begin{proposition} \label{prop:Yam10}
 	Let $\xsp\to X$ be the spectral cover associated to $\varrho:\pi_1(X)\to G(K)$ defined in \cref{def:spectral}.  Let $\mu:  Y\to \xsp$ a resolution of singularities. Then there is a proper Zariski closed subset $E\subsetneqq X$ such that, for any holomorphic map  $f:\bC_{>\delta}\to X$ which is not contained in $E$,  
	one has  
\begin{align}\label{eq:ramification}
	 	N_{{\rm ram}\, p_{\cY}}(r) =o( T_{g}(r)) + O(\log r) ||,
\end{align} 
 	where $g:\cY\to Y$ is the induced holomorphic map in \eqref{figure:curve}.
 \end{proposition}


\begin{proof} 
	We first determine the proper Zariski closed subset $E\subsetneqq X$ in the proposition.  Write $Z_{ij}:=(\omega_i-\omega_j=0)$ for $\omega_i\neq \omega_j$, which is a proper Zariski closed subset of $Y$.  
	For the quasi-Albanese variety $\cA_Y$ of $Y$, we know that there exists log one forms $\widetilde{\omega}_i\in T_1(\cA_Y)$ so that
	$
	\omega_i=\alpha^*\widetilde{\omega}_i
	$ for every $i$, where $\alpha:Y\to \cA_Y$ is the quasi-Albanese map. Consider the maximal semi-abelian subvariety $B$ of $\cA_Y$  on which $\widetilde{\omega}_i-\widetilde{\omega}_j$ vanishes.  Denote by $q:\cA_Y\to \cA_Y/B$ the quotient map.  For the composed morphism $\beta:Y\to \cA_Y/B$,  by \cref{lem:critpointalb}, for any closed    subvariety $F\subset Y$, one has $\beta(F)$ is a point if and only if $(\omega_i-\omega_j)|_{F}=0$.  Then for any irreducible component  $Z$ of the closed subvariety $Z_{ij}$,  
	$
	\beta(Z)
	$ is a point $x\in A':=\cA_Y/B$.   Since $Z$ has any finitely many irreducible components, it follows that $\beta(Z_{ij})$ is a set of finite points  $F_{ij}$ in $A'$. Denote by $E_{ij}:=\beta^{-1}(F_{ij})$, which is a proper Zariski closed subset of $Y$.   Define  $E= \pi\big( \bigcup_{\omega_i\neq \omega_j}E_{ij}\big)$ in $X$, which is a proper Zariski closed subset  of $X$.  
	
We fix an ample line bundle $L$ on a smooth projective compactification $\overline{Y}$ of $Y$. Let $\bC_{>\delta}\to X$  be a holomorphic map  whose image is not contained in $E$. Then for the induced holomorphic map $g:\cY\to Y$, it is not contained in any $Z_{ij}$.  Hence by \cref{claim:ramified},   
	\begin{align}\label{eq:count}
		N_{{\rm ram}\,  p_{\cY}}(r)\leq \deg p_\cY\cdot 	N^{(1)}_g(r, \sum_{\omega_i\neq\omega_j}Z_{ij}).  
	\end{align}
It then suffices to prove that $	N^{(1)}_g(r,  Z_{ij})= o(T_{g}(r, L))+O(\log r)||$ for every  $Z_{ij}$.

	\medspace
	
	\noindent
	{\it Case 1.} If  $g^*\omega_i\neq g^*\omega_j$, write 
	$$
	\eta:=\frac{ g^*(\omega_i-\omega_j)}{p_\cY^*dz}
	$$   
Since $\omega_i-\omega_j\in H^0(Y,\pi^*\Omega_X)$, it follows that $\eta$ is a holomorphic function on $\cY$ and 
	 $g^{-1}(Z_{ij})\subset (\eta=0).$ 
Note that      $\eta$ can be seen as a holomorphic map $g_\eta:{\cY}\to  \bP^1\backslash\{\infty\}$. Hence $N_{g_\eta}(r,\infty)=0$. By \eqref{eq:First} one has 
	\begin{align*} 
		T_{g_\eta}(r, \cO_{\bP^1}(1))=N_{g_\eta}(r,\infty)+ m(r,\eta)+O(\log r)\leq m(r,\eta)+O(\log r).
	\end{align*}
	Since $\omega_i-\omega_j\in H^0(\overline{Y}, \Omega_{\overline{Y}}(\log D))$,  by \eqref{eq:Noguchi}, one concludes 
	$$
	T_{g_\eta}(r, \cO_{\bP^1}(1))\leq  m(r,\eta)+O(\log r)\leq o( T_g(r))+	O(\log r)\ ||. 
	$$
	By \eqref{eq:First} again, one has
	$$
	N^{(1)}_g(r, Z_{ij})\leq 	N_{g_\eta}(r, 0)\leq T_{g_\eta}(r, \cO_{\bP^1}(1))+O(\log r). 
	$$
	In conclusion, 
	$$
	N^{(1)}_g(r, Z_{ij})\leq  o( T_g(r))+	O(\log r)\ ||. 
	$$
	
	\medskip
	
	\noindent
	{\it Case 2.} Assume now $g^*\omega_i=g^*\omega_j$ for some $\omega_i\neq \omega_j$.    Recall that for the composed morphism $\beta:Y\to \cA_Y/B$ defined above,    for any irreducible component  $Z$ of the closed subvariety $Z_{ij}$,  
	$
	\beta(Z)
	$ is a point $x\in A':=\cA_Y/B$.   We take the factorisation of $\beta$ as follows
	\[
	Y\stackrel{\psi}{\to} W\stackrel{\phi}{\to} A'
	\]
	where $\phi:W\to A'$ is a finite morphism and $\psi:Y\to W$ is a dominant morphism whose general fibers are connected.  Hence $\dim W\leq \dim Y$.   Since $\phi$ is finite and $\beta(Z)=\{x\}$, there is a point $y\in W$ so that $\psi(Z)=\{y\}$. 
	Since there is a log one form $\omega'$ on $A'$ so that $\widetilde{\omega}_i-\widetilde{\omega}_j=q^*\omega'$, it follows that 
	\begin{align}\label{eq:nontrivial}
		(\psi\circ g)^*\phi^*\omega'=  (\beta\circ g)^*\omega'=g^*(\omega_i-\omega_j)=0.
	\end{align}
	Let $S$ be the Zariski closure of $\psi\circ g(\cY)$ in $W$. Note that $\dim S\leq \dim W\leq \dim Y$.    Note that $N_g^{(1)}(r,Z)=0$ if $g^{-1}(Z)=\varnothing$, and the proposition follows trivially.  Hence we may assume that $g^{-1}(Z)\neq\varnothing$.   
	\begin{claim}\label{claim:nontrivial}
		$\phi^*\omega'|_{S}\neq 0$.
	\end{claim}
	\begin{proof}[Proof of \cref{claim:nontrivial}]
		Assume by contradiction that   $\phi^*\omega'|_{S}= 0$. Then for the algebraic subvariety $\psi^{-1}(S)$ of $Y$,  one has
		$$
		(\omega_i-\omega_j)|_{\psi^{-1}(S)}=\beta^*\omega'|_{\psi^{-1}(S)}=0.
		$$
		By \cref{lem:critpointalb} $\beta(\psi^{-1}(S))$ is a set of finite points in $A'$. Note that the Zariski closure  $\overline{g(\cY)}^{\rm Zar}$ in  $Y$ is contained in an irreducible component of $\psi^{-1}(S)$.   Hence $\beta\big(\overline{g(\cY)}^{\rm Zar}\big)=\{x\}=\beta(Z)$ by our assumption that $g^{-1}(Z)\neq\varnothing$.  This contradicts with our choice of $E$ at the beginning. 
	\end{proof}
	
Since $\omega'$ is a linear log one form on $A'$, $d\omega'=0$. Hence by the Poincar\'e lemma in some analytic neighborhood $U$ of $x$ in $A'$ there is a holomorphic function $h\in \cO(U)$ so that $dh=\omega'$ on $U$ and $h(x)=0$.   According to \eqref{eq:nontrivial}, $\phi^*\omega'|_S$ is not identically equal to zero. Hence $h\circ \phi|_{S}$ is not identically equal to zero.

	For every $n\in \bZ_{>0}$, consider the zero dimensional subscheme on $S$ defined by $V_n:=\spec \cO_S/(\cI(h\circ\phi|_S)+\mathfrak{m}_y^n)$, where $\mathfrak{m}_y$ is the maximal ideal at $y$ and $\cI(h\circ\phi|_S)$ is the ideal sheaf defined by the holomorphic function $h\circ\phi|_S$.  $V_n$ is supported at $y$.  We take a   projective compactification $\overline{W}$ for $W$, and let $\overline{S}$ the closure of $S$ in $\overline{W}$. We fix an ample line bundle $L'$ on $\overline{W}$ and denote by $M:=L'|_{\overline{S}}$ its restriction on $\overline{S}$, which is also an ample line bundle on $\overline{S}$.  
	Consider the following short exact sequence 
	$$
	0\to M^k\otimes (\mathcal{I}(h\circ \phi|_S)+\mathfrak{m}_y^n)\to   M^k\to M^k\otimes \cO_S/(\cI(h\circ\phi|_S)+\mathfrak{m}_y^n)\to 0.
	$$
	which yields a a short exact sequence
	$$
	0\to H^0(\overline{S}, M^k\otimes (\mathcal{I}(h\circ \phi|_S)+\mathfrak{m}_y^n))\to   H^0(\overline{S}, M^k)\to H^0(\overline{S}, M^k\otimes \cO_S/(\cI(h\circ\phi|_S)+\mathfrak{m}_y^n))
	$$
	Note that 
	$$
	H^0(V_n,  \cO_{V_n}))=H^0(\overline{S}, M^k\otimes \cO_S/(\cI(h\circ\phi|_S)+\mathfrak{m}_y^n))
	$$
	which implies that $h^0(V_n,  \cO_S/(\cI(h\circ\phi|_S)+\mathfrak{m}_y^n))\sim  (n^{\dim S-1})$ when $n\to \infty$. 
	By Riemann-Roch theorem it follows that 
	$
	h^0(\overline{S}, M^k)\sim  O(k^{\dim S}). 
	$ 
	If we take $k_n\sim n^{1-\frac{1}{2\dim S}}$ when $n\to \infty$, it follows that 
	$$
	h^0(\overline{S}, M^{k_n})>h^0(V_n, \cO_{V_n})
	$$
	when $n\gg 1$. Hence there are sections 
	$$
	s_n\in H^0(\overline{S}, M^{k_n}\otimes (\mathcal{I}(h\circ \phi)+\mathfrak{m}_y^n))
	$$
	when $n\gg 1$.  We now consider $\psi\circ g:\cY\to S$ as a holomorphic map with image in $S$ .  Then by \cref{thm:first}
	$$
	N_{\psi\circ g}(r, V_n)\leq N_{\psi\circ g}(r, D_n)\leq T_{\psi\circ g}(r, M^{k_n})+O(\log r)=k_n T_{\psi\circ g}(r, M)+O(\log r)
	$$
	where $D_n$ is the zero divisor of $s_n$, and the second inequality above  follows from   that $\psi\circ g$ is Zariski dense in $S$.  
	For every $t\in \cY$ so that $g(t)\in Z$, it follows that $\beta(g(t))=x$. Since $g^*(\omega_i-\omega_j)=0$, it follows that
	$$
	0=g^*\beta^*\omega' =g^*\beta^*dh.
	$$ 
	Hence $h\circ\phi\circ \psi \circ g$ is constant on the connected component of $(\beta\circ g)^{-1}(U)$ containing $t$ which is thus zero.
	This implies that  $\ord_t(\psi \circ g)^*V_n\geq n$. Therefore, 
	$$
	nN^{(1)}_{g}(r, Z)\leq N_{\psi\circ g}(r, V_n).
	$$
	In conclusion,
	$$
	N^{(1)}_{\psi\circ g}(r, Z)\leq \frac{k_n}{n} T_{\psi\circ g}(r, M)+O(\log r).
	$$
Let us now consider $\psi\circ g:\cY\to W$ as a holomorphic map with image in $W$. By the very definition of the characteristic function and   $L'|_{\overline{S}}=M$, one gets
	$$
	T_{\psi\circ g}(r, M)=T_{\psi\circ g}(r, L').
	$$
	On the other hand, $T_{\psi\circ g}(r, L')\leq cT_{g}(r, L)+O(\log r)$ for some constant $c>0$ since order functions decrease under rational map $\overline{Y}\dashrightarrow \overline{S}$ induced by $\psi$.  Let $n\to \infty$ we get \begin{equation}\label{eqn:202304151}
	N^{(1)}_{\psi\circ g}(r, Z)\leq \ep T_{g}(r, L)+O_{\varepsilon}(\log r) 
	\end{equation}
	for every $\ep>0$.

	Suppose that $\varliminf_{r\to\infty}T_g(r,L)/\log r<+\infty$.
	Then by \cref{lem:20230415}, we have $T_g(r,L)=O(\log r)$.
			Then by \eqref{eqn:202304151}, we have $N^{(1)}_{\psi\circ g}(r, Z)=O(\log r)$, in particular 
	\begin{equation}\label{eqn:202304152}
	N^{(1)}_{\psi\circ g}(r, Z)=o(T_{g}(r,L))+O(\log r).
	\end{equation}	
		Next we assume $\varliminf_{r\to\infty}T_g(r,L)/\log r=+\infty$.
		Then $\log r=o(T_g(r,L))$.		
		Then by \eqref{eqn:202304151}, we have $N^{(1)}_{\psi\circ g}(r, Z)\leq \ep T_{g}(r, L)+	o(T_g(r,L))$ for all $\varepsilon>0$.	
		Hence we get $N^{(1)}_{\psi\circ g}(r, Z)=o(T_g(r,L))$, in particular we get \eqref{eqn:202304152}.		
							
	In summary, by \eqref{eq:count} we prove  that
	 $
	N_{{\rm ram}\,  p_{\cY}}(r)= o(T_{g}(r, L))+O(\log r)||.$
\end{proof}

\begin{rem}\label{rem:small ramification}
In \cite[p.557, Claim]{Yam10}, the third author proved \eqref{eq:ramification} for \emph{Zariski dense entire curves} when $X$ is projective. 
Here we modified the proof in \cite{Yam10} such that the stronger result in \cref{prop:Yam10} holds even if the map $f:\bC_{>\delta}\to X$ is not Zariski dense and $X$ is quasi-projective. This is crucial in the proof of  \cref{main2}. 

In this regard, we should mention that in \cite[Lemma 4.2]{Sun22}, it seems that Sun incorrectly applied \cite[p.557, Claim]{Yam10} to any holomorphic map from any quasi-projective curve $C$ to $X$ whose image is not Zariski dense, provided only that it is not included in the ramification locus of $\xsp\to X$.   
\end{rem}

 \subsection{Spread hyperbolicity from spectral covering}\label{subsection:spread2}
 
\begin{thm}\label{thm:main33}
	Let $X$ be a complex connected quasi-projective manifold and let $G$ be an almost simple algebraic group over some non-archimedean local field $K$. If  $\varrho:\pi_1(X)\to G(K)$ is  a big  representation which is   Zariski dense and unbounded, then $X$ is pseudo Picard hyperbolic. 
	\end{thm}

\begin{proof}
	By \cref{thm:spectral cover}, we know that there is a finite Galois cover $\xsp\to X$ associated to $\varrho$.  Such normal quasi-projective variety $\xsp$ is of log general type.  Moreover,  the spectral one forms $\{\eta_1,\ldots,\eta_m\}$ in \cref{def:spectral one forms} induced by the $\varrho$-equivariant pluriharmonic mapping $u:\widetilde{X}\to \Delta(G)$   $\alpha$ gives rise to a  quasi-Albanese map $a:\xsp\to \cA$ defined in \cref{def:partial}. By \cref{thm:spectral cover} one has $\dim X=\dim a(\xsp)$.   Let $\mu:Y\to \xsp$ be  a resolution of singularities. We identify the puncture disk $\bD^*$ with $\bC_{>1}$  by taking a transformation $z\mapsto\frac{1}{z}$.    By \cref{prop:Yam10},    there is a proper Zariski closed subset $E\subsetneqq X$ so that, for any holomorphic map $f:\bC_{>1}\to X$ which is not contained in $E$,  one has  
  $N_{{\rm ram}\, p_{\cY}}(r) = o(T_{g}(r, L)) + O(\log r) ||.
$ 
  Here $g:\cY\to Y$ is the induced holomorphic map  defined in the right commutative diagram of  \eqref{figure:curve}, $p_\cY:\cY\to \bC_{>1}$ is the surjective proper finite morphism and $L$ is an ample line bundle on $Y$.   By  \cref{thm2nd}, if $g(\mathcal{Y})\not\subset E\cup \Xi$, where $\Xi$ is the proper Zariski closed subset of $Y$ defined in \cref{thm2nd},  then   there are   an extension $\bar{g}:\overline{\cY}\to  \overline{Y}$ of $g$ and a proper holomorphic surjective map $\overline{p_\cY}:\overline{\cY}\to \mathbb C_{>1}\cup\{\infty\}$ which is an extension of $p_\cY:\cY\to \mathbb C_{>1}$. 
  It follows that $f$ extends to a holomorphic map  $\mathbb C_{>1}\cup\{\infty\}\to \overline{X}$. 
\end{proof}

\section{Rigid representation and $\bC$-VHS} \label{sec:rigid}
In this section we will prove that rigid representations of the fundamental groups of quasi-projective manifolds come from a complex variation of Hodge structures (abbreviated as $\bC$-VHS). The proof relies on Mochizuki's Kobayashi--Hitchin correspondence in the non-compact case and his correspondence between reductive representations of $\pi_1(X)$ into ${\rm GL}_N(\bC)$ and tame pure imaginary harmonic bundles over $X$, as presented in \cite{Moc06,Moc07b}.
\subsection{Regular filtered Higgs bundles}  
In this subsection, we recall the notions of regular filtered Higgs bundles. For more details refer to \cite{Moc06}. Let $X$ be a complex manifold with a reduced simple normal crossing divisor $D=\sum_{i=1}^{\ell}D_i$, and let $U=X-D$ be the complement of $D$. We denote the inclusion map of $U$ into $X$ by $j$.

\begin{dfn}\label{dfn:parab-higgs}
	A \emph{regular filtered Higgs bundle} $(\bm{E}_*,\theta)$ on $(X, D)$ is holomorphic vector bundle $E$ on $X-D$, together with an $\mathbb{R}^\ell$-indexed
	filtration ${}_{ \bm{a}}E$ (so-called {\em parabolic structure}) by coherent subsheaves of $j_*E$ such that
	\begin{enumerate}[leftmargin=0.7cm]
		\item $\bm{a}\in \mathbb{R}^\ell$ and ${}_{\bm{a}}E|_U=E$. 
		\item  For $1\leq i\leq \ell$, ${}_{\bm{a}+\bm{1}_i}E = {}_{\bm{a}}E\otimes \cO_X(D_i)$, where $\bm{1}_i=(0,\ldots, 0, 1, 0, \ldots, 0)$ with $1$ in the $i$-th component.
		\item $_{\bm{a}+\bm{\epsilon}}E = {}_{\bm{a}}E$ for any vector $\bm{\epsilon}=(\epsilon, \ldots, \epsilon)$ with $0<\epsilon\ll 1$.
		\item  The set of {\em weights} \{$\bm{a}$\ |\  $_{\bm{a}}E/_{\bm{a}-\bm{\epsilon}}E\not= 0$  for any vector $\bm{\epsilon}=(\epsilon, \ldots, \epsilon)$ with $0<\epsilon\ll 1$\}  is  discrete in $\mathbb{R}^\ell$.
		\item There is a $\cO_X$-linear map, so-called Higgs field, 
		$$\theta:\diae\to \Omega_X^1(\log D)\otimes \diae$$
		such that
	 $\theta\wedge \theta=0$, 
		and
		$$\theta(_{\bm{a}}E)\subseteq \Omega_X^1(\log D)\otimes {}_{\bm{a}}E.$$ 
	\end{enumerate}
\end{dfn}
 Denote $_{\bm{0}}E$ by $\diae$, where $\bm{0}=(0, \ldots, 0)$.  By the work of Borne-Vistoli the parabolic structure of a parabolic bundle  is  \emph{locally abelian}, \emph{i.e.} it admits a local frame compatible with the filtration (see e.g. \cite{IS07} and \cite{BV12}).

A natural class of regular filtered   Higgs bundles comes from prolongations of tame harmonic bundles. We first   recall some notions in \cite[\S 2.2.1]{Moc07}.  Let $E$ be a holomorphic vector bundle with a $\cC^\infty$ hermitian metric $h$ over $X-D$.
Pick any $x\in D$. Let  $(\Omega; z_1, \ldots, z_n)$  be a coordinate system of $X$ centered at  $x$. For any section $\sigma\in \Gamma(\Omega-D,E|_{\Omega-D})$, let $|\sigma|_h$ denote the norm function of $\sigma$ with respect to the metric $h$. We denote $|\sigma|_h\in \cO(\prod_{i=1}^{\ell}|z_i|^{-b_i})$ if there exists a positive number $C$ such that $|\sigma|_h\leq C\cdot\prod_{i=1}^{\ell}|z_i|^{-b_i}$. For any $\bm{b}\in \bR^\ell$, say $-\mbox{ord}(\sigma)\leq \bm{b}$ means the following:
$$
|\sigma|_h=\cO(\prod_{i=1}^{\ell}|z_i|^{-b_i-\varepsilon})
$$
for any real number  $\varepsilon>0$ and $0<|z_i|\ll1$. For any $\bm{b}$, the sheaf ${}_{\bm{b}} E$ is defined as follows: 
\begin{align}\label{eq:prolongation}
	\Gamma(\Omega, {}_{\bm{b}} E):=\{\sigma\in\Gamma(\Omega-D,E|_{\Omega-D})\mid -\mbox{ord}(\sigma)\leq \bm{b} \}. 
\end{align}
The sheaf ${}_{\bm{b}} E$ is called the prolongment of $E$ by an increasing order $\bm{b}$. In particular, we use the notation ${}^\diamond E$ in the case $\bm{b}=(0,\ldots,0)$.

According to  Simpson \cite[Theorem 2]{Sim90} and Mochizuki \cite[Theorem 8.58]{Moc07}, the above prolongation gives a regular filtered Higgs bundle.
\begin{thm}[Simpson, Mochizuki] \label{thm:SM} Let $(X, D)$ be a complex manifold $X$ with a simple normal crossing divisor $D$. If $(E , \theta, h)$ is a tame harmonic bundle on $X-D$, then the corresponding filtration $_{\bm{b}}E$ defined above defines a regular filtered Higgs bundle $(\bm{E}_*, \theta)$ on $(X,D)$.  	\qed
\end{thm}

\subsection{Character variety and rigid representation}  
In this subsection, we discuss the character variety and rigid representations. Let $X$ be a quasi-projective manifold, and $G$ be a reductive algebraic group defined over a number field $k$.   
We have an affine scheme $\Hom(\pi_1(X), G)$ defined over $\overline{\bQ}$ that represents the functor
$$
S\mapsto \Hom(\pi_1(X), G(S))
$$
for any ring $S$. 
The \emph{character variety} $M_B(\pi_1(X),G):=\Hom(\pi_1(X), G)\slash\!\!\slash G$ is the GIT quotient of $\Hom(\pi_1(X), G)$ by $G$, where $G$ acts by conjugation. Note that it might be reducible while it is called variety.  Thus, $\Hom(\pi_1(X), G)\to M_B(\pi_1(X),G)$ is surjective.  For any representation $\varrho:\pi_1(X)\to G(\bC)$, we write $[\varrho]$ to denote  its image in $M_B(\pi_1(X),G)$.  We list some properties of  character varieties which will be used in this paper, and we refer the readers to the comprehensive paper \cite{Sik12} for more details.

Let $K$ be any algebraically closed field of   characteristic zero containing $k$.  A representation $\varrho:\pi_1(X)\to G(K)$ is called  \emph{reductive} if the  Zariski closure of $\varrho(\pi_1(X))$ is a   reductive group. 
In particular,  Zariski dense representations in  $\Hom(\pi_1(X), G)(K)$ are reductive for $G$ is reductive.

By \cite[Theorem 30]{Sik12}, the orbit of  any representation $\varrho$ in $ \Hom(\pi_1(X), G)(K)$ is closed if and only if  $\varrho$ is   reductive. Two reductive representations $\varrho,\varrho'$ are conjugate under $G(K)$ if and only if $[\varrho]=[\varrho']$.

A reductive representation   
$\varrho:\pi_1(X)\to G(K)$  is   \emph{rigid} if  the irreducible component  of $M_B(\pi_1(X),G)$ containing $[\varrho]$  is zero dimensional,  otherwise it is called \emph{non-rigid}.  For a  rigid reductive  representation $\varrho$, by  above arguments any continuous deformation $\varrho':\pi_1(X)\to G(K)$ of $\varrho$ which is reductive is conjugate to $\varrho$  under $G(K)$. 

\subsection{Rigid representations underlie $\bC$-VHS}
We will prove that any rigid representation underlies a $\bC$-VHS. We refer the readers to    \cite[\S 3.1.3]{Moc06} for  the definition of  $\mu_L$-stability of regular filtered Higgs bundles with respect to some ample polarization $L$.
\begin{thm}\label{thm:vhs}
	Let $X$ be a quasi-projective manifold.  Let  $\sigma:\pi_1(X)\to G(\bC)$ be a Zariski dense representation where $G$ is a reductive  algebraic group over  $\bC$. If  $\sigma$ is  rigid, then $\sigma$ underlies a $\bC$-VHS.   
\end{thm}

\begin{proof}
	Fixing a faithful representation \(G \to {\rm GL}_{N}\), \(\sigma\) induces a representation \(\varrho : \pi_{1}(X) \to {\rm GL}_{N}(\mathbb{C})\) whose Zariski closure is \(G\).
	\medskip
	
\noindent	{\em Step 1. We may assume that $\varrho$ is a simple representation.} Since the Zariski closure of $\varrho$ is $G$, it follows that $\varrho$ is semisimple and thus $\varrho=\oplus_{i=0}^{\ell}\varrho_i$ where   $\varrho_i:\pi_1(X)\to GL(V_i)$ is a simple representation with $\oplus_{i=0}^{\ell}V_i=\bC^N$. Let $G_i\subset GL(V_i)$ be the Zariski closure of $\varrho_i$ which is reductive. Then $G:=\prod_{i=0}^{\ell}G_i$, and it follows that	$M_B(X,G)=\prod_{i=1}^{\ell}M_B(X,G_i)$. Since $[\varrho]$ is an isolated point in $M_B(X,G)$, it implies that each $[\varrho_i]$ is an isolated point hence rigid.   We thus can assume that $\varrho:\pi_1(X)\to {\rm GL}_N(\bC)$ is a simple representation to prove the theorem. 
	
\noindent	{\em Step 2. Higgs field and scaling by a real number.}
	Take a projective compactification $\overline{X}$ of $X$ so that $D:=\overline{X}\backslash X$  is a simple normal crossing divisor.  Fix an ample polarization $L$ 
on $\overline{X}$. By \cref{moc}, we can find a tame purely imaginary harmonic bundle \((E, \theta, h)\) with monodromy \(\varrho\) by \cref{moc}. Then, we can let $(\mathbf{E}_*, \theta)$ be the associated regular filtered Higgs bundle  
of $(E,\theta,h)$  on the log pair $(\overline{X},D)$ defined in \cref{thm:SM}.  This filtered Higgs bundle is $\mu_L$-stable  regular filtered  Higgs bundle with trivial characteristic numbers.   
	
	Let $t$ be a real number in $]0,1]$.  Then   $(\mathbf{E}_*, t\theta)$ is also $\mu_L$-stable   with trivial characteristic numbers. By the Kobayashi-Hitchin correspondence \cite[Theorem 9.4]{Moc06}, there is a pluriharmonic metric $h_t$ for $(E,t\theta)$ which is adapted to the parabolic structure of $(\mathbf{E}_*, t\theta)$. In particular, $(E,t\theta,h_t)$ is a tame harmonic bundle. It is moreover pure imaginary since $t$ is real. Hence the associated monodromy representation $\varrho_t:\pi_1(X)\to {\rm GL}_N(\bC)$ of the flat connection $\bD_t:=\nabla_{h_t}+t\theta+(t\theta)_{h_t}^\dagger$ is  semisimple by \cref{moc}. Here $\nabla_t$ is the Chern connection of $(E,h_t)$.   
	By \cite[Lemma 10.10]{Moc06}, the Zariski closure of $\varrho_t$ coincides with that of $\varrho$, hence is also $G$.  
Hence $[\varrho_t]\in M_B(\pi_1(X),G)(\bC)$.  

	
\noindent	{\em Step 3. Existence of isometries between the deformations.} By the proofs of \cite[Theorem 10.1 \&  Lemma 10.11]{Moc06}, we know that the map
	\begin{align*}
		]0,1]&\to M_B(\pi_1(X),G)(\bC)\\
		t&\mapsto [\varrho_t]
	\end{align*}
	is \emph{continuous}, where we endow the complex affine variety $M_B(\pi_1(X),G)(\bC)$ with the analytic topology. Since $ [\varrho] \in M_B(X,G)(\bC)$ is isolated and $\varrho_t$ is reductive, it follows that $\varrho_t$ is conjugate to $\varrho$. This implies that $\varrho_t$ is also simple for each $t\in ]0,1]$.  Let  $V$ be the underlying smooth vector bundle of $E$.  Fix some $t\in ]0,1[$. One can thus construct a  smooth automorphism $\varphi:V\to V$ so that $\bD_t=\varphi^*\bD_1:=\varphi^{-1}\bD_1\varphi$.  Hence  $\varphi^*h$  defined by
	$
	\varphi^*h(u,v):=h(\varphi(u),\varphi(v)) 
	$  is the harmonic metric for the flat bundle $(V, \varphi^*\bD_1)=(V, \bD_t)$.   By the unicity of harmonic metric of  simple flat bundle in  \cref{moc},  there is a constant $c>0$ so that $c\varphi^*h=h_t$.  Let us  replace the automorphism $\varphi$ of $V$ by $\sqrt{c}\varphi$ so that   we will have $\varphi^*h=h_t$.
	\medskip
	
\noindent	{\em Step 4. \(\theta\) is nilpotent.} Since the decomposition of $\bD_t$ with respect to the metric $h_t$ into a sum of unitary connection   and self-adjoint operator of $V$ is unique, it implies that 
	\begin{align*}
		\nabla_{h_t}&=\nabla_{\varphi^*h}=\varphi^*\nabla_h\\
		t\theta+ (t\theta)^\dagger_{h_t}&=\varphi^*(\theta+\theta_{h}^\dagger).
	\end{align*} 
	Since $E=(V, \nabla_h^{0,1})=(V,\nabla_{h_t}^{(0,1)})$, the first equality means that $\varphi$ is a holomorphic isomorphism of $E$. The second one implies that  
	\begin{align}\label{eq:same}
		t\theta=\varphi^*\theta:=\varphi^{-1}\theta \varphi.
	\end{align} 
	Let $T$ be a formal variable. Consider the characteristic polynomial 
	$$
	\det(T-\varphi^{-1}\theta \varphi)=\det(T-\theta)=T^N+a_1T^{N-1}+\cdots+a_0
	$$
	with $a_i\in H^0(\overline{X}, \Omega^i_{\overline{X}}(\log D))$. Note that
	$$
	\det(T-t\theta)=T^N+ta_1T^{N-1}+\cdots+t^Na_0.
	$$
	By \eqref{eq:same} and $t\in (0,1)$, it follows that $a_i\equiv 0$ for each $i=1,\ldots,N$. Hence $\theta$ is \emph{nilpotent}.
	\medskip
	
\noindent	{\em Step 5. Scaling of $\theta$ by an element of \(U(1)\).} By the previous step,  $(E,\theta,h)$ has nilpotent residues in the sense of \cref{def:nilpotency}.  Then for any $\lambda\in  U(1)$, $(E,\lambda\theta,h)$ also has nilpotent residues, hence is tame pure imaginary 
harmonic bundle.  We apply \cite[Lemmas 10.9 \&  10.10]{Moc06} to conclude that for the  monodromy representation $\varrho_\lambda$ of the flat connection  $\bD_\lambda:=\nabla_h+\lambda\theta+\overline{\lambda}\theta^\dagger_h$, its Zariski closure is also $G$, which implies that $\varrho_\lambda$ is reductive.  Hence $[\varrho_\lambda]\in M_B(\pi_1(X),G)(\bC)$.   
As $\varrho_\lambda$ is a continuous deformation of $\varrho$ in $\Hom(\pi_1(X), G)(\bC)$ and $ [\varrho] \in M_B(X,G)(\bC)$ is isolated, it follows that $\varrho_\lambda$ is conjugate to $\varrho$. 
	\medskip
	
\noindent	{\em Step 6. End of proof.} The rest of the proof is exactly the same as \cite[Proposition 4.8]{BDDM}; we provide it for completeness sake.  Now fix $\lambda\in U(1)$ which is not root of unity. By the same argument as above, there is a  a  smooth automorphism $\phi:V\to V$ so that $\bD_\lambda=\phi^*\bD_1:=\phi^{-1}\bD_1\phi$ and $\phi^*h=h$.  Moreover, 
	\begin{align*}
		\nabla_{h}&=\nabla_{\phi^*h}=\phi^*\nabla_h\\
		\lambda\theta+ \overline{\lambda}\theta^\dagger_{h}&=\phi^*(\theta+\theta_{h}^\dagger).
	\end{align*} 
	In other words, $\phi:(E,h)\to (E,h)$ is a holomorphic automorphism which is moreover an isometry.  Moreover, $\phi^*\theta=\lambda\theta$. 
	Consider the prolongation $\diae_h$  over $\overline{X}$ via norm growth defined in \cref{sec:tame}. Since $\phi: (E,h)\to (E,h)$ is   an isometry, it thus extends to a holomorphic isomorphism $\diae_{h}\to \diae_h$, which we still denote by $\phi$. Recall that $(E,\theta,h)$ and $(E,\lambda\theta,h)$ 
is tame,  we thus have the following commutative diagram
	\begin{equation*}
		\begin{tikzcd}
			\diae_{h} \arrow[r, "\lambda\theta"] \arrow[d, "\phi"] & \diae_{h}\otimes \Omega_{\overline{X}}(\log D) \arrow[d, "\phi\otimes {\rm id}"] \\
			\diae_{h} \arrow[r, "\theta"]  & \diae_{h}\otimes \Omega_{\overline{X}}(\log D).
		\end{tikzcd}
	\end{equation*} 
Let $T$ be a formal variable. Consider the polynomial  $\det (\phi-T)=0$ of $T$. Since its coefficients are holomorphic functions defined over  $\overline{X}$, they are all constant.  Let $\eta$ be an eigenvalue of $\det (\phi-T)=0$.  Consider the generalized eigenspace $\diae_{h,\eta}$ defined by $\ker (\phi-\eta)^\ell=0$ for some sufficiently big $\ell$. One can check that $\theta:  \diae_{h,\eta} \to \diae_{h,\lambda^{-1}\eta}\otimes \Omega_{\overline{X}}(\log D)$. Since $\lambda$ is not root of unity,  the eigenvalues of $\phi$ break up into a finite number of  chains of the form $\lambda^{i}\eta,\ldots, \lambda^{-j}\eta$  so that $\lambda^{i+1}\eta$ and $\lambda^{-j-1}\eta$
	are not eigenvalues of $\phi$.  Therefore, there is decomposition $\diae_{h}=\oplus_{i=0}^{m} \diae_i$ so that 
	$$
	\theta: \diae_i\to \diae_{i+1}\otimes \Omega_{\overline{X}}(\log D).
	$$
	In the language of \cite[Definition 2.10]{Den22} $(\diae_{h}=\oplus_{i=0}^{m} E_i,\theta)$ is a \emph{system of log Hodge bundles}. One can apply \cite[Proposition 2.12]{Den22} to show that the harmonic metric  $h$ is moreover a \emph{Hodge metric} in the sense of \cite[Definition 2.11]{Den22}, \emph{i.e.}  $h(u,v)=0$ if $u\in E_i|_{X}$ and $v\in E_j|_{X}$ with $i\neq j$. Hence by \cite[\S 8]{Sim88}, $(E,\theta,h)$ corresponds to a complex variation of Hodge structures in $X$.     This proves the theorem. 
\end{proof}
\begin{rem}
The above theorem has been known to us for the following cases:	when $X$ is projective, it is proved by Simpson \cite{Sim92}; when $\sigma$ has quasi-unipotent monodromies at infinity and $G={\rm SL}_n$, it is proved by Corlette-Simpson \cite{CS08}; when $G={\rm GL}_N$ and the corresponding harmonic bundle of $\sigma$ has nilpotent residues at infinity, it is proved in \cite{BDDM}. In Steps 1-4 of the above proof we utilise Mochizuki's work to prove that $\theta$ is nilpotent. 
\end{rem}

\section{Proof of  \Cref{main2}} \label{sec:proof1}
In this section, we  prove our first main result of this paper.
\subsection{Existence of unbounded representations}
 We first prove a theorem on constructing unbounded representation in almost simple algebraic groups over non-archimedean local field. This refines a previous result in \cite[\S 4.2]{Yam10}.

\begin{proposition}\label{lem:20220819}
	Let $X$ be a quasi-projective manifold and  let $G$ be an almost simple algebraic group defined over $\bC$. 
Assume that there is a Zariski dense, non-rigid representation $\varrho:\pi_1(X)\to G(\bC)$, then there is a Zariski dense, unbounded representation $\varrho':\pi_1(X)\to G(F)$ where $F$ is some non-archimedean local field of zero characteristic. 
	If moreover $\varrho$ is big, then $\varrho'$ is taken to be big.
\end{proposition}

Before prove this, we prepare two lemmas.

\begin{lem}
Let $G$ be an almost simple algebraic group over a field $K$ of characteristic zero. Then its Lie algebra $\mathrm{Lie}(G)$ is simple. 
\end{lem}
\begin{proof} 
		Since $G$ is almost simple, $\mathrm{Lie}(G)$ is semi-simple thanks to the assumption $\mathrm{char}(K)=0$ (cf. \cite[Prop 4.1.]{Mil17}).
		We may decompose as $\mathrm{Lie}(G)=\mathfrak{g}_1\oplus \cdots\oplus \mathfrak{g}_r$, where $\mathfrak{g}_i$ is a simple ideal of $\mathrm{Lie}(G)$.
		We need to show $r=1$.
		So assume contrary that $r>1$.
		Then $\mathfrak{g}_1\subset \mathrm{Lie}(G)$ is neither trivial nor $\mathrm{Lie}(G)$ itself.
Set $H=C_G(\mathfrak{g}_2\oplus \cdots\oplus \mathfrak{g}_r)$, which is a subgroup of $G$ defined by 
		$$
	 C_G(\mathfrak{g}_2\oplus \cdots\oplus \mathfrak{g}_r)=(R \leadsto g \in G(R) \mid g x=x \text { for all } x \in \mathfrak{g}_2\oplus \cdots\oplus \mathfrak{g}_r(R))
		$$
		for all $k$-algebra $R$ 
		(cf. \cite[Proposition 3.40]{Mil17}). It means that  $H$ acts trivially on $\mathfrak{g}_2\oplus \cdots\oplus \mathfrak{g}_r$.
		Then $\mathrm{Lie}(H)=\mathfrak{g}_1$. 
		Let $H^0\subset H$ be the identity component of $H$.
		Then $\mathrm{Lie}(H^0)=\mathrm{Lie}(H)=\mathfrak{g}_1$.
		Hence $\mathrm{Lie}(H^0)\subset \mathrm{Lie}(G)$ is an ideal.
		Hence $H^0\subset G$ is normal (cf. \cite[Thm 3.31]{Mil17}), which is neither trivial nor $G$ itself.
		Since $G$ is almost simple, this is a contradiction.
		Thus $r=1$ and $\mathrm{Lie}(G)$ is simple. 
\end{proof}
\begin{lem}\label{lem:20220827}
Let $G$ be an almost simple algebraic group defined over a field $K$ of characteristic zero.
Let $\mathrm{Ad}:G\to \mathrm{Aut}(\mathrm{Lie}(G))$  be the adjoint representation.
Let $W\subset \mathrm{Lie}(G)$ be a $K$-linear subspace which is $G$-invariant.
Then either $W=\mathrm{Lie}(G)$ or $W=\{ 0\}$. 
\end{lem}

\begin{proof}
Write $\kg:=\mathrm{Lie}(G)$. 	Consider the adjoint representation $G \rightarrow \mathrm{GL}_{\kg}$. For any $K$-subspace $P$  of $V$,  the functor
	$$
	R \rightsquigarrow\left\{g \in G(R) \mid g P_R=P_R\right\}\  \mbox{for all}\ K\mbox{-algebra } R
	$$
	is a subgroup of $G$, denoted $G_P$. Then by our assumption $G_W=G$.  By \cite[Proposition 3.38]{Mil17}, it follows that $\mathrm{Lie}(G_W)=\mathrm{Lie}(G)_W$,
	where 
	$$
	\mathrm{Lie}(G)_W:=\{x\in \kg\mid  \mathrm{ad}(x)(W)\subset W \}. 
	$$   
In other words, $[x,y]\in W$ for all $x\in\kg$, $y\in W$, i.e., $W\subset  \kg$ is an ideal.
Note that $\kg$ is a simple Lie algebra. 
Hence $W=\kg$ or $W=\{ 0\}$. 
\end{proof}

\medskip

\begin{proof}[Proof of \Cref{lem:20220819}] 
Since $G$ is almost simple,  by the classification of almost simple linear algebraic groups defined over $\bC$,   $G$ is isogenous to exactly one of the following:  $A_n, B_n, C_n, D_n, E_6, E_7, E_8, F_4, G_2$. Therefore, $G$ is defined over some number field $k\subset \bar{\bQ}$. Moreover, it is absolutely almost simple, i.e. for any field extension $L/k$, the base change $G_L$  is also an almost simple algebraic group over $L$.  

Since $\pi _1(X)$ is finitely presented, there exists an affine scheme $R$ defined over $k$ such that
$$R (L)=\Hom (\pi _1(X),G(L))$$
for every field extension $L/k$.
This space is defined as follows: We choose generators $\gamma _1,\ldots ,\gamma _\ell$ for $\pi _1(X)$. Let $\mathcal R$ be the set of relations among the generators $\gamma _i$. Then 
$$R \subset \underbrace{G\times \cdots \times G}_{\ell \ \text{times}}$$
is the closed subscheme defined by the equations $r(m_1,\ldots ,m_\ell)=1$ for $r\in \mathcal R$. A representation $\tau :\pi _1(X)\to G(L)$ corresponds to the point $(m_1,\ldots ,m_\ell)\in R (L)$ with $m_i=\tau (\gamma _i)$. 
Note that $R$ is an affine scheme, since it is a closed subscheme of an affine variety.

\medskip

\begin{claim}\label{claim:ZD}
There exists a Zariski closed subset $E\subset R$ defined over $k$ with the following property:
Let $L$ be a field extension of $k$ and $\tau :\pi _1(X)\to G(L)$ be a representation such that $\tau(\pi_1(X))$ is infinite.
Then $\tau(\pi_1(X))$ is not Zariski dense in $G({L})$ if and only if the corresponding point $[\tau]\in R(L)$ satisfies $[\tau]\in E(L)$. 
\end{claim} 

\begin{proof}[Proof of \cref{claim:ZD}]
A stronger result is proved in \cite[Proposition 8.2]{AB}.
Here we give a proof for the sake of completeness.
 Consider  the adjoint action $G\curvearrowright \mathrm{Lie}(G)$ which is defined over $k$. 
This action induces the action $G\curvearrowright \mathrm{Gr}_d(\mathrm{Lie}(G)$ over $k$, where $0< d<\dim G$.
Here $\mathrm{Gr}_d(\mathrm{Lie}(G)$ is the grassmannian variety, which we denote by $X_d$. Recall that $G_{L}$ is  almost simple for the extension $L$ of $k$.   
We remark that the action $G(L)\curvearrowright X_d(L)$ has no fixed point. 
Indeed if there exists a fixed point $[P]\in X_d(L)$, where $P\subsetneqq \mathrm{Lie}(G)_L$ is a $d$-dimensional subspace defined over $L$, we get a $G(L)$ invariant subspace $P\subset \mathrm{Lie}(G)_L=\mathrm{Lie}(G_L)$ which is impossible by Lemma \ref{lem:20220827}. 
On the other hand, if $H\subset G_L$ is an algebraic subgroup of dimension $d$, then $H(L)$ fixes the point $[\mathrm{Lie}(H)]\in X_d(L)$. 

Let $S=\{\gamma_1,\ldots,\gamma_l\}\subset \pi_1(X)$ be a finite subset which generates $\pi_1(X)$.
We define a Zariski closed subset $W_{d}\subset X_d\times G^{S}$ by
$$W_{d}=\{ (x,h_1,\ldots,h_l); h_ix=x, \forall i\in\{1,\ldots,l\}\}.
$$
Then if a representation $\tau :\pi _1(X)\to G(L)$ satisfies $\tau(\pi_1(X))\subset H(L)$ for some algebraic subgroup of dimension $d$, where $0<d<\dim G$, then we have $([\mathrm{Lie(H)}],\tau(\gamma_1),\ldots,\tau(\gamma_l))\in W_{d}(L)$.

Let $p:X_d\times G^{S}\to G^{S}$ be the projection.
Then $p$ is a proper map, for $X_d$ is complete.
Hence $p(W_{d})\subset G^{S}$ is a Zariski closed subset defined over $k$.
We set $E_{d}=p(W_{d})\cap R$, which is a Zariski closed subset of $R$.
Hence if a representation $\tau :\pi _1(X)\to G(L)$ satisfies $\tau(\pi_1(X))\subset H(L)$ for some algebraic subgroup of dimension $d$, then the corresponding point satisfies $[\tau]\in E_{d}(L)$.

Now let $\tau :\pi _1(X)\to G(L)$ be a representation such that $\tau(\pi_1(X))$ is infinite.
Assume first that $\tau(\pi_1(X))$ is not Zariski dense in $G(L)$.
Let $H\subset G$ be the Zariski closure of $\tau(\pi_1(X))\subset G(L)$.
Then $H$ is defined over $L$.
Since $\tau(\pi_1(X))$ is infinite, we have $d=\dim H>0$.
Then by the above consideration, we have $[\tau]\in E_{d}(L)$.  
Next suppose $[\tau]\in E_{d}(L)$ for some $d$ with $0<d<\dim G$.
Then by $E_{d}(L)\subset p(W_{d}(L))$, there exists a $d$-dimensional $L$-subspace $P\subset \mathrm{Lie}(G)_L$  such that $([P],\tau(\gamma_1),\ldots,\tau(\gamma_l))\in W_{d}(L)$.
Hence $\tau(\pi_1(X))$ fixes the point $[P]\in X_d(L)$.
If $\tau(\pi_1(X))\subset G(L)$ is Zariski dense, then $G(L)$ also fixes $[P]\in X_d(L)$.
This is impossible as we see above.
Hence $\tau(\pi_1(X))\subset G(L)$ is not Zariski dense.

We set $E=E_1\cup \cdots\cup E_{\dim G-1}$.
This concludes the proof of the claim.
\end{proof} 

Now we return to the proof of the proposition. 
The group $G$ acts on $R$ by simultaneous conjugation.
Put $M=R // G$, and let $p:R\to M$ be the quotient map which is surjective.
Then $M$ is an affine scheme defined over $k$.
Let $\varrho\in R (\mathbb C)$ be the point which correspond to the Zariski dense representation $\varrho :\pi _1(X)\to G(\mathbb C)$.
\par

Let $W$ be the set of words in $x_1,\ldots,x_\ell$.
Given $w\in W$, we define a closed subscheme $Z_{w}\subset R$ defined by $Z_w=\{ (m_1,\ldots,m_\ell)\in R;\ w(m_1,\ldots,m_\ell)=1\}$.
Then $Z_w$ is defined over $k$. 

Let $M'$ be the irreducible component of $M_{\bar{\bQ}}$ such that $[\varrho]\in M'(\bC)$. 
Since $\varrho $ is non-rigid,   one has $\dim M'>0$.    Let $R'$ be the irreducible component   of $R_{\bar{\bQ}}$ containing $\varrho$. Since  $p:R_{\bar{\bQ}}\to M_{\bar{\bQ}}$  is surjective,   $p|_{R'}: R'\to M'$  is surjective.  Then $R'$ and $M'$ are defined over some finite extension  $K$  of $k$.  Note that $R'$ and $M'$ are geometrically irreducible.  
Let $\eta\in R'$ be the schematic generic point. 
Then $K(\eta)$ is a finitely generated over $K$.
Let $\mathfrak{K}$ be a field extension of $K$, which is finitely generated over $K$, such that an embedding $K(\eta)\hookrightarrow \mathfrak{K}$ exists and that $\varrho$ is defined over $\mathfrak{K}$, i.e., $\varrho:\pi_1(X)\to G(\mathfrak{K})$.

Let $p$ be a prime number and let $\mathbb Q_p$ be the completion.
Then since the transcendental degree of $\mathbb Q_p$ over $\mathbb Q$ is infinite and $\mathfrak{K}$ is finitely generated over $\mathbb Q$, there exists a finite extension $L/\mathbb Q_p$ such that an embedding $\mathfrak{K}\hookrightarrow L$ exists.
Then we have $[\varrho ]\in R (L)$. Recall that $G_L$ is   almost simple as an algebraic group defined over $L$. 
We remark that  $\varrho:\pi_1(X)\to G(L)$ is  Zariski dense.  
Thus by \cref{claim:ZD}, we have $[\varrho]\not\in E(L)$.
In particular, we have $E\subsetneqq R$.

We define $W_o\subset W$ by $w\in W_o$ iff $Z_w(\bar{L})\cap R'(\bar{L})\subsetneqq R'(\bar{L})$. 
We note that $W$, hence $W_o$ is countable. 
We may take $\eta_0\in R'(L)$ such that the corresponding map $\eta_0:\mathrm{Spec}\ L\to R'$ has image $\eta\in R'$.
Then $\eta_0\not\in Z_w(\bar{L})$ for all $w\in W_o$ and $\eta_0\not\in E(\bar{L})$.

\par
Since  $\dim M'>0$, there exists a morphism of $L$-scheme  $\psi :M'\to \mathbb A^1$ such that the image $\psi (M')$ is Zariski dense in $\mathbb A^1$.
Recall that  $p|_{R'}:R'\to M'$ is surjective.  Hence  the image $\psi \circ p(R')$ is Zariski dense in $\mathbb A^1$, and 
there exists an affine curve $C\subset R'$ defined over $L$ such that the restriction $\psi \circ p|_C:C\to \mathbb A^1$ is generically finite and $\eta_0\in C(L)$.
Since $R'$ is geometrically irreducible,,  one has $C\not\subset Z_w$ for all $w\in W_o$ and $C\not\subset E$.
We may take a Zariski open subset $U\subset \mathbb A^1$ such that $\psi \circ p|_C$ is finite over $U$.
Let $x\in U(L)$ be a point, and let $y\in C(\bar{L})$ be a point over $x$.
Then $y$ is defined over some extension of $L$ whose extension degree is bounded by the degree of $\psi \circ p|_C:C\to \mathbb A^1$.
Note that there are only finitely many such field extensions.
Hence there exists a finite extension $F/L$ such that the points over $U(L)$ are all contained in $C(F)$.
Since $U(L)\subset \mathbb A^1(F)$ is unbounded, the image $\psi \circ p(C(F))\subset \mathbb A^1(F)$ is unbounded.
\par
Let $R_0\subset R(F)$ be the subset whose points correspond to $p$-bounded representations.
Let $M_0\subset M(F)$ be the image of $R_0$ under the quotient $p:R\to M$.
Then by \cite[Lemma 4.2]{Yam10}, $M_0$ is compact. 
Hence $\psi (M_0)$ is compact.
In particular it is bounded.
On the other hand, $\psi \circ p (C(F))\subset \mathbb A^1(F)$ is unbounded.
Hence $\psi \circ p (C(F))\backslash \psi (M_0)$ is an uncountable set.
Hence $C(F)\backslash R_0$ is an uncountable set.
Thus we may take $y\in C(F)$ such that $y\not\in R_0$ and $y\not\in Z_w(F)$ for all $w\in W_0$ and $y\not\in E(F)$.
Here we note that $C(F)\cap Z_w$ and $C(F)\cap E$ is a finite set, hence $\cup_{w\in W_o} (C(F)\cap Z_w)\cup (C(F)\cap E)$ is a countable set.
Let $\varrho':\pi _1(X)\to G(F)$ be the representation which corresponds to $y\in C(F)\subset R(F)$, i.e., $[\varrho']=y$.
Then $\varrho'$ is a $p$-unbounded representation such that $\varrho'$ is Zariski dense or has finite image  by \cref{claim:ZD}.

Next we show $\mathrm{ker}(\varrho')\subset \mathrm{ker}(\varrho)$.
Indeed, every element in $\mathrm{ker}(\varrho')$ is written as $w(\gamma_1,\ldots,\gamma_\ell)$ for some $w\in W$.
Then $[\varrho']\in Z_w$.
Hence $w\not\in W_0$.
Hence $R(F)\subset Z_w$.
In particular $[\varrho]\in Z_w$, hence $w(\gamma_1,\ldots,\gamma_\ell)\in \mathrm{ker}(\varrho)$.
This shows $\mathrm{ker}(\varrho')\subset \mathrm{ker}(\varrho)$.
Hence if $\varrho$ is big, then $\varrho'$ is also big.
In particular, $\varrho'$ has infinite image, hence Zariski dense image.
This completes the proof of the proposition.
\end{proof}

\begin{rem}\label{rem:20230417}
The idea of the proof of \cref{lem:20220819} goes back to the works of Zuo \cite{Zuo96}, as well as \cite{CS08, Eys04}.
On the other hand, we should mention that there are several drawbacks, when we apply these works to the hyperbolicity problems.
In the following, we first briefly review Zuo's influential idea to construct unbounded representations from given non-rigid representations.
Indeed the proof of \cref{lem:20220819} is inspired by his idea.
Next we discuss the importance of reducing mod $p$. 
Finally, we highlight the difficult problems that still occurs in proving hyperbolicity problems using reduction mod $p$ methods in \cite{CS08, Eys04}.
	
	 Assume that $\varrho:\pi_1(X)\to G(\bC)$ is non-rigid, where $G$ is an almost simple algebraic group. For simplicity we assume $\varrho:\pi_1(X)\to G(k)$ for some number field $k$. 
In \cite{Zuo96,CS08,Eys04}, it is explained that we can choose an affine curve $T\subset \Hom(\pi_1(X),G)$ containing $\varrho$ defined over   $k$ (after we replace $k$ by some finite extension) such that $\pi(T)$ is non-constant for the GIT quotient $\pi:\Hom(\pi_1(X),G)\to M_{\rm B}(\pi_1(X),G)$. Then $T$ induces a representation $\varrho_T:\pi_1(X)\to G(k[T])$ such that $\varrho$ is the composition of $\varrho_T$ with some $\lambda_\varrho:k[T]\to k$. Let $\overline{T}$  be a projective compactification of $T$. Since $M_B(\pi_1(X),G)$ is affine, there exists  a point $x$ in $\overline{T}$ such that $\pi|_{T}:T\to  M_B(\pi_1(X),G)$ cannot extend over $x$. Let $v:k(T)\to\bZ$ be the valuation defined by $x$, where $k(T)$ is the function field of $T$, and $\widehat{k(T)}_v$ be the completion of $k(T)$ with respect to $v$. Then we can prove that the induced representation $\varrho'_T: \pi_1(X)\to G(\widehat{k(T)}_v)$ is unbounded based on Procesi's theorem and is Zariski dense and big. 

Unfortunately, the field $\widehat{k(T)}_v$ is not locally compact, which prevents the application of Gromov-Schoen's theory.  
To address this issue, we can follow the approach mentioned in \cite{Eys04, CS08} and perform reduction modulo $p$ for $T$. To do so, we choose some model   $\mathcal{T}\to \spec{\cO_{k}[\frac{1}{n}]}$ of $T$. For any   maximal ideal $\mathfrak{p}$ of $\cO_k[\frac{1}{n}]$, we consider the base change $T_{\mathfrak{p}}$ of $T$ over $\mathfrak{p}$.  Then $\mathcal{T}_{\mathfrak{p}}$ is defined over some finite field $\bF_{\mathfrak{p}}$. For a general $\mathfrak{p}$, $\varrho_T$  gives rise to a curve of representation $\varrho_{\mathfrak{p}}: \pi_1(X)\to G(\bF_{\mathfrak{p}}[\mathcal{T}_{\mathfrak{p}}])$ which is non-constant.   We complete the curve $\mathcal{T}_{\mathfrak{p}}$ by adding  points $P_1,\ldots,P_k$   at infinity. Let $\bF_q$ be a finite extension of $\bF_{\mathfrak{p}}$ such that $P_1,\ldots,P_k$ are defined over $\bF_q$.   Each $P_i$ induces a non-archimedean valuation $v_i$ on $\bF_{q}(T_{\mathfrak{p}})$.    Denote by $\widehat{\bF_{q}(T_{\mathfrak{p}})}_{v_i}$   the completion of $\bF_{q}(T_{\mathfrak{p}})$ with respect to the  non-archimedean  valuation $v_i$. Then $\widehat{\bF_{q}(T_{\mathfrak{p}})}_{v_i}\simeq \bF_{q}((t))$ for each $i$, and thus they are all locally compact.   One can prove that   for some $i=1,\ldots,k$, the induced representation $\varrho_{\mathfrak{p},i}:\pi_1(X)\to G(\widehat{\bF_{q}(T_{\mathfrak{p}})}_{v_i})$ by $\varrho_{\mathfrak{p}}$ is  unbounded.

	However, it is presently unknown whether the induced representation $\varrho_{\mathfrak{p},i}: \pi_1(X)\to G(\bF_q((t)))$ is big and has a \emph{semisimple} Zariski closure for general $\mathfrak{p}$. In fact, it is even unclear  whether $\varrho_{\mathfrak{p},i}$ is reductive. 
This issue is a critical point in the proof of hyperbolicity of $X$.
For this regards, we should mention that the proofs in \cite[Sec. 4.1]{Sun22}, \cite[Sec. 6.2.1]{Bru22} 
are incomplete because of these unfixed problems.   
We will come back to this issue again later in \cref{rem:previous}. 
	\end{rem}

\subsection{Constructing linear representations of fundamental groups of Galois conjugate varieties}
Let $X$ be a complex smooth projective variety.  Recall that given any automorphism  $\sigma\in {\rm Aut}(\bC/\bQ)$, we can form the conjugate variety ${X}^\sigma$ defined as the complex variety ${X} \times_\sigma\spec \bC$, that is, by the cartesian diagram
\begin{equation}\label{eq:galois}
	\begin{tikzcd}
		{X}^\sigma\arrow[r,"\sigma^{-1}"]\arrow[d] & {X}\arrow[d]\\
		\spec\bC \arrow[r, "\sigma^*"]&\spec \bC 
	\end{tikzcd}
\end{equation}
It is a smooth  projective variety. If $X$ is defined by homogeneous polynomials $P_1, \ldots, P_r$ in some projective space, then $X^\sigma$ is defined by the conjugates of the $P_i$ by $\sigma$. In this case, the morphism from $X^\sigma$ to $X$ in the cartesian diagram sends the closed point with coordinates $\left(x_0: \ldots: x_n\right)$ to the closed point with homogeneous coordinates $\left(\sigma^{-1}\left(x_0\right): \ldots: \sigma^{-1}\left(x_n\right)\right)$, which allows us to denote it by $\sigma^{-1}$.

The morphism $\sigma^{-1}: {X}^\sigma \rightarrow {X}$ is an isomorphism of abstract schemes, but it is not a morphism of complex varieties. It is important to note that, in general, the fundamental groups of the complex variety $X$ and $X^\sigma$ may be quite different, as demonstrated by the famous examples of Serre \cite{Ser64}. Despite this, their algebraic fundamental groups, which are the profinite completions of the topological fundamental groups, are canonically isomorphic.  

The following proposition plays a crucial role in the proof of \cref{main2}:  
\begin{proposition}\label{thm:conjugate}
	Let $X$ be a   smooth quasi-projective   variety and let $\rho:\pi_1(X)\to\mathrm{GL}_n(\mathbb C)$ be a representation.
	Let $\sigma\in\mathrm{Aut}(\mathbb C/\mathbb Q)$.
	Then there exists a representation $\tau:\pi_1(X^{\sigma})\to\mathrm{GL}_n(\mathbb C)$ such that the Zariki closures satisfy 
	\begin{equation}\label{eqn:202303111}
		\overline{\rho(\pi_1(X))}^{\mathrm{Zar}}=\overline{\tau(\pi_1(X^{\sigma}))}^{\mathrm{Zar}}.
	\end{equation}
	More precisely, $\tau$ satisfies the following property:
	If $Y\to X$ is a morphism from a smooth quasi-projective variety $Y$, we have
	\begin{equation}\label{eqn:202303115}
		\overline{\rho(\mathrm{Im}[\pi_1(Y)\to\pi_1(X)])}^{\mathrm{Zar}}=\overline{\tau(\mathrm{Im}[\pi_1(Y^{\sigma})\to\pi_1(X^{\sigma})])}^{\mathrm{Zar}}.
	\end{equation}
	In particular, if $\rho$ is big (resp. large), then $\tau$ is big (resp. large).
\end{proposition} 
\begin{proof}
	Let $\gamma_1,\ldots,\gamma_k\in \pi_1(X)$ be a system of generators of $\pi_1(X)$ (as monoid).
	For $l=1,\ldots,k$, we set $\rho(\gamma_l)=(a_{ij}(\gamma_l))_{1\leq i,j,\leq n}\in \mathrm{GL}_n(\mathbb C)$.
	Let $S\subset \mathbb C$ be a finite subset defined by $S=\{a_{ij}(\gamma_l); 1\leq i,j\leq n,1\leq l\leq k\}$.
	Let $\mathbb Q(S)\subset \mathbb C$ be the subfield generated by $S$ over $\mathbb Q$.
	Then $\mathbb Q(S)$ is a finitely generated field over $\mathbb Q$.
	By Cassels' p-adic embedding theorem (cf. \cite{Cas76})
	there exist a prime number $p\in\mathbb N$ and an embedding $\iota:\mathbb Q(S)\hookrightarrow \mathbb Q_p$ such that 
	\begin{equation}\label{eqn:20230311}
		|\iota(a)|_p=1
	\end{equation}
	for all $a\in S$.
	
	We claim that there exists an isomorphism 
	$$\mu:\overline{\mathbb Q_p}\overset{\sim}{\to} \mathbb C$$
	such that $\mu\circ\iota(a)=a$ for all $a\in S$, where $\overline{\mathbb Q_p}$ is an algebraic closure of $\mathbb Q_p$.
	We prove this.
	
	Let $B\subset \mathbb C$ be a transcendence basis of $\mathbb C/\mathbb Q(S)$ 
	Let $B'\subset \overline{\mathbb Q_p}$ be a transcendence basis of $\overline{\mathbb Q_p}/ \iota(\mathbb Q(S))$.
	Then for the cardinality, we have $\#B=\#B'=\#\mathbb R$.
	Hence there exists an extension $\iota_1: \mathbb Q(S)(B)\hookrightarrow \overline{\mathbb Q_p}$ of $\iota$ such that $\iota_1(\mathbb Q(S)(B))=\iota(\mathbb Q(S))(B')$.
	Since $\mathbb C=\overline{\mathbb Q(S)(B)}$ and $\overline{\mathbb Q_p}=\overline{\iota(\mathbb Q(S))(B')}$, $\iota_1$ extends to an isomorphism $\iota_2:\mathbb C\to \overline{\mathbb Q_p}$.
 	Then we set $\mu=\iota_2^{-1}$, which is an isomorphism such that $\mu\circ\iota(a)=a$ for all $a\in S$.
	By this isomorphism $\mu$, we consider
	$$
	\mathrm{GL}_n(\mathbb Z_p)\subset \mathrm{GL}_n(\mathbb Q_p)\subset \mathrm{GL}_n(\overline{\mathbb Q_p})=\mathrm{GL}_n(\mathbb C).
	$$
	
	By \eqref{eqn:20230311}, we have $\rho(\pi_1(X))\subset \mathrm{GL}_n(\mathbb Z_p)$.
	Thus we may consider $\rho$ as $\rho:\pi_1(X)\to \mathrm{GL}_n(\mathbb Z_p)$ so that
	\begin{equation*}
		\rho(\gamma_l)=(\iota(a_{ij}(\gamma_l)))_{1\leq i,j,\leq n}\in \mathrm{GL}_n(\mathbb Z_p).
	\end{equation*}
	Since $\mathrm{GL}_n(\mathbb Z_p)$ is a profinite group, $\rho$ extends to $\widehat{\rho}:\widehat{\pi_1(X)}\to\mathrm{GL}_n(\mathbb Z_p)$, where $\widehat{\pi_1(X)}$ is the profinite completion of $\pi_1(X)$.
	Note that $\widehat{\pi_1(X)}$ is the etale fundamental group of $X$ and the etale fundamental groups of $X$ and $X^{\sigma}$ are naturally isomorphic:
	\begin{equation}\label{eqn:202303112}
		\widehat{\pi_1(X)}\simeq \widehat{\pi_1(X^{\sigma})}.
	\end{equation}
	Hence we define $\tau:\pi_1(X^{\sigma})\to \mathrm{GL}_n(\overline{\mathbb Q_p})$ by the composite of the followings:
	$$
	\pi_1(X^{\sigma})\to \widehat{\pi_1(X^{\sigma})}\simeq \widehat{\pi_1(X)}\overset{\widehat{\rho}}{\to} \mathrm{GL}_n(\mathbb Z_p)\hookrightarrow \mathrm{GL}_n(\overline{\mathbb Q_p}).
	$$
	
	Next we prove \eqref{eqn:202303111}.
	We first prove 
	\begin{equation}\label{eqn:20230313}
		\widehat{\rho}(\widehat{\pi_1(X)})\subset \overline{\rho(\pi_1(X))}^{\mathrm{Zar}}\subset \mathrm{GL}_n(\overline{\mathbb Q_p}).
	\end{equation}
	Note that $\overline{\rho(\pi_1(X))}^{\mathrm{Zar}}\subset \mathrm{GL}_n(\overline{\mathbb Q_p})$ is Zariski closed, hence $p$-adically closed.
	Since $\widehat{\rho}:\widehat{\pi_1(X)}\to\mathrm{GL}_n(\mathbb Z_p)$ is continuous, $\widehat{\rho}^{-1}(\overline{\rho(\pi_1(X))}^{\mathrm{Zar}})\subset \widehat{\pi_1(X)}$ is a closed subset.
	Since the image of $\pi_1(X)\to\widehat{\pi_1(X)}$ is dense and contained in $\widehat{\rho}^{-1}(\overline{\rho(\pi_1(X))}^{\mathrm{Zar}})$.
	Hence $\widehat{\rho}^{-1}(\overline{\rho(\pi_1(X))}^{\mathrm{Zar}})=\widehat{\pi_1(X)}$.
	This shows \eqref{eqn:20230313}.
	Hence $\overline{\widehat{\rho}(\widehat{\pi_1(X)})}^{\mathrm{Zar}}\subset\overline{\rho(\pi_1(X))}^{\mathrm{Zar}}$.
	The converse inclusion is obvious. 
	Hence we have
	\begin{equation}\label{eqn:202303133}
		\overline{\rho(\pi_1(X))}^{\mathrm{Zar}}=\overline{\widehat{\rho}(\widehat{\pi_1(X)})}^{\mathrm{Zar}}\subset \mathrm{GL}_n(\overline{\mathbb Q_p}).
	\end{equation}
	Similarly we have
	$$
	\overline{\tau(\pi_1(X^{\sigma}))}^{\mathrm{Zar}}=\overline{\widehat{\rho}(\widehat{\pi_1(X)})}^{\mathrm{Zar}}\subset \mathrm{GL}_n(\overline{\mathbb Q_p}).$$
	Thus we have \eqref{eqn:202303111} in $\mathrm{GL}_n(\overline{\mathbb Q_p})$.
	Thus \eqref{eqn:202303111} holds in $\mathrm{GL}_n(\mathbb C)$.
	
	Finally we take a morphism $Y\to X$ from a smooth quasi-projective variety $Y$.
	Then we have a natural isomorphism $\widehat{\pi_1(Y)}=\widehat{\pi_1(Y^{\sigma})}$ which commutes with \eqref{eqn:202303112}:
\begin{equation*}
\begin{tikzcd}
	 	\widehat{\pi_1(Y^{\sigma})} \arrow[r,equal] \arrow[d]& \widehat{\pi_1(Y)} \arrow[d]\\ 
	 	\widehat{\pi_1(X^{\sigma})}\arrow[r,equal]  & \widehat{\pi_1(X)}
\end{tikzcd}  
\end{equation*}
	Since the image of $\mathrm{Im}[\pi_1(Y)\to\pi_1(X)]\to \mathrm{Im}[\widehat{\pi_1(Y)} \to   \widehat{\pi_1(X)}]$ is dense, a similar argument for the proof of \eqref{eqn:202303133} yields
	$$
	\overline{\rho(\mathrm{Im}[\pi_1(Y)\to\pi_1(X)])}^{\mathrm{Zar}}=
	\overline{\widehat{\rho}(\mathrm{Im}[\widehat{\pi_1(Y)}\to\widehat{\pi_1(X)}])}^{\mathrm{Zar}}\subset \mathrm{GL}_n(\overline{\mathbb Q_p}).
	$$  
	Similarly we have
	$$
	\overline{\tau(\mathrm{Im}[\pi_1(Y^{\sigma})\to\pi_1(X^{\sigma})])}^{\mathrm{Zar}}=
	\overline{\widehat{\rho}(\mathrm{Im}[\widehat{\pi_1(Y)}\to\widehat{\pi_1(X)}])}^{\mathrm{Zar}}\subset \mathrm{GL}_n(\overline{\mathbb Q_p}).
	$$
	Hence we get \eqref{eqn:202303115} in $\mathrm{GL}_n(\overline{\mathbb Q_p})$, thus in $\mathrm{GL}_n(\mathbb C)$.      
	
	Note that $	\overline{\tau(\mathrm{Im}[\pi_1(Y^{\sigma})\to\pi_1(X^{\sigma})])}^{\mathrm{Zar}}$ is positive dimensional if and only if $\tau(\mathrm{Im}[\pi_1(Y^{\sigma})\to\pi_1(X^{\sigma})])$ is infinite. This concludes the last claim of the theorem.  
\end{proof} 

 \subsection{Proof of \Cref{main2}}
 Now we are able to prove the first main result  of this paper. 
\begin{thm}[=Theorem \ref{main2}]\label{thm:20220819}
Let $X$ be a  complex quasi-projective normal variety and $G$ be a semi-simple linear algebraic group over $\bC$. Suppose that $\varrho: \pi_1(X) \to G(\bC)$ is a Zariski dense and big representation. For any Galois conjugate variety $X^\sigma$ of $X$ under $\sigma \in \mathrm{Aut}(\bC/\bQ)$, 
\begin{thmlist}
	\item \label{item:log general type} there exists a proper Zariski closed subset $Z\subsetneqq X^\sigma$ such that any closed subvariety $V$ of $X^\sigma$ not contained in $Z$ is of log general type. 
	\item \label{item:pseudo Picard}Furthermore, $X^\sigma$ is pseudo Picard hyperbolic.
\end{thmlist}
   \end{thm}
\begin{proof} 
Using \cref{lem:fun}, we can replace $X$ by a desingularization, and thus we may assume that $X$ is smooth.  We first prove the theorem for $X$ itself.   We prove in two steps.

\medskip

\noindent{\it Step 1. We assume that $G$ is almost simple.}   Note that $G$  is   isogenous to exactly one of the following:  $A_n, B_n, C_n, D_n, E_6, E_7, E_8, F_4, G_2$, and is defined over some number field $L'$. Moreover, it is absolutely almost simple.  
We prove the theorem in two cases.

\noindent{\em Case 1: $\varrho$ is rigid.}\quad 	 
Since $\varrho$ is  $G$-rigid,   it is conjugate to some Zariski dense representation $\tau:\pi_1(X)\to G(L)$ where $L$ is a finite extension of $L'$. Moreover, for every embedding $v:L\to \bC$, the representation $v\tau:\pi_1(X)\to G(\bC)$ is rigid, \emph{i.e.} $[v\tau]$ is an isolated point in the complex affine variety $M_B(\pi_1(X),G)_{\bC}$.   

\noindent {\em Case 1.1:} Assume that for some non-archimedean place $v$ of $L$, the associated representation $\tau_v:\pi_1(X)\to G(L_v)$ is unbounded, where $L_v$ denotes the completion of $L$ with respect to $v$. Note that $\tau_v$ is still Zariski dense and big, and $G_{L_{v}}$ is an absolutely simple algebraic group over $L_{v}$.   We apply \cref{thm:log general type,thm:main33} to conclude \cref{thm:20220819}. 

\noindent {\em Case 1.2:} Fix a    faithful representation $G\to {\rm GL}_N$. Assume that  for  every non-archimedean place $v$ of $L$, the associated representation $\tau_v:\pi_1(X)\to G(L_v)$ is  bounded.  Then there is a factorisation $\tau:\pi_1(X)\to {\rm GL}_N(\cO_L)$, where $\cO_L$ is the ring of integers.    For every embedding $\nu:L\to \bC$, since   $\nu\tau:\pi_1(X)\to G(\bC)$ is rigid, by virtue of \cref{thm:vhs},  $\nu\tau:\pi_1(X)\to {\rm GL}_N(\bC)$ underlies a $\bC$-VHS.  We apply \cite[Proposition 7.1 \& Lemma 7.2]{LS18} to conclude that $\tau$ is  a complex direct factor of a $\bZ$-variation of Hodge structures $\tau''$. Since $\tau$ is big,  so is $\tau''$. Hence the period mapping $p:X\to \mathcal{D}/\Gamma$ of this  $\bZ$-VHS satisfies $\dim X=\dim p(X)$, where $\Gamma$ is the monodromy group.      Let $Z\subset X$ be a proper Zariski closed subset of $X$ so that $p|_{X-Z}$ is finite. 
Then by \cref{thm:PicardVHS}, we obtain the   pseudo Picard hyperbolicity of $X$.  It follows from \cite{BC20} (see also \cite{CD21}) that any closed   subvariety $V$ of $X$ not contained in $Z$ is  of log general type.

\smallskip

\noindent 
{\em Case 2: $\varrho$ is non-rigid.} \quad By \cref{lem:20220819}, it follows that there is a Zariski dense, big and unbounded repsentation $\pi_1(X)\to G(K)$ where $K$ is some non-archimedean local field of zero characteristic. 
 By \cref{thm:main33,thm:log general type} again,  we conclude the proof.
 
 Thus we have proved the theorem for $X$  when $G$ is almost simple. 

\medskip

\noindent {\it Step 2. 
We treat the general case that $G$ is semi-simple.}
 Note that $G$ has only finitely many minimal normal subgroup varieties $H_1, \ldots, H_k$, and   the multiplication map
$$
H_1 \times \cdots \times H_k \rightarrow G
$$
is an isogeny. Each $H_i$ is almost-simple and it is centralized by the remaining ones. Then the product $H_1\cdots H_{i-1}H_{i+1}\cdots H_k$ is also a normal subgroup of $G$.  Define $G_i:=G/(H_1\cdots H_{i-1}H_{i+1}\cdots H_k)$. It follows that the natural map $G\to G_1\times\cdots \times G_k$ is an isogeny.   This enables us to replace $G$ by $ G_1\times\cdots \times G_k$,   and the representation $ \pi_1(X)\to G_1(\bC)\times\cdots \times G_k(\bC)$ induced by $\varrho$ is also Zariski dense and big. We use the same letter $\varrho$ to denote this representation.   Take the projection $\varrho_i:\pi_1(X)\to G_i(\mathbb C)$.
Then $\varrho_i$ is Zariski dense for all $i$.

We apply \cref{lem:kollar}.
Then after passing to an \'etale cover of $X$, there exists a rational map $p_i:X\dashrightarrow Y_i$ and a big representation $\tau_i:\pi_1(Y_i)\to G_i(\mathbb C)$ such that $p_i^*\tau_i=\varrho_i$. 
Then $\tau_i$ is big and Zariski dense.
So we may apply Step 1 above to get the proper Zariski closed set $Z_i\subsetneqq Y_i$ such that \cref{item:pseudo Picard} holds for $Y_i$ with $Z_i$ the exceptional set. 
Let $q:X\dashrightarrow Y_1\times \cdots\times Y_k$ be the natural map and let $\alpha:X\dashrightarrow S$ be the quasi-Stein factorisation of $q$ (cf. \cref{lem:Stein}).
Then $\alpha$ is birational.
Indeed let $F\subset X$ be a general fiber of $\alpha$.
To show $\dim F=0$, we assume contrary $\dim F>0$.
Since the induced map $F\to Y_i$ is constant, we have $\varrho_i(\mathrm{Im}[\pi_1(F)\to \pi_1(X)])=\{ 1\}$.
Hence $\varrho(\mathrm{Im}[\pi_1(F)\to \pi_1(X)])=\{ 1\}$.
This contradicts to the assumption that $\varrho$ is big.
Hence $\dim F=0$, so $\alpha$ is birational. 

Now we set $Z=\mathrm{Ex}(\alpha)\cup (\cup_ip_i^{-1}(Z_i))$.
Let $f:\bD^*\to X$ be holomorphic such that the image is not contained in $Z$.
Then by Step 1 above, the map $p_i\circ f:\bD^*\to Y_i$ does not have essential singularity at $0\in\bD$, hence the same holds for $\alpha\circ f:\bD^*\to S$.
This proves \cref{item:pseudo Picard} for $X$.

\medspace

Let us prove \cref{item:log general type}.  
We will prove the result by induction on $k$. The case $k=1$ is proved by Step 1.  Assume now the statement is true for $k-1$. For the representation $\varrho':\pi_1(X)\to G_2(\bC)\times\cdots\times G_k(\bC)$ defined by the composition of $\varrho$ and the quotient $G(\bC)\to G_2(\bC)\times\cdots\times G_k(\bC)$, by \cref{lem:kollar},  after we replace $X$ by a finite \'etale cover and a birational proper morphism, there is a dominant morphism  $p:X\to Y$   (resp. $p_1:X\to Y_1$) with connected general fibers and  a big and Zariski dense representation $\tau:\pi_1(Y)\to G_2(\bC)\times\cdots\times G_k(\bC)$ (resp. $\tau_1:\pi_1(Y_1)\to G_1(\bC)$) such that $p^*\tau=\varrho$ (resp. $p_1^*\tau_1=\varrho_1$). By the induction, there is a  proper Zariski closed set  $Z_0\subsetneqq Y$  (resp. $Z_1\subsetneqq Y_1$) such that any closed   subvariety $V\not\subset Z_0$ (resp. $V_1\not\subset  Z_1$) is of log general type.  As we have seen above, the natural morphism
\begin{align*}
	 q: X&\to Y_1\times Y\\
	 x&\mapsto (p_1(x), p(x))
\end{align*} 
satisfies $\dim X=\dim q(X)$.   Let $\alpha:X\dashrightarrow S$ be the quasi-Stein factorisation of $q$. 
Then $\alpha$ is birational.
  Set $Z:=p^{-1}(Z_0)\cup p_1^{-1}(Z_1)\cup {\rm Exc}(\alpha)$. Let $V\subset X$ be  any closed   subvariety  not contained in $Z$.  Then the  closure   $\overline{p(V)}$ 
 is of log general type.  Since $p_1(V)\not\subset Z_1$,      $\overline{p_1(F)}$  is not contained in $Z_1$ for general fibers $F$ of $p|_{V}:V\to \overline{p(V)}$. Hence $\overline{p_1(F)}$ is of log general type. Note that    $\alpha|_F:F\to \overline{\alpha(F)}$ of $p$ is  generically finite. Hence $p_1|_F:F\to \overline{p_1(F)}$  is also generically finite. It follows that $F$ is also of log general type.   
We use Fujino's addition formula for logarithmic Kodaira dimensions \cite[Theorem 1.9]{Fuj17} to show that $V$ is of log general type, which proves \cref{item:log general type} for $X$. Therefore, we have established the theorem for $X$.

Let $\sigma\in {\rm Aut}(\bC/\bQ)$. Using \cref{thm:conjugate}, we can construct a representation $\varrho^\sigma:\pi_1(X^\sigma)\to G(\bC)$ that is also Zariski dense and big, satisfying the conditions of the theorem. Hence, we have proven the theorem for $X^\sigma$.  
\end{proof}

\begin{rem}\label{rem:sharp}
Note that the condition in \cref{thm:20220819} is sharp. For example, consider an abelian variety $X$ of dimension $n$. The representation 
\begin{align*}
	\bZ^{2n}\simeq \pi_1(X)&\to (\bC^*)^{2n}\\
	(a_1,\ldots,a_{2n})&\mapsto (\exp({a_1}),\ldots, \exp(a_{2n}))
	\end{align*}
 is a Zariski dense representation. However, $X$ is not of general type and contains Zariski dense entire curves. This example demonstrates that the semisimplicity of $G$ is necessary for \cref{thm:20220819} to hold.
	
Another example to consider is a curve $C$ of genus at least 2. There exists a Zariski dense representation $\varrho:\pi_1(C)\to G(\bC)$ where $G$ is some semisimple algebraic group over $\bC$. The representation $\varrho:\pi_1(C\times \bP^1)\to G(\bC)$ is Zariski dense but not big. It is clear that $C\times \bP^1$ is neither pseudo Brody hyperbolic nor of general type. Thus, the bigness condition in \cref{thm:20220819} is also essential.
\end{rem}

\begin{rem}\label{rem:previous}
Let us conclude this section with some remarks on the proof of \cref{main2}. When $X$ is projective, proving the bigness of $K_X$ in \cite{CCE15} is a rather challenging and intricate task. It relies on Eyssidieux's celebrated work \cite{Eys04} on the precise structure of the Shafarevich morphism for an absolute constructible subset in $M_B(\pi_1(X),G)$. Another crucial ingredient in their proof is the following deep result proven in \cite[Proposition 5.4.6]{Eys04} (see also the paragraph before \cite[Lemme 6.1]{CCE15}):
\begin{enumerate}[label=($\spadesuit$),leftmargin=0.6cm]
	\item \label{Eys} There exists a finite family of reductive representations into non-Archimedean local fields lying in the absolutely constructible subsets of the character varieties, such that the generic rank of spectral one-forms ${\eta_1,\ldots,\eta_m}$ induced by the harmonic mappings associated with all these representations is equal to the dimension of $\alpha(\xsp)$, where $\alpha:\xsp\to\cA$ is the partial-Albanese morphism defined in \cref{def:partial}.
\end{enumerate}
It should be noted that a similar statement was claimed by Zuo  in \cite[p. 149]{Zuo96} for a \emph{single} Zariski dense and unbounded representation into almost algebraic group defined over a non-archimedean local field: \guillemotleft \emph{...Therefore, the complex differential of $u \pi$ gives rise to non-zero holomorphic 1-forms $\omega_1, \ldots, \omega_l$ on $X^s$. Since $\varrho$ is big, by the above factorization theorem the dimension of   ${\rm Span}\left\langle\omega_1, \ldots, \omega_l\right\rangle$ in $\Omega_{X^s}^1$ at the generic point in $X^s$ is maximal, and is equal to $\operatorname{dim} X=: n$.}\guillemotright. It is noteworthy that this claim remains unproven to date. The proof of \Cref{Eys} in \cite{Eys04} uses a more elaborate version of Simpson's Lefschetz theorem for leaves defined by one-forms (see \cite[Proposition 5.1.7]{Eys04} and \cite{Sim93}). It was used in \cite{CCE15} to prove that $\kappa(\xsp)=\kappa(X)$ when $\kappa(X)=0$.
	
Zuo also claimed in \cite{Zuo96} that when $X$ in \cref{main2} is projective, $K_X$ is big. However, besides the incompleteness of the aforementioned claim, another weak point in his paper is the case when $\varrho$ is non-rigid. In \cite[p. 152]{Zuo96}, Zuo claimed that \guillemotleft \emph{a non-rigid representation $\varrho$ can be associated to a representation $\varrho_T$ over a function field, such that $\left.\varrho_T\right|_{t_0}=\varrho$ and $\varrho_T$ is unbounded with respect to a discrete valuation. Applying Theorem B to those unbounded representations, we prove Theorem 1}\guillemotright. 
A more precise construction of $\varrho_T$ is already discussed in \cref{rem:20230417}.
However, since the completion of the function field with respect to a discrete valuation is not locally compact as $T$ is defined over some infinite field, one cannot apply Gromov-Schoen's theory of harmonic mappings, and thus \cite[Theorem B]{Zuo96} cannot be applied for such unbounded representations. This particular flaw in the argument leads to the collapse of his proof, resulting unfortunately  several incorrect papers, such as \cite{Sun22}. In \cite{Eys04,CS08}, the authors apply the method of reduction mod $p$ to construct unbounded representations over some finite extension of $\mathbb{F}_p((t))$ such that Gromov-Schoen's theory works. But in this case, the new representations might not big, 
and the Zariski closure of their image might not be almost simple, as we already discussed in \cref{rem:20230417}.
Hence one still cannot prove the bigness of $K_X$ along Zuo's strategy.

Our proof of \cref{main:log general type} does not rely on either \cref{Eys} or the precise structure of Shafarevich morphisms of quasi-projective varieties (both of which are open problems in the quasi-projective setting). Instead, we use \cref{lem:20220819} and \cref{thm:vhs} to reduce the proof to either the case of $\bZ$-VHS or the case described in \cref{thm:log general type}. This approach has the advantage of avoiding the reduction mod $p$ method used in previous works like \cite{CS08,Eys04}.

Another novel aspect of our work is the proof of \cref{thm:log general type}, which relies on \cref{thm:main33} regarding the generalized Green-Griffiths-Lang conjecture, as well as the positivity of the spectral cover $\xsp$ established in \cref{thm:spectral cover}. By utilizing the techniques developed by Campana-P\u{a}un \cite{CP19}, we are able to spread the strong positivity of $\xsp$ to $X$. We believe that our approach to proving \cref{thm:log general type}  provides a new perspective on extending the strong positivity of certain covers to the original variety. 
\end{rem}

\section{On the generalized Green-Griffiths-Lang conjecture II}\label{sec:proof2}

In this section we prove \cref{main:GGL,main:special}.
We first prove the following lemma.

\begin{lem}\label{lem:202305101} 
Let $X$ be a quasi-projective normal variety and let $\varrho:\pi_1(X)\to {\rm GL}_N(\bC)$ be a reductive and big representation.  
Then there exist 
\begin{itemize}
\item a semi-abelian variety $A$,
\item a smooth quasi projective variety $Y$ satisfying $\Spp(Y)\subsetneqq Y$ and $\Spalg(Y)\subsetneqq Y$,
\item a birational modification $\widehat{X}'\to \widehat{X}$ of a finite \'etale cover $\widehat{X}\to X$,
\item a morphism $g:\widehat{X}'\to A\times Y$
\end{itemize}
such that $\dim g(\widehat{X}')=\dim \widehat{X}'$.
\end{lem}

\begin{proof}
Let $G$ be the Zariski closure of $\varrho$ which is reductive.  
Let $G_0$ be the connected component of $G$ which contains the identity element of $G$. 
Then after replacing $X$ by a finite \'etale cover   corresponding to the finite index subgroup $\rho^{-1}(\varrho(\pi_1(X))\cap G_0(\bC))$ of $\pi_1(X)$, we can assume that the Zariski closure $G$ of $\varrho$ is connected.  
Hence the radical $R(G)$ of $G$ is a torus, and the derived group $\cD G$ is semisimple or trivial.
Write $G_2:=G/\cD G$ and $G_1=G/R(G)$. 
Then   $G_1$ is either semisimple or trivial and $G_2$ is a torus.  
Moreover, the natural  morphism $G\to G_1\times G_2$ is an isogeny.  
Let $\varrho':\pi_1(X)\to G_1(\bC)\times G_2(\bC)$  be the composition of $\varrho$ and $G(\bC)\to G_1(\bC)\times G_2(\bC)$. 
Then it is also big and Zariski dense. 
Denote by  $\varrho_i:\pi_1(X)\to G_i(\bC)$  the  composition of $\varrho:\pi_1(X)\to G_1(\bC)\times G_2(\bC)$ and $ G_1(\bC)\times G_2(\bC)\to G_i(\bC)$, which is Zariski dense. 
After replacing $X$ by a finite \'etale cover, we may assume that $ H_1(X,\bZ)$ is torsion free.
Hence $\varrho_2$ factors the quasi-albanese map $a:X\to A$, i.e., there exists $\varrho_2':H_1(X,\bZ)\to G_2(\bC)$ such that $a^*\varrho_2'=\varrho_2$.

If $G_1$ is not trivial, we apply \cref{lem:kollar}.
Then after replacing $X$ by a finite \'etale cover $\nu:\widehat{X}\to X$ and a birational modification $\mu:\widehat{X}'\to X$, there exists a dominant morphism $p:\widehat{X}'\rightarrow Y$ with connected general fibers and a representation $\tau:\pi_1(Y)\to G_1(\mathbb C)$ such that $p^*\tau=(\nu\circ\mu)^*\varrho_1$ and
	$\tau$ is big and Zariski dense. 
By \cref{main2}, $\Spalg(Y)\subsetneqq Y$ and $\Spp(Y)\subsetneqq Y$.
If $G_1$ is trivial, then we set $Y$ to be a point and   $p:\widehat{X}'\to Y$ is the constant map.  
In this case, we also have $\Spalg(Y)\subsetneqq Y$ and $\Spp(Y)\subsetneqq Y$.

Consider the morphism $g:\widehat{X}'\to A\times Y$ defined by $x\mapsto (a\circ\nu\circ\mu(x),p(x))$.
Let $\beta:\widehat{X}'\to S$ be the quasi-Stein factorisation of $g$ defined in \cref{lem:Stein}.  
 Then $\beta$ is birational.
	Indeed let $F\subset \widehat{X}'$ be a general fiber of $\beta$.
We shall show $\dim F=0$.
	Since the induced map $F\to Y$ is constant, we have $(\nu\circ\mu)^*\varrho_1(\mathrm{Im}[\pi_1(F^{\mathrm{norm}})\to \pi_1(\widehat{X}')])=\{ 1\}$. 
Since the induced map $F\to A$ is constant, we have $(\nu\circ\mu)^*\varrho_2(\mathrm{Im}[\pi_1(F^{\mathrm{norm}})\to \pi_1(\widehat{X}')])=\{ 1\}$. 
Therefore, one has $(\nu\circ\mu)^*\varrho(\mathrm{Im}[\pi_1(F^{\mathrm{norm}})\to \pi_1(\widehat{X}')])=\{ 1\}$.
Since $(\nu\circ\mu)^*\varrho:\pi_1(\widehat{X}')\to G$ is big, we have $\dim F=0$.
Hence $\beta$ is birational.
Thus $\dim g(\widehat{X}')=\dim S=\dim \widehat{X}'$.
\end{proof}
Let us prove \cref{main:GGL}.
\begin{thm}[=\cref{main:GGL}]\label{thm:GGL}
	Let $X$ be a complex  smooth  quasi-projective  variety admitting a big and reductive representation  $\varrho:\pi_1(X)\to {\rm GL}_N(\bC)$. 
Then for any  automorphism $\sigma\in {\rm Aut}(\bC/\bQ)$, the following properties are equivalent:
	\begin{enumerate}[ font=\normalfont, label=(\alph*)] 
		\item $X^\sigma$ is of log general type. 
		\item   $\Spp(X^\sigma)\subsetneqq X^\sigma$.
\item  $\Sph(X^\sigma)\subsetneqq X^\sigma$.
\item  $\Spab(X^\sigma)\subsetneqq X^\sigma$.
\item  $\Spalg(X^\sigma)\subsetneqq X^\sigma$.
	\end{enumerate} 
\end{thm}

\begin{proof}  
	We use \cref{thm:conjugate} to show that it suffices to prove the theorem for $X$ itself. 
	We apply \cref{lem:202305101}.
Then by replacing $X$ with a finite \'etale cover and a birational modification, we obtain a smooth quasi-projective variety $Y$ (might be zero-dimensional), a semiabelian variety $A$, and a morphism $g:X\to A \times Y$ that satisfy the following properties:
	\begin{itemize}
		\item  $\dim X=\dim g(X)$.
		\item Let $p:X\to Y$ be  the composition of $g$ with the projective map $A\times Y\to Y$. 
		Then  $p$   is dominant.
		\item $\Spp(Y)\subsetneqq Y$ and $\Spalg(Y)\subsetneqq Y$.
	\end{itemize} 
Therefore, the conditions in \cref{cor:GGL} are satisfied.
This yields the two implications: $(a)\implies (b)$ and $(d)\implies (e)$. 
By \cref{lem:inclusion}, we have the implications: $(b)\implies (c)$ and $(c)\implies (d)$.
 Since the implication of $(e)\implies (a)$ is direct, we have proved the equivalence of $(a)$, $(b)$, $(c)$, $(d)$ and $(e)$.
\end{proof}

\begin{thm}[=\cref{main:special}]\label{thm:special}
	Let $X$ be a smooth quasi-projective   variety  and $\varrho:\pi_1(X)\to {\rm GL}_N(\bC)$ be a large and reductive representation. Then   for any  automorphism $\sigma\in {\rm Aut}(\bC/\bQ)$,  
	\begin{thmlist}
		\item \label{itemize:same} the four special subsets defined in \cref{def:special2} are the same, i.e., $$\Spalg(X^\sigma)=\Spab(X^\sigma)=\Sph(X^\sigma)=\Spp(X^\sigma).$$  
		\item These special subsets are conjugate under  automorphism   $\sigma$,  i.e., \begin{align}\label{eq:conjugate}
			 \Sp_{\bullet}(X^\sigma)=\Sp_{\bullet}(X)^\sigma,
		\end{align} 
		where $\Sp_{\bullet}$  denotes any of $\Spab$, $\Sph$, $\Spalg$ or $ \Spp$.
	\end{thmlist}
\end{thm}
\begin{proof}
\noindent {\em Step 1. }  Let $Y$ be a closed subvariety of $X$ that is not of log general type. Let $\iota:Z\to Y$ be a desingularization. 
Then $\iota^*\varrho$ is big and reductive. 
By \cref{thm:GGL}, we have $\Spab(Z)=Z$.
Hence $Y\subset \Spab(X)$, which implies $\Spalg(X)\subset \Spab(X)$.

\medspace

\noindent {\em Step 2. }   Let $f:\bD^*\to X$ be a holomorphic map with essential singularity at the origin. Let $Z$ be a desingularization of the Zariski closure of $f(\bD^*)$. 
Note that the natural morphism $\iota:Z\to X$ induces a big and reductive representation $\iota^*\varrho$.   
Since $\Spp(Z)=Z$,
\cref{thm:GGL} implies that $Z$ is not of log general type. 
Hence $\iota(Z)\subset \Spalg(X)$, which implies $\Spp(X)\subset \Spalg(X)$.
By  \cref{lem:inclusion} and Step 1, we conclude  that $\Spalg(X)=\Spab(X)=\Sph(X)=\Spp(X)$.

\medspace

\noindent {\em Step 3. }  	 We use \cref{thm:conjugate} to show that there is a large and reductive representation $\varrho^\sigma:\pi_1(X^\sigma)\to {\rm GL}_N(\bC)$. By Step 2, we conclude that $\Spalg(X^\sigma)=\Spab(X^\sigma)=\Sph(X^\sigma)=\Spp(X^\sigma)$.

\medspace

\noindent {\em Step 4. }   A quasi-projective variety $V$ is of log general type if and only if its conjugate $V^\sigma$ is of log general type.   It follows that $\Spalg(X^\sigma)=\Spalg(X)^\sigma.$ By Step 3 we conclude \eqref{eq:conjugate}. We complete the proof of the theorem. 
\end{proof}
 We provide a class of quasi-projective varieties whose fundamental groups have reductive and large representations.  
\begin{proposition}\label{prop:large}
	Let $X$ be  a normal quasi-projective variety.  If $a:X\to A$ is a morphism to a semiabelian variety $A$ that satsifies $\dim a(X)=\dim X$, then there exists a big and reductive representation $\varrho:\pi_1(X)\to \mathrm{GL}_N(\mathbb C)$. Moreover, if $a$ is quasi-finite, then $\varrho$ is furthermore large.  
\end{proposition}  
\begin{proof} 
We start from the following two claims.

\begin{claim}\label{claim:abelian large}
	If $A_0$ is an abelian variety, then $\pi_1(A_0)$ is large.  
\end{claim}
\begin{proof}[Proof of \cref{claim:abelian large}]
	Let $Z$ be any closed positive-dimensional subvariety of $A_0$. Suppose	 that  ${\rm Im}[\pi_1(Z^{\rm norm})\to \pi_1(A_0)]$ is finite. Then there is a finite \'etale cover $Y\to Z^{\rm norm}$ such that ${\rm Im}[\pi_1(Y)\to \pi_1(A_0)]=\{1\}$. Therefore, the natural morphism $Y\to A_0$ lifts to a holomorphic map $Y\to \widetilde{A}_0$, where $\widetilde{A}_0$ is the universal covering of $A_0$ that is  isomorphic to $\bC^N$.    Therefore $Y\to \widetilde{A}_0$ is constant, a contradiction.
\end{proof} 
\begin{claim}\label{claim:large}
Let $Y$ be a closed subvariety of $X$. If $a(Y)$ is not a point, then ${\rm Im}[\pi_1(Y^{\rm norm})\to \pi_1(A)]$ is infinite.   
\end{claim} 
\begin{proof}[Proof of \cref{claim:large}] 
 Note that $A$ admits a short exact sequence
\begin{align}\label{eq:short}
 0\to (\bC^*)^k\to A\stackrel{\pi}{\to} A_0\to 0.
\end{align}
 Let $Z$ be the closure of $\pi\circ a(Y)$.  If $Z$ is positive-dimensional, then ${\rm Im}[\pi_1(Z^{\rm norm})\to \pi_1(A_0)]$  is infinite by \cref{claim:abelian large}. It follows from    \cref{lem:finiteindex} that  ${\rm Im}[\pi_1(Y^{\rm norm})\to \pi_1(A)]$  is also infinite and the claim is proved.  
 
Assume now $\pi\circ a(Y)$ is a point.   Let $W$ be the closure of $a(Y)$. By  \eqref{eq:short}, $W$ is contained in $(\bC^*)^k$.   Since $W$ is assumed to be positive-dimensional, then it must dominate  some factor $\bC^*$ of   $(\bC^*)^k$.  By \cref{lem:finiteindex} again, ${\rm Im}[\pi_1(W^{\rm norm})\to \pi_1((\bC^*)^k)]$ is infinite.  Since	 we have the following short exact sequence:
 $$
 0=\pi_2(A_0)\to \pi_1((\bC^*)^k)\to \pi_1(A),
 $$
it follows that ${\rm Im}[\pi_1(W^{\rm norm})\to \pi_1(A)]$ is infinite.  Applying \cref{lem:finiteindex} once again, we conclude that ${\rm Im}[\pi_1(Y^{\rm norm})\to \pi_1(A)]$ is infinite. The claim is proved. 
\end{proof}
Since $\dim (X)=\dim a(X)$, there exists a proper Zariski closed subset  $\Xi\subsetneqq X$ such that $a|_{X\backslash \Xi}:X\backslash \Xi\to A$ is quasi-finite. Note that $\Xi=\varnothing$ if $a$ is quasi-finite.  Therefore, for any positive-dimensional closed subvariety $Y$ with $Y\not\subset  \Xi$,  we have $\dim(a(Y))>0$. According to Claim \ref{claim:large}, we can conclude that ${\rm Im}[\pi_1(Y^{\rm norm})\to \pi_1(A)]$ is infinite. 

Since $\pi_1(A)$ is abelian, we can embed it faithfully into $(\bC^*)^N\hookrightarrow {\rm GL}_N(\bC)$, where the later is the diagonal embedding.  Thus, the composition $\varrho:\pi_1(X)\to {\rm GL}_N(\bC)$ with $a_*:\pi_1(X)\to \pi_1(A)$   is a big representation. Furthermore, when $a$ is quasi-finite,  $\varrho$ is large. As  $\varrho(\pi_1(X))$ is contained in $(\bC^*)^N$, its Zariski closure is   a torus in ${\rm GL}_N(\bC)$. Hence, $\varrho$ is reductive. This completes the proof of the proposition.
\end{proof} 
  Note that \cref{cor:20221102} is incorporated into   \cref{thm:GGL} through  \cref{prop:large}.

\section{Some properties of $h$-special varieties}\label{sec:20230406}
One can easily see that if $X$ is weakly special, then any finite \'etale cover $X'$ of $X$ is also weakly special.  
It is proven by Campana that this result also holds for special varieties. 
\begin{thm}[Campana]\label{special}
If $X$ is  special quasi-projective manifold, then any finite \'etale cover $X'$ of $X$ is  special. Moreover, $X$ is weakly special. 
\end{thm} 

For $h$-special varieties, we have the following.

\begin{lem}\label{lem:202304061}
Let $X$ be an $h$-special quasi-projective variety, and let $p:X'\to X$ be a finite \'etale morphism from a quasi-projective variety $X'$.
Then $X'$ is $h$-special.
\end{lem}

\begin{proof}
Let $\sim'$ be the equivalence relation on $X'$ and $R'\subset X'\times X'$ be the set defined by $R'=\{(x',y')\in X'\times X'; x'\sim'y'\}$.
Let $q:X'\times X'\to X\times X$ be the induced map.

We shall show $R\subset q(R')$.
So let $(x,y)\in R$.
Then there exists a sequence $f_1,\ldots,f_l:\mathbb C\to X$ such that 
$$x\in Z_1, Z_1\cap Z_2\not=\emptyset,\ldots,Z_{l-1}\cap Z_l\not=\emptyset, y\in Z_l,$$
where $Z_i\subset X$ is the Zariski closure of $f_i(\mathbb C)\subset X$.
Let $Z_i^{\mathrm{norm}}\to Z_i$ be the normalization.
Then we may take a connected component $W_i$ of $Z_i^{\mathrm{norm}}\times_XX'$ such that 
\begin{equation}\label{eqn:202304061}
\varphi_1(W_1)\cap \varphi_{2}(W_{2})\not=\emptyset,\ldots, \varphi_{l-1}(W_{l-1})\cap \varphi_{l}(W_{l})\not=\emptyset,
\end{equation}
where $\varphi_i:W_i\to X'$ is the natural map.
Note that the induced map $W_i\to Z_i^{\mathrm{norm}}$ is \'etale.
Hence $W_i$ is normal.
Since $W_i$ is connected, $W_i$ is irreducible (cf. \cite[\href{https://stacks.math.columbia.edu/tag/0357}{Tag 0357}]{stacks-project}). 
Since $f_i:\mathbb C\to Z_i$ is Zariski dense, we have a lift $f_i':\mathbb C\to Z_i^{\mathrm{norm}}$.
Since $W_i\to Z_i^{\mathrm{norm}}$ is finite \'etale, we may take a lift $g_i:\mathbb C\to W_i$, which is Zariki dense.
Then $\varphi_i\circ g_i:\mathbb C\to X'$ has Zariski dense image in $\varphi_i(W_i)\subset X'$.

We take $a\in W_1$ and $b\in W_l$ such that $p\circ\varphi_1(a)=x$ and $p\circ\varphi_l(b)=y$.
Then $q((\varphi_1(a),\varphi_l(b)))=(x,y)$.
We have $\varphi_1(a)\in\varphi_1(W_1)$ and $\varphi_l(b)\in\varphi_l(W_l)$.
Hence by \eqref{eqn:202304061}, we have $(\varphi_1(a),\varphi_l(b))\in R'$.
Thus $R\subset q(R')$.

Now let $\overline{R'}\subset X'\times X'$ be the Zariski closure.
To show $\overline{R'}= X'\times X'$, we assume contrary that $\overline{R'}\not= X'\times X'$.
Then $\mathrm{dim}\overline{R'}<2\dim X'$.
This contradicts to $X\times X=\overline{R}\subset q(\overline{R'})$.
Hence $\overline{R'}= X'\times X'$.
\end{proof}

\begin{lem}\label{lem:202304063}
Let $X$ be an $h$-special quasi-projective variety.
Let $S$ be a quasi-projective variety and let $p:X\to S$ be a dominant morphism.
Then $S$ is $h$-special.
\end{lem}

\begin{proof}
Let $\sim_S$ be the equivalence relation on $S$ and $R_S\subset S\times S$ be the set defined by $R_S=\{(x',y')\in S\times S; x'\sim_Sy'\}$.
Let $q:X\times X\to S\times S$ be the induced map.
Then $q$ is dominant.

We shall show $q(R)\subset R_S$.
Indeed let $q((x,y))\in q(R)$, where $(x,y)\in R$.
Then there exists a sequence $f_1,\ldots,f_l:\mathbb C\to X$ such that 
$$x\in Z_1, Z_1\cap Z_2\not=\emptyset,\ldots,Z_{l-1}\cap Z_l\not=\emptyset, y\in Z_l,$$
where $Z_i\subset X$ is the Zariski closure of $f_i(\mathbb C)\subset X$.
Then the Zariski closure of $p\circ f_i:\mathbb C\to S$ is $\overline{p(Z_i)}$.
We have
$$
p(x)\in \overline{p(Z_1)}, \overline{p(Z_1)}\cap \overline{p(Z_2)}\not=\emptyset,\ldots,\overline{p(Z_{l-1})}\cap \overline{p(Z_l)}\not=\emptyset, p(y)\in \overline{p(Z_l)}.
$$
Hence $(p(x),p(y))\in R_S$.
Thus $q(R)\subset R_S$.
Hence $q(\overline{R})\subset \overline{R_S}$.
By $\overline{R}=X\times X$, we have $\overline{R_S}=S\times S$, for $q$ is dominant.
\end{proof}

\begin{lem}\label{lem:20230407}
Let $X$ be a smooth, $h$-special quasi-projective variety and let $p:X'\to X$ be a proper birational morphism from a quasi-projective variety $X'$.
Then $X'$ is $h$-special.
\end{lem}

\begin{rem}
We can not drop the smoothness assumption for $X$.
See \Cref{ex:20230418} below.
\end{rem}

\begin{proof}[Proof of \cref{lem:20230407}]
\noindent {\em Step 1. } 
In this step, we assume that $p:X'\to X$ is a blow-up along a smooth subvariety $C\subset X$.
We first prove that every entire curve $f:\mathbb C\to X$ has a lift $f':\mathbb C\to X'$, i.e., $p\circ f'=f$.
The case $f(\mathbb C)\not\subset C$ is well-known, so we assume that $f(\mathbb C)\subset C$.
There exists a vector bundle $E\to C$ so that its projectivization $P(E)\to C$ is isomorphic to $p^{-1}(C)\to C$.
The pull-back $f^*E\to \mathbb C$ is isomorphic to the trivial line bundle over $\mathbb C$.
Thus we may take a non-zero section of $f^*E\to \mathbb C$.
This yields a holomorphic map $f':\mathbb C\to P(E)$.
Hence $f$ has a lift $f':\mathbb C\to X'$.

Now let $\sim'$ be the equivalence relation on $X'$ and $R'\subset X'\times X'$ be the set defined by $R'=\{(x',y')\in X'\times X'; x'\sim'y'\}$.
Let $q:X'\times X'\to X\times X$ be the induced map.

To show $R\subset q(R')$, we take $(x,y)\in R$.
Then there exists a sequence $f_1,\ldots,f_l:\mathbb C\to X$ such that 
$$x\in Z_1, Z_1\cap Z_2\not=\emptyset,\ldots,Z_{l-1}\cap Z_l\not=\emptyset, y\in Z_l,$$
where $Z_i\subset X$ is the Zariski closure of $f_i(\mathbb C)\subset X$.
For each $f_i:\mathbb C\to X$, we take a lift $f_i':\mathbb C\to X'$.
Let $Z_i'\subset X'$ be the Zariski closure of $f_i'(\mathbb C)\subset X'$.
Then the induced map $Z_i'\to Z_i$ is proper surjective.

For each $i=1,2,\ldots,l-1$, we take $z_i\in Z_i\cap Z_{i+1}$.
We define a holomorphic map $\varphi_i:\mathbb C\to X'$ as follows.
If $z_i\not\in C$, then $p^{-1}(z_i)$ consists of a single point $z_i'\in X'$ and thus $z_i'\in Z_i'\cap Z_{i+1}'$.
In particular $Z_i'\cap Z_{i+1}'\not=\emptyset$.
In this case, we define $\varphi_i$ to be the constant map such that $\varphi_i(\mathbb C)=\{z_i'\}$.
If $z_i\in C$, we have $p^{-1}(z_i)=\mathbb P^d$, where $d=\mathrm{codim}(C,X)-1$.
We have $Z_i'\cap p^{-1}(z_i)\not=\emptyset$ and $Z_{i+1}'\cap p^{-1}(z_i)\not=\emptyset$.
In this case, we take $\varphi_i:\mathbb C\to p^{-1}(z_i)$ so that the image is Zariski dense in $p^{-1}(z_i)$.
We define $W_i\subset X'$ to be the Zariski closure of $\varphi_i(\mathbb C)\subset X'$.
Then we have
$$
Z_1'\cap W_1\not=\emptyset, W_1\cap Z_2'\not=\emptyset, Z_2'\cap W_2\not=\emptyset, \ldots, Z_{l-1}'\cap W_{l-1}\not=\emptyset, W_{l-1}\cap Z_l'\not=\emptyset.
$$
We take $a\in Z_1'$ and $b\in Z_l'$ such that $p(a)=x$ and $p(b)=y$.
Then $q((a,b))=(x,y)$ and $(a,b)\in R'$.
Thus $R\subset q(R')$.
This induces $\overline{R'}= X'\times X'$ as in the proof of \cref{lem:202304061}.

\medskip

\noindent {\em Step 2. } 
We consider the general proper birational morphism $p:X'\to X$.  
 Then we can apply a theorem of Hironaka 
 (cf. \cite[Corollary 3.18]{Kol07}) 
to $p^{-1}:X\dashrightarrow X'$ to conclude 
there exists a sequence of blowing-ups
$$
X_k\overset{\psi_k}{\longrightarrow}X_{k-1}\overset{\psi_{k-1}}{\longrightarrow}X_{k-2}\longrightarrow\cdots\longrightarrow X_1\overset{\psi_1}{\longrightarrow}X_0=X$$
such that
\begin{itemize}
\item
each $\psi_i:X_{i}\to X_{i-1}$ is a blow-up along a smooth subvariety of $X_{i-1}$, and
\item
there exists a morphism $\pi:X_k\to X'$ such that $p\circ\pi=\psi_1\circ\psi_2\circ\cdots\circ\psi_k$.
\end{itemize}
Then by the step 1, each $X_i$ is $h$-special.
In particular, $X_k$ is $h$-special.
Thus by \cref{lem:202304063}, $X'$ is $h$-special.
\end{proof}

\begin{lem}\label{lem:202304065}
If a positive dimensional quasi-projective variety $X$ is pseudo Brody hyperbolic, then $X$ is not $h$-special.
\end{lem} 
 \begin{proof}
  we take a proper Zariski closed subset $E\subsetneqq X$ such that every non-constant holomorphic map $f:\mathbb C\to X$ satisfies $f(\mathbb C)\subset E$.
Let $x\in X$ satisfies $x\not\in E$.
Assume that $y\in X$ satsifies $x\sim y$.
Then there exists a sequence $f_1,\ldots,f_l:\mathbb C\to X$ such that 
$$x\in Z_1, Z_1\cap Z_2\not=\emptyset,\ldots,Z_{l-1}\cap Z_l\not=\emptyset, y\in Z_l,$$
where $Z_i\subset X$ is the Zariski closure of $f_i(\mathbb C)\subset X$.
Then $f_1$ is constant map and $Z_1=\{x\}$.
By $Z_1\cap Z_2\not=\emptyset$, we have $x\in Z_2$.
This yields that $f_2$ is constant and $Z_2=\{x\}$.
Similarly, we have $Z_3=\cdots=Z_l=\{ x\}$.
Hence $y=x$.
Thus we have $R\subset (X\times E)\cup (E\times X)\cup \Delta$, where $\Delta\subset X\times X$ is the diagonal.
Hence $R\subset X\times X$ is not Zariski dense. 
\end{proof}

\begin{cor}[$\supsetneqq$ \cref{main}]\label{cor:202304071}	 	
 	Let $X$ be a  complex   normal quasi-projective  
	variety and let $G$ be a semisimple algebraic group over $\bC$. 
 If  $\varrho:\pi_1(X)\to G(\bC)$ is a Zariski dense representation, then there  exist
 a  finite \'etale cover \(\nu:\widehat{X}\to X\), a birational and proper morphism \(\mu:\widehat{X}'\to \widehat{X}\), a dominant morphism $f:\widehat{X}'\to Y$  with connected general fibers, and a   big   representation \(\tau : \pi_{1}(Y) \to G(\bC)\)  such that
 \begin{itemize}
 	\item   \(f^{\ast} \tau = (\nu\circ\mu)^{\ast}\varrho\). 
 	\item There is a proper Zariski closed subset $Z\subsetneqq Y$ such that any closed   subvariety of $Y$ not contained in $Z$ is  of log general type.  
 	\item  $Y$ is pseudo Picard hyperbolic, and in particular pseudo Brody hyperbolic.  
 \end{itemize}  
Specifically,   $X$ is not weakly special and does not contain Zariski dense entire curves.  Furthermore, if $X$ is assumed to be smooth, then it cannot be $h$-special.
 \end{cor}
 \begin{proof} 
 	By \cref{lem:kollar}, there exists a commutative diagram of quasi-projective varieties
 	\[
 	\begin{tikzcd}
 		\widehat{X}' \arrow[r, "\mu"] \arrow[d, "f"] & \widehat{X} \arrow[r, "\nu"] & X \\
 		Y
 	\end{tikzcd}
 	\]
 	where \(\nu\) is finite \'etale, \(\mu\) is birational and proper, and a big representation \(\tau : \pi_{1}(Y) \to G(\bC)\) such that \(f^{\ast} \tau = (\nu\circ\mu)^{\ast}\varrho\). Moreover, $\widehat{X}'$ and $Y$ are smooth. 
 	Since $\varrho$ is Zariski dense, so is $(\nu \circ \mu)^*\varrho$ for $(\nu \circ \mu)_*(\pi_1(\widehat{X}'))$ is a finite index subgroup in $\pi_1(X)$. Thus, since the image $\tau(\pi_1(Y))$ includes that of $f^*\tau=(\nu\circ\mu)^*\varrho$, it follows that $\tau$ is also Zariski dense. By \cref{main2}, $Y$ is of log general type and pseudo-Picard hyperbolic, implying that $X$ is not weakly special. If there is a Zariski dense entire curve $\gamma:\bC\to X$, it can be lifted to $\gamma':\bC\to \widehat{X}'$, from which $f\circ\gamma':\bC\to Y$ would be a Zariski dense entire curve due to $f$ being dominant. However, this leads to a contradiction, thereby indicating that $X$ does not admit Zariski dense entire curves.
 	
 	Assuming $X$ is smooth, let us suppose that $X$ is $h$-special. 
By \cref{lem:202304061}, $\widehat{X}$ is $h$-special.
Moreover $\widehat{X}$ is smooth.
Hence by \cref{lem:20230407}, $\widehat{X}'$ is $h$-special.
Hence by \cref{lem:202304063}, $Y$ is $h$-special.
This contradicts to \cref{lem:202304065}.
Hence $X$ is not $h$-special.
 \end{proof}

\begin{example}\label{ex:20230418}
There exists a singular, normal projective surface $X$ such that
\begin{itemize}
\item $X$ is not weakly special,
\item $X$ does not contain Zariski dense entire curve,
\item $X$ is $h$-special and $H$-special,
\item
there exists a proper birational modification $X'\to X$ such that $X'$ is neither $h$-special nor $H$-special.
\end{itemize}

The construction is as follows.
Let $C\subset \mathbb P^2$ be a smooth projective curve of genus greater than one.
Then $C$ is of general type and hyperbolic.
Let $p\in \mathbb P^3$ be a point and $\varphi:\mathbb P^3\backslash\{p\}\to\mathbb P^2$ be the projection from the point $p\in\mathbb P^3$.
Namely for each $y\in\mathbb P^2$, $\varphi^{-1}(y)\cup\{p\}\subset \mathbb P^3$ is a projective line  $\mathbb P^1\subset \mathbb P^3$ passing through $p\in\mathbb P^3$.
We denote this line by $\ell_y\in\mathbb P^3$.
Set $X=\varphi^{-1}(C)\cup\{p\}$, which is the cone over $C$.
Then $X$ is projective and normal.

Let $X'=\mathrm{Bl}_pX$ be the blow-up of $X$ at $p\in X$.
Then we have a morphism $\varphi':X'\to C$.
Since $C$ is of general type, this shows that $X$ is not weakly special.
If $X$ contains Zariski dense entire curve $f:\mathbb C\to X$, it lifts as a Zariski dense entire curve $f':\mathbb C \to X'$.
Hence $\varphi'\circ f':\mathbb C\to C$ becomes a non-constant holomorphic map, which is a contradiction.
Hence $X$ does not contain Zariski dense entire curve.

Note that $X=\bigcup_{y\in C}\ell_y$.
Let $x,x'\in X$ be two points.
We take $y,y'\in C$ such that $x\in\ell_{y}$ and $x'\in\ell_{y'}$.
Then we have $d_X(x,p)=d_X(x',p)=0$, where $d_X$ is the Kobayashi pseudo-distance on $X$.
Hence $d_X(x,x')=0$.
This shows that $X$ is $H$-special.
There exist entire curves $f:\mathbb C\to\ell_y$ and $f':\mathbb C\to \ell_{y'}$.
By $\ell_y\cap\ell_{y'}=\{p\}$, we conclude that $X$ is $h$-special.

Now the existence of the morphism $\varphi':X'\to C$ shows that $X'$ is not $H$-special.
Since $C$ is not $h$-special, \cref{lem:202304063} shows that $X'$ is not $h$-special.
\end{example}

\section{Fundamental groups of special varieties}\label{sec:VN}
In this section we study fundamental groups of special varieties. As we will see in \Cref{example}, we construct a special and $h$-special quasi-projective  manifold whose fundamental group is linear nilpotent but not virtually abelian.  Hence  \cref{conj:Campana}  is modified as follows. 
\begin{conjecture}\label{conj:revised2}
	A   special or $h$-special quasi-projective manifold has \emph{virtually nilpotent} fundamental group.  
\end{conjecture}
In this section we  prove  the linear version of the above conjecture.  
\begin{thm}[=\Cref{main5}]\label{thm:VN}
	Let $X$ be a   special  or $h$-special quasi-projective manifold.   Let  $\varrho:\pi_1(X)\to {\rm GL}_N(\bC)$ be a  linear representation.   Then  $\varrho(\pi_1(X)) $ is {virtually nilpotent}.  If $\varrho$ is reductive,  then  $\varrho(\pi_1(X)) $ is  {virtually abelian}.
\end{thm} 
It is indeed based on the following  theorem. 
\begin{thm}[=\Cref{main:geomety group}]\label{thm:202210123}
Let $X$ be a special or $h$-special quasi-projective manifold.
Let $G$ be a connected, solvable algebraic group defined over $\mathbb C$.
Assume that there exists a Zariski dense representation $\varphi:\pi_1(X)\to G$.
Then $G$ is nilpotent.
In particular, $\varphi(\pi_1(X))$ is nilpotent.
\end{thm} 
The proof consists of the following three inputs: 
\begin{itemize} 
\item
Algebraic property of solvable algebraic groups (\cref{claim:202210121}).
\item
$\pi_1$-exactness of quasi-Albanese map for  special or $h$-special quasi-projective manifold (\cref{pro:202210131}).
\item
Deligne's unipotency theorem for monodromy action.
\end{itemize}
\begin{rem}
	Note that when $X$ is compact K\"ahler,  \cref{thm:202210123} is proved by Campana  \cite{Cam01} and Delzant \cite{Del10}.  Our proof of \cref{thm:202210123} is inspired by \cite[\S 4]{Cam01}.
\end{rem}
The structure of this section is organized as follows. 
In \cref{sec:qA} we prove a structure theorem for the quasi-Albanese morphism of  $h$-special or weakly special quasi-projective manifolds. \cref{s1,s2,s3,s4} are devoted to the proof of \cref{thm:202210123}. In \cref{s5} we prove \cref{thm:VN}. The last section is  on some examples of $h$-special complex manifolds.

\subsection{Structure of the quasi-Albanese morphism}\label{sec:qA}

\begin{lem}\label{prop:factor}
	Let $X$ be an $h$-special or weakly special quasi-projective manifold.  Then  the quasi-Albanese morphism $\alpha:X\to \cA$ is dominant with connected general fibers. 
\end{lem}
\begin{proof} 

We first assume that $X$ is $h$-special.
	Let $\beta:X\to Y$  and $g:Y\to \cA$ be the quasi-Stein factorisation of $\alpha$ in \cref{lem:Stein}. 
Then $Y$ is $h$-special (cf. \cref{lem:202304063}) and  $\overline{\kappa}(Y)\geq 0$. 
To show $\overline{\kappa}(Y)=0$, we assume contrary that $\overline{\kappa}(Y)>0$.  By a theorem of Kawamata \cite[Theorem 27]{Kaw81},  there are a semi-abelian variety $B\subset \cA$,   finite \'etale Galois covers $\widetilde{Y}\to Y$ and $\widetilde{B}\to B$, and a normal algebraic variety $Z$ such that
	\begin{itemize}
		\item there is a finite morphism from $Z$ to the quotient $\cA/B$;
		\item  $\widetilde{Y}$ is a fiber bundle over $Z$ with fibers $\widetilde{B}$;
		\item $\overline{\kappa}(Z)=\dim Z=\overline{\kappa}(Y)$.
	\end{itemize} 	  
By \cref{lem:202304061}, $\widetilde{Y}$ is $h$-special.
Hence by \cref{lem:202304063}, $Z$ is $h$-special.
Thus by \cref{lem:202304065}, $Z$ is not pseudo-Brody hyperbolic.
On the other hand, by \cref{cor:20221102}, $Z$ is pseudo-Brody hyperbolic, a contradiction. 
Hence $\overline{\kappa}(Y)=0$.

	Assume now $X$ is weakly special.  Consider a connected component $\widetilde{X}$ of $X\times_Y\widetilde{Y}$. Then $\widetilde{X}\to  X$ is finite \'etale.  The composed morphism $\widetilde{X}\to Z$ of $\widetilde{X}\to \widetilde{Y}$ and $\widetilde{Y}\to Z$ is dominant with connected general fibers. We obtain a contradiction since $Z$ is of log general type and $X$ is weakly special. Hence $\overline{\kappa}(Y)=0$. 
	
	By \cite[Theorem 26]{Kaw81}, $Y$ is a semi-abelian variety and according to the universal property of quasi-Albanese morphism, $Y=\cA$. Hence $\alpha$ is dominant with  connected general fibers.  
\end{proof}

\subsection{A nilpotency condition for solvable linear group} \label{s1}
All the ground fields for algebraic groups and linear spaces are $\mathbb C$ in this subsection.

\begin{lem}  \label{lem:trivial morphism}
		Let $T$ be an algebraic torus and let $U$ be a uniponent group.
		Then every morphism $f:T\to U$ of algebraic groups is constant.
	\end{lem}
\begin{proof}
		Since $U$ is unipotent, we have a sequence $\{e\}=U_0\subset U_1\subset \cdots\subset U_n=U$ of normal subgroups of $U$ such that $U_k/U_{k-1}=\mathbb G_a$, where $\mathbb G_a$ is the additive group.
		To show $f(T)\subset U_0=\{e\}$, suppose contrary $f(T)\not\subset U_0$.
		Then we may take the largest $k$ such that $f(T)\not\subset U_k$.
		Then $k<n$ and $f(T)\subset U_{k+1}$.
		Thus we get a non-trivial morphism $T\to U_{k+1}/U_k=\mathbb G_a$.
		By taking a suitable subgroup $\mathbb G_m\subset T$, we get a non-trivial morphism $g:\mathbb G_m\to \mathbb G_a$. 
		But this is impossible.
		Indeed let $\mu=\cup_n\mu_n$, where $\mu_n=\{ a\in \mathbb C^*;a^n=1\}$.
		Then $|\mu|=\infty$.
		On the other hand, for $a\in \mu_n$, we have $ng(a)=g(a^n)=g(1)=0$, hence $g(a)=0$.
		Thus $\mu\subset g^{-1}(0)$.
		This is impossible since we are assuming that $g$ is non-constant. 
		\end{proof}

\begin{lem}\label{lem:20221012}
Let $T$ be an algebraic torus.
Let $0\to L'\to L\to L''\to 0$ be an exact sequence of vector spaces with equivariant $T$-actions.
If $L'$ and $L''$ have trivial $T$-actions, then $L$ has also a trivial $T$-action.
\end{lem}

\begin{proof}
Let $\varphi:T\to \mathrm{GL}(L)$ be the induced morphism of algebraic groups.
We take a (non-canonical) splitting $L=L'\oplus L''$ of vector spaces.
Let $t\in T$.
For $(v',0)\in L'$, we have $(\varphi(t)-\mathrm{id}_L)v'=0$, for $T$ acts trivially on $L'\subset L$.
For $(0,v'')\in L''$, we have $(\varphi(t)-\mathrm{id}_L)v''\in L'$.
Hence for $(v',v'')\in L$ and $t\in T$, we have 
$$(\varphi(t)-\mathrm{id}_L)^2\cdot (v',v'')=(\varphi(t)-\mathrm{id}_L)\cdot (u,0)=0,$$ where $u=(\varphi(t)-\mathrm{id}_L)v''\in L'$.
Hence $\varphi:T\to \mathrm{GL}(L)$ factors through the unipotent group $U\subset \mathrm{GL}(L)$.
Since the map $T\to U$ is trivial (cf. \cref{lem:trivial morphism}), 
$L$ has a trivial $T$-action.
\end{proof}

Let $G$ be a connected, solvable linear group.
We have an exact sequence 
\begin{equation}\label{eqn:20230321}
1\to U\to G\to T\to 1,
\end{equation}
where $U=R_u(G)$ is the unipotent radical and $T\subset G$ is a maximal torus.
Then $T$ acts on $U/U'$ by the conjugate.
The following lemma is from \cite[Lemma 1.8]{AN99}. 

\begin{lem}\label{eqn:202210154}
If $T$ acts trivially on $U/U'$, then $G$ is nilpotent.
\end{lem}

\begin{proof}
By $T\subset G$, the conjugate yields a $T$-action on $U$.
We shall show that this $T$-action is trivial.
To show this, we set $N=\mathrm{Lie}(U)$.
Then the $T$-action yields $\alpha:T\to \mathrm{Aut}(N)$.
We set 
$$S=\{ \sigma\in \mathrm{Aut}(N); (\mathrm{id}_N-\sigma)N\subset [N,N]\}.$$
Then the elements of $S$ are unipotent.
Since $T$ acts trivially on $U/U'$, we have $\alpha(T)\subset S$.
Hence $\alpha$ is trivial (cf \cref{lem:trivial morphism}).
Hence $T$ acts trivially on $U$.
Thus the elements of $T$ and $U$ commute.
Thus $G=U\times T$.
Since $U$ is nilpotent, $G$ is nilpotent.
\end{proof}

Since $T$ is commutative, we have 
$G'\subset U$, where $G'=[G,G]$ is the commutator subgroup.
Hence we have
$$
1\to U/G'\to G/G'\to T\to 1. 
$$
Since $G/G'$ is commutative and $U/G'$ is unipotent, we have $G/G'=(U/G')\times T$. 
By
$$
1\to G'/G''\to G/G''\to G/G'\to 1,
$$
$G/G'$ acts on $G'/G''$ by the conjugate.
By $T\subset (U/G')\times T=G/G'$, we get $T$-action on $G'/G''$.

\begin{lem}\label{claim:202210121}
Assume $T$ acts trivially on $G'/G''$.
Then $G$ is nilpotent.
\end{lem} 

\begin{proof}
By $G'\subset U$, we have $G''\subset U'$, where $U'=[U,U]$.
Then we get
$$
1\to U'/G''\to G'/G''\to G'/U'\to 1.
$$
Hence $T$ acts trivially on $G'/U'$.
By $G/G'\simeq T\times (U/G')$, we note that $T$ acts trivially on $U/G'$.
Now we have the following exact sequence
$$
1\to G'/U'\to U/U'\to U/G'\to 1.
$$
This induces
$$
0\to \mathrm{Lie}(G'/U')\to \mathrm{Lie}(U/U')\to \mathrm{Lie}(U/G')\to 0.
$$
We know that both $T\to \mathrm{GL}(\mathrm{Lie}(G'/U'))$ and $T\to \mathrm{GL}(\mathrm{Lie}(U/G'))$ are trivial.
Hence by Lemma \ref{lem:20221012}, $T$ acts trivially on $\mathrm{Lie}(U/U')$.
Hence $T$ acts trivially on $U/U'$.
By \cref{eqn:202210154}, $G$ is nilpotent.
\end{proof}

\subsection{A lemma on Kodaira dimension of fibration}
 In this section, we prove an extension to the quasi-projective setting of \cite[Theorem 1.8]{Cam04}, and derive a useful consequence concerning algebraic fiber spaces of special varieties over semi-abelian varieties (see Corollary~\ref{cor:specialalg}).

We refer to  \Cref{sec:spechspecial} and to \cite{Cam11} for several definitions concerning orbifold bases. For our purposes, we only need to recall the following.
\medskip

Let \(f : X \to Y\) be a dominant morphism with general connected fibers between quasi-projective manifolds. Let $\Delta \subset Y$ be a prime divisor. We denote by $J(\Delta )$ the set of all prime divisors $D$ in $X$ such that $\overline{f(D)}=\Delta$. 
We write 
\begin{equation}\label{eqn:div}
f^*(\Delta )=\sum_{j\in J(\Delta )}m_jD_j+E_{\Delta}
\end{equation}
where $E_{\Delta}$ is an exceptional divisor of $f$, i.e., the codimension of $f(E_{\Delta})$ is greater than one.
We put 
$$m_{\Delta}=
\begin{cases}
\min_{j\in J(\Delta )}m_j & \text{if $J(\Delta)\not=\emptyset$} \\
+\infty & \text{if $J(\Delta)=\emptyset$}
\end{cases}
$$
Let $I(p)$ be the set of all $\Delta$ such that $m_{\Delta}\geq 2$. The {\em orbifold ramification divisor} of \(f\) on \(Y\) is $\Delta (f)=\sum_{\Delta \in I(p)} \left( 1 - \frac{1}{m_{\Delta}} \right)\Delta$. Consider now \(\overline{f} : \overline{X} \to \overline{Y}\), any extension of \(f\) between log-smooth compactifications, and let \(D := \overline{X} - X\) and \(G := \overline{Y} - Y\). Then the orbifold base of \(\bar{f}: (\overline{X}| D)\to \overline{Y}\)   is given by \(\Delta(\bar{f}, D) := G + \overline{\Delta(f)}\) (where \(\overline{\Delta(f)}\) is the \(\mathbb{Q}\)-divisor on \(\overline{Y}\) obtained from \(\Delta(f)\) by taking the closure of every component).

\begin{proposition}\label{prop:KDF}  Let $\{f_i:X_i\to Y_i\}_{i=1,2}$   be  dominant morphisms between smooth quasi-projective varieties with connected general fibers, and let \(\{\overline{f}_{i}:\overline{X}_{i}\to \overline{Y}_{i}\}_{i=1,2}\) be two extensions to log-smooth projective compactifications. We assume that we have two compatible commutative diagrams
	\begin{equation*}
		\begin{tikzcd}
			X_2 \arrow[r, "u_\circ"]\arrow[d, "f_2"] & X_1\arrow[d, "f_1"] \\
			Y_2 \arrow[r, "v_\circ"] & Y_1 
		\end{tikzcd}
		\hspace*{2em}
		\text{and}
		\hspace*{2em}
		\begin{tikzcd}
			\overline{X}_2 \arrow[r, "u"]\arrow[d, "\overline{f}_2"] & \overline{X}_1\arrow[d, "\overline{f}_1"] \\
			\overline{Y}_2 \arrow[r, "v"] & \overline{Y}_1 
		\end{tikzcd}
	\end{equation*}
	where $u_\circ$ and $v_\circ$ are  proper birational morphisms. Denoting by \(D_{i} := \overline{X}_{i} - X_{i}\), we let \((\overline{Y}_{i}| \Delta(\bar{f}_{i}, D_{i}))\) (resp. \((Y| \Delta(f_{i}))\) be the orbifold base of the fibration \(\overline{f}_{i} : (\overline{X}_{i}| D_{i}) \to \overline{Y}_{i}\) (resp. of the fibration \(f_{i} : X_{i} \to Y_{i}\)). Then, one has
	\begin{thmlist}
	\item the orbifold Kodaira dimension satisfy \(\kappa(\overline{Y}_{2}| \Delta(\bar{f}_{2}, D_{2})) \leq \kappa(\overline{Y}_{1}| \Delta(\bar{f}_{1}, D_{1}))\);  
\item If \(\bar{\kappa}(Y_{1}) \geq 0\) and $\bar{\kappa}(\overline{Y}_1-{\rm Supp}\, \Delta(\bar{f}_1,D_1))=\dim Y_1$, then   
 $K_{\overline{Y}_1}+\Delta(\bar{f}_1,D_1)$ and  $ K_{\overline{Y}_2}+\Delta(\bar{f}_2,D_2)$ are both big line bundles.  
 In particular,  $\kappa(Y_1,f_1)=\dim Y_1$, where $\kappa(Y_1,f_1)$ is the Kodaira dimension of   $f_1:X_1\to Y_1$ defined in  \eqref{eq:Kodaira2}.
	\end{thmlist}
\end{proposition}

\begin{proof}[Proof of Proposition~\ref{prop:KDF}]
Let us prove the first statement. According to \cite[Lemme 4.6]{Cam11},  
we have
\begin{align} \label{eq:Cam}
	 v^*\Delta(\bar{f}_{1}, D_{1})=\Delta(\bar{f}_{2}, D_{2})+E,
\end{align} 
where $E$ is some effective $\bQ$-divisor which is $v$-exceptional.  

Since \(\overline{Y}_{1}\) is smooth, hence has terminal singularities, we have therefore
\begin{equation} \label{eq:difforbcanonical}
K_{\overline{Y}_2}+\Delta(\bar{f}_{2}, D_{2})+E=v^*(K_{\overline{Y}_1}+\Delta(\bar{f}_{1}, D_{1}))+F,
\end{equation}
where  $F$ is also an effective  divisor.
Therefore, we have
\begin{align*}
	\kappa(K_{\overline{Y}_1}+\Delta(\bar{f}_{1}, D_{1})) & = \kappa(v^*(K_{\overline{Y}_1}+\Delta(\bar{f}_{1}, D_{1}))) =\kappa(v^*(K_{\overline{Y}_1}+\Delta(\bar{f}_{1}, D_{1}))+F)\\  &=\kappa(K_{\overline{Y}_2}+\Delta(\bar{f}_{2}, D_{2})+E)\geq \kappa(K_{\overline{Y}_2}+\Delta(\bar{f}_{2}, D_{2})).
\end{align*} 
The first claim follows.
\medskip

	We will prove the second claim.  We let \(G_{i} := \overline{Y}_{i} - Y_{i}\). By assumption, this is a simple normal crossing divisor, and we can write
\begin{align}\label{eq:canonical formula}
	 	K_{\overline{Y}_{2}} + G_{2} = v^{\ast}(K_{\overline{Y}_{1}} + G_{1}) +E_b
\end{align} 
	where \(E_b\) is a  \(v\)-exceptional effective divisor. We note that if $E_0$ is a prime \(v\)-exceptional   divisor that is not contained in $G_2$, we have $E_b\geq E_0$. We note that 
	$\Delta(\bar{f}_i,D_i)=G_i+\Delta_i$, where $\Delta_i$ is an effective divisor such that each of its irreducible component is not contained in $G_i$.      
Let $\Delta_1'$ be the strict transform of $\Delta_1$. It follows that  
 $
\Delta_2\geq \Delta_1'.
$  

We denote by $\Delta_1''$ the positive part of $v^*\Delta_1-G_2$. Then    
$|\Delta_1''|=|\Delta_1'|+F_2 
$ 
by \eqref{eq:Cam}, where $F_2$ is  a \(v\)-exceptional   effective divisor such that each of its irreducible component is not contained in $G_2$.   Here $|\Delta_1'|$ denotes the support of $\Delta'_1$.   
 
 Write $Y_1^\circ:=\overline{Y}_1-{\rm Supp}\, \Delta(\bar{f}_1,D_1)$ and $Y_2^\circ:=v^{-1}(Y_1^\circ)$.  By our assumption, $\bar{\kappa}(Y_2^\circ)=\bar{\kappa}(Y_1^\circ)=\dim Y_2$.  Therefore, 
 $
K_{\overline{Y}_1}+|\Delta_1|+G_1
$  and  $
 K_{\overline{Y}_2}+|\Delta_1''|+G_2
 $ are both   big line bundles by \cite[Lemma 3]{NWY13}.
\begin{claim}\label{claim:big}
	The line bundle 
 $
K_{\overline{Y}_2}+ G_2+\Delta_2
$  is big.
\end{claim}
\begin{proof}
	Let $m\in \bZ_{>0}$ sufficiently large such that $m E_b\geq F_2$.  By \eqref{eq:canonical formula}, we have
	$$
	\kappa(	K_{\overline{Y}_{2}} + G_{2}-E_b)\geq 0. 
	$$
	Therefore,
	$$
	\kappa(	m(K_{\overline{Y}_{2}} + G_{2}-E_b)+ K_{\overline{Y}_2}+|\Delta_1''|+G_2 )=\dim Y_2.
	$$
	Note that
	\begin{align*}
		m(K_{\overline{Y}_{2}} + G_{2}-E_b)+ K_{\overline{Y}_2}+|\Delta_1''|+G_2\leq 	(m+1)(K_{\overline{Y}_{2}} + G_2)  +|\Delta_1'|.
	\end{align*}
Hence 
$$
\kappa(		(m+1)(K_{\overline{Y}_{2}} + G_2)  +|\Delta_1'|)=\dim Y_2.
$$
We take $\ep\in \bQ_{>0}$ small enough such that  $\ep(m+1)<1$ and $\ep|\Delta_1'|\leq \Delta_1'$. Since $K_{\overline{Y}_{2}} + G_2$ is $\bQ$-effective, it follows that
$$
\ep((m+1)(K_{\overline{Y}_{2}} + G_2)  +|\Delta_1'|)\leq K_{\overline{Y}_{2}} + G_2+\Delta_1'\leq K_{\overline{Y}_{2}} + G_2+\Delta_2.
$$
Therefore, $K_{\overline{Y}_{2}} + G_2+\Delta_2$ is big. 
\end{proof} 
\Cref{claim:big} implies that  $ K_{\overline{Y}_2}+\Delta(\bar{f}_2,D_2)$ is big. We thus proved the second claim. The last claim follows from the very definition of Kodaira dimension of fibration in \eqref{eq:Kodaira2}. 
\end{proof}
 
 We have the following consequence. 
\begin{cor} \label{cor:specialalg}
	Let \(X\) be a special smooth quasi-projective variety, let \(A\) be a semi-abelian variety, and let \(p : X \to A\) be a dominant morphism with connected general fibers. Let \(\Delta(p)\) be the orbifold base divisor on \(A\) defined at the beginning of this section, and let \(\mathrm{St}(\Delta(p))\) be its stabilizer under the action of \(A\). If \(\mathrm{dim}\; \mathrm{St}(\Delta(p)) = 0\), then \(X\) is not special. 
\end{cor}
\begin{proof}
	Let $\overline{A}$  be an equivariant smooth compactification of $A$ such that $G:=\overline{A}-A$ is a simple normal crossing divisor such that $K_{\overline{A}}+G=\cO_{\overline{A}}$. We take a smooth projective compactification $\overline{X}$ of $X$   such that $D:=\overline{X}-X$ is a simple normal crossing divisor and $p$ extends to a morphism $\bar{p}:\overline{X}\to \overline{A}$.   Denote by $\Delta(\bar{p},D)$ the orbifold divisor of $\bar{p}:(\overline{X}|D)\to \overline{A}$.   Let $\Delta(p)$ be the orbifold divisor of $p:X\to A$ defined at the beginning of this subsection. Then we have
	$$
	\Delta(\bar{p},D) =G+\overline{\Delta(p)}.
	$$
Hence $\overline{A}-|\Delta(\bar{p},D)|=A-|\Delta(p)|$.
	
Since \(\mathrm{dim}\; \mathrm{St}(\Delta(p)) = 0\),	by \cite[Proposition 5.6.21]{NW13}, the variety $A-|\Delta(p)|$, hence  $\overline{A}-|\Delta(\bar{p},D)|$ is of log general type.  Therefore,  conditions in  \Cref{prop:KDF} are fulfilled. This implies that \(\kappa( {A},  {p}) = \dim A\), so \(X\) is not special by \cref{def:special}.
\end{proof}

\subsection{$\pi_1$-exactness of quasi-Albanese morphisms}\label{s2}
Let $X$ and $Y$ be smooth quasi-projective varieties. 
We say that a morphism $p:X\to Y$ is an algebraic fiber space 
if 
\begin{itemize}
\item
$p$ is dominant, and
\item
$p$ has general connected fibers. 
\end{itemize}

Let $p:X\to Y$ be an algebraic fiber space, and let $F$ be a general fiber of $p$ (assumed to be smooth). Then there is a natural sequence of morphisms of groups:
\begin{equation}\label{eqn:gp}
\pi _1(F)\overset{i_*}{\to}\pi _1(X)\overset{p_*}{\to} \pi _1(Y).
\end{equation}
Note that $p_*$ is surjective (because of the second condition of the algebraic fiber spaces) and that the image of $i_*$ is contained in the kernel of $p_*$.
\begin{dfn}
For an algebraic fiber space $p:X\to Y$, we say that $p$ is $\pi _1$-exact if the above sequence \eqref{eqn:gp} is exact.
\end{dfn}

By Proposition \ref{prop:factor}, if $X$ is $h$-special, then the quasi-Albanese map $X\to A(X)$ is an algebraic fiber space.

In this subsection, we prove the following proposition.

\begin{proposition}\label{pro:202210131}
Let $X$ be a $h$-special or special quasi-projective manifold. 
Let $p:X\to A$ be an algebraic fiber space where $A$ is a semi-abelian variety. 
Then $p$ is $\pi _1$-exact. 
In particular, the quasi-Albanese map $a_X:X\to A(X)$ is $\pi _1$-exact.
\end{proposition}

The rest of this subsection is devoted to the proof of \cref{pro:202210131}, for which we need three lemmas.
\begin{lem}\label{lem:3}
Let $X$ be a $h$-special or special quasi-projective variety. 
Let $A$ be a semi-abelian variety and let $p:X\to A$ be an algebraic fiber space. If $\Delta (p)\not= \emptyset$, then $\dim \mathrm{St}(\Delta (p))> 0$, where $ \mathrm{St}(\Delta (p))=\{ a\in A; \ a+\Delta (p)=\Delta (p)\}$.
\end{lem}

\begin{proof}
Suppose $\Delta (p)\not= \emptyset$.
To show $\dim \mathrm{St}(\Delta (p))> 0$, we assume contrary $\dim \mathrm{St}(\Delta (p))= 0$.

	The case where \(X\) is special has been dealt with in \cref{cor:specialalg}, which implies readily that $\dim \mathrm{St}(\Delta (p))> 0$ in this situation. 

Next, we consider the case that $X$ is $h$-special.
Let $E\subset X$ be the exceptional divisor of $p$.
Let $Z\subset A$ be the Zariski closure of $p(E)$.
Then $\mathrm{codim}(Z,A)\geq 2$.
We apply \cref{prop:20230405} for the divisor $\Delta(p)\subset A$ and $Z\cap\Delta(p)$ to get a proper Zariski closed set $\Xi\subsetneqq A$.
Let $f:\mathbb C\to X$ be a holomorphic map such that $p\circ f$ is non-constant.
We first show that $p\circ f(\mathbb C)\subset \Xi\cup \Delta(p)$.
Indeed, we have $\mathrm{ord}_y(p\circ f)^{*}(\Delta(p))\geq 2$ for all $y\in (p\circ  f)^{-1}(\Delta(p)\backslash Z)$.
Let $g:\mathbb C\to A$ be defined by $g(z)=p\circ f(e^z)$.
Then $g$ has essential singularity over $\infty$.
Thus by \cref{prop:20230405}, we have $g(\mathbb C)\subset \Xi\cup \Delta(p)$.
Thus $p\circ f(\mathbb C)\subset \Xi\cup \Delta(p)$.

Now since $X$ is $h$-special, we may take two points $x,y\in X$ so that $x\sim y$, $p(x)\not=p(y)$ and $p(x)\not\in \Xi\cup \Delta(p)$.
Then there exists a sequence $f_1,\ldots,f_l:\mathbb C\to X$ such that 
$$x\in F_1, F_1\cap F_2\not=\emptyset,\ldots,F_{l-1}\cap F_l\not=\emptyset, y\in F_l,$$
where $F_i\subset X$ is the Zariski closure of $f_i(\mathbb C)\subset X$.
By $p(x)\not\in \Xi\cup \Delta(p)$, we have $p\circ f_1(\mathbb C)\not\subset  \Xi\cup \Delta(p)$.
Hence $p\circ f_1$ is constant and $F_1\subset  p^{-1}(p(x))$.
By $F_1\cap F_2\not=\emptyset$, we have $p(x)\in p(F_2)$.
Hence $p\circ f_2(\mathbb C)\not\subset  \Xi\cup \Delta(p)$.
Hence $p\circ f_2$ is constant and $F_2\subset  p^{-1}(p(x))$.
Similary, we get $F_i\subset  p^{-1}(p(x))$ for all $i=1,2,\ldots,l$ inductively.
In particular, we have $F_l\subset  p^{-1}(p(x))$.
Thus $p(y)=p(x)$.
This is a contradiction.
Thus we have proved $\dim \mathrm{St}(\Delta (p))> 0$.
\end{proof}

\begin{lem}\label{lem:4}
Let $p:X\to Y$ be an algebraic fiber space. Assume that $\Delta (p)=\emptyset$. Then $p$ is $\pi _1$-exact. 
\end{lem}

\begin{proof}
This lemma is proved implicitly in  \cite[A.C.10]{Cam98} 
and  \cite[Lemme 1.9.9]{Cam01}. 
Let $Y^*$ be a Zariski open subset of $Y$ on which $p$ is a locally trivial $C^\infty$ fibration. 
We may assume that $Y\backslash Y^*$ is a divisor. Set $X^*=p^{-1}(Y^*)$. For each irreducible component $\Delta $ of $Y\backslash Y^*$, we write $p^*(\Delta )=\sum_{j\in J(\Delta )}m_jD_j+E_{\Delta}$ as in \eqref{eqn:div}.
Then since $m_{\Delta}=1$, there is a divisor $D_j$, $j\in J(\Delta )$, such that $m_j=1$. We denote this $D_j$ by $D_{\Delta }$.
Fix a base point $x$ in $X^*$ and let $\gamma _{\Delta}$ be a small loop going once around the divisor $D_{\Delta}$ in the counterclockwise direction. Then $\delta _{\Delta}=p_*(\gamma _{\Delta})$ is a small loop going once around the divisor $\Delta$ for the corresponding base point $y=p(x)$ in $Y^*$.
Since the restriction $p|_{X^*}:X^*\to Y^*$ is smooth and surjective, 
we have the following exact sequence:
\begin{equation}\label{eqn:1}
\pi _1(F,x)\to \pi _1(X^*,x)\to \pi _1(Y^*,y)\to 1,
\end{equation}
where $F$ is the fiber of $p$ over $y$. Let $\Phi$ be the image of $\pi _1(F,x)\to \pi _1(X^*,x)$.
\par
Now we look the following exact sequence:
\begin{equation*}
\begin{CD}
@. @. 1 @. 1\\
@. @. @VVV @VVV \\
@. @. \Phi @>>> \Psi @>>> \widetilde{C}@>>> 1\\
@. @. @VVV @VVV \\
1 @>>> K @>>> \pi _1(X^*,x) @>>> \pi _1(X,x)@>>> 1\\
@. @VVV @VVV @VVV @.\\
1 @>>> L @>>> \pi _1(Y^*,y) @>>> \pi _1(Y,y)@>>> 1\\
@. @VVV @VVV  @VVV\\
@. C @. 1 @. 1\\
@. @VVV\\
@. 1
\end{CD}
\end{equation*}
Note that $\widetilde{C}$ and $C$ are naturally isomorphic.
To prove our lemma, it is enough to show that $\widetilde{C}$, hence $C$ is trivial.
Note that by Van Kampen's theorem, $L$ is generated by $\delta _{\Delta}$. Thus the map $K\to L$ is surjective, hence $C$ is trivial.
\end{proof}

Let $p:X\to Y$ and $q:Y\to S$ be algebraic fiber spaces.
Given $s\in S$, we consider $p_s:X_s\to Y_s$, where $X_s$ is the fiber of $q\circ p:X\to S$ and $Y_s$ is the fiber of $q:Y\to S$.

\begin{lem}\label{lem:20221020}
For generic $s\in S$, we have $\Delta(p_s)\subset \Delta (p)\cap Y_s$.
\end{lem}

To prove this lemma, we start from the following general discussion.

\begin{claim}\label{claim:202210221}
Let $p:V\to W$ and $q:W\to S$ be dominant morphisms of irreducible quasi projective varieties.
Then for generic $s\in S$, the map $p_s:V_s\to W_s$ is dominant.
\end{claim}

\begin{proof}
We apply \cref{lem:Stein} for $q:W\to S$ to get a factorization $W\overset{\alpha}{\to}\Sigma\overset{\beta}{\to}S$ of $q$, where $\alpha:W\to \Sigma$ has connected general fibers and $\beta:\Sigma\to S$ is finite. 
Since $p(V)$ is a dense contractible set, there exists a non-empty Zariski open set $W^o\subset W$ such that $W^o\subset p(V)$.
Similarly, we may take a non-empty Zariski open set $\Sigma^o\subset \Sigma$ such that $\Sigma^o\subset \alpha(W^o)$.
By replacing $\Sigma^o$ by a smaller non-empty Zariski open set, we may assume that the fiber $W_{\sigma}$ is irreducible for every $\sigma\in \Sigma$.
Then for every $\sigma\in \Sigma$, $W^o\cap W_{\sigma}\subset W_{\sigma}$ is a non-empty, hence dense Zariski open set.
We set $S^o=S-\beta(\Sigma-\Sigma^o)$.
Then for every $s\in S^o$, we have $\beta^{-1}(s)\subset \Sigma^o$.
Hence for every $s\in S^o$, $W^o\cap W_s\subset W_s$ is dense.
By $W^o\subset p(V)$, we have $W^o\cap W_s\subset p_s(V_s)$.
Hence $p_s:V_s\to W_s$ is dominant.
\end{proof}

\begin{proof}[Proof of \cref{lem:20221020}]
We start from the following observation.

\begin{claim}
For generic $s\in S$, $p_s:X_s\to Y_s$ is an algebraic fiber space.
\end{claim}
\begin{proof}
By replacing $S$ by a smaller non-empty Zariski open set, we assume that $X_s$ and $Y_s$ are smooth varieties for all $s\in S$.
We apply \cref{claim:202210221}.
Then by replacing $S$ by a smaller non-empty Zariski open set, we may assume that $p_s:X_s\to Y_s$ is dominant for every $s\in S$.
Next we take a non-empty Zariski open set $Y^o\subset Y$ such that $p^{-1}(Y^o)\to Y^o$ has connected fibers.
We take a non-empty Zariski open set $S^o\subset S$ such that $S^o\subset q(Y^o)$ and that $Y_s$ is irreducible for every $s\in S^o$.
Hence for every $s\in S^o$, $Y^o\cap Y^s\subset Y_s$ is a dense Zariski open set.
Note that $p_s:X_s\to Y_s$ has connected fibers over $Y^o\cap Y^s\subset Y_s$.
Hence for every $s\in S^o$, $p_s:X_s\to Y_s$ is an algebraic fiber space.
\end{proof}

Since the problem is local in $S$, by replacing $S$ by a smaller non-empty Zariski open set, we assume that $p_s:X_s\to Y_s$ are algebraic fiber spaces for all $s\in S$.
In particular, both $X_s$ and $Y_s$ are smooth.

\begin{claim}\label{claim:202210222}
Let $D\subset Y$ be an irreducible and reduced divisor such that $q(D)\subset S$ is Zariski dense.
Assume that $D\not\subset \Delta(p)$.
Then for generic $s\in S$,
\begin{itemize}
\item
$D_s\subset Y_s$ is a reduced divisor, and
\item
for every irreducible component $\Delta$ of $D_s$, we have $\Delta\not\subset \Delta(p_s)$.
\end{itemize} 
\end{claim}

\begin{proof}
We first prove the first assertion.
Let $D^o\subset D$ be the smooth locus of $D$.
Then $D^o\subset D$ is a dense Zariski open set.
We take a non-empty Zariski open set $S^o\subset S$ such that both $Y\to S$ and $D^o\to S$ are smooth and surjective over $S^o$.
By \cref{claim:202210221} applied to the open immersion $D^o\subset D$, we may also assume that $D^o_s\subset D_s$ is dense.
We take $s\in S^o$.
Then we have $\dim Y_s=\dim Y-\dim S$ and $\dim D_s=\dim D-\dim S$.
Hence we have $\dim D_s=\dim Y_s-1$.
Hence $D_s\subset Y_s$ is a divisor.
Note that the fiber $D^o_s=D^o\times_S\mathrm{Spec}(\mathbb C(s))$ is a regular scheme.
Hence $D_s\subset Y_s$ is a reduced divisor.
This shows the first assertion.

Next we prove the second assertion.
Let $p^*D=\sum m_jW_j+E$ be the decomposition as in \eqref{eqn:div}.
By $D\not\subset \Delta(p)$, there exists $W_{j_0}$ such that $m_{j_0}=1$.
We remove $\sum_{j\not=j_0} m_jW_j+E$ from $X$ to get a Zariski open set $Z\subset X$.
Then $Z$ is quasi-projective.
Set $W=Z\cap W_{j_0}$, which is a reduced divisor on $Z$.
Let $\pi:Z\to Y$ be the restriction of $p$ onto $Z$.
Then we have $\pi^*D=W$ and $\overline{\pi(W)}=D$.
We apply the first assertion of \cref{claim:202210222} for $W\subset X$.
Then for generic $s\in S$, $W_s\subset Z_s$ is a reduced divisor.
By \cref{claim:202210221}, $\overline{\pi_s(W_s)}=D_s$ for generic $s\in S$.
We take generic $s\in S$ and an irreducible component $\Delta$ of $D_s$.
Then by $\pi_s^*D_s=W_s$, we have $\Delta\not\subset \Delta(\pi_s)$, hence $\Delta\not\subset \Delta(p_s)$.
\end{proof}

Now let $Y^o\subset Y$ be a non-empty Zariski open set such that $p|_{X^o}:X^o\to Y^o$ is smooth and surjective, where $X^o=p^{-1}(Y^o)$.
We may assume that $D=Y-Y^o$ is a divisor.
Let $D_1,\ldots,D_n$ be the irreducible components of $D$.
By replacing $S$ by a smaller non-empty Zariski open set, we may assume that $q(D_i)$ is dense in $S$ for all $i=1,\ldots,n$.
We assume that $D_i$, $1\leq i\leq l$, are all the components of $D$ such that $D_i\not\subset \Delta(p)$.
By replacing $S$ by a smaller non-empty Zariski open set, we may assume that \cref{claim:202210222} is valid for all $s\in S$ and $D_{1},\ldots,D_l$.
We take $s\in S$.
Let $\Delta\subset Y_s$ be a prime divisor such that $\Delta\subset \Delta(p_s)$.

We first show $\Delta\subset D_s$.
For this purpose, we suppose contrary $\Delta\not\subset D_s$.
Then $\Delta\cap Y^o_s\not=\emptyset$.
Set $X^o_s=p_s^{-1}(Y^o_s)$.
Note that $p_s|_{X^o_s}:X^o_s\to Y^o_s$ is smooth and surjective.
Then the divisor $p_s^*\Delta\cap X^o_s\subset X^o_s$ is reduced and $p_s(\Delta\cap X^o_s)=\Delta\cap Y^o_s$.
Thus $\Delta\not\subset \Delta(p_s)$, a contradiction.
Hence $\Delta\subset D_s$.

Now we may take $D_i$, $i=1,\ldots,n$, such that $\Delta\subset D_i$.
Then by \cref{claim:202210222}, we have $l+1\leq i\leq n$.
Hence by $D_i\subset \Delta(p)$ for $i=l+1,\ldots,n$, we have $\Delta\subset \Delta(p)$.
This shows $\Delta(p_s)\subset \Delta(p)\cap Y_s$.
\end{proof}

\begin{lem}\label{lem:5}
Let $A$ be a semi-abelian variety, and let $p:X\to A$ be an algebraic fiber space. Let $B$ be the quotient semi-abelian variety $A/\mathrm{St}^o(\Delta (p))$.
Let $q:X\to B$ be the composition of $p$ and the quotient $r:A\to B$. 
If $q$ is $\pi _1$-exact, then $p$ is $\pi _1$-exact.
\end{lem}

\begin{proof}
We apply \cref{lem:20221020}.
Then for generic $b\in B$, we have $\Delta(p_b)\subset \Delta(p)$, where $p_b:X_b\to A_b$ is the induced algebraic fiber space.
Note that $\overline{r(\Delta (p))}\subsetneqq B$ by $r(\Delta (p))=\Delta(p)/\mathrm{St}^o(\Delta (p))$.
Hence we may assume $b\in B\backslash r(\Delta (p))$.
Then we have $A_b\cap \Delta(p)=\emptyset$.
Hence $\Delta(p_b)=\emptyset$.
We take generic $a\in A$ such that the fiber $X_a$ of $p:X\to A$ over $a$ is smooth.
Then $X_a$ is a smooth fiber of the algebraic fiber space $p_b:X_b\to A_b$.
By \cref{lem:4}, we have the exact sequence:
\begin{equation}\label{eqn:s}
\pi _1(X_a,x)\to \pi _1(X_b,x)\to \pi _1(A_b,a)\to 1,
\end{equation}
where $x\in X_a$.
We consider the following sequence:
\begin{equation}\label{eqn:m}
\begin{CD}
\pi _1(X_a,x) @>>> \pi _1(X,x) @>>> \pi _1(A,a) \\
@V{c}VV @V{i}VV @V{e}VV \\
\pi _1(X_b,x) @>>> \pi _1(X,x) @>>> \pi _1(B,b) 
\end{CD}
\end{equation}
where the second line is exact from the assumption and $i$ is an isomorphism.
By \eqref{eqn:s}, we have $\text{coker}(c)=\pi _1(A_b,a)=\text{ker}(e)$.
Thus the kernel of the natural map $\pi _1(X_b,x)\to \text{ker}(e)$ is the image of $c:\pi _1(X_a,x)\to \pi _1(X_b,x)$.
Hence the first line of \eqref{eqn:m} is also exact.
\end{proof}

\begin{proof}[Proof of \cref{pro:202210131}]
Inductively, we define algebraic fiber spaces $p_i:X\to A_i$, where $A_i$ are semi-abelian varieties, as follows: First, set $A_1=A$ and $p_1=p$. Next if $\dim \mathrm{St}(\Delta (p_i))>0$, then set $A_{i+1}=A_i/\mathrm{St}^o(\Delta (p_i))$ and let $p_{i+1}$ be the composition of $p_i$ and the quotient map $A_i\to A_{i+1}$. Then since $\dim A_i>\dim A_{i+1}>0$, this process should stop, i.e., $\dim \mathrm{St}(\Delta (p_i))=0$ for some $i$. Then by Lemma \ref{lem:3}, we have $\Delta (p_i)=\emptyset$. Thus by Lemmas \ref{lem:4} and \ref{lem:5}, we conclude that $p$ is $\pi _1$-exact.
\end{proof}

\subsection{Two Lemmas for the proof of \cref{thm:202210123}}\label{s3}

Before going to give a proof of \cref{thm:202210123}, we prepare three lemmas.

\begin{lem}\label{claim:20221014}
Let $\Pi\subset \mathbb C^n$ be a finitely generated additive subgroup which is Zariski dense.
Set $\Sigma=\{ \sigma\in\mathrm{GL}(\mathbb C^n);\ \sigma\Pi= \Pi\}$.
Let $\lambda:\Sigma\to \mathrm{GL}(\Pi\otimes_{\mathbb Z}\overline{\mathbb Q})$ be the induced map.
Let $\widetilde{\Sigma}\subset \Sigma$ be a subgroup and let $E\subset \mathrm{GL}(\Pi\otimes_{\mathbb Z}\overline{\mathbb Q})$ be the Zariski closure of $\lambda(\widetilde{\Sigma})\subset \mathrm{GL}(\Pi\otimes_{\mathbb Z}\overline{\mathbb Q})$.
Let $Y$ be the Zariski closure of $\widetilde{\Sigma}\subset \mathrm{GL}(\mathbb C^n)$.
Assume that the identity component $E^o\subset E$ is unipotent.
Then the identity component $Y^o\subset Y$ is unipotent.  
\end{lem}

\begin{proof}
Since $\Pi\subset \mathbb C^n$ is Zariski dense, $\lambda$ is injective and the natural linear map $q:\Pi\otimes_{\mathbb Z}\mathbb C\to \mathbb C^n$ is surjective.
Set $K=\mathrm{ker}(q)$.
Then we have the following exact sequence of finite dimensional $\mathbb C$-vector spaces:
$$
0\to K\to \Pi\otimes_{\mathbb Z}\mathbb C\to \mathbb C^n\to 0.
$$
Let $H\subset \mathrm{GL}(\Pi\otimes_{\mathbb Z}\mathbb C)$ be the algebraic subgroup such that $g\in H$ iff $gK=K$.
We claim that $\lambda(\Sigma)\subset H$ under the embedding $\lambda:\Sigma\hookrightarrow \mathrm{GL}(\Pi\otimes_{\mathbb Z}\mathbb C)$.
Indeed, we take $\sigma\in \Sigma$ and $v_1\otimes c_1+\cdots+v_r\otimes c_r\in K$.
Then we have $c_1v_1+\cdots +c_rv_r=0$ in $\mathbb C^n$.
We have
\begin{equation*}
\begin{split}
q(\lambda(\sigma)(v_1\otimes c_1+\cdots+v_r\otimes c_r))&=q(\sigma(v_1)\otimes c_1+\cdots+\sigma(v_r)\otimes c_r)\\
&=c_1q(\sigma(v_1))+\cdots+c_rq(\sigma(v_r))
\\
&=c_1\sigma(v_1)+\cdots+c_r\sigma(v_r)
\\
&=\sigma(c_1v_1+\cdots+c_rv_r)=0.
\end{split}
\end{equation*}
Hence $\sigma(v_1\otimes c_1+\cdots+v_r\otimes c_r)\in K$.
Thus $\sigma\in H$, so $\Sigma\subset H$.

Now we have a morphism $p:H\to \mathrm{GL}(\mathbb C^n)$.
Then $p\circ\lambda:\Sigma\to \mathrm{GL}(\mathbb C^n)$ is the original embedding.
This induces a surjection $E_{\mathbb C}\to Y$, where $E_{\mathbb C}=E\times_{\bar{\mathbb Q}}\mathbb C$.
Hence we get a surjection $E^o_{\mathbb C}\to Y^o$.
Since every quotient of unipotent group is again unipotent, 
 $Y^o$ is unipotent. 
\end{proof}

\begin{lem}\label{lem:20221015}
Let $Z\to E$ be a surjective morphism of algebraic groups.
Assume that the radical of $Z$ is unipotent.
Then the radical of $E$ is unipotent.
\end{lem}

\begin{proof}
We have an exact sequence
$$
1\to R(Z)\to Z\to Z/R(Z)\to 1,
$$
where $R(Z)$ is the radical of $Z$.
Then $Z/R(Z)$ is semi-simple.
Let $N\subset Z$ be the kernel of $Z\to E$.
Then we have
$$
1\to R(Z)/(R(Z)\cap N)\to E\to (Z/R(Z))/N'\to 1
$$
where $N'\subset Z/R(Z)$ is the image of $N$.
Then $(Z/R(Z))/N'$ is semi-simple and $R(Z)/(R(Z)\cap N)$ is solvable.
Hence $R(E)=R(Z)/(R(Z)\cap N)$.
Hence $R(E)$ is a quotient of $R(Z)$.
Note that $R(Z)$ is unipotent.
Hence $R(E)$ is unipotent.
\end{proof}

\subsection{Proof of \cref{thm:202210123}}\label{s4}
\begin{proof}[Proof of \cref{thm:202210123}]
We take a finite {\'e}tale cover $X'\to X$ such that $\pi_1(X')^{ab}$ is torsion free.
Note that $\varphi(\pi_1(X'))\subset G$ is Zariski dense and $X'$ is special or $h$-special.
Hence by replacing $X$ by $X'$, we may assume that $\pi_1(X)^{ab}$ is torsion free.

Let $X\to A$ be the quasi Albanese map.
We have the following sequence 
\begin{equation}\label{eqn:202210131}
\pi_1(F)\overset{\kappa}{\to} \pi_1(X)\to \pi_1(A)\to 1,
\end{equation}
where $F\subset X$ is a generic fiber.
By \cref{pro:202210131}, this sequence is exact.
Since $\pi_1(X)^{ab}$ is torsion free, we have $\pi_1(A)=\pi_1(X)^{ab}$.
Hence $\kappa$ induces $\pi_1(F)\to\pi_1(X)'$, where $\pi_1(X)'= [\pi_1(X),\pi_1(X)]$.
By $\varphi(\pi_1(X)')\subset G'$, we get 
$$\varphi\circ\kappa:\pi_1(F)\to G'.$$
Let $\Pi\subset G'/G''$ be the image of $\pi_1(F)\to G'/G''$.
Since $\pi_1(F)$ is finitely generated, $\Pi$ is a finitely generated, abelian group.

\begin{claim}\label{claim:202210153}
$\Pi\subset G'/G''$ is Zariski dense.
\end{claim}

\begin{proof}
Set $\Gamma=\varphi(\pi_1(X))$.
Note that $\Gamma' \subset G'$, where $\Gamma'=[\Gamma,\Gamma]$ and $G'=[G,G]$.
We first show that $\Gamma'$ is Zariski dense in $G'$.
Let $H\subset G'$ be the Zariski closure of $\Gamma'$.
Then we have $\Gamma/\Gamma'\to G/H$, whose image is Zariski dense.
Since $\Gamma/\Gamma'$ is commutative, $G/H$ is commutative.
Hence $G'\subset H$.
This shows $H=G'$.
Thus $\Gamma'\subset G'$ is Zariski dense.

Now by \cref{pro:202210131}, $\kappa$ induces the surjection $\pi_1(F)\twoheadrightarrow \pi_1(X)'$.
Since $\pi_1(X)\to\Gamma$ is surjective, the induced map $\pi_1(X)'\to\Gamma'$ is surjective.
Hence $\pi_1(X)\to \Gamma$ induces a surjection $\pi_1(F)\twoheadrightarrow \Gamma'$.
Hence the image $\pi_1(F)\to G'$ is Zariski dense, for $\Gamma'\subset G'$ is Zariski dense.
Hence $\Pi\subset G'/G''$ is Zariski dense.
\end{proof}

Let $\Phi\subset \pi_1(X)$ be the image of $\pi_1(F)\to \pi_1(X)$.
By \eqref{eqn:202210131}, we have the following exact sequence:
$$
1\to \Phi\to\pi_1(X)\to \pi_1(A)\to 1.
$$
Note that $\Phi'\subset \pi_1(X)$ is a normal subgroup.
Hence we get
\begin{equation*}
\begin{CD}
1@>>> \Phi^{ab} @>>> \pi _1(X)/\Phi' @>>> \pi _1(A) @>>> 1\\
@. @VVV @VVV @VVV \\
1@>>> G'/G'' @>>> G/G'' @>>> G/G' @>>> 1
\end{CD}
\end{equation*}
By the conjugation, we get 
\begin{equation*}
\pi_1(A)\to \mathrm{Aut}(\Phi^{ab}).
\end{equation*}
This induces
\begin{equation*}
\rho:\pi_1(A)\to \mathrm{Aut}(\Pi).
\end{equation*}
Note that $G'/G''$ is a commutative unipotent group.
Hence $G'/G''\simeq (\mathbb G_a)^n$, where $\mathbb G_a$ is the additive group.
The exponential map $\mathrm{Lie}(G'/G'')\to G'/G''$ is an isomorphism. 
Note that $(\mathbb G_a)^n=\mathbb C^n$ as additive group.
We have 
$$\mathrm{Aut}((\mathbb G_a)^n)=\mathrm{GL}(\mathrm{Lie}(G'/G''))=\mathrm{GL}(\mathbb C^n).
$$
Hence by the conjugate, we have 
$$
\mu:G/G'\to \mathrm{Aut}(G'/G'')=\mathrm{GL}(\mathbb C^n).
$$

Let $1\to U\to G\to T\to 1$ be the sequence as in \eqref{eqn:20230321}.
We have $G/G'=(U/G')\times T$, from which we obtain $\mu|_T:T\to\mathrm{GL}(\mathbb C^n) $.
In the following, we are going to prove that $\mu|_T$ is trivial.
(Then \cref{claim:202210121} will yields that $G$ is nilpotent.)
For this purpose, we shall show that $\mu|_T(T)$ is contained in some unipotent subgroup $Y\subset \mathrm{GL}(\mathbb C^n)$, which we describe below, and then apply \cref{lem:trivial morphism}.

Now we define a subgroup 
$$\Sigma=\{ \sigma\in\mathrm{Aut}(G'/G'');\ \sigma\Pi= \Pi\}\subset \mathrm{GL}(\mathbb C^n).$$
Note that $\rho:\pi_1(A)\to \mathrm{Aut}(\Pi)$ factors through $\mu:G/G'\to \mathrm{Aut}(G'/G'')$.
This induces the following commutative diagram:

\begin{equation}\label{eqn:202210141}
\begin{CD}
\pi _1(A)@>\rho>> \Sigma \\
@VVV @VVV \\
G/G'@>>\mu> \mathrm{GL}(\mathbb C^n)
\end{CD}
\end{equation}

Since $\Pi\subset \bC^n$ is finitely generated, $\Pi$ is a free abelian group of finite rank. 
Since $\Pi\subset \mathbb C^n$ is Zariski dense (cf. \cref{claim:202210153}), the linear subspace spanned by $\Pi$ is $\mathbb C^n$.
Hence we may embed  $\Sigma\subset \mathrm{GL}(\Pi\otimes_{\mathbb Z}\bar{\mathbb Q})$.
Let $E\subset \mathrm{GL}(\Pi\otimes_{\mathbb Z}\bar{\mathbb Q})$ be the Zariski closure of $\rho(\pi_1(A))\subset \mathrm{GL}(\Pi\otimes_{\mathbb Z}\bar{\mathbb Q})$.
Then $E$ is commutative.
Let $E^o\subset E$ be the identity component.

\begin{claim}
$E^o$ is unipotent.
\end{claim}

\begin{proof}
We apply Deligne's theorem.
Let $A^o\subset A$ be a non-empty Zariski open set such that the restriction $f_o:X^o\to A^o$ of $X\to A$ over $A^o$ is a locally trivial $C^\infty$-fibration. 
Then we have the following exact sequence 
$$
\pi_1(F)\to \pi_1(X^o)\to\pi_1(A^o)\to 1.
$$
Let $\Psi\subset \pi_1(X^o)$ be the image of $\pi_1(F)\to \pi_1(X^o)$.
Note that $\Psi'\subset \pi_1(X^o)$ is a normal subgroup.
Then we get the following commutative diagram:

\begin{equation*}
\begin{CD}
1@>>> \Psi^{ab} @>>> \pi _1(X^o)/\Psi' @>>> \pi _1(A^o) @>>> 1\\
@. @VVV @VVV @VVV \\
1@>>>  \Phi^{ab} @>>> \pi _1(X)/\Phi' @>>> \pi _1(A)@>>> 1 \\
@. @VVV @VVV @VVV\\
1@>>> G'/G'' @>>> G/G'' @>>> G/G' @>>> 1
\end{CD}
\end{equation*}
By the conjugation, we get 
\begin{equation*}
\lambda:\pi_1(A^o)\to \mathrm{Aut}(\Psi^{ab}).
\end{equation*} 
This induces
\begin{equation*}
\bar{\lambda}:\pi_1(A^o)\to \mathrm{Aut}(\Pi),
\end{equation*}
which is the composite of $\pi_1(A^o)\to \pi_1(A)$ and $\rho:\pi_1(A)\to  \mathrm{Aut}(\Pi)$.
We note that $\bar{\lambda}$ induces
\begin{equation*}
\bar{\lambda}_{\bar{\mathbb Q}}:\pi_1(A^o)\to\mathrm{GL}(\Pi\otimes_{\mathbb Z}\bar{\mathbb Q}).
\end{equation*}
Then $E\subset \mathrm{GL}(\Pi\otimes_{\mathbb Z}\bar{\mathbb Q})$ is the Zariski closure of the image $\bar{\lambda}_{\bar{\mathbb Q}}(\pi_1(A^o))$.

Now we have the monodromy action
\begin{equation*}
\tau:\pi_1(A^o)\to\mathrm{GL}(\pi_1(F)^{ab}\otimes_{\mathbb Z}\bar{\mathbb Q})
\end{equation*}
induced from the family $X^o\to A^o$.
Let $Z\subset \mathrm{GL}(\pi_1(F)^{ab}\otimes_{\mathbb Z}\bar{\mathbb Q})$ be the Zariski closure of the image $\tau(\pi_1(A^o))$.
Let $R(Z)\subset Z$ be the radical of $Z$. 
Since $X^o\to A^o$ is a locally trivial $C^\infty$ fibration, $\tau$ is dual to the local system $R^1(f_o)_*(\bar{\bQ})$. Note that $R^1(f_o)_*(\bar{\bQ})$, hence $\tau$ underlies an  admissible variation of mixed Hodge structures (cf. \cite{BZ14} for the definition). 
Then by Deligne's theorem (cf. \cite[Corollary 2]{And92} or \cite[4.2.9b]{Del71}), $R(Z)$ is unipotent. 
Let $H\subset \mathrm{GL}(\pi_1(F)^{ab}\otimes_{\mathbb Z}\bar{\mathbb Q})$ be an algebraic subgroup defined by
$$
H=\{\sigma \in \mathrm{GL}(\pi_1(F)^{ab}\otimes_{\mathbb Z}\bar{\mathbb Q}); \sigma K= K\},
$$
where $K=\mathrm{ker}(\pi_1(F)^{ab}\otimes_{\mathbb Z}\bar{\mathbb Q}\to \Pi\otimes_{\mathbb Z}\bar{\mathbb Q})$.
Then $\tau(\pi_1(A^o))\subset H$, hence $Z\subset H$.
Note that $\bar{\lambda}_{\bar{\mathbb Q}}$ factors through $\tau:\pi_1(A^o)\to H$. 
This induces a morphism $Z\to E$, which is
surjective, for $\bar{\lambda}_{\bar{\mathbb Q}}(\pi_1(A^o))$ is Zariski dense.
Hence we get a surjection $Z\to E$.
Hence by \cref{lem:20221015}, $R(E)$ is unipotent.
Since $E$ is commutative, we have $R(E)= E^o$.
Hence $E^o$ is unipotent.
\end{proof}

Let $Y\subset \mathrm{GL}(\mathbb C^n)$ be the Zariski closure of $
\rho(\pi_1(A))\subset \mathrm{GL}(\mathbb C^n)$.
Let $Y^o\subset Y$ be the identity component.
By \cref{claim:202210153}, we may apply \cref{claim:20221014} for 
$\widetilde{\Sigma}=\rho(\pi_1(A))\subset\Sigma$ to conclude that $Y^o$ is unipotent.
Since the image of $\pi_1(A)\to G/G'$ is Zariski dense, the commutativity of \eqref{eqn:202210141} implies $\mu(G/G')\subset Y$.
This is Zariski dense, in particular $Y^o=Y$.
By $G/G'= (U/G')\times T$, $\mu$ induces $\mu|_T:T\to Y$.
Since $Y$ is unipotent, this is trivial by \cref{lem:trivial morphism}. 
Hence the action of $T$ onto $G'/G''$ is trivial.
By Lemma \ref{claim:202210121}, $G$ is nilpotent.
\end{proof}

\subsection{Proof of \cref{thm:VN}}  \label{s5}
\cref{thm:VN} is a consequence of \cref{main,thm:202210123}. 
\begin{proof}[Proof of \cref{thm:VN}]
	Observe that if $X$ is $h$-special (resp.  special), then   any finite \'etale cover  of $X$ is also  $h$-special (resp.   special). Denote by $G$ be the  Zariski closure of $\varrho$. 	
	Then after replacing $X$ by a finite \'etale cover   corresponding to the finite index subgroup $\varrho^{-1}(\varrho(\pi_1(X))\cap G_0(\bC))$ of $\pi_1(X)$, we can assume that the Zariski closure $G$ of $\varrho$ is connected.  Denote by $R(G)$   the radical $R(G)$ of $G$, which is the unique normal solvable subgroup such that $G/R(G)$ is semisimple.  If $R(G)\neq G$, the  induced representation $\varrho':\pi_1(X)\to G/R(G)(\bC)$ is still Zariski dense.   By \cref{cor:202304071}, $X$ is neither $h$-special nor weakly special which contradicts our assumption. This implies that $R(G)=G$.     By \cref{thm:202210123}, $G$ is nilpotent. 
	
	If  $\varrho$ is assumed to be semisimple, then $G$ is reductive. Hence
	$R(G)$ of $G$ is an algebraic torus.  It follows that $\varrho(\pi_1(X))$ is a virtually abelian group. 
\end{proof}

\subsection{Some examples}\label{sec:examples} 
As we mentioned above, the following example disproves \cref{conj:Campana}. 
 We construct a quasi-projective surface, which is both $h$-special and special. Its  fundamental group is linear and nilpotent, but not almost abelian.

\begin{example}\label{example} 
	Fix $\tau\in\mathbb H$ from the upper half plane.
	Then $\mathbb C/<\mathbb Z+\mathbb Z\tau>$ is an elliptic curve.
	We define a nilpotent group $G$ as follows.
	$$
	G=\left\{g(l,m,n)=\begin{pmatrix}
		1 & 0 & m & n \\
		-m & 1 & -\frac{m^2}{2} & l \\
		0 & 0 & 1 & 0 \\
		0 & 0 & 0 & 1
	\end{pmatrix}
	\in\mathrm{GL}_4(\mathbb Z)
	\
	\vert 
	\
	l,m,n\in \mathbb Z
	\right\}
	$$
	Thus as sets $G\simeq \mathbb Z^3$.
	However, $G$ is non-commutative as direct computation shows:
	\begin{equation}\label{eqn:202210061}
		g(l,m,n)\cdot g(l',m',n')=g(-mn'+l+l',m+m',n+n').
	\end{equation}
	We define $C\subset G$ by letting $m=0$ and $n=0$.
	$$
	C=\left\{
	g(l,0,0)=
	\begin{pmatrix}
		1 & 0 & 0 & 0 \\
		0 & 1 & 0 & l \\
		0 & 0 & 1 & 0\\
		0 & 0 & 0 & 1
	\end{pmatrix}
	\in\mathrm{GL}_4(\mathbb Z)
	\
	\vert 
	\
	l\in \mathbb Z
	\right\}
	$$
	Then $C$ is a free abelian group of rank one, thus $C\simeq \mathbb Z$ as groups.
	By \eqref{eqn:202210061}, $C$ is a center of $G$.  
	We have an exact sequence
	$$
	1\to C\to G\to L\to 1,
	$$
	where $L\simeq \mathbb Z^2$ is a free abelian group of rank two.
	This is a central extension.
	The quotient map $G\to L$ is defined by $g(l,m,n)\mapsto (m,n)$.
	
	\begin{claim}
		$G$ is not almost abelian.
	\end{claim}
	
	\begin{proof}
		For $(\mu,\nu)\in\mathbb Z^2$, we set $G_{\mu,\nu}=\{ g(l,m,n);\ \mu n=\nu m\}$.
		Then by \eqref{eqn:202210061}, $G_{\mu,\nu}$ is a subgroup of $G$.
		By \eqref{eqn:202210061}, $g(l,m,n)$ commutes with $g(l',m',n')$ if and only if $g(l',m',n')\in G_{m,n}$.
		We have $C\subset G_{\mu,\nu}$ and $G_{\mu,\nu}/C=\{ (m,n)\in L; \mu n=\nu m\}$.
		Hence the index of $G_{\mu,\nu}\subset G$ is infinite for $(\mu,\nu)\not=(0,0)$.

		Now assume contrary that $G$ is almost abelian.
		The we may take a finite index subgroup $H\subset G$ which is abelian.
		We may take $g(l,m,n)\in H$ such that $(m,n)\not=(0,0)$.
		Then we have $H\subset G_{m,n}$.
		This is a contradiction, since the index of $G_{m,n}\subset G$ is infinite.
	\end{proof}

	Now we define an action $G\curvearrowright \mathbb C^2$ as follows:
	For $(z,w)\in\mathbb C^2$, we set
	$$
	\begin{pmatrix}
		z  \\
		w\\
		\tau \\
		1 
	\end{pmatrix}
	\mapsto
	\begin{pmatrix}
		1 & 0 & m & n \\
		-m & 1 & -\frac{m^2}{2} & l \\
		0 & 0 & 1 & 0 \\
		0 & 0 & 0 & 1\end{pmatrix}
	\begin{pmatrix}
		z  \\
		w\\
		\tau  \\
		1 
	\end{pmatrix}
	$$
	Hence
	$$
	g(l,m,n)\cdot (z,w)=\left(z+m\tau+n,-mz+w-\frac{m^2}{2}\tau+l\right).
	$$
	This action is properly discontinuous.
	We set $X=\mathbb C^2/G$.
	Hence 
	$\pi_1(X)=G.$ 
	Then $X$ is a smooth complex manifold.
	We have $\mathbb C^2/C\simeq \mathbb C\times \mathbb C^*$.
	The action $L\curvearrowright \mathbb C\times \mathbb C^*$ is written as 
	\begin{equation}\label{eqn:20221005}
		(z,\xi)\mapsto (z+m\tau+n,e^{-2\pi i mz-\pi im^2\tau}\xi),
	\end{equation}
	where $\xi=e^{2\pi iw}$.
	The first projection $\mathbb C\times \mathbb C^*\to \mathbb C$ is equivariant $L\curvearrowright \mathbb C\times \mathbb C^*\to \mathbb C\curvearrowleft L$.
	By this, we have $X\to E$, where $E=\mathbb C/<\mathbb Z+\mathbb Z\tau>$ is an elliptic curve.
	The action \eqref{eqn:20221005} gives the action on $\mathbb C\times \mathbb C$ by the natural inclusion $\mathbb C\times \mathbb C^*\subset \mathbb C\times \mathbb C$.
	We consider this as a trivial line bundle by the first projection $\mathbb C\times \mathbb C\to \mathbb C$. 
	We set $Y=(\mathbb C\times \mathbb C)/L$.
	This gives a holomorphic line bundle $Y\to E$. 
	By Serre's GAGA, $Y$ is algebraic.
	Hence $X=Y-Z$ is quasi-projective, where $Z$ is the zero section of $Y$. 
	\begin{claim}
		The quasi-projective surface $X$ is special and contains a Zariski dense entire curve. In particular, it is $h$-special.
	\end{claim}
	\begin{proof}
We take a dense set $\{(x_n,y_n)\}_{n\in \bZ}\subset  \bC^2$. By  Mittag-Leffler interpolation we can find entire functions $f_1(z)$ and $f_2(z)$   such that $f_1(n)=x_n$ and   $f_2(n)=y_n$. Then $f:\bC\to \bC^2$ defined by $f=(f_1,f_2)$ is a metrically dense entire curve. 
	Note that $\pi:\bC^2\to X$ is the universal covering map. It follows that $\pi\circ f$  is  a metrically dense entire curve in  $X$. Hence $X$ is also $h$-special. 
	
	If $X$ is not special, then  by definition after replacing $X$ by a proper birational modification  there is an algebraic fiber space $g:X\to C$ from $X$ to a  quasi-projective curve such that the orbifold base $(C,\Delta)$ of $g$ defined in \cite{Cam11} is of log general type, hence hyperbolic.  However, the composition $g\circ \pi\circ f$ is a orbifold entire curve of  $(C,\Delta)$. This contradicts with the hyperbolicity of $(C,\Delta)$.  Therefore, $X$ is special. 
	\end{proof}
	
\end{example}

\begin{rem}\label{re:20230428}
Let $Y\to E$ be the line bundle described in \Cref{example}.
Here $Y=(\mathbb C\times \mathbb C)/L$ under the action given by \eqref{eqn:20221005}.
By the $\tau$-quasiperiodic relation, the Jacobi theta function $\vartheta(z,\tau)$ gives an equivariant section of the first projection $\mathbb C\times \mathbb C\to \mathbb C$.
		Hence the degree of the line bundle $Y\to E$ is equal to one.
Thus $Y$ is ample.
\end{rem}

\begin{rem}
The above $X$ in \Cref{example} is homotopy equivalent to a Heisenberg manifold which is a circle bundle over the 2-torus.
It is well known that Heisenberg manifolds have nilpotent fundamental groups.
\end{rem}




We will construct an example of $h$-special complex manifold with linear solvable, but non virtually nilpotent fundamental group. By \cref{thm:202210123}, it is thus non quasi-projective.
\begin{example} 
	We start from algebraic argument.
	Set
	\begin{equation*}
		M=	\begin{bmatrix}
			1 & 2\\
			1 & 3
		\end{bmatrix}
		\in\mathrm{SL}_2(\mathbb Z)
	\end{equation*} 	
	All the needed property for $M$ is contained in the following \cref{claim:202210091}.
	
	\begin{claim}\label{claim:202210091}
		For every non-zero $n\in \mathbb Z-\{0\}$, $M^n$ has no non-zero eigenvector in $\mathbb Z^2$.
	\end{claim}
	
	\begin{proof}
		Fix $n\in \mathbb Z-\{0\}$.
		Assume contrary that there exists $v\in \mathbb Z^2\backslash \{ 0\}$, such that $M^nv=\alpha^nv$, where $\alpha$ is an eigenvalue of $M$.
		Then $\alpha^n\not\in\mathbb Q$ as direct computation shows. 
		Hence $\alpha^nv\not\in \mathbb Z^2$, while $M^nv\in\mathbb Z^2$.
		This is a contradiction.
	\end{proof}
	
	\begin{claim}\label{claim:202210092}
		Let $N\subset \mathbb Z^2$ be a submodule.
		Let $n\in \mathbb Z-\{0\}$.
		If $N$ is invariant under the action of $M^n$, then either $N=\{0\}$ or $N$ is finite index in $\mathbb Z^2$.
	\end{claim}
	
	\begin{proof}
		Suppose $N\not=\{0\}$.
		By \cref{claim:202210091}, we have $\dim_{\mathbb R}N\otimes_{\mathbb Z}\mathbb R=2$.
		Then $N$ is finite index in $\mathbb Z^2$.
	\end{proof}
	
	\begin{claim}\label{claim:202210093}
		Let $G$ be a group which is an extension 
		$$
		1\to \mathbb Z^2\to G\overset{p}{\to}\mathbb Z\to 1.
		$$
		Assume that the corresponding action $\mathbb Z\to {\rm Aut}(\bZ^2)$ is given by $n\mapsto M^n$.
		Then $G$ is not virtually nilpotent.
	\end{claim}
	
	\begin{proof}
		Assume contrary that there exists a finite index subgroup $G'\subset G$ which is nilpotent.
		Then there exists a central sequence 
		$$G'\cap \mathbb Z^2=N_0\supset N_1\supset N_2\supset \cdots\supset N_{k+1}=\{0\},$$ 
		i.e., for each $i=0,\ldots, k$, $N_i\subset G'$ is a normal subgroup and $N_i/N_{i+1}\subset G'/N_{i+1}$ is contained in the center of $G'/N_{i+1}$.
		By induction on $i$, we shall prove that $N_i\subset \mathbb Z^2$ is a finite index subgroup.
		For $i=0$, this is true by $N_0=G'\cap \mathbb Z^2$.
		We assume that $N_i\subset \mathbb Z^2$ is finite index.
		There exists $l\in\mathbb Z_{>0}$ such that $p(G')=l\mathbb Z$, where $p:G\to\mathbb Z$.
		Then we have an exact sequence
		$$
		1\to G'\cap \mathbb Z^2\to G'\to l\mathbb Z\to 1.
		$$
		Since $N_{i+1}\subset G'$ is normal, $N_{i+1}\subset \mathbb Z^2$ is invariant under the action $M^l$.
		By \cref{claim:202210092}, either $N_{i+1}=\{0\}$ or $N_{i+1}\subset \mathbb Z^2$ is finite index.
		In the second case, we complete the induction step.
		Hence it is enough to show $N_{i+1}\not=\{0\}$.
		So suppose $N_{i+1}=\{0\}$.
		Note that $N_i/N_{i+1}$ is contained in the center of $G'/N_{i+1}$.
		Hence $N_i$ is contained in the center of $G'$.
		So $M^l$ acts trivially on $N_i$.
		This is a contradiction (cf. \cref{claim:202210091}).
		Hence $N_{i+1}\not=\{0\}$.
		Thus we have proved that $N_i\subset \mathbb Z^2$ is a finite index subgroup.
		This contradicts to $N_{k+1}=\{0\}$.
		Thus $G$ is not virtually nilpotent.
	\end{proof}
	
	Now for $n\in\mathbb Z$, we define integers $a_n$, $b_n$, $c_n$ ,$d_n$ by
	\begin{equation*}
		\begin{bmatrix}
			a_n & b_n\\
			c_n & d_n
		\end{bmatrix}
		=M^n
	\end{equation*} 	
	Define a $\bZ$-action on $\bC^*\times\bC^*\times \bC$ by  
	$$
	(\xi_1,\xi_2,z)\mapsto (\xi_1^{a_n}\xi^{b_n}_2,\xi_1^{c_n}\xi^{d_n}_2,z+n).
	$$
	The quotient by this action is a holomorphic fiber bundle $X$ over $\mathbb C^*$  with fibers $\bC^*\times\bC^*$. For the monodromy representation of this fiber bundle  $\varrho:\bZ\to \mathrm{Aut}(\bZ^2)$,  
	\begin{equation*}
		\varrho(n)=M^n
	\end{equation*} 
	We have
	$$
	1\to \pi_1(\mathbb C^*\times \mathbb C^*)\to \pi_1(X)\to \pi_1(\mathbb C^*)\to 1
	$$
	Hence $\pi_1(X)$ is solvable.
	By \cref{claim:202210093}, $\pi_1(X)$ is not virtually nilpotent.
	
\end{example}

\section{A structure theorem of varieties with $\pi_1$ admitting big $\&$ reductive Representations}      
In this section, we prove \cref{thm:20230510}.
\subsection{A structure theorem}
Before going to prove this, we prepare the following generalization of a structure theorem for quasi-projective varieties of maximal quasi-albanese dimension in \cref{lem:abelian pi0}. 
\begin{thm}\label{thm:202305101} 
Let $X$ be a smooth quasi-projective variety and let $g:X\to \cA\times Y$ be a morphism, where $\cA$ is a semi-abelian variety and $Y$ is a smooth quasi-projective variety such that $\Spalg(Y)\subsetneqq Y$.
Let $p:X\to Y$ be the composition of $g:X\to \cA\times Y$ and the second projection $\cA\times Y\to Y$.
Assume that $p:X\to Y$ is dominant and $\dim g(X)=\dim X$.
Then after replacing $X$ by a finite \'etale cover and a birational modification, there are a semiabelian variety $A$, a quasi-projective manifold $V$  and a birational morphism $a:X\to V$  such that the following commutative diagram holds:
		\begin{equation*} 
			\begin{tikzcd}
				X \arrow[rr,    "a"] \arrow[dr, "j"] & & V \arrow[ld, "h"]\\
				& J(X)&
			\end{tikzcd}
		\end{equation*}
		where $j$ is the logarithmic Iitaka fibration of $X$ and $h:V\to J(X)$ is a  locally trivial fibration with fibers isomorphic to $A$.  
Moreover, for a   general fiber $F$ of $j$, $a|_{F}:F\to A$ is proper in codimension one.  
\end{thm}

\begin{proof}
Consider the logarithmic Iitaka fibration $j: X\dashrightarrow J(X)$. We may replace $X$ by a birational modification such that $j$ is regular.    
Write $X_t:=j^{-1}(t)$.

\begin{claim}\label{claim:20230518}
 $p(X_t)$ is a point for very generic $t\in J(X)$.
 \end{claim}
 
 \begin{proof}
Since $p:X\to Y$ is dominant and $\Spalg(Y)\subsetneqq Y$, we have $\overline{p(X_t)}\not\subset \Spalg(Y)$ for generic $t\in J(X)$.
Hence, $\overline{p(X_t)}$ is of log general type for generic $t\in J(X)$.
Now we take very generic $t\in J(X)$.
To show that $p(X_t)$ is a point, we assume contrary that $\dim p(X_t)>0$.
Then $\bar{\kappa}(\overline{p(X_t)})>0$.
Since $\bar{\kappa}(X_t)=0$, general fibers of $p|_{X_t}:X_t\to \overline{p(X_t)}$ has  non-negative logarithmic Kodaira dimension.  By \cite[Theorem 1.9]{Fuj17} it follows that $\bar{\kappa}(X_t)\geq \bar{\kappa}(\overline{p(X_t)})>0$. We obtain the  contradiction. 
Thus $p(X_t)$ is a point. 
\end{proof}

Let $\alpha:X\to \cA$ be the composition of $g:X\to \cA\times Y$ and the first projection $\cA\times Y\to \cA$.
	Since $\dim X=\dim g(X)$,  one has $\dim X_t=\dim \alpha(X_t)$ for general $t\in J(X)$ by \cref{claim:20230518}.		
	For very generic $t\in J(X)$, note that $\bar{\kappa}(X_t)=0$, hence by \cref{prop:Koddimabb}, the closure of $\alpha(X_t)$ is a translate of a  semi-abelian variety $A_t$  of $\cA$.  
	Note that $\cA$ has only at most countably many  semi-abelian subvarieties, it follows that $A_t$ does not depend on very general $t$ which we denote by $B$.  
	\begin{claim}\label{claim:20221105}
		There are 
		\begin{equation*}
			\begin{tikzcd}
				X' \arrow[r] \arrow[d, "\gamma"] & X\arrow[d,"\alpha"]\\
				\cA' \arrow[r] & \cA
			\end{tikzcd}
		\end{equation*}
		where the two rows are a finite     \'etale cover $X'\to X$ and an isogeny $\cA'\to \cA$  such that for a very general fiber $F$ of the logarithmic Iitaka fibration $j':X'\to J(X')$ of $X'$,  $\gamma|_{F}:F\to \cA'$ is mapped birationally to a translate of a semiabelian subvariety $C$ of $\cA'$, and the induced map $F\to C$ is proper in codimension one.
	\end{claim}
	\begin{proof}
		For a very general fiber $X_t$ of $j$,  we know that $\alpha|_{X_t}:X_t\to \cA$ factors through a birational morphism $X_t\to C$ which is proper in codimension one and an isogeny   $C\to B$.  
 Let $\cA'\to \cA$ be the isogeny in \cref{lem:finite cover}  below such that $C\times_{\cA}\cA'$ is a disjoint union of $C$ and the natural morphism of $C\to \cA'$ is injective.  
 It follows that  for a connected component $X'$ of $X\times_A\cA'$, the fiber of $X'\to J(X)$ at $t$ is a disjoint union of $X_t$.  
 Consider  the quasi-Stein factorisation of $X'\to J(X)$  and we obtain an algebraic fiber space $X'\to I$ and a finite morphism $\beta:I\to J(X)$.   
 Then for each $t'\in I$ with $t=\beta(t')$, the fiber $X'_{t'}$ of $X'\to I$ is isomorphic to $X_t$ hence it has zero logarithmic Kodaira dimension.   
 By the universal property of logarithmic Iitaka fibration,   $X'_{t'}$  is contracted by the logarithmic Iitaka fibration of $X'\to J(X')$. 
 Since $\bar{\kappa}(X')=\bar{\kappa}(X)$ by \cref{lem:same Kd}, it follows that $X'\to I$ is the logarithmic Iitaka fibration of $X'$.   Moreover, by our construction, for the natural morphism $\gamma:X'\to \cA'$,  the induced map $\gamma|_{X'_{t'}}:X'_{t'}\to \cA'$ is mapped birationally to a translate of $C\subset \cA'$, and proper in codimension one.
 \end{proof}

We replace $X$ and $\cA$ by the above \'etale covers.   
Write $W$ to be the closure of $\alpha(X)$. Then $W$ is invariant under the action of $C$ by the translate.  Denote by $T:=W/C$. The natural morphism $X\to T$ contracts general fibers of $j$. Therefore after replacing by a birational model of $X\to J(X)$, one has
	\begin{equation*}
		\begin{tikzcd}
			X \arrow[r] \arrow[d] &  W\arrow[d] \\
			J(X)  \arrow[r] & T
		\end{tikzcd}
	\end{equation*} 
	Consider the natural morphism $a: X\to J(X)\times_T W$. Then it is birational by the above claim.   Moreover,  for a very general fiber $X_t$, the morphism $a|_{X_t}:X_t\to (J(X)\times_T W)_t$ 
	 is proper in codimension one by \cref{claim:20221105}.  
	 	 
Set $V:=J(X)\times_T W$. 	 Let $\overline{X}$ be a partial compactification of $X$ such that $a:X\to V$ extends to  a proper morphism $\bar{a}:\overline{X}\to V$.   Then $\Xi:=\bar{a}(\overline{X}\backslash X)$ is a Zariski closed subset of $V$.   By the above result  for a very general fiber $V_t$ of the fibration $V\to J(X)$  we know that $V_t\cap \Xi$ is of codimension at least two in $V_t$. By the semi-continuity it holds for a \emph{general} fiber $V_t$.   Since $ \overline{X}\backslash \bar{a}^{-1}(\Xi)\subset X$,  we conclude that for a  general fiber $X_t$ of $X\to J(X)$, the morphism $a|_{X_t}:X_t\to V_t$ 
	 is proper in codimension.  The theorem is proved. 
\end{proof}

\begin{lem}\label{lem:finite cover}
	Let $A$ be a semiabelian variety and let $B$ be a semiabelian subvariety of $A$. Let  $C\to B$ be a finite \'etale cover.  Then there is an isogeny $A'\to A$ such that   $C\times_AA'$ is a disjoint union of $C$ and the natural morphism $C\to A'$ is injective. 
\end{lem}
\begin{proof}
	We may write $A=\bC^n/\Gamma$ where $\Gamma$ is a lattice in $\bC^n$ such that there is a natural isomorphic $i: \pi_1(A)\to \Gamma$. Then there is a $\bC$-vector space $V\subset \bC^n$ such that $B=V/V\cap \Gamma$. It follows that $i({\rm Im}[\pi_1(B)\to \pi_1(A)])=\Gamma\cap V$.   Since $C\to B$ is an isogeny, then $i({\rm Im}[\pi_1(C)\to \pi_1(A)])$ is a finite index subgroup of $\Gamma\cap V$.  Let $\Gamma'\subset \Gamma$ be the finite index subgroup such that $\Gamma'\cap V=i({\rm Im}[\pi_1(C)\to \pi_1(A)])$. Therefore, for the semiabelian variety $A':=\bC^n/\Gamma'$,   the morphism $C\to A$  lifts to $C\to A'$, and it is moreover injective. Then the base change $C\times_A A'$ is identified with $C\times_BC$, which is a disjoint union of $C$.  The lemma is proved. 
\end{proof}

Now we prove \cref{thm:20230510} (i), (ii), (iii), which we restate as follows.

\begin{thm}[= \cref{thm:20230510} (i)-(iii)]\label{thm:structure}
	Let $X$ be a quasi-projective normal variety and let $\varrho:\pi_1(X)\to {\rm GL}_N(\bC)$ be a reductive and big representation.   Then 
	\begin{thmlist}
		\item \label{item:LKD} the logarithmic Kodaira dimension satsifies $\bar{\kappa}(X)\geq 0$.
Moreover, if $ \bar{\kappa}(X)= 0$, then $\pi_1(X)$ is virtually abelian.
		\item \label{item:Mori}There is a proper Zariski closed subset $\Xi$ of $X$ such that each non-constant morphism $\bA^1\to X$ has image in $\Xi$.
		\item \label{item:local trivial}After replacing $X$ by a finite \'etale cover and a birational modification, there are a    semiabelian variety $A$, a   quasi-projective manifold $V$  and a birational morphism $a:X\to V$  such that the following commutative diagram holds:
		\begin{equation*} 
			\begin{tikzcd}
				X \arrow[rr,    "a"] \arrow[dr, "j"] & & V \arrow[ld, "h"]\\
				& J(X)&
			\end{tikzcd}
		\end{equation*}
		where $j$ is the logarithmic Iitaka fibration of $X$ and $h:V\to J(X)$ is a  locally trivial fibration with fibers isomorphic to $A$.  Moreover, for a   general fiber $F$ of $j$, $a|_{F}:F\to A$ is proper in codimension one.  
	\end{thmlist}   
\end{thm}    
\begin{proof}
	To prove the theorem we are free to replace $X$ by a birational modification  and by a finite \'etale cover thanks to \cref{lem:fun,lem:KodairaDim}. 
	We apply \cref{lem:202305101}.
Then by replacing $X$ with a finite \'etale cover and a birational modification, we obtain a smooth quasi-projective variety $Y$ (might be zero-dimensional), a semiabelian variety $\cA$, and a morphism $g:X\to \cA \times Y$ that satisfy the following properties:
	\begin{itemize}
		\item  $\dim X=\dim g(X)$.
		\item Let $p:X\to Y$ be  the composition of $g$ with the projective map $\cA\times Y\to Y$. 
		Then  $p$   is dominant.
		\item $\Spalg(Y)\subsetneqq Y$ and $\Spp(Y)\subsetneqq Y$.
	\end{itemize}

	Let us  prove \cref{item:LKD}.
	Let $\alpha:X\to\cA$ be the composite of $g:X\to \cA\times Y$ and the first projection $\cA\times Y\to Y$.
	Let  $Z$ be a general fiber of $p$. 
	Then $\alpha|_Z:Z\to \cA$ satisfies $\dim Z=\dim \alpha(Z)$. 
	It follows that $\bar{\kappa}(Z)\geq 0$. By\cite[Theorem 1.9]{Fuj17} we obtain $\bar{\kappa}(X)\geq \bar{\kappa}(Y)+\bar{\kappa}(Z)$.  
Hence $\bar{\kappa}(X)\geq 0$.

Suppose $\bar{\kappa}(X)= 0$.
Then $\bar{\kappa}(Y)=0$.
By $\Spalg(Y)\subsetneqq Y$, we conclude that $\dim Y=0$.
Hence $\dim\alpha(X)=\dim X$.
By \cref{lem:abelian pi}, $\pi_1(X)$ is abelian.
The first claim is proved.

	\medspace
	
Let us  prove \cref{item:Mori}.  
Let $E\subsetneqq X$ be a proper Zariski closed set such that $g|_{X\backslash E}:X\backslash E\to \cA\times Y$ is quasi-finite.
Set $\Xi=E\cup p^{-1}(\Spp(Y))$.
Then $\Xi\subsetneqq X$.

We shall show that every non-constant algebraic morphism $\bA^1\to X$ has its image in $\Xi$.
Indeed, suppose $f:\bA^1\to X$ satisfies $f(\bA^1)\not\subset \Xi$.
Since $\cA$ does not contain $\bA^1$-curve, the composite $\alpha\circ f:\bA^1\to \cA$ is constant.
By $p\circ f(\bA^1)\not\subset \Spp(Y)$, $p\circ f:\bA^1\to Y$ is constant.
Hence $g\circ f:\bA^1\to \cA\times Y$ is constant.
By $f(\bA^1)\not\subset E$, $f$ is constant.
Thus we have proved that every non-constant algebraic morphism $\bA^1\to X$ has its image in $\Xi$. 

Finally \cref{item:local trivial} follows from \cref{thm:202305101}.
\end{proof}

\begin{rem}
If we assume the logarithmic abundance conjecture: a quasi-projective manifold is $\bA^1$-uniruled if and only if $\bar{\kappa}(X)=-\infty$, then it predicts that   $\bar{\kappa}(X)\geq 0$ if there is a big representation $\varrho:\pi_1(X)\to {\rm GL}_N(\bC)$, which is slightly stronger than the first claim in \cref{thm:structure}. Indeed, since $\varrho$ is big, $X$ is not $\bA^1$-uniruled and thus $\bar{\kappa}(X)\geq 0$ by this conjecture.  
\end{rem}

\subsection{A characterization of varieties birational to semi-abelian variety} 
In \cite{Yam10}, the third author established the following theorem: Let $X$ be a smooth projective variety equipped with a big representation $\varrho:\pi_1(X)\to {\rm GL}_N(\bC)$. If $X$ admits a Zariski dense entire curve, then after replacing $X$ by a  finite \'etale cover, its Albanese morphism $\alpha_X:X\to \cA_X$ is birational.

In \cref{thm:20230510}.(iv) we state a similar result  for smooth quasi-projective varieties $X$, provided that $\varrho$ is a reductive representation.   
\begin{proposition}[=\cref{thm:20230510}.(iv)]\label{thm:char}
	Let $Y$ be an $h$-special or   special quasi-projective manifold.    If $\varrho:\pi_1(Y)\to {\rm GL}_N(\bC)$ is  a  reductive and  big representation, then there is a finite \'etale cover $X$ of $Y$ such that   the quasi-Albanese morphism $\alpha:X\to \cA$ is  birational  and the induced morphism $\alpha_*:\pi_1(X)\to\pi_1(\cA)$ is an isomorphism.
In particular, $\pi_1(Y)$ is virtually abelian.
\end{proposition}
\begin{proof}
 By \cref{prop:factor} $\alpha$ is dominant with connected general fibers.	Since $\varrho$ is reductive, by \cref{main5},  there is a finite \'etale cover $X$ of $Y$   such that  $G:=\varrho(\pi_1(X))$ is  abelian and torsion free. It follows that $\varrho$ factors through  $H_1(X,\bZ)/{\rm torsion}$. 
	Since $\alpha_*:H_1(X,\bZ)/{\rm torsion}\to H_1(\cA,\bZ)$ is isomorphic,  $\varrho$ further factors through 	$H_1(\cA,\bZ)$.
	\begin{equation*}
		\begin{tikzcd}
			&	\pi_1(X) \arrow[d]\arrow[r, "\rho"] \arrow[ldd, bend right=30]& G\\
			&	H_1(X,\bZ)/{\rm torsion} \arrow[ru]\arrow[d, "\alpha_*"']&\\
			\pi_1(\cA)\arrow[r, "="]	&	H_1(\cA,\bZ)\arrow[ruu, bend right=30,, "\beta"']&
		\end{tikzcd}
	\end{equation*}
	From the above diagram for every fiber $F$ of $\alpha$,   $\varrho({\rm Im}[\pi_1(F)\to \pi_1(X)])$ is trivial.   Since $\varrho$ is big,  the general fiber of $\alpha$ is thus a point. Hence $\alpha$ is birational.  Since $\alpha:X\to \cA$ is $\pi_1$-exact by \cref{pro:202210131}, it follows that  $\alpha_*:\pi_1(X)\to \pi_1(\cA)$ is an isomorphism. 
\end{proof}
\begin{rem}\label{rem:sharp abelian} 
We would like to point out that \cref{thm:char} is a sharp result:
	  \begin{itemize}
	 	\item In contrast to the projective case which was proven in \cite{Yam10}, we require additionally $\varrho$ to be reductive for the result to hold.
	 	\item Unlike the situation described in \cref{lem:abelian pi0},  we cannot expect that the quasi-Albanese morphism $\alpha:X\to \cA$ is proper in codimension one.
	 \end{itemize}
 In the next subsection, we will provide examples to illustrate the above points.
\end{rem}

 \subsection{Remarks on \cref{thm:char}} 
In this subsection, we will provide examples to demonstrate the facts in \cref{rem:sharp abelian}.
\begin{lem}
	Let $X$ be the 	smooth quasi-projective surface constructed defined in \Cref{example}. Then 
	\begin{itemize}
		\item the variety $X$ is a log Calabi-Yau variety, i.e. there is a smooth projective compactification $\overline{X}$ of  $X$ with $D:=\overline{X}\backslash X$ a simple normal crossing divisor such that $K_{\overline{X}}+D$ is trivial. 
		In particular, $\bar{\kappa}(X)=0$.
		\item The  fundamental group  $\pi_1(X)$ is linear and large, i.e. for any closed  subvariety $Y\subset X$, the image ${\rm Im}[\pi_1(Y^{\rm norm})\to \pi_1(X)]$ is infinite.
		\item The quasi-Albanese morphism of $X$ is a fibration over an elliptic curve $B$ with fibers $\bC^*$.
		\item For any finite \'etale cover $\nu:\widehat{X}\to X$, its quasi-Albanese morphism $\alpha_{\widehat{X}}:\widehat{X}\to \cA_{\widehat{X}}$ is a surjective morphism to an elliptic curve $\cA_{\widehat{X}}$.
\end{itemize}
\end{lem}
\begin{proof}
	By the construction of $X$ in \Cref{example}, there exists a holomorphic line bundle $L$ over an elliptic curve $B$  such that $X=L\backslash D_1$ where $D_1$ is the zero section of $L$.  Denote by $\overline{X}:=\bP(L^*\oplus \cO_B)$ which is a smooth projective surface and write $\xi:=\cO_{\overline{X}}(1)$ for the tautological line bundle. Denote by  $\overline{\pi}:\overline{X}\to B$   the projection map.  Then   $\cO_{\overline{X}}(D_1)=\pi^*L+ \xi$.   Denote by $D_2:=\overline{X}\backslash L$. Then   $\cO_{\overline{X}}(D_2)=\xi$.   Note that $K_{\overline{X}}=-2\xi+\overline{\pi}^*(K_B+\det (L^*\oplus \cO_B))$. 	It follows that $K_{\overline{X}}+D_1+D_2=\cO_{\overline{X}}$. The first claim follows.

    By the Gysin sequence, we have 
    \begin{align}\label{eq:Gysin}
    	0\to H^1(B,\bZ)\stackrel{\pi^*}{\to} H^1(X,\bZ)\to H^0(B,\bZ)\stackrel{\cdot c_1(L)}{\to} H^2(B,\bZ)\to H^2(X,\bZ)\to H^1(B,\bZ)\to\cdots
    \end{align}   
    where $\pi:X\to B$ is the projection map.    By the functoriality of the quasi-Albanese morphism, we have the following diagram
\begin{equation*}
	\begin{tikzcd}
		X \arrow[r,"\pi"]\arrow[d,"\alpha_X"] & B\arrow[d,"\alpha_B","\simeq"']\\
		\cA_X\arrow[r,"\simeq"] & \cA_{B}
	\end{tikzcd}
\end{equation*}
where $\alpha_X$ and $\alpha_B$ are (quasi-)Albanese morphisms of $X$ and $B$ respectively.     By \Cref{re:20230428}, we know that $c_1(L)\neq 0$.
It follows from \eqref{eq:Gysin} that $\pi^*:H^1(B,\bZ)\to H^1(X,\bZ)$ is an isomorphism, and thus $\cA_X\to \cA_B$ is an isomorphism.   Since $B$ is an elliptic curve, $\alpha_B:B\to \cA_B$ is also an isomorphism. It proves that $\pi$ coincides with the quasi-Albanese morphism   $\alpha_X$. 

We will prove that $\pi_1(X)$ is large. Since $\pi_1(X)$ is infinite, it suffices to check all irreducible curves $Y$ of $X$.    If $Y$ is a fiber of $\pi$, then $Y\simeq \bC^*$ and the claim follows from the exact sequence
$$
0=\pi_2(B)\to \pi_1(\bC^*)\to \pi_1(X)\to \pi_1(B).
$$   
If $Y$ is not a fiber of $\pi$, then $\pi|_{Y}:Y\to B$ is a  finite morphism, and thus by \cref{lem:finiteindex}  ${\rm Im}[\pi_1(Y^{\rm norm})\to \pi_1(B)]$ is a finite index subgroup of $\pi_1(B)$ which is thus infinite. It follows that   ${\rm Im}[\pi_1(Y^{\rm norm})\to \pi_1(X)]$ is infinite. 

Let us prove the last assertion. It is important to note that $\nu^*: H^1(X, \mathbb{C}) \to H^1(\widehat{X}, \mathbb{C})$ is   injective with a finite cokernel. From this, we deduce that $\dim_\bC H^1(\cA_{\widehat{X}},\bC)=\dim_\bC H^1(\widehat{X},\bC)=\dim_\bC H^1({X},\bC)=2$.  Therefore, $ \cA_{\widehat{X}}$ is   either an elliptic curve or $(\bC^*)^2$.  However, it is worth noting that $\nu$ induces a non-constant morphism $\cA_{\widehat{X}} \to \cA_X=B$. This implies that $\cA_{\widehat{X}}$ cannot be $(\mathbb{C}^*)^2$. To see this, consider the algebraic morphism from $\mathbb{C}^*$ to an elliptic curve $B$, which can be extended to $\mathbb{P}^1 = \mathbb{C}^* \cup \{0,\infty\}$. This extension must be constant, ruling out the possibility of $\cA_{\widehat{X}}$ being $(\mathbb{C}^*)^2$.
\end{proof}
The above lemma shows that \cref{thm:char} does not hold if $\varrho$ is not reductive. 

\begin{lem} \label{lem:20230507}
	Let $X$ be the quasi-projective surface constructed in \Cref{example:20221105}, which is special and $h$-special.   Then
	\begin{itemize}
	 \item for the quasi-Albanese morphism $\alpha:X\to \cA_X$,   $\alpha_*:\pi_1(X)\to \pi_1(\cA_X)$ is an isomorphism.  
\item $\pi_1(X)$ is linear reductive and  large.
	\item For any finite \'etale cover $\nu:\widehat{X}\to X$, its quasi-Albanese morphism $\alpha_{\widehat{X}}:\widehat{X}\to \cA_{\widehat{X}}$ is birational but not proper in codimension one. 
	\end{itemize} 
\end{lem}
\begin{proof}
	We will use the notations in \Cref{example:20221105}.  The first statement follows from \Cref{thm:char}. 
 
 	We aim to show that $\pi_1(X)$ is large. Since $\cA_X$ is positive-dimensional by the construction of $X$, the first statement implies that $\pi_1(X)$ is infinite. Thus, it suffices to check all irreducible curves $Y$ of $X$. Consider the projection maps $q_1:X\to C_1$ and $q_2:X\to C_2$ where $C_1$ and $C_2$ are two elliptic curves constructed in \Cref{example:20221105}. Since $X\to C_1\times C_2$ is birational,  $q_i|_{Y}:Y\to C_i$ is dominant for some $i=1,2$. By Lemma \ref{lem:finiteindex}, for some $i=1,2$, ${\rm Im}[\pi_1(Y^{\rm norm})\to \pi_1(C_i)]$ is a finite index subgroup of $\pi_1(C_i)$ which is thus infinite. It follows that ${\rm Im}[\pi_1(Y^{\rm norm})\to \pi_1(X)]$ is infinite. Since $\pi_1(\cA_X)$ is an abelian group, it follows that $\pi_1(X)$ is linear reductive and large.

 Since $X$ is special and $h$-special, we know from \cref{thm:char} that its quasi-Albanese morphism $\alpha:X\to \cA_X$ is birational. Moreover, by the construction of $X$ in \Cref{example:20221105}, we note that there is a birational morphism $g:X\to A$ where $A=C_1\times C_2$ is an abelian variety. Therefore, $g$ coincides with $\alpha$.      We further observe that $g(X)\subset (C_1-{0})\times C_2\cup {(0,0)}$. As a result, it is not proper in codimension one.
 
 Let us prove the last assertion. Since $\pi_1(\widehat{X})$ is a finite index subgroup of $\pi_1(\cA_X)$ and $\alpha_*:\pi_1(X)\to \pi_1(\cA_X)$ is an isomorphism, we can consider a finite étale cover $ \widehat{\cA} \to \cA_X$ associated with the finite index subgroup  $\alpha_*(\pi_1(\widehat{X}))$   of $\pi_1(\cA_X)$. Then  $\widehat{\cA}$ is also an abelian variety and there exists  a morphism $f:\widehat{X}\to \widehat{\cA}$ satisfying the following commutative diagram 
 \begin{equation*}
 	\begin{tikzcd}
 		\widehat{X} \arrow[r,"\pi"]\arrow[d,"f"] & X\arrow[d,"\alpha"]\\
 		\widehat{\cA}\arrow[r] & \cA_{X}
 	\end{tikzcd}
 \end{equation*}
 As $\alpha$ is birational, we have $\widehat{X}=X\times_{\cA_X}\widehat{\cA}$ which implies that $f$ is birational but not proper is codimension one as $\alpha$ is not proper is codimension one.   Furthermore, since $f_*: \pi_1(\widehat{X}) \to \pi_1(\widehat{\cA})$ is an isomorphism, we conclude that  $f$ coincides with the quasi-Albanese morphism $\alpha_{\widehat{X}}:\widetilde{X}\to \cA_{\widehat{X}}$.   
This completes the proof of the last claim.
\end{proof}
The above lemma shows that in \cref{thm:char} we cannot expect that the quasi-Albanese morphism is   proper in codimension one.


\begin{thebibliography}{BDDM22}
 	\newcommand{\enquote}[1]{``#1''}
 	\providecommand{\url}[1]{\texttt{#1}}
 	\providecommand{\urlprefix}{URL }
 	\providecommand{\bibinfo}[2]{#2}
 	\providecommand{\eprint}[2][]{\url{#2}}
 	
 	\bibitem[AB94]{AB}
 	\bibinfo{author}{N.~A'Campo} \& \bibinfo{author}{M.~Burger}.
 	\newblock \enquote{\bibinfo{title}{Arithmetic lattices and commensurator
 			according to {G}. {A}. {Margulis}}}.
 	\newblock \emph{\bibinfo{journal}{Invent. Math.}},
 	\textbf{\bibinfo{volume}{116}(\bibinfo{year}{1994})(\bibinfo{number}{1-3})}:\bibinfo{pages}{1--25}.
 	\newblock \urlprefix\url{http://dx.doi.org/10.1007/BF01231555}.
 	
 	\bibitem[AB08]{AB08}
 	\bibinfo{author}{P.~Abramenko} \& \bibinfo{author}{K.~S. Brown}.
 	\newblock \emph{\bibinfo{title}{Buildings. {Theory} and applications.}}, vol.
 	\bibinfo{volume}{248} of \emph{\bibinfo{series}{Grad. Texts Math.}}
 	\newblock \bibinfo{publisher}{Berlin: Springer} (\bibinfo{year}{2008}).
 	\newblock \urlprefix\url{http://dx.doi.org/10.1007/978-0-387-78835-7}.
 	
 	\bibitem[ADH16]{Ara16}
 	\bibinfo{author}{D.~Arapura}, \bibinfo{author}{A.~Dimca} \&
 	\bibinfo{author}{R.~Hain}.
 	\newblock \enquote{\bibinfo{title}{On the fundamental groups of normal
 			varieties}}.
 	\newblock \emph{\bibinfo{journal}{Commun. Contemp. Math.}},
 	\textbf{\bibinfo{volume}{18}(\bibinfo{year}{2016})(\bibinfo{number}{4})}:\bibinfo{pages}{17}.
 	\newblock \bibinfo{note}{Id/No 1550065},
 	\urlprefix\url{http://dx.doi.org/10.1142/S0219199715500650}.
 	
 	\bibitem[AN99]{AN99}
 	\bibinfo{author}{D.~Arapura} \& \bibinfo{author}{M.~Nori}.
 	\newblock \enquote{\bibinfo{title}{Solvable fundamental groups of algebraic
 			varieties and {K{\"a}hler} manifolds}}.
 	\newblock \emph{\bibinfo{journal}{Compos. Math.}},
 	\textbf{\bibinfo{volume}{116}(\bibinfo{year}{1999})(\bibinfo{number}{2})}:\bibinfo{pages}{173--188}.
 	\newblock \urlprefix\url{http://dx.doi.org/10.1023/A:1000879906578}.
 	
 	\bibitem[And92]{And92}
 	\bibinfo{author}{Y.~Andr{\'e}}.
 	\newblock \enquote{\bibinfo{title}{Mumford-{Tate} groups of mixed {Hodge}
 			structures and the theorem of the fixed part}}.
 	\newblock \emph{\bibinfo{journal}{Compos. Math.}},
 	\textbf{\bibinfo{volume}{82}(\bibinfo{year}{1992})(\bibinfo{number}{1})}:\bibinfo{pages}{1--24}.
 	
 	\bibitem[BC20]{BC20}
 	\bibinfo{author}{Y.~Brunebarbe} \& \bibinfo{author}{B.~Cadorel}.
 	\newblock \enquote{\bibinfo{title}{Hyperbolicity of varieties supporting a
 			variation of {Hodge} structure}}.
 	\newblock \emph{\bibinfo{journal}{Int. Math. Res. Not.}},
 	\textbf{\bibinfo{volume}{2020}(\bibinfo{year}{2020})(\bibinfo{number}{6})}:\bibinfo{pages}{1601--1609}.
 	\newblock \urlprefix\url{http://dx.doi.org/10.1093/imrn/rny054}.
 	
 	\bibitem[BDDM22]{BDDM}
 	\bibinfo{author}{D.~Brotbek}, \bibinfo{author}{G.~Daskalopoulos},
 	\bibinfo{author}{Y.~Deng} \& \bibinfo{author}{C.~Mese}.
 	\newblock \enquote{\bibinfo{title}{{Representations of fundamental groups and
 				logarithmic symmetric differential forms}}}.
 	\newblock \emph{\bibinfo{journal}{HAL preprint}},
 	\textbf{(\bibinfo{year}{2022})}.
 	\newblock \urlprefix\url{https://hal.archives-ouvertes.fr/hal-03839053}.
 	
 	\bibitem[BEZ14]{BZ14}
 	\bibinfo{author}{P.~Brosnan} \& \bibinfo{author}{F.~El~Zein}.
 	\newblock \enquote{\bibinfo{title}{Variations of mixed {Hodge} structure}}.
 	\newblock \enquote{\bibinfo{booktitle}{Hodge theory. Based on lectures
 			delivered at the summer school on Hodge theory and related topics, Trieste,
 			Italy, June 14 -- July 2, 2010}}, \bibinfo{pages}{333--409}.
 	\bibinfo{publisher}{Princeton, NJ: Princeton University Press}
 	(\bibinfo{year}{2014}):\hspace{0pt}.
 	
 	\bibitem[Bor72]{Bor72}
 	\bibinfo{author}{A.~Borel}.
 	\newblock \enquote{\bibinfo{title}{Some metric properties of arithmetic
 			quotients of symmetric spaces and an extension theorem}}.
 	\newblock \emph{\bibinfo{journal}{J. Differential Geometry}},
 	\textbf{\bibinfo{volume}{6}(\bibinfo{year}{1972})}:\bibinfo{pages}{543--560}.
 	\newblock \urlprefix\url{http://projecteuclid.org/euclid.jdg/1214430642}.
 	
 	\bibitem[{Bru}22]{Bru22}
 	\bibinfo{author}{Y.~{Brunebarbe}}.
 	\newblock \enquote{\bibinfo{title}{{Hyperbolicity in presence of a large local
 				system}}}.
 	\newblock \emph{\bibinfo{journal}{arXiv e-prints}},
 	\textbf{(\bibinfo{year}{2022})}:\bibinfo{eid}{arXiv:2207.03283}.
 	\newblock \eprint{2207.03283},
 	\urlprefix\url{http://dx.doi.org/10.48550/arXiv.2207.03283}.
 	
 	\bibitem[BV12]{BV12}
 	\bibinfo{author}{N.~Borne} \& \bibinfo{author}{A.~Vistoli}.
 	\newblock \enquote{\bibinfo{title}{Parabolic sheaves on logarithmic schemes}}.
 	\newblock \emph{\bibinfo{journal}{Adv. Math.}},
 	\textbf{\bibinfo{volume}{231}(\bibinfo{year}{2012})(\bibinfo{number}{3-4})}:\bibinfo{pages}{1327--1363}.
 	\newblock \urlprefix\url{http://dx.doi.org/10.1016/j.aim.2012.06.015}.
 	
 	\bibitem[Cam91]{Cam91}
 	\bibinfo{author}{F.~Campana}.
 	\newblock \enquote{\bibinfo{title}{{On twistor spaces of the class
 				$\mathcal{C}$}}}.
 	\newblock \emph{\bibinfo{journal}{Journal of Differential Geometry}},
 	\textbf{\bibinfo{volume}{33}(\bibinfo{year}{1991})(\bibinfo{number}{2})}:\bibinfo{pages}{541
 		-- 549}.
 	\newblock \urlprefix\url{http://dx.doi.org/10.4310/jdg/1214446329}.
 	
 	\bibitem[Cam98]{Cam98}
 	---{}---{}---.
 	\newblock \enquote{\bibinfo{title}{Negativity of compact curves in infinite
 			covers of projective surfaces}}.
 	\newblock \emph{\bibinfo{journal}{J. Algebr. Geom.}},
 	\textbf{\bibinfo{volume}{7}(\bibinfo{year}{1998})(\bibinfo{number}{4})}:\bibinfo{pages}{673--693}.
 	
 	\bibitem[Cam01]{Cam01}
 	---{}---{}---.
 	\newblock \enquote{\bibinfo{title}{Green-{Lazarsfeld} sets and resolvable
 			quotients of {K{\"a}hler} groups}}.
 	\newblock \emph{\bibinfo{journal}{J. Algebr. Geom.}},
 	\textbf{\bibinfo{volume}{10}(\bibinfo{year}{2001})(\bibinfo{number}{4})}:\bibinfo{pages}{599--622}.
 	
 	\bibitem[Cam04]{Cam04}
 	---{}---{}---.
 	\newblock \enquote{\bibinfo{title}{Orbifolds, special varieties and
 			classification theory.}}
 	\newblock \emph{\bibinfo{journal}{Ann. Inst. Fourier}},
 	\textbf{\bibinfo{volume}{54}(\bibinfo{year}{2004})(\bibinfo{number}{3})}:\bibinfo{pages}{499--630}.
 	\newblock \urlprefix\url{http://dx.doi.org/10.5802/aif.2027}.
 	
 	\bibitem[Cam11a]{Cam11}
 	---{}---{}---.
 	\newblock \enquote{\bibinfo{title}{Orbifoldes g\'{e}om\'{e}triques
 			sp\'{e}ciales et classification bim\'{e}romorphe des vari\'{e}t\'{e}s
 			k\"{a}hl\'{e}riennes compactes}}.
 	\newblock \emph{\bibinfo{journal}{J. Inst. Math. Jussieu}},
 	\textbf{\bibinfo{volume}{10}(\bibinfo{year}{2011})(\bibinfo{number}{4})}:\bibinfo{pages}{809--934}.
 	\newblock \urlprefix\url{http://dx.doi.org/10.1017/S1474748010000101}.
 	
 	\bibitem[Cam11b]{Cam11b}
 	---{}---{}---.
 	\newblock \enquote{\bibinfo{title}{Special orbifolds and birational
 			classification: a survey}}.
 	\newblock \enquote{\bibinfo{booktitle}{Classification of algebraic varieties.
 			Based on the conference on classification of varieties, Schiermonnikoog,
 			Netherlands, May 2009.}}, \bibinfo{pages}{123--170}.
 	\bibinfo{publisher}{Z{\"u}rich: European Mathematical Society (EMS)}
 	(\bibinfo{year}{2011}):\hspace{0pt}.
 	
 	\bibitem[Cas76]{Cas76}
 	\bibinfo{author}{J.~W.~S. Cassels}.
 	\newblock \enquote{\bibinfo{title}{An embedding theorem for fields}}.
 	\newblock \emph{\bibinfo{journal}{Bull. Austral. Math. Soc.}},
 	\textbf{\bibinfo{volume}{14}(\bibinfo{year}{1976})(\bibinfo{number}{2})}:\bibinfo{pages}{193--198}.
 	\newblock \urlprefix\url{http://dx.doi.org/10.1017/S000497270002503X}.
 	
 	\bibitem[CCE15]{CCE15}
 	\bibinfo{author}{F.~Campana}, \bibinfo{author}{B.~Claudon} \&
 	\bibinfo{author}{P.~Eyssidieux}.
 	\newblock \enquote{\bibinfo{title}{Linear representations of {K{\"a}hler}
 			groups: factorizations and linear {Shafarevich} conjecture}}.
 	\newblock \emph{\bibinfo{journal}{Compos. Math.}},
 	\textbf{\bibinfo{volume}{151}(\bibinfo{year}{2015})(\bibinfo{number}{2})}:\bibinfo{pages}{351--376}.
 	\newblock \urlprefix\url{http://dx.doi.org/10.1112/S0010437X14007751}.
 	
 	\bibitem[CD21]{CD21}
 	\bibinfo{author}{B.~{Cadorel}} \& \bibinfo{author}{Y.~{Deng}}.
 	\newblock \enquote{\bibinfo{title}{{Picard hyperbolicity of manifolds admitting
 				nilpotent harmonic bundles}}}.
 	\newblock \emph{\bibinfo{journal}{arXiv e-prints}},
 	\textbf{(\bibinfo{year}{2021})}:\bibinfo{eid}{arXiv:2107.07550}.
 	\newblock \eprint{2107.07550}.
 	
 	\bibitem[CP19]{CP19}
 	\bibinfo{author}{F.~Campana} \& \bibinfo{author}{M.~P{\u{a}}un}.
 	\newblock \enquote{\bibinfo{title}{Foliations with positive slopes and
 			birational stability of orbifold cotangent bundles}}.
 	\newblock \emph{\bibinfo{journal}{Publ. Math., Inst. Hautes {\'E}tud. Sci.}},
 	\textbf{\bibinfo{volume}{129}(\bibinfo{year}{2019})}:\bibinfo{pages}{1--49}.
 	\newblock \urlprefix\url{http://dx.doi.org/10.1007/s10240-019-00105-w}.
 	
 	\bibitem[CS08]{CS08}
 	\bibinfo{author}{K.~Corlette} \& \bibinfo{author}{C.~Simpson}.
 	\newblock \enquote{\bibinfo{title}{On the classification of rank-two
 			representations of quasiprojective fundamental groups}}.
 	\newblock \emph{\bibinfo{journal}{Compos. Math.}},
 	\textbf{\bibinfo{volume}{144}(\bibinfo{year}{2008})(\bibinfo{number}{5})}:\bibinfo{pages}{1271--1331}.
 	\newblock \urlprefix\url{http://dx.doi.org/10.1112/S0010437X08003618}.
 	
 	\bibitem[CW16]{CW16}
 	\bibinfo{author}{F.~Campana} \& \bibinfo{author}{J.~Winkelmann}.
 	\newblock \enquote{\bibinfo{title}{Rational connectedness and order of
 			non-degenerate meromorphic maps from {{\(\mathbb C^n\)}}}}.
 	\newblock \emph{\bibinfo{journal}{Eur. J. Math.}},
 	\textbf{\bibinfo{volume}{2}(\bibinfo{year}{2016})(\bibinfo{number}{1})}:\bibinfo{pages}{87--95}.
 	\newblock \urlprefix\url{http://dx.doi.org/10.1007/s40879-015-0083-z}.
 	
 	\bibitem[Del71]{Del71}
 	\bibinfo{author}{P.~Deligne}.
 	\newblock \enquote{\bibinfo{title}{Th{\'e}orie de {Hodge}. {II}. ({Hodge}
 			theory. {II})}}.
 	\newblock \emph{\bibinfo{journal}{Publ. Math., Inst. Hautes {\'E}tud. Sci.}},
 	\textbf{\bibinfo{volume}{40}(\bibinfo{year}{1971})}:\bibinfo{pages}{5--57}.
 	\newblock \urlprefix\url{http://dx.doi.org/10.1007/BF02684692}.
 	
 	\bibitem[Del10]{Del10}
 	\bibinfo{author}{T.~Delzant}.
 	\newblock \enquote{\bibinfo{title}{The {Bieri}-{Neumann}-{Strebel} invariant of
 			fundamental groups of {K{\"a}hler} manifolds}}.
 	\newblock \emph{\bibinfo{journal}{Math. Ann.}},
 	\textbf{\bibinfo{volume}{348}(\bibinfo{year}{2010})(\bibinfo{number}{1})}:\bibinfo{pages}{119--125}.
 	\newblock \urlprefix\url{http://dx.doi.org/10.1007/s00208-009-0468-8}.
 	
 	\bibitem[{Den}22]{Den22}
 	\bibinfo{author}{Y.~{Deng}}.
 	\newblock \enquote{\bibinfo{title}{A characterization of complex
 			quasi-projective manifolds uniformized by unit balls}}.
 	\newblock \emph{\bibinfo{journal}{Math. Ann.}},
 	\textbf{(\bibinfo{year}{2022})}.
 	\newblock
 	\urlprefix\url{http://dx.doi.org/https://doi.org/10.1007/s00208-021-02334-z}.
 	
 	\bibitem[{Den}23]{Denarxiv}
 	---{}---{}---.
 	\newblock \enquote{\bibinfo{title}{{Big Picard theorems and algebraic
 				hyperbolicity for varieties admitting a variation of Hodge structures}}}.
 	\newblock \emph{\bibinfo{journal}{{Épijournal de Géométrie Algébrique}}},
 	\textbf{\bibinfo{volume}{{Volume 7}}(\bibinfo{year}{2023})}.
 	\newblock \urlprefix\url{http://dx.doi.org/10.46298/epiga.2023.volume7.8393}.
 	
 	\bibitem[DM21]{DMunique}
 	\bibinfo{author}{G.~{Daskalopoulos}} \& \bibinfo{author}{C.~{Mese}}.
 	\newblock \enquote{\bibinfo{title}{{Uniqueness of equivariant harmonic maps to
 				symmetric spaces and buildings}}}.
 	\newblock \emph{\bibinfo{journal}{arXiv e-prints}},
 	\textbf{(\bibinfo{year}{2021})}:\bibinfo{eid}{arXiv:2111.11422}.
 	\newblock \eprint{2111.11422}.
 	
 	\bibitem[DY24]{DY23b}
 	\bibinfo{author}{Y.~{Deng}} \& \bibinfo{author}{K.~{Yamanoi}}.
 	\newblock \enquote{\bibinfo{title}{{Linear Shafarevich conjecture in positive
 				characteristic, hyperbolicity and applications}}}.
 	\newblock \emph{\bibinfo{journal}{Preprint to appear}},
 	\textbf{(\bibinfo{year}{2024})}.
 	
 	\bibitem[DYK23]{DY23}
 	\bibinfo{author}{Y.~{Deng}}, \bibinfo{author}{K.~{Yamanoi}} \&
 	\bibinfo{author}{L.~{Katzarkov}}.
 	\newblock \enquote{\bibinfo{title}{{Reductive Shafarevich Conjecture}}}.
 	\newblock \emph{\bibinfo{journal}{arXiv e-prints}},
 	\textbf{(\bibinfo{year}{2023})}:\bibinfo{eid}{arXiv:2306.03070}.
 	\newblock \eprint{2306.03070},
 	\urlprefix\url{http://dx.doi.org/10.48550/arXiv.2306.03070}.
 	
 	\bibitem[EKPR12]{EKPR12}
 	\bibinfo{author}{P.~Eyssidieux}, \bibinfo{author}{L.~Katzarkov},
 	\bibinfo{author}{T.~Pantev} \& \bibinfo{author}{M.~Ramachandran}.
 	\newblock \enquote{\bibinfo{title}{Linear {Shafarevich} conjecture}}.
 	\newblock \emph{\bibinfo{journal}{Ann. Math. (2)}},
 	\textbf{\bibinfo{volume}{176}(\bibinfo{year}{2012})(\bibinfo{number}{3})}:\bibinfo{pages}{1545--1581}.
 	\newblock \urlprefix\url{http://dx.doi.org/10.4007/annals.2012.176.3.4}.
 	
 	\bibitem[{Eys}04]{Eys04}
 	\bibinfo{author}{P.~{Eyssidieux}}.
 	\newblock \enquote{\bibinfo{title}{{Sur la convexit\'e holomorphe des
 				rev\^etements lin\'eaires r\'eductifs d'une vari\'et\'e projective
 				alg\'ebrique complexe}}}.
 	\newblock \emph{\bibinfo{journal}{{Invent. Math.}}},
 	\textbf{\bibinfo{volume}{156}(\bibinfo{year}{2004})(\bibinfo{number}{3})}:\bibinfo{pages}{503--564}.
 	\newblock \urlprefix\url{http://dx.doi.org/10.1007/s00222-003-0345-0}.
 	
 	\bibitem[Fuj15]{Fuj15}
 	\bibinfo{author}{O.~Fujino}.
 	\newblock \enquote{\bibinfo{title}{On quasi-albanese maps}}.
 	\newblock \emph{\bibinfo{journal}{preprint}}, \textbf{(\bibinfo{year}{2015})}.
 	\newblock
 	\urlprefix\url{https://www.math.kyoto-u.ac.jp/~fujino/quasi-albanese2.pdf}.
 	
 	\bibitem[Fuj17]{Fuj17}
 	---{}---{}---.
 	\newblock \enquote{\bibinfo{title}{Notes on the weak positivity theorems}}.
 	\newblock \enquote{\bibinfo{booktitle}{Algebraic varieties and automorphism
 			groups. Proceedings of the workshop held at RIMS, Kyoto University, Kyoto,
 			Japan, July 7--11, 2014}}, \bibinfo{pages}{73--118}.
 	\bibinfo{publisher}{Tokyo: Mathematical Society of Japan (MSJ)}
 	(\bibinfo{year}{2017}):\hspace{0pt}.
 	
 	\bibitem[GGK22]{GGK22}
 	\bibinfo{author}{M.~{Green}}, \bibinfo{author}{P.~{Griffiths}} \&
 	\bibinfo{author}{L.~{Katzarkov}}.
 	\newblock \enquote{\bibinfo{title}{{Shafarevich mappings and period
 				mappings}}}.
 	\newblock \emph{\bibinfo{journal}{arXiv e-prints}},
 	\textbf{(\bibinfo{year}{2022})}:\bibinfo{eid}{arXiv:2209.14088}.
 	\newblock \eprint{2209.14088}.
 	
 	\bibitem[GS92]{GS92}
 	\bibinfo{author}{M.~{Gromov}} \& \bibinfo{author}{R.~{Schoen}}.
 	\newblock \enquote{\bibinfo{title}{{Harmonic maps into singular spaces and
 				$p$-adic superrigidity for lattices in groups of rank one}}}.
 	\newblock \emph{\bibinfo{journal}{{Publ. Math., Inst. Hautes \'Etud. Sci.}}},
 	\textbf{\bibinfo{volume}{76}(\bibinfo{year}{1992})}:\bibinfo{pages}{165--246}.
 	\newblock \urlprefix\url{http://dx.doi.org/10.1007/BF02699433}.
 	
 	\bibitem[IS07]{IS07}
 	\bibinfo{author}{J.~N.~N. Iyer} \& \bibinfo{author}{C.~T. Simpson}.
 	\newblock \enquote{\bibinfo{title}{A relation between the parabolic {C}hern
 			characters of the de {R}ham bundles}}.
 	\newblock \emph{\bibinfo{journal}{Math. Ann.}},
 	\textbf{\bibinfo{volume}{338}(\bibinfo{year}{2007})(\bibinfo{number}{2})}:\bibinfo{pages}{347--383}.
 	\newblock \urlprefix\url{http://dx.doi.org/10.1007/s00208-006-0078-7}.
 	
 	\bibitem[Kat97]{Kat97}
 	\bibinfo{author}{L.~Katzarkov}.
 	\newblock \enquote{\bibinfo{title}{On the {Shafarevich} maps}}.
 	\newblock \enquote{\bibinfo{booktitle}{Algebraic geometry. Proceedings of the
 			Summer Research Institute, Santa Cruz, CA, USA, July 9--29, 1995}},
 	\bibinfo{pages}{173--216}. \bibinfo{publisher}{Providence, RI: American
 		Mathematical Society} (\bibinfo{year}{1997}):\hspace{0pt}.
 	
 	\bibitem[Kaw81]{Kaw81}
 	\bibinfo{author}{Y.~Kawamata}.
 	\newblock \enquote{\bibinfo{title}{Characterization of abelian varieties}}.
 	\newblock \emph{\bibinfo{journal}{Compos. Math.}},
 	\textbf{\bibinfo{volume}{43}(\bibinfo{year}{1981})}:\bibinfo{pages}{253--276}.
 	
 	\bibitem[Kli03]{Kli03}
 	\bibinfo{author}{B.~Klingler}.
 	\newblock \enquote{\bibinfo{title}{On the rigidity of certain fundamental
 			groups, arithmeticity of complex hyperbolic lattices, and the ``fake
 			projective planes''}}.
 	\newblock \emph{\bibinfo{journal}{Invent. Math.}},
 	\textbf{\bibinfo{volume}{153}(\bibinfo{year}{2003})(\bibinfo{number}{1})}:\bibinfo{pages}{105--143}.
 	\newblock \urlprefix\url{http://dx.doi.org/10.1007/s00222-002-0283-2}.
 	
 	\bibitem[KO71]{KO71}
 	\bibinfo{author}{S.~Kobayashi} \& \bibinfo{author}{T.~Ochiai}.
 	\newblock \enquote{\bibinfo{title}{Satake compactification and the great
 			{P}icard theorem}}.
 	\newblock \emph{\bibinfo{journal}{J. Math. Soc. Japan}},
 	\textbf{\bibinfo{volume}{23}(\bibinfo{year}{1971})}:\bibinfo{pages}{340--350}.
 	\newblock \urlprefix\url{http://dx.doi.org/10.2969/jmsj/02320340}.
 	
 	\bibitem[Kol93]{Kol93}
 	\bibinfo{author}{J.~Koll{\'a}r}.
 	\newblock \enquote{\bibinfo{title}{Shafarevich maps and plurigenera of
 			algebraic varieties}}.
 	\newblock \emph{\bibinfo{journal}{Invent. Math.}},
 	\textbf{\bibinfo{volume}{113}(\bibinfo{year}{1993})(\bibinfo{number}{1})}:\bibinfo{pages}{177--215}.
 	\newblock \urlprefix\url{http://dx.doi.org/10.1007/BF01244307}.
 	
 	\bibitem[Kol95]{Kol95}
 	---{}---{}---.
 	\newblock \emph{\bibinfo{title}{Shafarevich maps and automorphic forms}}.
 	\newblock \bibinfo{publisher}{Princeton University Press},
 	\bibinfo{address}{Princeton (N.J.)} (\bibinfo{year}{1995}).
 	
 	\bibitem[Kol07]{Kol07}
 	\bibinfo{author}{J.~Koll\'{a}r}.
 	\newblock \emph{\bibinfo{title}{Lectures on resolution of singularities}}, vol.
 	\bibinfo{volume}{166} of \emph{\bibinfo{series}{Annals of Mathematics
 			Studies}}.
 	\newblock \bibinfo{publisher}{Princeton University Press, Princeton, NJ}
 	(\bibinfo{year}{2007}).
 	
 	\bibitem[Lan91]{Lan97}
 	\bibinfo{author}{S.~Lang}.
 	\newblock \emph{\bibinfo{title}{Number theory. {III}}},
 	vol.~\bibinfo{volume}{60} of \emph{\bibinfo{series}{Encyclopaedia of
 			Mathematical Sciences}}.
 	\newblock \bibinfo{publisher}{Springer-Verlag, Berlin} (\bibinfo{year}{1991}).
 	\newblock \bibinfo{note}{Diophantine geometry},
 	\urlprefix\url{http://dx.doi.org/10.1007/978-3-642-58227-1}.
 	
 	\bibitem[Laz04]{Laz04}
 	\bibinfo{author}{R.~Lazarsfeld}.
 	\newblock \emph{\bibinfo{title}{Positivity in algebraic geometry. {I}.
 			{Classical} setting: line bundles and linear series}},
 	vol.~\bibinfo{volume}{48} of \emph{\bibinfo{series}{Ergeb. Math. Grenzgeb.,
 			3. Folge}}.
 	\newblock \bibinfo{publisher}{Berlin: Springer} (\bibinfo{year}{2004}).
 	
 	\bibitem[LS18]{LS18}
 	\bibinfo{author}{A.~Langer} \& \bibinfo{author}{C.~Simpson}.
 	\newblock \enquote{\bibinfo{title}{Rank 3 rigid representations of projective
 			fundamental groups}}.
 	\newblock \emph{\bibinfo{journal}{Compos. Math.}},
 	\textbf{\bibinfo{volume}{154}(\bibinfo{year}{2018})(\bibinfo{number}{7})}:\bibinfo{pages}{1534--1570}.
 	\newblock \urlprefix\url{http://dx.doi.org/10.1112/S0010437X18007182}.
 	
 	\bibitem[Mil17]{Mil17}
 	\bibinfo{author}{J.~S. Milne}.
 	\newblock \emph{\bibinfo{title}{Algebraic groups. {The} theory of group schemes
 			of finite type over a field}}, vol. \bibinfo{volume}{170} of
 	\emph{\bibinfo{series}{Camb. Stud. Adv. Math.}}
 	\newblock \bibinfo{publisher}{Cambridge: Cambridge University Press}
 	(\bibinfo{year}{2017}).
 	\newblock \urlprefix\url{http://dx.doi.org/10.1017/9781316711736}.
 	
 	\bibitem[Moc06]{Moc06}
 	\bibinfo{author}{T.~Mochizuki}.
 	\newblock \enquote{\bibinfo{title}{Kobayashi-{H}itchin correspondence for tame
 			harmonic bundles and an application}}.
 	\newblock \emph{\bibinfo{journal}{Ast\'{e}risque}},
 	\textbf{(\bibinfo{year}{2006})(\bibinfo{number}{309})}:\bibinfo{pages}{viii+117}.
 	
 	\bibitem[Moc07a]{Moc07}
 	---{}---{}---.
 	\newblock \enquote{\bibinfo{title}{Asymptotic behaviour of tame harmonic
 			bundles and an application to pure twistor {$D$}-modules. {I}}}.
 	\newblock \emph{\bibinfo{journal}{Mem. Amer. Math. Soc.}},
 	\textbf{\bibinfo{volume}{185}(\bibinfo{year}{2007})(\bibinfo{number}{869})}:\bibinfo{pages}{xii+324}.
 	\newblock \urlprefix\url{http://dx.doi.org/10.1090/memo/0869}.
 	
 	\bibitem[Moc07b]{Moc07b}
 	---{}---{}---.
 	\newblock \enquote{\bibinfo{title}{Asymptotic behaviour of tame harmonic
 			bundles and an application to pure twistor {$D$}-modules. {II}}}.
 	\newblock \emph{\bibinfo{journal}{Mem. Amer. Math. Soc.}},
 	\textbf{\bibinfo{volume}{185}(\bibinfo{year}{2007})(\bibinfo{number}{870})}:\bibinfo{pages}{xii+565}.
 	\newblock \urlprefix\url{http://dx.doi.org/10.1090/memo/0870}.
 	
 	\bibitem[Mok92]{Mok92}
 	\bibinfo{author}{N.~Mok}.
 	\newblock \enquote{\bibinfo{title}{Factorization of semisimple discrete
 			representations of {K{\"a}hler} groups}}.
 	\newblock \emph{\bibinfo{journal}{Invent. Math.}},
 	\textbf{\bibinfo{volume}{110}(\bibinfo{year}{1992})(\bibinfo{number}{3})}:\bibinfo{pages}{557--614}.
 	\newblock \urlprefix\url{http://dx.doi.org/10.1007/BF01231345}.
 	
 	\bibitem[NW14]{NW13}
 	\bibinfo{author}{J.~Noguchi} \& \bibinfo{author}{J.~Winkelmann}.
 	\newblock \emph{\bibinfo{title}{Nevanlinna theory in several complex variables
 			and diophantine approximation}}, vol. \bibinfo{volume}{350} of
 	\emph{\bibinfo{series}{Grundlehren Math. Wiss.}}
 	\newblock \bibinfo{publisher}{Tokyo: Springer} (\bibinfo{year}{2014}).
 	\newblock \urlprefix\url{http://dx.doi.org/10.1007/978-4-431-54571-2}.
 	
 	\bibitem[NWY08]{NWY08}
 	\bibinfo{author}{J.~Noguchi}, \bibinfo{author}{J.~Winkelmann} \&
 	\bibinfo{author}{K.~Yamanoi}.
 	\newblock \enquote{\bibinfo{title}{The second main theorem for holomorphic
 			curves into semi-abelian varieties. {II}}}.
 	\newblock \emph{\bibinfo{journal}{Forum Math.}},
 	\textbf{\bibinfo{volume}{20}(\bibinfo{year}{2008})(\bibinfo{number}{3})}:\bibinfo{pages}{469--503}.
 	\newblock \urlprefix\url{http://dx.doi.org/10.1515/FORUM.2008.024}.
 	
 	\bibitem[NWY13]{NWY13}
 	---{}---{}---.
 	\newblock \enquote{\bibinfo{title}{Degeneracy of holomorphic curves into
 			algebraic varieties. {II}}}.
 	\newblock \emph{\bibinfo{journal}{Vietnam J. Math.}},
 	\textbf{\bibinfo{volume}{41}(\bibinfo{year}{2013})(\bibinfo{number}{4})}:\bibinfo{pages}{519--525}.
 	\newblock \urlprefix\url{http://dx.doi.org/10.1007/s10013-013-0051-1}.
 	
 	\bibitem[Ser64]{Ser64}
 	\bibinfo{author}{J.-P. Serre}.
 	\newblock \enquote{\bibinfo{title}{Exemples de vari\'{e}t\'{e}s projectives
 			conjugu\'{e}es non hom\'{e}omorphes}}.
 	\newblock \emph{\bibinfo{journal}{C. R. Acad. Sci. Paris}},
 	\textbf{\bibinfo{volume}{258}(\bibinfo{year}{1964})}:\bibinfo{pages}{4194--4196}.
 	
 	\bibitem[Shi68]{Shi68}
 	\bibinfo{author}{B.~Shiffman}.
 	\newblock \enquote{\bibinfo{title}{On the removal of singularities of analytic
 			sets}}.
 	\newblock \emph{\bibinfo{journal}{Mich. Math. J.}},
 	\textbf{\bibinfo{volume}{15}(\bibinfo{year}{1968})}:\bibinfo{pages}{111--120}.
 	\newblock \urlprefix\url{http://dx.doi.org/10.1307/mmj/1028999912}.
 	
 	\bibitem[Sik12]{Sik12}
 	\bibinfo{author}{A.~S. Sikora}.
 	\newblock \enquote{\bibinfo{title}{Character varieties}}.
 	\newblock \emph{\bibinfo{journal}{Trans. Am. Math. Soc.}},
 	\textbf{\bibinfo{volume}{364}(\bibinfo{year}{2012})(\bibinfo{number}{10})}:\bibinfo{pages}{5173--5208}.
 	\newblock \urlprefix\url{http://dx.doi.org/10.1090/S0002-9947-2012-05448-1}.
 	
 	\bibitem[Sim88]{Sim88}
 	\bibinfo{author}{C.~T. Simpson}.
 	\newblock \enquote{\bibinfo{title}{Constructing variations of {H}odge structure
 			using {Y}ang-{M}ills theory and applications to uniformization}}.
 	\newblock \emph{\bibinfo{journal}{J. Amer. Math. Soc.}},
 	\textbf{\bibinfo{volume}{1}(\bibinfo{year}{1988})(\bibinfo{number}{4})}:\bibinfo{pages}{867--918}.
 	\newblock \urlprefix\url{http://dx.doi.org/10.2307/1990994}.
 	
 	\bibitem[Sim90]{Sim90}
 	---{}---{}---.
 	\newblock \enquote{\bibinfo{title}{Harmonic bundles on noncompact curves}}.
 	\newblock \emph{\bibinfo{journal}{J. Amer. Math. Soc.}},
 	\textbf{\bibinfo{volume}{3}(\bibinfo{year}{1990})(\bibinfo{number}{3})}:\bibinfo{pages}{713--770}.
 	\newblock \urlprefix\url{http://dx.doi.org/10.2307/1990935}.
 	
 	\bibitem[Sim92]{Sim92}
 	---{}---{}---.
 	\newblock \enquote{\bibinfo{title}{Higgs bundles and local systems}}.
 	\newblock \emph{\bibinfo{journal}{Inst. Hautes \'{E}tudes Sci. Publ. Math.}},
 	\textbf{(\bibinfo{year}{1992})(\bibinfo{number}{75})}:\bibinfo{pages}{5--95}.
 	\newblock \urlprefix\url{http://www.numdam.org/item?id=PMIHES_1992__75__5_0}.
 	
 	\bibitem[Sim93]{Sim93}
 	---{}---{}---.
 	\newblock \enquote{\bibinfo{title}{Lefschetz theorems for the integral leaves
 			of a holomorphic one-form}}.
 	\newblock \emph{\bibinfo{journal}{Compos. Math.}},
 	\textbf{\bibinfo{volume}{87}(\bibinfo{year}{1993})(\bibinfo{number}{1})}:\bibinfo{pages}{99--113}.
 	
 	\bibitem[Siu75]{Siu75}
 	\bibinfo{author}{Y.~T. Siu}.
 	\newblock \enquote{\bibinfo{title}{Extension of meromorphic maps into
 			{K}\"{a}hler manifolds}}.
 	\newblock \emph{\bibinfo{journal}{Ann. of Math. (2)}},
 	\textbf{\bibinfo{volume}{102}(\bibinfo{year}{1975})(\bibinfo{number}{3})}:\bibinfo{pages}{421--462}.
 	\newblock \urlprefix\url{http://dx.doi.org/10.2307/1971038}.
 	
 	\bibitem[{Sta}22]{stacks-project}
 	\bibinfo{author}{T.~{Stacks project authors}}.
 	\newblock \enquote{\bibinfo{title}{The stacks project}}.
 	\newblock \bibinfo{howpublished}{\url{https://stacks.math.columbia.edu}}
 	(\bibinfo{year}{2022}).
 	
 	\bibitem[{Sun}22]{Sun22}
 	\bibinfo{author}{R.~{Sun}}.
 	\newblock \enquote{\bibinfo{title}{{Hyperbolicity of varieties with big linear
 				representation of $\pi_1$}}}.
 	\newblock \emph{\bibinfo{journal}{arXiv e-prints}},
 	\textbf{(\bibinfo{year}{2022})}:\bibinfo{eid}{arXiv:2202.00196}.
 	\newblock \eprint{2202.00196},
 	\urlprefix\url{http://dx.doi.org/10.48550/arXiv.2202.00196}.
 	
 	\bibitem[Sza94]{Sza94}
 	\bibinfo{author}{E.~Szab{\'o}}.
 	\newblock \enquote{\bibinfo{title}{Divisorial log terminal singularities}}.
 	\newblock \emph{\bibinfo{journal}{J. Math. Sci., Tokyo}},
 	\textbf{\bibinfo{volume}{1}(\bibinfo{year}{1994})(\bibinfo{number}{3})}:\bibinfo{pages}{631--639}.
 	
 	\bibitem[Uen75]{Uen75}
 	\bibinfo{author}{K.~Ueno}.
 	\newblock \emph{\bibinfo{title}{Classification theory of algebraic varieties
 			and compact complex spaces. {Notes} written in collaboration with {P}.
 			{Cherenack}}}, vol. \bibinfo{volume}{439} of \emph{\bibinfo{series}{Lect.
 			Notes Math.}}
 	\newblock \bibinfo{publisher}{Springer, Cham} (\bibinfo{year}{1975}).
 	
 	\bibitem[Yam04]{Yam04}
 	\bibinfo{author}{K.~Yamanoi}.
 	\newblock \enquote{\bibinfo{title}{Algebro-geometric version of {Nevanlinna}'s
 			lemma on logarithmic derivative and applications}}.
 	\newblock \emph{\bibinfo{journal}{Nagoya Math. J.}},
 	\textbf{\bibinfo{volume}{173}(\bibinfo{year}{2004})}:\bibinfo{pages}{23--63}.
 	\newblock \urlprefix\url{http://dx.doi.org/10.1017/S0027763000008710}.
 	
 	\bibitem[Yam06]{Yam06}
 	---{}---{}---.
 	\newblock \enquote{\bibinfo{title}{On the truncated small function theorem in
 			{Nevanlinna} theory}}.
 	\newblock \emph{\bibinfo{journal}{Int. J. Math.}},
 	\textbf{\bibinfo{volume}{17}(\bibinfo{year}{2006})(\bibinfo{number}{4})}:\bibinfo{pages}{417--440}.
 	\newblock \urlprefix\url{http://dx.doi.org/10.1142/S0129167X06003540}.
 	
 	\bibitem[Yam10]{Yam10}
 	---{}---{}---.
 	\newblock \enquote{\bibinfo{title}{On fundamental groups of algebraic varieties
 			and value distribution theory}}.
 	\newblock \emph{\bibinfo{journal}{Ann. Inst. Fourier}},
 	\textbf{\bibinfo{volume}{60}(\bibinfo{year}{2010})(\bibinfo{number}{2})}:\bibinfo{pages}{551--563}.
 	\newblock \urlprefix\url{http://dx.doi.org/10.5802/aif.2532}.
 	
 	\bibitem[Yam15]{Yam15}
 	---{}---{}---.
 	\newblock \enquote{\bibinfo{title}{Holomorphic curves in algebraic varieties of
 			maximal albanese dimension}}.
 	\newblock \emph{\bibinfo{journal}{Int. J. Math.}},
 	\textbf{\bibinfo{volume}{26}(\bibinfo{year}{2015})(\bibinfo{number}{6})}:\bibinfo{pages}{45}.
 	\newblock \bibinfo{note}{Id/No 1541006},
 	\urlprefix\url{http://dx.doi.org/10.1142/S0129167X15410062}.
 	
 	\bibitem[Yam23]{Y22}
 	---{}---{}---.
 	\newblock \enquote{\bibinfo{title}{{Bloch's principle for holomorphic maps into
 				subvarieties of semi-abelian varieties}}}.
 	\newblock \emph{\bibinfo{journal}{arXiv e-prints}},
 	\textbf{(\bibinfo{year}{2023})}:\bibinfo{eid}{arXiv:2304.05715}.
 	\newblock \eprint{2304.05715},
 	\urlprefix\url{http://dx.doi.org/10.48550/arXiv.2304.05715}.
 	
 	\bibitem[Zuo96]{Zuo96}
 	\bibinfo{author}{K.~Zuo}.
 	\newblock \enquote{\bibinfo{title}{Kodaira dimension and {Chern} hyperbolicity
 			of the {Shafarevich} maps for representations of {{\(\pi_ 1\)}} of compact
 			{K{\"a}hler} manifolds}}.
 	\newblock \emph{\bibinfo{journal}{J. Reine Angew. Math.}},
 	\textbf{\bibinfo{volume}{472}(\bibinfo{year}{1996})}:\bibinfo{pages}{139--156}.
 	\newblock \urlprefix\url{http://dx.doi.org/10.1515/crll.1996.472.139}.
 	
 \end{thebibliography}

\end{document}